\documentclass[12pt]{article}
\usepackage{amssymb,amsmath,amsthm,ascmac}
\usepackage{enumerate}
\usepackage{array}
\newtheorem{thm}{Theorem}

\newtheorem{lem}{Lemma}[section]
\newtheorem{pro}[lem]{Proposition}
\newtheorem{rem}[lem]{Remark}

\newcommand{\dis}{\displaystyle}

\newcommand{\R}{{\Bbb R}}
\newcommand{\N}{{\Bbb N}}

\newcommand{\pa}{\partial}
\newcommand{\Ker}{\operatorname{Ker}}

\makeatletter

\@addtoreset{equation}{section}
\makeatother

\usepackage{fancyhdr}
\title{\Large\sf Dynamics near the ground state for the Sobolev critical Fujita type heat equation in 6D}
\author{Junichi Harada\\[4mm]
{\normalsize \it Dedicated to the memory of Professor Marek Fila}}
\date{}

\begin{document}
\maketitle
\thispagestyle{empty}
\large
%%%%%%%%%%%%%%%%%%%%%%%%%%%%%%%%%%%%%%%%%%%%%%%%%%%%%%%%%%%%%%%%

%%%%%%%%%%%%%%%%%%%%%%%%%%%%%%%%%%%%%%%%%%%%%%%%%%%%%%%%%%%%%%%%
 \begin{abstract}
 This paper investigates the asymptotic behavior of solutions to
 $u_t=\Delta u+|u|^{p-1}u$ in the Sobolev critical case.
 Our main result is a classification of the dynamics near the ground states in the six dimensional case.
% It is shown that solutions starting near the ground states fall into one of the following three scenarios:
 It is shown that
 if the initial data $u_0\in H^1(\R^6)$ satisfies $\|u_0-{\sf Q}\|_{\dot H^1(\R^6)}\ll1$,
 then the solution falls into one of the following three scenarios:
 \begin{enumerate}[1)]
 \item It is globally defined and converge to one of the ground states as $t\to\infty$.
 \item It is globally defined and converge to $0$ in $\dot H^1(\R^6)$ as $t\to\infty$.
 \item It exhibits finite time blowup with a type I rate.
 \end{enumerate}
 This paper extends the classification result in the case $n\geq7$, previously obtained by Collot-Merle-Rapha\"el,
 to the borderline case $n=6$.
 \end{abstract}
%%%%%%%%%%%%%%%%%%%%%%%%%%%%%%%%%%%%%%%%%%%%%%%%%%%%%%%%%%%%%%%%

%%%%%%%%%%%%%%%%%%%%%%%%%%%%%%%%%%%%%%%%%%%%%%%%%%%%%%%%%%%%%%%%
 \noindent
 {\normalsize{\bf Keyword}: Fujita type heat equation; Sobolev critical case; dynamics near the ground states}
%%%%%%%%%%%%%%%%%%%%%%%%%%%%%%%%%%%%%%%%%%%%%%%%%%%%%%%%%%%%%%%
 \tableofcontents

% The typical behaviors of global in time solutions to \eqref{equation_1.1} are listed as follows.
% \begin{enumerate}[(i)]
% \item
% a solution $u(x,t)$ converges to $0$ as $t\to\infty$ in some sense,
% \item
% a solution $u(x,t)$ converges to one of non zero stationary solutions as $t\to\infty$,
% \item
% a solution $u(x,t)$ asymptotically behaves like  as $t\to\infty$,
% \end{enumerate}
%%%%%%%%%%%%%%%%%%%%%%%%%%%%%%%%%%%%%%%%%%%%%%%%%%%%%%%%%%%%%%%%

%%%%%%%%%%%%%%%%%%%%%%%%%%%%%%%%%%%%%%%%%%%%%%%%%%%%%%%%%%%%%%%%
% Our motivation of this paper is to understand the dynamics of solutions to \eqref{equation_1.1}.

 \section{Introduction}
 \label{sec_1}
%%%%%%%%%%%%%%%%%%%%%%%%%%%%%%%%%%%%%%%%%%%%%%%%%%%%%%%%%%%%%%%%
 We consider the Fujita type heat equation.
 \begin{equation}
 \label{equation_1.1}
 \begin{cases}
 u_t = \Delta u+|u|^{p-1}u
 &
 \text{for } x\in\R^n,\ t\in(0,T),\\
 u|_{t=0}=u_0(x)
 &
 \text{for }
 x\in\R^n
 \end{cases}
 \end{equation}
 with the Sobolev critical case:
 \[
 p=\tfrac{n+2}{n-2}
 \quad\
 (n\geq3).
 \]
 For any initial data $u_0\in\dot H^1(\R^n)$,
 there exists a maximal existence time $T\in(0,\infty]$,
 and
 a unique solution $u(x,t)$ of \eqref{equation_1.1} satisfying
 \[
 u(t)\in C([0,T);\dot H^1(\R^n))\cap C((0,T);L^\infty(\R^n)).
 \]
 When $T$ is finite (which is equivalent to $\limsup_{t\to T}\|u(t)\|_\infty=\infty$),
 we say that the solution $u(x,t)$ blows up in finite time $T$.
 Our interest lies in investigating the global dynamics of solutions to \eqref{equation_1.1}.
 It is known that these dynamics  are highly intricate in the Sobolev critical case.
 In particular,
 these dynamics strongly depend on the dimension $n$.
 To begin with, we first focus on the higher dimensional case.
 Let ${\sf Q}(x)$ denote the Aubin-Talenti solution given by
 \begin{align}
 \label{EQUATION_1.2}
 {\sf Q}(x)
 =
 \left( 1+\frac{|x|^2}{n(n-2)} \right)^{-\frac{n-2}{2}}.
 \end{align}
 Notably,
 the following classification theorem for $n\geq7$ was established by Collot, Merle and Rapha\"el in \cite{Collot-Merle-Raphael}.
%%%%%%%%%%%%%%%%%%%%%%%%%%%%%%%%%%%%%%%%%%%%%%%%%%%%%%%%%%%%%%%%
 \begin{thm}[Theorem 1.1 in {\rm\cite{Collot-Merle-Raphael} p.\,217 - 218}]
 \label{THEOREM_1}
 Let $n\geq7$ and $p=\frac{n+2}{n-2}$.
 There exists a positive constant $\eta$ such that
 if the initial data $u_0(x)\in\dot H^1(\R^n)$ satisfies
 \begin{align}
 \label{EQUATION_1.3}
 \|u_0(x)-{\sf Q}(x)\|_{\dot H^1(\R^n)}<\eta,
 \end{align}
 then the corresponding solution
 \[
 u(x,t)\in
 C([0,T);\dot H^1(\R^n))
 \cap
 C((0,T);\dot H^2(\R^n))
 \]
 to \eqref{equation_1.1} follows one of the three regimes.
 \begin{enumerate}[\rm 1.]
 \item
 {\rm(Soliton)}
 The solution is global and asymptotically attracted by a solitary wave.
 Namely
 there exist $\lambda_\infty>0$ and $z_\infty\in\R^n$ such that
 \begin{align*}
 \lim_{t\to\infty}
 \left\|
 u(x,t)
 -
 {\sf Q}\left(\frac{x-z_\infty}{\lambda_\infty}\right)
 \quad\
 \right\|_{\dot H_x^1(\R^n)}
 =0.
 \end{align*}
 Moreover,
 $|\lambda_\infty-1|+|z_\infty|\to0$
 as $\eta\to0$.
 \item
 {\rm(Dissipation)}
 The solution is global and dissipates
 \begin{align*}
 \lim_{t\to\infty}
 (\|u(t)\|_{\dot H^1(\R^n)}+\|u(t)\|_{L^\infty(\R^n)})
 =0.
 \end{align*}
 \item
 {\rm(Type I blowup)}
 The solution blows up in finite time $T>0$ in the ODE type I selfsimilar blowup regime near the singularity
 \begin{align*}
 \lim_{t\to T}
 (T-t)^\frac{1}{p-1}
 \|u(t)\|_{L^\infty(\R^n)}
 =
 (p-1)^{-\frac{1}{p-1}}.
 \end{align*}
 \end{enumerate}
 \end{thm}
%%%%%%%%%%%%%%%%%%%%%%%%%%%%%%%%%%%%%%%%%%%%%%%%%%%%%%%%%%%%%%%%
 Let $\mathcal{M}$ be a set of all positive solutions of $-\Delta u=u^p$ on $\R^n$,
 which is written as
 \begin{align}
 \label{EQUATION_1.4}
 \mathcal{M}
 =
 \{\tfrac{1}{\lambda^\frac{n-2}{2}}{\sf Q}(\tfrac{x-\sigma}{\lambda})\in \dot H^1(\R^n);\lambda>0,\ \sigma\in\R^n\}.
 \end{align}
 In the case $p=\frac{n+2}{n-2}$,
 the space $\dot H^1(\R^n)$ is the scaling critical Sobolev space.
 Therefore
 Theorem \ref{THEOREM_1} remains true
 if the assumption \eqref{EQUATION_1.3} is replaced by
 \begin{align*}
 \inf_{v\in\mathcal{M}}\|u_0-v\|_{\dot H^1(\R^n)}
 <
 \eta.
 \end{align*}
 As indicated in Theorem \ref{THEOREM_1},
 the dynamics in the higher dimensional case are relatively simple.
 In contrast,
 the dynamics in lower dimensional cases are considerably more complex.
 \\[2mm]
%%%%%%%%%%%%%%%%%%%%%%%%%%%%%%%%%%%%%%%%%%%%%%%%%%%%%%%%%%%%%%%%
 \noindent
 {\bf Type II blowup in the case $n=3,4,5,6$}

 In lower dimensions,
 solutions exhibiting various peculiar behaviors have been observed.
 In fact,
 for the case $3\leq n\leq 6$,
 there exist finite time blowup solutions satisfying
 \begin{align*}
 \lim_{t\to T}
 (T-t)^\frac{1}{p-1}
 \|u(t)\|_\infty
 =
 \infty
 \quad\
 \text{(type II)}.
 \end{align*}
 We say that the blowup is of type I if $\limsup_{t\to T}(T-t)^\frac{1}{p-1}\|u(t)\|_\infty<\infty$,
 and type II if it is not of type I.
 In the pioneering work \cite{Filippas-Herrero-Velazquez},
 Filippas-Herrero-Vel\'azquez
 predicted the existence of type II blowup solutions of \eqref{equation_1.1} when $p=\frac{n+2}{n-2}$,
 and gave the following list of blowup rates of their solutions.
 \begin{align}
 \label{EQUATION_1.5}
 \|u(t)\|_\infty
 &\sim
 \begin{cases}
 (T-t)^{-k} & n=3,
 \\
 (T-t)^{-k}|\log(T-t)|^\frac{2k}{2k-1} & n=4,
 \\
 (T-t)^{-k} & n=5,
 \\
 (T-t)^{-\frac{1}{4}}|\log(T-t)|^{-\frac{15}{8}} & n=6
 \end{cases}
 \qquad
 (k\in\N).
 \end{align}
 It is known that  the list \eqref{EQUATION_1.5} contains minor errors,
 which are pointed out in \cite{Harada_dim5}.
%%%%%%%%%%%%%%%%%%%%%%%%%%%%%%%%%%%%%%%%%%%%%%%%%%%%%%%%%%%%%%%%
 The first rigorous proof for the existence of type II blowup solutions was given by Schweyer \cite{Schweyer}
 in dimension $n=4$.
 He constructed a type II blowup solution with $k=1$ based on a robust energy method.
 After his work,
 del Pino-Musso-Wei
 developed a new method, called the inner-outer gluing method
 to deal with type II blowup solutions for the critical parabolic equations
 (see \cite{Cortazar-del Pino-Musso}, \cite{Davila-del Pino-Wei}).
 Employing the gluing method, 
 del Pino, Musso, Wei, Zhang and Zhou succeeded in constructing most of the solutions listed in \eqref{EQUATION_1.5}.
 \begin{itemize}
 \item $n=3$ ($k\in\N$) \ by del Pino-Musso-Wei-Zhang-Zhou in \cite{delPino-Musso-Wei-Qidi-Yifu_dim3},
 \item $n=4$ ($k=1$) \ by del Pino-Musso-Wei-Zhou in \cite{delPino-Musso-Wei-Yifu_dim4}
 \item $n=5$ ($k=1$) \ by del Pino-Musso-Wei in \cite{delPino-Musso-Wei_dim5}
 \end{itemize}
 Motivated by \cite{delPino-Musso-Wei_dim5},
 the author constructed solutions for several cases not addressed above, using the gluing method.
 \begin{itemize}
 \item $n=5$ ($k\geq2$) \ by Harada in \cite{Harada_dim5}
 \item $n=6$ \ by Harada in \cite{Harada_dim6}
 \end{itemize}
 The final remaining case was settled by Li, Sun and Wang.
 \begin{itemize}
 \item $n=4$ ($k\geq2$) \ by Li-Sun-Wang in \cite{Li-Sun-Wang}
 \end{itemize}
 Furthermore
 the precise asymptotic behaviors of all the blowup solutions listed in \eqref{EQUATION_1.5}
 were derived in \cite{Filippas-Herrero-Velazquez},
 and these solutions are given by
 \begin{align}
 \label{EQUATION_1.6}
 u(x,t)
 =
 \frac{1}{\lambda(t)^{\frac{n-2}{2}}}
 {\sf Q}\left(\frac{x}{\lambda(t)}\right)
 +
 \text{g}(x,t)
 \quad\
 \text{with }
 \end{align}
 \begin{enumerate}[({g}1)]
 \item 
 $\dis\lim_{t\to T}\lambda(t)^\frac{n-2}{2}\|g(x,t)\|_{L_x^\infty(\R^n)}=0$ \
 for $n=\{3,4,5,6\}$,
 and
 \item
 define the energy by
 \begin{align}
 \label{equation_1.7}
 E[u]
 =
 \frac{1}{2}
 \int_{\R^n}
 |\nabla u(x)|^2
 dx
 -
 \frac{1}{p+1}
 \int_{\R^n}
 |u|^{p+1}
 dx.
 \end{align}
 Then
 we have
 \begin{align}
 \label{equation_1.8}
 \lim_{t\to T}E[u(t)]=
 \begin{cases}
 E[{\sf Q}]+e_\infty & 
 \text{for } n=3,4,5,\\
 -\infty
 & \text{for } n=6,
 \end{cases}
 \end{align}
 where $e_\infty$ is a certain constant satisfying $|e_\infty|\ll E[{\sf Q}]$.
 \end{enumerate}
 As shown in \eqref{equation_1.8},
 the behavior of the solution in $n=6$ is quite different from other cases $n=3$, $4$ and $5$
 among type II blowup solutions listed in \eqref{EQUATION_1.5}.
%%%%%%%%%%%%%%%%%%%%%%%%%%%%%%%%%%%%%%%%%%%%%%%%%%%%%%%%%%%%%%%%
 \\[2mm]
 \noindent
 {\bf Long time behavior}

 In \cite{Fila-King},
 Fila and King studied large time behavior of solutions to \eqref{equation_1.1} with $p=\frac{n+2}{n-2}$.
 For any nonnegative, smooth function $\phi(x)$ with $\phi(x)\not\equiv0$,
 consider initial data of the form
 \begin{align}
 \label{equation_1.9}
 u_0(x)
 =
 \alpha
 \phi(x),
 \end{align}
 where $\alpha$ is a positive parameter.
 Define
 \begin{align*}
 \alpha_\phi^*
 =
 \sup
 \{
 \alpha>0;\
 &\text{the solution } u(x,t) \text{ of \eqref{equation_1.1}  starting from \eqref{equation_1.9}}
 \\
 &
 \text{is globally defined}
 \}.
 \end{align*}
 They predicted that
 if $\phi(x)\geq0$ satisfies
 \begin{align}
 \label{equation_1.10}
 \lim_{|x|\to\infty}
 |x|^\gamma
 \phi(x)
 =
 A
 \quad\
 \text{for some }
 A>0,
 \end{align}
 then
 the solution $u(x,t)$ starting from $u_0(x)=\alpha_\phi^*\phi(x)$
 is globally defined,
 and
 asymptotically behaves like
 \begin{align*}
 \lim_{t\to\infty}
 \frac{\|u(t)\|_\infty}{\varphi(t;n,\gamma)}
 =
 C
 \quad\
 \text{for some }
 C>0,
 \end{align*}
 where $\varphi(t;n,\gamma)$ represents the decay and growth rates of the solution,
 as shown in Table \ref{Table}.
%%%%%%%%%%%%%%%%%%%%%%%%%%%%%%%%%%%%%%%%%%%%%%%%%%%%%%%%%%%%%%%%
 \begin{table}[h]
 \centering
 \renewcommand{\arraystretch}{1.5}
 \begin{tabular}{p{2.5em}|p{6.5em}|p{7.5em}|p{4em}|c}
 & \centering $\frac{1}{2}<\gamma<2$ & \centering $\gamma=2$ & \centering $\gamma>2$ &
 \\ \hline
 \centering $n=3$ & \centering $\varphi=t^\frac{\gamma-1}{2}$ & \centering $\varphi=t^{-\frac{1}{2}}(\log t)^{-1}$
 & \centering $\varphi=t^{-\frac{1}{2}}$ &
 \\ \hline
 \end{tabular}
 \\[4mm]
 \begin{tabular}{p{2.5em}|p{6.5em}|p{7.5em}|p{4em}|c}
 &  \centering $1<\gamma<2$ &  \centering $\gamma=2$ & \centering $\gamma>2$ &
 \\ \hline
 \centering $n=4$ &  \centering $\varphi=t^{-\frac{2-\gamma}{2}}\log t$ & \centering $\varphi=1$ &
 \centering $\varphi=\log t$ &
 \\ \hline
 \end{tabular}
 \\[4mm]
 \begin{tabular}{p{2.5em}|p{6.5em}|p{7.5em}|p{4em}|c}
 & \centering $\frac{3}{2}<\gamma<2$ & \centering $\gamma=2$ & \centering $\gamma>2$ &
 \\ \hline
 \centering $n=5$ & \centering $\varphi=t^{-\frac{3(2-\gamma)}{2}}$ & \centering $\varphi=(\log t)^{-3}$
 & \centering $\varphi=1$ &
 \\ \hline
 \end{tabular}
 \\[4mm]
 \begin{tabular}{p{2.5em}|p{6.5em}|p{7.5em}|p{4em}|c}
 & \centering $\gamma>\frac{n-2}{2}$ & & & \\ \hline
 \centering $n\geq6$ & \centering $\varphi=1$ & & & \\ \hline
 \end{tabular}
 \caption{Fila-King conjecture \cite{Fila-King}}
 \label{Table}
 \end{table}
% \newpage

 \noindent
 Most of cases in Table \ref{Table} have been affirmatively solved by
 Del Pino, Musso, Wei, Li, Zhang and Zhou.
 They constructed solutions which behave as described in Table \ref{Table}
 using the gluing method.
 \begin{itemize}
 \item $n=3$ ($\gamma>1$) \ by del Pino-Musso-Wei in \cite{delPino-Musso-Wei_dim3_long}
 \item $n=4$ ($\gamma>2$) \ by Wei-Zhang-Zhou in \cite{Wei-Qidi-Yifu}
 \item $n=4$ ($1<\gamma\leq2$) \ by Wei-Zhou in \cite{Wei-Yifu}
 \item $n=5$ ($\gamma>\frac{3}{2}$) \ by Li-Wei-Zhang-Zhou in \cite{Li-Wei-Qidi-Yifu}
 \item $n=6$ ($\gamma>2$) \ by Wei-Zhou in \cite{Wei-Yifu}
 \end{itemize}
 In \cite{Wei-Yifu},
 Wei-Zhou did not include the full proofs for the remaining cases $n=3$ with $\frac{1}{2}<\gamma\leq1$ and
 $n\geq7$ with $\gamma>\frac{n-2}{2}$, but they provided remarks on these cases.
 All solutions described in Table \ref{Table} satisfy
 \begin{align*}
 \lim_{t\to\infty}
 t^{\frac{1}{p-1}}\|u(t)\|_\infty
 =
 \infty.
 \end{align*}
 Therefore
 these solutions are not selfsimilar.
 The asymptotic forms of these solutions are described by the rescaling of ${\sf Q}(x)$,
 as in the case of type II blowup,
 and are expressed as \eqref{EQUATION_1.6} with
 \begin{enumerate}[(g1)]
 \setcounter{enumi}{2}
 \item
 $\dis\lim_{t\to\infty}\lambda(t)^\frac{n-2}{2}\|g(x,t)\|_{L_x^\infty(\R^n)}=0$ and
 \item
 $\dis
 \lim_{t\to\infty}E[u(t)]
 =
 E[{\sf Q}]$.
 \end{enumerate}
%%%%%%%%%%%%%%%%%%%%%%%%%%%%%%%%%%%%%%%%%%%%%%%%%%%%%%%%%%%%%%%%
 {\bf Classification in the case $n=6$}

 The type II blowup described in \eqref{EQUATION_1.5} and the long time behavior presented in Table \ref{Table}
 provide the counterexamples to Theorem \ref{THEOREM_1} for $n=3$, $4$ and $5$.
 Therefore
 the same type of classification result as in Theorem \ref{THEOREM_1} cannot be expected in these dimensions.
 However,
 the case $n=6$ is particularly delicate.
 According to the asymptotic formula \eqref{equation_1.8},
 the solution in \eqref{EQUATION_1.5} for $n=6$
 seems unlikely to satisfy the assumption \eqref{EQUATION_1.3} in Theorem \ref{THEOREM_1}.
 In other words, the solution for $n=6$ might not be a counterexample.
 Indeed,
 our main result presented below shows that the same classification as in Theorem \ref{THEOREM_1} still holds for the case $n=6$.
%%%%%%%%%%%%%%%%%%%%%%%%%%%%%%%%%%%%%%%%%%%%%%%%%%%%%%%%%%%%%%%%
%%%%%%%%%%%%%%%%%%%%%%%%%%%%%%%%%%%%%%%%%%%%%%%%%%%%%%%%%%%%%%%%
 \begin{thm}
 \label{THEOREM_2}
 Let $n=6$ and $p=\frac{n+2}{n-2}$.
 There exists a positive constant $\eta$ such that
 if the initial data $u_0(x)\in H^1(\R^n)$ satisfies
 \[
 \|u_0(x)-{\sf Q}(x)\|_{\dot H^1(\R^n)}<\eta,
 \]
 then the corresponding solution $u(x,t)\in C([0,T);H^1(\R^n))$
 to \eqref{equation_1.1} follows one of the three regimes.
 \begin{enumerate}[\rm 1)]
 \item
 There exist $\lambda_\infty>0$ and $z_\infty\in\R^n$ such that
 \begin{align*}
 \lim_{t\to\infty}
 \left\|
 u(x,t)
 -
 {\sf Q}\left(\frac{x-z_\infty}{\lambda_\infty}\right)
 \right\|_{\dot H_x^1(\R^n)}
 =0.
 \end{align*}
 \item
 The solution is global and satisfies
 \begin{align*}
 \lim_{t\to\infty}
 (\|u(t)\|_{\dot H^1(\R^n)}+\|u(t)\|_{L^\infty(\R^n)})
 =0.
 \end{align*}
 \item
 The solution blows up in finite time $T>0$
 and satisfies
 \begin{align*}
 \sup_{t\in(0,T)}
 (T-t)^\frac{1}{p-1}
 \|u(t)\|_{L^\infty(\R^n)}
 <
 \infty.
 \end{align*}
 \end{enumerate}
 \end{thm}
%%%%%%%%%%%%%%%%%%%%%%%%%%%%%%%%%%%%%%%%%%%%%%%%%%%%%%%%%%%%%%%%

%%%%%%%%%%%%%%%%%%%%%%%%%%%%%%%%%%%%%%%%%%%%%%%%%%%%%%%%%%%%%%%%
 \begin{rem}
 In this theorem,
 a stronger condition on the initial data is imposed compared to Theorem {\rm\ref{THEOREM_1}},
 namely $u_0\in L^2(\R^n)$.
 Nevertheless,
 the smallness of $\|u_0(x)-{\sf Q}(x)\|_{L^2(\R^n)}$ is not required.
 \end{rem}
%%%%%%%%%%%%%%%%%%%%%%%%%%%%%%%%%%%%%%%%%%%%%%%%%%%%%%%%%%%%%%%%
 \noindent
 {\bf Strategy and difficulties for the case $n=6$}

 The strategy of our proof is based on the method presented in {\rm\cite{Collot-Merle-Raphael}}.
 Let $\mathcal{M}$ be the set give in \eqref{EQUATION_1.4},
 and define
 \begin{align*}
 \text{dist}_{\dot H^1(\R^n)}
 (u,\mathcal M)
 =
 \inf_{v\in\mathcal M}
 \|u-v\|_{\dot H^1(\R^n)}
 \end{align*}
 and
 \begin{align*}
 \mathcal{U}_{\eta}
 =
 \{u(x)\in\dot H^1(\R^n);\text{\rm dist}_{\dot H^1(\R^n)}(u(x),\mathcal{M})<\eta\}.
 \end{align*}
 We fix parameters $\eta$ and $\bar\eta$ such that $0<\eta\ll\bar\eta\ll1$,
 and assume that
 \begin{align*}
 u_0\in\mathcal{U}_\eta.
 \end{align*}
 Let $T$ be the maximal existence time of $u(x,t)$.
 There are four possible cases for the dynamics of solution $u(x,t)$.
 \vspace{2mm}
 \begin{enumerate}[(c1)]
 \item
 \label{(c1)}
 $T<\infty$ and $\text{dist}_{\dot H^1(\R^n)}(u(t),\mathcal{M})<\bar\eta$ for $t\in(0,T)$,

 \item
 \label{(c2)}
 $T<\infty$ and there exists $T_1\in(0,T)$ such that $\text{dist}_{\dot H^1(\R^n)}(u(T_1),\mathcal{M})=\bar\eta$,
 
 \item
 \label{(c3)}
 $T=\infty$
 and $\text{dist}_{\dot H^1(\R^n)}(u(t),\mathcal{M})<\bar\eta$ for $t\in(0,\infty)$,

 \item
 \label{(c4)}
 $T=\infty$
 and there exists $T_1\in(0,T)$ such that $\text{dist}_{\dot H^1(\R^n)}(u(T_1),\mathcal{M})=\bar\eta$.
 \end{enumerate}
 \vspace{2mm}
%%%%%%%%%%%%%%%%%%%%%%%%%%%%%%%%%%%%%%%%%%%%%%%%%%%%%%%%%%%%%%%%
 The asymptotic behavior of the solution in both cases,
 (c2) and (c4),
 can be fully determined in the same way as in \cite{Collot-Merle-Raphael},
 with only minor modifications.
 As a result, (c2) corresponds to 3) in Theorem \ref{THEOREM_2},
 while (c4) corresponds to 2) in Theorem \ref{THEOREM_2}.
 The main contribution of our work lies in the analysis of (c1) and (c3).
 Let $\mathcal{Y}(y)$ denote the eigenfunction corresponding to the unstable eigenvalue of the linearized operator
 $H_y=\Delta_y+p{\sf Q}(y)^{p-1}$ on $\R^n$
 (see Proposition \ref{PROPOSITION_2.4} on p.\pageref{PROPOSITION_2.4} for details).
 In both cases (c1) and (c3),
 a solution $u(x,t)$ can be decomposed as
 \begin{align*}
 u(x,t)
 =
 \lambda(t)^{-\frac{n-2}{2}}
 \{
 {\sf Q}(y)
 +
 a(t)
 \mathcal{Y}(y)
 +
 \epsilon(y,t)
 \}
 \end{align*}
 with
 $y=\frac{x-z(t)}{\lambda(t)}$
 for $t\in(0,T)$,
 where $\epsilon(y,s)$ represents a remainder satisfying certain orthogonal conditions.
 To determine the dynamics of the solution,
 we analyze the coupled system of differential equations for
 $\lambda(t)$,
 $a(t)$,
 $z(t)$ and
 $\epsilon(y,t)$.
 In this step,
 it is necessary to handle the term $(\epsilon,\Lambda_y{\sf Q})_2$ to derive the bound for $\lambda(t)$.
 In \cite{Collot-Merle-Raphael}, for the case $n\geq7$,
 it can be treated as
 \begin{align*}
 |(\epsilon,\Lambda_y{\sf Q})_2|
 <
 \|\epsilon\|_\frac{2n}{n-2}
 \|\Lambda_y{\sf Q}\|_\frac{2n}{n+2}.
 \end{align*}
 Here
 they used the fact that
 $\|\Lambda_y{\sf Q}\|_\frac{2n}{n+2}<\infty$ if $n\geq7$.
 The author believes this part to be essential rather than merely technical.
 To overcome this difficulty,
 we additionally introduce a differential equation for $\|\epsilon(t)\|_2$.
 As a result,
 the estimates for $\lambda(t)$, $a(t)$, $z(t)$ and $\epsilon(y,t)$
 form a closed system, allowing us to obtain the desired bounds (see Section \ref{section_3}).
 Here,
 we impose the additional assumption that $u_0\in L^2(\R^n)$.
 Finally,
 it is shown that (c3) corresponds to 1) in Theorem \ref{THEOREM_2},
 while (c1) does not occur.
%%%%%%%%%%%%%%%%%%%%%%%%%%%%%%%%%%%%%%%%%%%%%%%%%%%%%%%%%%%%%%%%
 \ \\[2mm]
 \noindent
 {\bf Bubbling tower in the case $n\geq6$}

 Del pino, Musso and Wei in \cite{delPino-Musso-Wei_bubble} found a new type of global solutions in the case $n\geq6$ such that
 \begin{align}
 \label{equation_1.11}
 \lim_{t\to\infty}
 \|u(t)\|_\infty=\infty.
 \end{align}
 These solutions look like a tower of bubbles.
 \begin{align}
 \label{equation_1.12}
 u(x,t)
 =
 \sum_{j=1}^k
 (-1)^{j-1}
 \lambda_j(t)^{-\frac{n-2}{2}}
 {\sf Q}\left( \frac{x}{\lambda_j(t)} \right)
 +
 g(x,t).
 \end{align}
 They proved in \cite{delPino-Musso-Wei_bubble} that,
 for any $n\geq6$ and $k\geq2$,
 there exists a radially symmetric solution $u(x,t)$ of \eqref{equation_1.1}
 that possesses a profile of the form \eqref{equation_1.12} with the following properties
 (p1) - (p5).
 \begin{enumerate}[(p1)]
 \item Estimate \eqref{equation_1.11} holds,
 \item ${\dis\lim_{t\to\infty}}\frac{\lambda_j(t)}{\lambda_{j+1}(t)}=0$ \ for
 $1\leq j\leq k-1$ \
 and $\dis\lim_{t\to\infty}\lambda_1(t)=1$,
 \item $\dis\lim_{t\to\infty}\lambda_k(t)^\frac{n-2}{2}\|g(x,t)\|_{L_x^\infty(\R^n)}=0$,
 \item $\dis\lim_{t\to\infty}\|g(x,t)\|_{\dot H^1(\R^n)}=0$,
 \item $\dis\lim_{t\to\infty}E[u(t)]=kE[{\sf Q}]$,
 where $E[u]$ is the energy functional defined in \eqref{equation_1.7}.
 \end{enumerate}
 In the case $n\geq6$,
 a single bubble solution
 corresponding to the case $k=1$ with \eqref{equation_1.11}
 is ruled out by Theorem \ref{THEOREM_1} and Theorem \ref{THEOREM_2}.
 Very recently,
 Kim and Merle
 obtained a complete classification of $\dot H^1(\R^n)$ bounded
 radial solutions to \eqref{equation_1.1} in $n\geq7$.
 They proved in \cite{Kim-Merle} that,
 when $n\geq7$,
 any radial solution $u(x,t)$ of \eqref{equation_1.1} satisfing
 \begin{align}
 \label{equation_1.13}
 \dis\sup_{t\in(0,T)}\|u(t)\|_{\dot H^1(\R^n)}<\infty,
 \end{align}
 where $T\in(0,\infty]$ denotes the maximal existence time of $u(x,t)$,
 is globally defined and asymptotically behaves like one of the solutions
 obtained in \cite{delPino-Musso-Wei_bubble} as $t\to\infty$.
%%%%%%%%%%%%%%%%%%%%%%%%%%%%%%%%%%%%%%%%%%%%%%%%%%%%%%%%%%%%%%%%

 \section{Preliminary}
 \label{SECTION_2}
 \subsection{Notations}
 \label{SECTION_2.1}
%%%%%%%%%%%%%%%%%%%%%%%%%%%%%%%%%%%%%%%%%%%%%%%%%%%%%%%%%%%%%%%%
 Throughout this paper,
 we write
 \begin{itemize}
 \item 
 $\R_+=(0,\infty)$,
 \item
 $p=\tfrac{n+2}{n-2}$,
 \item
 $f(u)=|u|^{p-1}u$,
 \item
 $\chi(\xi)\in C^\infty(\R)$ stands for a standard cut off function satisfying
 \begin{align}
 \label{EQUATION_2.1}
 \chi(\xi)=
 \begin{cases}
 1 & \text{if } |\xi|<1,\\
 0 & \text{if } |\xi|>2.
 \end{cases}
 \end{align}
 \end{itemize}
%For any function of $t$,
%we 
% \[
% \dot \lambda=\frac{d\lambda}{dt}
% \]
%%%%%%%%%%%%%%%%%%%%%%%%%%%%%%%%%%%%%%%%%%%%%%%%%%%%%%%%%%%%%%%%
 Let ${\sf Q}(x)$ be the positive solution to $-\Delta_x{\sf Q}={\sf Q}^p$ given by
 \[
 {\sf Q}(x)
 =
 (1+\tfrac{|x|^2}{n(n-2)})^{-\frac{n-2}{2}}.
 \]
%%%%%%%%%%%%%%%%%%%%%%%%%%%%%%%%%%%%%%%%%%%%%%%%%%%%%%%%%%%%%%%%
%%%%%%%%%%%%%%%%%%%%%%%%%%%%%%%%%%%%%%%%%%%%%%%%%%%%%%%%%%%%%%%%

 \subsection{Sobolev spaces}
 \label{SECTION_2.2}
%%%%%%%%%%%%%%%%%%%%%%%%%%%%%%%%%%%%%%%%%%%%%%%%%%%%%%%%%%%%%%%%
 We recall definition of standard Sobolev spaces.
 \begin{align*}
 \dot H^1(\R^n)
 &=
 \{u\in L^\frac{2n}{n-2}(\R^n);D_xu\in L^2(\R^n)\}
 \quad
 (n\geq3),
 \\
 \dot H^2(\R^n)
 &=
 \{u\in L^\frac{2n}{n-4}(\R^n);D_xu\in L^\frac{2n}{n-2}(\R^n), D_x^2u\in L^2(\R^n)\}
 \quad
 (n\geq5),
 \\
 \dot H^3(\R^n)
 &=
 \{u\in L^\frac{2n}{n-6}(\R^n);D_xu\in L^\frac{2n}{n-4}(\R^n), D_x^2u\in L^\frac{2n}{n-2}(\R^n),
 \\
 &\qquad
 D_x^3u\in L^2(\R^n)\}
 \quad
 (n\geq7).
 \end{align*}
%%%%%%%%%%%%%%%%%%%%%%%%%%%%%%%%%%%%%%%%%%%%%%%%%%%%%%%%%%%%%%%%
 We first collect the classical Hardy inequalities.
 \begin{lem}[Hardy inequalities, Theorem 13 in \cite{Davies-Hinz} p.\,521]
 \label{LEMMA_2.1}
 There exist three constants $C_1,C_2$ and $C_3>0$ depending only on $n$ such that
 \begin{align*}
 \int_{\R^n}
 \frac{u(x)^2}{|x|^2}
 dx
 &<
 C_1
 \|\nabla u\|_{L^2(\R^n)}^2
 \quad
 \text{ \rm for }
 u\in \dot H^1(\R^n)
 \quad
 (n\geq3),
 \\
 \int_{\R^n}
 \frac{u(x)^2}{|x|^4}
 dx
 &<
 C_2
 \|\Delta u\|_{L^2(\R^n)}^2
 \quad
 \text{ \rm for }
 u\in \dot H^2(\R^n)
 \quad
 (n\geq5),
 \\
 \int_{\R^n}
 \frac{u(x)^2}{|x|^6}
 dx
 &<
 C_3
 \|\nabla \Delta u\|_{L^2(\R^n)}^2
 \quad
 \text{ \rm for }
 u\in \dot H^3(\R^n)
 \quad
 (n\geq7).
 \end{align*}
 \end{lem}
%%%%%%%%%%%%%%%%%%%%%%%%%%%%%%%%%%%%%%%%%%%%%%%%%%%%%%%%%%%%%%%%

 \subsection{Local solvability of solutions to $u_t=\Delta u+|u|^{p-1}u$ for $p=\frac{n+2}{n-2}$}
 \label{SECTION_2.3}
%%%%%%%%%%%%%%%%%%%%%%%%%%%%%%%%%%%%%%%%%%%%%%%%%%%%%%%%%%%%%%%%
 We recall the solvability results for $u_t=\Delta u+|u|^{p-1}u$ establish in \cite{Weissler,Brezis-Cazenave}.
 Since our setting differs slightly from theirs,
 for completeness, we give a brief proof.
 \begin{pro}[Local solvability of $u_t=\Delta u+|u|^{p-1}u$]
 \label{PROPOSITION_2.2}
 Let $n\geq3$ and $p=\frac{n+2}{n-2}$,
 and let $X$ denote either $\dot H^1(\R^n)$ or $H^1(\R^n)$.
 For any $u_0\in X$,
 there exist $T=T(u_0)>0$ and a unique solution $u(x,t)$ of \eqref{equation_1.1} satisfying {\rm (e1) - (e3)}.
 \begin{enumerate}[\rm(e1)]
 \item $u(t)\in C([0,T);X)$,
 \item $u(x,t)\in C^{2,1}(\R^n\times(0,T))$,
 \item $u(x,t)\in L^\infty(\R^n\times(T_1,T_2))$ { \rm for any} $0<T_1<T_2<T$.
 \end{enumerate}
 Furthermore
 let
 \begin{align*}
 \Theta
 =
 \{(0,\alpha);|\alpha|\leq3\}
 \cup
 \{(1,\alpha);|\alpha|\leq2\}
 \cup
 \{(2,0)\}.
 \end{align*}
 The above solution $u(x,t)$ additionally satisfies {\rm(e4) - (e7)}.
 \begin{enumerate}[\rm(e1)]
 \setcounter{enumi}{3}
 \item
 $\pa_t^kD_x^\alpha u(x,t)\in C(\R^n\times(0,T))$
 { \rm  for}
 $(k,\alpha)\in\Theta$,
 \item
 $\pa_t^kD_x^\alpha u(x,t)\in L^\infty(\R^n\times(T_1,T_2))$
 { \rm for any} $0<T_1<T_2<T$
 {\rm  and}
 $(k,\alpha)\in\Theta$,
 \quad
 \item $u_t(t)\in C((0,T);H^1(\R^n))$,
 \item
 $D_x^{\alpha}u(t)\in C((0,T);L^2(\R^n))$
 { \rm for}
 $2\leq|\alpha|\leq3$.
 \end{enumerate}
 \end{pro}
%%%%%%%%%%%%%%%%%%%%%%%%%%%%%%%%%%%%%%%%%%%%%%%%%%%%%%%%%%%%%%%%
 \begin{proof}
 The existence part can be verified from standard contraction mapping arguments.
 Let $X$ denote either $\dot H^1(\R^n)$ or $H^1(\R^n)$.
 Fix $r\in(\frac{2n}{n-2},\frac{2n}{n-2}\frac{n+2}{n-2})$ and
 set $\alpha=\frac{n}{2}(\frac{n-2}{2n}-\frac{1}{r})>0$.
 For each $(T,M,\delta)\in(\R_+)^3$,
 we set
 \begin{align*}
 K_T(M,\delta)
 =
 \{&
 u(t)\in C([0,T];X)\cap C((0,T];L^r(\R^n));
 \\
 &
 \sup_{t\in[0,T]}\|u(t)\|_{X}\leq M,\
 \lim_{t\to0}t^\alpha\|u(t)\|_{L^r(\R^n)}=0
 \\
 &
 \sup_{t\in[0,T]}t^\alpha\|u(t)\|_{L^r(\R^n)}\leq\delta
 \}.
 \end{align*}
 We define the metric on $K_T=K_T(M,\delta)$ by
 \begin{align*}
 \text{dist}_{K_T}(u,v)
 =
 \sup_{t\in[0,T]}\|u(t)-v(t)\|_{X}
 +
 \sup_{t\in(0,T]}t^\alpha\|u(t)-v(t)\|_{L^r(\R^n)}.
 \end{align*}
 Then
 $(K_T(M,\delta),\text{dist}_{K_T})$ is a nonempty complete metric space for each $(T,M,\delta)\in(\R_+)^3$.
 For given $u_0\in X$,
 we put $M=2\|u_0\|_{X}$
 and define
 \begin{align*}
 \Phi(u)(t)
 =
 T(t)[u_0]
 +
 \int_0^t
 T(t-s)
 [f(u(s))]
 ds
 \end{align*}
 for $u(t)\in K_T(M,\delta)$,
 where $\{T(t)\}_{t>0}$ represents the linear heat semigroup.
 Using the same computations as in Section 4 of \cite{Brezis-Cazenave} p.\,283 - p.\,286,
 we can check that
 $\Phi(u)(t)$ gives a contraction mapping on $(K_T(M,\delta),\text{dist}_{K_T})$
 if $T$ and $\delta$ are sufficiently small.
 This shows the existence part of solutions satisfying (e1) - (e3).
 The uniqueness of solutions satisfying (e1) - (e3) follows from the argument in \cite{Brezis-Cazenave} p.\,286 - p.\,287.
 Put $w(x,t)=\pa_{x_j}u(x,t)$ and fix $t_1,t_2$ such that $0<t_1<t_2<T$.
 Then $w(x,t)$ satisfies
 \begin{align*}
 \begin{cases}
 w_t
 =
 \Delta w+f'(u(x,t))w & \text{for } (x,t)\in\R^n\times(t_1,t_2),
 \\
 w(x,t_1)=u_{x_j}(x,t_1)\in L^2(\R^n) & \text{for } x\in\R^n.
 \end{cases}
 \end{align*}
 Since $f(u)=|u|^{p-1}u$ ($p>1$),
 for any $M>0$,
 there exists $C_{p,M}>0$ such that
 \begin{align*}
 |f'(u_1)-f'(u_2)|
 &<
 C_{p.M}
 \begin{cases}
 |u_1-u_2| & \text{if } p\geq2,\\
 |u_1-u_2|^{p-1} & \text{if } p<2
 \end{cases}
 \\
 &\qquad
 \text{for any }
 u_1,u_2\in\{u\in\R;|u|<M\}.
 \end{align*}
 Hence
 by a standard parabolic estimate,
 we see that
 \begin{itemize}
 \item 
 $w(t)\in C([t_1,t_2];H^1(\R^n))$,
 \item
 $w(x,t)\in C^{2,1}(\R^n\times(t_1,t_2))$.
 \end{itemize}
 This implies $D_x^\alpha u(t)\in C((0,T);L^2(\R^n))$ for all $|\alpha|=2$.
 Since $u_t=\Delta u+f(u)$ for $t\in(0,T)$,
 it follows that $\pa_tu(t)\in C((0,T);L^2(\R^n))$.
 In the same procedure,
 we conclude that
 \begin{itemize}
 \item 
 $\pa_tu(t)\in C([t_1,t_2];H^1(\R^n))$,
 \item
 $\pa_tu(x,t)\in C^{2,1}(\R^n\times(t_1,t_2))$.
 \end{itemize}
 This completes the proof for (e4) - (e7).
 \end{proof}
%%%%%%%%%%%%%%%%%%%%%%%%%%%%%%%%%%%%%%%%%%%%%%%%%%%%%%%%%%%%%%%%

%%%%%%%%%%%%%%%%%%%%%%%%%%%%%%%%%%%%%%%%%%%%%%%%%%%%%%%%%%%%%%%%
 \begin{pro}
 \label{PROPOSITION_2.3}
 Let $n\geq3$ and $p=\frac{n+2}{n-2}$.
 Let $h(x)\in L^\frac{2n}{n-2}(\R^n)$,
 and consider the following equation.
 \begin{align}
 \label{equation_e2.2}
 \begin{cases}
 w_t=\Delta w+|w|^{p-1}w
 & \text{\rm for } (x,t)\in\R^n\times(0,T),
 \\
 w(x,0)={\sf Q}(x)+h(x)
 & \text{\rm for } x\in\R^n.
 \end{cases}
 \end{align}
 Let $T(h)$ be the maximal existence time of a solution of \eqref{equation_e2.2}.
 There exist constants ${\sf h}_1>0$, $(\Delta{\sf t})_1>0$
 and a function ${\sf M}(t)\in C((0,(\Delta{\sf t})_1];\R_+)$
 such that
 if $\|h\|_\frac{2n}{n-2}<{\sf h}_1$,
 then
 it holds that
 \begin{enumerate}[\rm(i)]
 \item $T(h)>(\Delta{\sf t})_1$ and
 \item
 $\|w(x,t)\|_{L_x^\infty(\R^n)}<{\sf M}(t)$
 \ \
 \text{\rm for}
 $t\in(0,(\Delta{\sf t})_1]$.
 \end{enumerate}
 The function ${\sf M}(t)$ is independent of $h(x)$,
 and might satisfy $\lim_{t\to0}{\sf M}(t)=\infty$.
 \end{pro}
%%%%%%%%%%%%%%%%%%%%%%%%%%%%%%%%%%%%%%%%%%%%%%%%%%%%%%%%%%%%%%%%
 \begin{proof}
 Let $W(x,t)=w(x,t)-{\sf Q}(x)$.
 Then
 $W(x,t)$ satisfies
 \begin{align}
 \label{equation_e2.3}
 \begin{cases}
 W_t
 =
 \Delta W
 +
 f({\sf Q}+W)
 -
 f({\sf Q})
 \quad\
 \text{for } (x,t)\in\R^n\times(0,T(h)),\\
 W(x,t)|_{t=0}
 =
 h(x)
 \quad\
 \text{for }
 x\in\R^n.
 \end{cases}
 \end{align}
 By the Duhamel formula,
 we have
 \begin{align*}
 W(t)
 =
 T(t)
 [h]
 +
 \int_0^t
 T(t-s)[
 f({\sf Q}+W(s))
 -
 f({\sf Q})
 ]
 ds,
 \end{align*}
 where $\{T(t)\}_{t>0}$ is the linear heat semigroup.
 Fix $r\in(\frac{2n}{n-2},\frac{2n}{n-2}\frac{n+2}{n-2})$.
 By a standard argument,
 we can check that
 \begin{align}
 \nonumber
 \|&W(t)\|_r
 <
 \|
 T(\tau)
 [h]
 \|_r
 +
 \int_0^\tau
 \|
 T(t-s)[
 f({\sf Q}+W)
 -
 f({\sf Q})
 ]
 \|_r
 ds
 \\
 \nonumber
 &<
 C
 t^{-\frac{n}{2}(\frac{n-2}{2n}-\frac{1}{r})}
 \|
 h
 \|_\frac{2n}{n-2}
 \\
 \label{equation_e2.4}
 &\quad
 +
 C
 \int_0^t
 (t-s)^{-\frac{n}{2}(\frac{p}{r}-\frac{1}{r})}
 \|
 f({\sf Q}+W)
 -
 f({\sf Q})
 \|_\frac{r}{p}
 ds.
 \end{align}
 Here we note that
 \begin{align*}
 \nonumber
 \|
 f({\sf Q}+W)
 &-
 f({\sf Q})
 \|_\frac{r}{p}
 <
 p
 \|
 ({\sf Q}+|W|)^{p-1}W
 \|_\frac{r}{p} 
 \\
 \nonumber
 &<
 p
 \|({\sf Q}+|W|)\|_r^{p-1}
 \|W\|_r
 \\
 &<
 C
 \|{\sf Q}\|_r^{p-1}
 \|W\|_r
 +
 C
 \|W\|_r^p.
 \end{align*}
 Set $\alpha=\frac{n}{2}(\frac{n-2}{2n}-\frac{1}{r})>0$.
 Therefore
 \eqref{equation_e2.4} can be rewritten as
 \begin{align}
 \label{equation_e2.5}
 \|W(t)\|_r
 &<C
 t^{-\alpha}
 \|
 h
 \|_\frac{2n}{n-2}
 +
 C
 \|{\sf Q}\|_r^{p-1}
 \int_0^t
 (t-s)^{-\frac{n}{2}(\frac{p}{r}-\frac{1}{r})}
 \|W\|_r
 ds
 \\
 \nonumber
 &\quad
 +
 C
 \int_0^t
 (t-s)^{-\frac{n}{2}(\frac{p}{r}-\frac{1}{r})}
 \|W\|_r^p
 ds.
 \end{align}
 We now define
 \begin{align*}
 K_\tau
 =
 \sup_{t\in(0,\tau)}
 t^\alpha\|W(t)\|_r.
 \end{align*}
 From \eqref{equation_e2.5},
 it holds that
 \begin{align*}
 \nonumber
 &\|W(t)\|_r
 <C
 t^{-\alpha}
 \|
 h
 \|_\frac{2n}{n-2}
 +
 C
 \|{\sf Q}\|_r^{p-1}
 K_\tau
 \int_0^t
 (t-s)^{-\frac{n}{2}\frac{p-1}{r}}
 s^{-\alpha}
 ds
 \\
 \nonumber
 &\quad
 +
 C
 K_\tau^p
 \int_0^t
 (t-s)^{-\frac{n}{2}\frac{p-1}{r}}
 s^{-\alpha p}
 ds
 \\
 \nonumber
 &=
 C
 t^{-\alpha}
 \|
 h
 \|_\frac{2n}{n-2}
 +
 C
 \|{\sf Q}\|_r^{p-1}
 K_\tau
 t^{-\alpha}
 t^{\alpha(p-1)}
 \int_0^1
 (1-\sigma)^{-\frac{n}{2}\frac{p-1}{r}}
 \sigma^{-\alpha}
 ds
 \\
 \nonumber
 &\quad
 +
 C
 K_\tau^p
 t^{-\alpha}
 \int_0^1
 (1-\sigma)^{-\frac{n}{2}\frac{p-1}{r}}
 \sigma^{-\alpha p}
 d\sigma
 \end{align*}
 for $t\in(0,\tau)$.
 From the choice of $\alpha$ and $r$,
 we note that $\frac{n}{2}\frac{p-1}{r}<1$, $\alpha<1$ and $\alpha p<1$.
 Therefore
 there exists a constant $C_1>0$ depending only on $n$ such that
 \begin{align*}
 &
 t^\alpha
 \|W(t)\|_r
 <
 C_1
 \|
 h
 \|_\frac{2n}{n-2}
 +
 C_1
 K_\tau
 t^{\alpha(p-1)}
 +
 C_1
 K_\tau^p
 \\
 &<
 C_1
 \|
 h
 \|_\frac{2n}{n-2}
 +
 C_1
 K_\tau
 \tau^{\alpha(p-1)}
 +
 C_1
 K_\tau^p
 \quad\
 \text{for }
 t\in(0,\tau).
 \end{align*}
 We put $(\Delta{\sf t})_1=(\frac{1}{2C_1})^\frac{1}{\alpha(p-1)}$.
 Then it holds that
 \begin{align*}
 \tfrac{1}{2}
 K_\tau
 &<
 C_1
 \|
 h
 \|_\frac{2n}{n-2}
 +
 C_1
 K_\tau^p
 \quad\
 \text{for }
 \tau\in(0,\min\{(\Delta{\sf t})_1,T(h)\}).
 \end{align*}
 Let $f(m)=\frac{1}{2}m-C_1m^p$ and denote its maximum value by $f_\text{max}>0$.
 The equation $f(m)=\frac{f_\text{max}}{2}$ has two roots,
 which are denoted by $m_1,m_2$ ($m_1<m_2$).
 We fix ${\sf h}_1>0$ small enough such that
 $C_1{\sf h}_1<\frac{f_\text{max}}{2}$.
 Hence
 if $\|h\|_\frac{2n}{n-2}<{\sf h}_1$,
 then it holds that
 \begin{align*}
 \tfrac{1}{2}
 K_\tau
 -
 C_1
 K_\tau^p
 <
 \tfrac{f_\text{max}}{2}
 \quad\
 \text{for }
 \tau\in(0,\min\{(\Delta{\sf t})_1,T(h)\}).
 \end{align*}
 We recall that $\lim_{t\to0}t^\alpha\|w(t)\|_r=0$ for any $h\in L^\frac{2n}{n-2}(\R^n)$
 (see the definition of $K_T(M,\delta)$ in the proof of Proposition \ref{PROPOSITION_2.2}).
 Hence
 it follows that $\lim_{\tau\to0}K_\tau=0$ for any $h\in L^\frac{2n}{n-2}(\R^n)$.
 Therefore
 since $K_\tau$ is continuous and nondecreasing in $\tau$,
 we conclude that if $\|h\|_\frac{2n}{n-2}<{\sf h}_1$,
 then $T(h)>(\Delta{\sf t})_1$ and
 \[
 K_\tau
 <
 m_1
 \quad\
 \text{for }
 \tau\in(0,(\Delta{\sf t})_1).
 \]
 This implies that
 $\|W(t)\|_r<m_1t^{-\alpha}$ for $t\in(0,(\Delta{\sf t})_1)$.
 We here rewrite \eqref{equation_e2.3} as $W_t=\Delta W+V(x,t)W$ with $V(x,t)=\frac{f({\sf Q}+W)-f({\sf Q})}{W}$.
 Since $|V(x,t)|<C({\sf Q}(x)+|W(x,t)|)^{p-1}$,
 it holds that
 \begin{align*}
 \|V(t)\|_\frac{r}{p-1}
 &<
 C
 (\|{\sf Q}\|_r^{p-1}+\|W(t)\|_r^{p-1})
 \\
 &<
 C
 (\|{\sf Q}\|_r^{p-1}+m_1^{p-1}t^{-\alpha(p-1)})
 \quad\
 \text{for }
 t\in(0,(\Delta{\sf t})_1).
 \end{align*}
 Since $\frac{r}{p-1}>\frac{n}{2}$,
 a local parabolic estimate establishes (ii),
 which completes the proof.
 \end{proof}
%%%%%%%%%%%%%%%%%%%%%%%%%%%%%%%%%%%%%%%%%%%%%%%%%%%%%%%%%%%%%%%%

 \subsection{Spectral properties of $H_y=\Delta_y+f'({\sf Q}(y))$}
 \label{ection_2.3}
%%%%%%%%%%%%%%%%%%%%%%%%%%%%%%%%%%%%%%%%%%%%%%%%%%%%%%%%%%%%%%%%
 Let us consider the eigenvalue problem associated with the linearized problem around
 the ground state ${\sf Q}(y)$.
 \begin{equation}
 \label{equation_e2.6}
 -H_y\Phi=e\Phi \quad\ \text{in } \R^n,
 \end{equation}
 where the operator $H_y$ is defined by
 \[
 H_y=\Delta_y+V(y),
 \quad\
 V(y)=p{\sf Q}(y)^{p-1}.
 \]
 We now summarize the fundamental properties of $H_y$
 stated in Section 2 of \cite{Collot-Merle-Raphael}.
%%%%%%%%%%%%%%%%%%%%%%%%%%%%%%%%%%%%%%%%%%%%%%%%%%%%%%%%%%%%%%%%
 \begin{pro}[Proposition 2.2 in \cite{Collot-Merle-Raphael} p.\,223]
 \label{PROPOSITION_2.4}
 Let $n\geq3$.
 The operator $H_y:H^2(\R^n)\to L^2(\R^n)$ is self-adjoint on $L^2(\R^n)$.
 The eigenvalue problem \eqref{equation_e2.6}
 has a unique negative eigenvalue $e=-e_0$ $(e_0>0)$,
 and
 an associated positive, radially symmetric eigenfunction denoted by $\mathcal{Y}(y)$
 decays exponentially {\rm(see \cite{Cortazar-del Pino-Musso} p.\,18)}.
 \begin{align}
 \label{equation_e2.7}
 \mathcal{Y}(y)
 +
 |D_y^\alpha\mathcal{Y}(y)|
 <
 C( 1+|y| )^{-\frac{n-1}{2}}
 e^{-\sqrt{e_0|y|}}
 \end{align}
 for any multi-index $\alpha$ satisfying $1\leq|\alpha|\leq2$.
 For the case $n\geq5$,
 $\Ker H_y\subset H^2(\R^n)$ is given by
 \begin{align}
 \label{equation_e2.8}
 \Ker H_y
 =
 \text{\rm span}
 \{\Lambda_y{\sf Q}(y),\pa_{y_1}{\sf Q}(y),\cdots,\pa_{y_n}{\sf Q}(y)\}.
 \end{align}
 Furthermore
 we have
 \begin{align*}
 \int_{\R^n}
 (|\nabla_y\varphi(y)|^2-V(y)\varphi(y)^2)
 dy
 &>
 0
 \quad
 \text{\rm for }
 \varphi\in X_{H_y}^\perp
 \quad
 (n\geq5),
 \\
 \int_{\R^n}
 (|\nabla_y\varphi(y)|^2-V(y)\varphi(y)^2)
 dy
 &>
 0
 \quad
 \text{\rm for }
 \varphi\in X_{H_y}^\perp
 \quad
 (n\geq7),
 \end{align*}
 where $X_{H_y}^\perp$ is defined by
 \begin{align*}
 X_{H_y}^\perp
 &=
 \{
 \varphi\in {\dot H^1(\R^n)};
 (\varphi,\mathcal{Y})_2=0
 \text{ \rm and }
 \}
 \\
 & \hspace{30mm}
 (\varphi,v)_2=0
 \text{ \rm for any } v\in\Ker H_y
 \}.
 \end{align*}
 \end{pro}
%%%%%%%%%%%%%%%%%%%%%%%%%%%%%%%%%%%%%%%%%%%%%%%%%%%%%%%%%%%%%%%%
 Following the approach in \cite{Collot-Merle-Raphael} p.\,223 - p.\,224,
 we introduce $n+1$ functions $\Psi_j(y)$ for $j=0,\cdots,n$,
 which are close to $\Lambda_y{\sf Q}(y)$, $\pa_{y_1}{\sf Q}(y)$, $\cdots$, $\pa_{y_n}{\sf Q}(y)$.
 Let $M>1$ be a sufficiently large constant and let
 $\chi_M
 =
 \chi(\tfrac{|y|}{M})$
 (see \eqref{EQUATION_2.1} for the definition of $\chi$).
 We define
 \begin{align}
 \label{equation_e2.9}
 \Psi_0(y)
 &=
 \chi_M
 \Lambda_y{\sf Q}(y)
 -
 (\chi_M\Lambda_y{\sf Q},\tfrac{\mathcal{Y}}{\|\mathcal{Y}\|_{L^2(\R^n)}})_{L^2(\R^n)}
 \tfrac{\mathcal{Y}(y)}{\|\mathcal{Y}\|_{L^2(\R^n)}},
 \\
 \label{equation_e2.10}
 \Psi_j(y)
 &=
 \chi_M
 \pa_{y_j}{\sf Q}(y)
 \qquad
 (j=1,\cdots,n).
 \end{align}
 These functions satisfy the following orthogonality relations.
 \begin{align}
 \label{equation_e2.11}
 \begin{cases}
 (\Psi_j,\mathcal{Y})_{L^2(\R^n)}
 =0 & \text{for } j\in\{0,\cdots,n\},
 \\
 (\Psi_i,\Psi_j)_{L^2(\R^n)}
 =0 & \text{for } i\not=j.
 \end{cases}
 \end{align}
%%%%%%%%%%%%%%%%%%%%%%%%%%%%%%%%%%%%%%%%%%%%%%%%%%%%%%%%%%%%%%%%
 We define
 \begin{align}
 \label{equation_e2.12}
 X_{\Psi}^\perp
 =&
 \{
 \varphi\in L_\text{loc}^1(\R^n);
 \varphi\in L^q(\R^n) \text{ for some } q\geq1,\
 (\varphi,\mathcal{Y})_{L^2(\R^n)}=0
 \\
 \nonumber
 & \quad
 \text{and }
 (\varphi,\Psi_j)_{L^2(\R^n)}=0
 \text{ for all }
 j\in\{0,\cdots,n\}
 \}.
 \end{align}
 Due to the presence of $\chi_M$,
 the operator $H_y$ exhibits coercivity on
 $\dot H^s(\R^n)\cap X_{\Psi}^\perp$ for certain values of $s$.
%%%%%%%%%%%%%%%%%%%%%%%%%%%%%%%%%%%%%%%%%%%%%%%%%%%%%%%%%%%%%%%%
 \begin{lem}[Lemma 2.3 in \cite{Collot-Merle-Raphael} p.\,224]
 \label{LEMMA_2.5}
 There exists a constant $M^*>0$ such that
 for any $M>M^*$,
 there exist constants $\bar C_1>0$ and $\bar C_2>0$ such that
 \begin{align}
 \label{equation_e2.13}
 \bar C_1
 \|\nabla v\|_2^2
 &<
 (-H_yv,v)_{L_y^2(\R^n)}
 <
 \bar C_2
 \|\nabla v\|_2^2
 \\
 \nonumber
 &\qquad\quad
 \text{\rm for } v\in\dot H^1(\R^n)\cap X_{\Psi}^\perp \quad (n\geq3),
 \\
 \label{equation_e2.14}
 \bar C_1
 \|\Delta v\|_2^2
 &<
 \|H_yv\|_{L_y^2(\R^n)}^2
 <
 \bar C_2
 \|\Delta v\|_2^2
 \\
 \nonumber
 &\qquad\quad
 \text{\rm for } v\in\dot H^2(\R^n)\cap X_{\Psi}^\perp \quad (n\geq5),
 \\
 \label{equation_e2.15}
 \bar C_1
 \|\nabla\Delta v\|_2^2
 &<
 (-H_y(H_yv),H_yv)_{L_y^2(\R^n)}
 <
 \bar C_2
 \|\nabla\Delta v\|_2^2
 \\
 \nonumber
 &\qquad\quad
 \text{\rm for } v\in\dot H^3(\R^n)\cap X_{\Psi}^\perp \quad (n\geq7).
 \end{align}
 Both constants $\bar C_1$ and $\bar C_2$ depend only on $n$ and $M$.
 \end{lem}
%%%%%%%%%%%%%%%%%%%%%%%%%%%%%%%%%%%%%%%%%%%%%%%%%%%%%%%%%%%%%%%%

 \subsection{Decomposition of solutions in a neighbourhood of the ground states}
%%%%%%%%%%%%%%%%%%%%%%%%%%%%%%%%%%%%%%%%%%%%%%%%%%%%%%%%%%%%%%%%
 Let $\mathcal{M}$ be the set of all positive solutions to $-\Delta u=u^p$ on $\R^n$,
 which is given by
 \begin{align*}
 \mathcal{M}
 =
 \{\lambda^{-\frac{n-2}{2}}{\sf Q}(\tfrac{x-\sigma}{\lambda})\in \dot H^1(\R^n);\lambda>0,\sigma\in\R^n\}.
 \end{align*}
 Define the distance in $\dot H^1(\R^n)$ from $u$ to $\mathcal M$ as
 \begin{align}
 \label{equation_e2.16}
 \text{dist}_{\dot H^1(\R^n)}(u,\mathcal{M})
 =
 \inf_{\lambda>0,\sigma\in\R^n}
 \|u(x)-\lambda^{-\frac{n-2}{2}}{\sf Q}(\tfrac{x-\sigma}{\lambda})\|_{\dot H^1(\R^n)}.
 \end{align}
%%%%%%%%%%%%%%%%%%%%%%%%%%%%%%%%%%%%%%%%%%%%%%%%%%%%%%%%%%%%%%%%
 We here prepare notations.
 \begin{align*}
 \mathcal{B}_{f_0(y)}(r)
 &=
 \{f(y)\in \dot H^1(\R^n);
 \|f(y)-f_0(y)\|_{\dot H^1(\R^n)}<r\},
 \\
 \Omega_{(1,0,0)}(r)
 &=
 \{(k,z,a)\in\R_+\times\R^n\times\R;
 \\
 &\qquad
 |k-1|^2+|z|^2+|a|^2\|\mathcal{Y}\|_{\dot H^1(\R^n)}^2<r^2\}.
 \end{align*}
%%%%%%%%%%%%%%%%%%%%%%%%%%%%%%%%%%%%%%%%%%%%%%%%%%%%%%%%%%%%%%%%
 \begin{lem}[Unique decomposition in a neighborhood of ${\sf Q}(y)$, Lemma 2.5 in \cite{Collot-Merle-Raphael}
 p.\,225 and p.\,275]
 \label{LEMMA_2.6}
 Let $n\geq3$.
 There exist two constants $\nu_1>0$ and $\nu_2>0$,
 and three smooth functions
 \begin{align*}
 (\mathcal{K}(u),\mathcal{Z}(u),\mathcal{A}(u))
 :
 \mathcal{B}_{{\sf Q}}(\nu_2)
 \to
 \R_+\times\R^n\times\R
 \end{align*}
 satisfying the following {\rm(I) - (II)}.
 \begin{enumerate}[\rm(I)]
 \item
 $|\mathcal{K}(u)-1|^2+|\mathcal{Z}(u)|^2+|\mathcal{A}(u)|^2\|\mathcal{Y}\|_{\dot H^1(\R^n)}^2
 <\nu_1^2$
 \quad
 {\rm for all}
 $u\in\mathcal{B}_{{\sf Q}}(\nu_2)$.
 \item
 For each $u\in\mathcal{B}_{{\sf Q}}(\nu_2)$,
 define
 \begin{align*}
 v(y)
 =
 \mathcal{K}(u)^\frac{n-2}{2}u(\mathcal{K}(u)y+\mathcal{Z}(u))
 -{\sf Q}(y)-\mathcal{A}(u)\mathcal{Y}(y).
 \end{align*}
 Then $v(y)$ satisfies the following orthogonality conditions.
 \begin{align*}
 \begin{cases}
 (
 v(y),\mathcal{Y}(y)
 )_{L_y^2(\R^n)}
 =0,
 \\
 (
 v(y),\Psi_j(y)
 )_{L_y^2(\R^n)}
 =0
 & \text{\rm for } j\in\{0,\cdots,n\}.
 \end{cases}
 \end{align*}
 \end{enumerate}
 Furthermore
 for $(u,k,z,a)\in\mathcal{B}_{{\sf Q}}(\nu_2)\times \Omega_{(1,0,0)}(\nu_1)$,
 the following conditions {\rm(cd1)} and {\rm(cd2)} are equivalent.
 \begin{enumerate}[\rm(cd1)]
 \item The function $v(y)=k^\frac{n-2}{2}u(ky+z)-{\sf Q}(y)-a\mathcal{Y}(y)$ satisfies
 \begin{align*}
 \begin{cases}
 (
 v(y),\mathcal{Y}(y)
 )_{L_y^2(\R^n)}
 =0,
 \\
 (
 v(y),\Psi_j(y)
 )_{L_y^2(\R^n)}
 =0
 & \text{\rm for } j\in\{0,\cdots,n\},
 \end{cases}
 \end{align*}
 \item $(k,z,a)=(\mathcal{K}(u),\mathcal{Z}(u),\mathcal{A}(u))$.
 \end{enumerate}
 \end{lem}
%%%%%%%%%%%%%%%%%%%%%%%%%%%%%%%%%%%%%%%%%%%%%%%%%%%%%%%%%%%%%%%%
 \begin{proof}
 Define the mapping
 ${\bf F}(u,k,z,a):\dot H^1(\R^n)\times\R\times\R^n\times\R\to\R^{n+2}$
 by
 \begin{align*}
 {\bf F}(u,k,z,a)
 =
 \begin{pmatrix}
 F_0(u,k,z,a)
 \\
 \vdots
 \\
 F_{n+1}(u,k,z,a)
 \end{pmatrix},
 \end{align*}
 where
 \begin{align*}
 F_j(u,k,z,a)
 =
 (
 k^\frac{n-2}{2}u(k y+z)
 -
 {\sf Q}(y)
 -
 a\mathcal{Y}(y),
 \Psi_j(y)
 )_{L_y^2(\R^n)}
 \end{align*}
 for
 $j\in\{0,\cdots,n\}$
 and
 \begin{align*}
 F_{n+1}(u,k,z,a)
 =
 (
 k^\frac{n-2}{2}u(k y+z)
 -
 {\sf Q}(y)
 -
 a\mathcal{Y}(y),
 \mathcal{Y}(y)
 )_{L_y^2(\R^n)}.
 \end{align*}
 For simplicity,
 let
 \begin{align*}
 x=k y+z.
 \end{align*}
 We easily see that
 $F_j(u,k,z,a)$ is a $C^1$ function for $j\in\{0,\cdots,n+1\}$
 and
 their derivatives are given by
 \begin{align*}
 ( \pa_uF_j(u,k,z,a), \varphi)_{\dot H^1(\R^n)}
 &=
 (
 k^\frac{n-2}{2}\varphi(x),
 \Psi_j(y)
 )_{L_y^2(\R^n)},
 \\
 ( \pa_uF_{n+1}(u,k,z,a), \varphi)_{\dot H^1(\R^n)}
 &=
 (
 k^\frac{n-2}{2}\varphi(x),
 \mathcal{Y}(y)
 )_{L_y^2(\R^n)}
 \end{align*}
 and
 \begin{align*}
 \tfrac{\pa F_j}{\pa k}(u,k,z,a)
 &=
 (
 \tfrac{n-2}{2}k^\frac{n-4}{2}u(x)+k^{\frac{n-2}{2}}y\cdot(\nabla_xu)(x),
 \Psi_j(y)
 )_{L_y^2(\R^n)},
 \\
 \tfrac{\pa F_{n+1}}{\pa k}(u,k,z,a)
 &=
 (
 \tfrac{n-2}{2}k^\frac{n-4}{2}u(x)+k^{\frac{n-2}{2}}y\cdot(\nabla_xu)(x),
 \mathcal{Y}(y)
 )_{L_y^2(\R^n)},
 \\
 \tfrac{\pa F_j}{\pa z_l}(u,k,z,a)
 &=
 (
 k^{\frac{n-2}{2}}u_{x_l}(x),
 \Psi_j(y)
 )_{L_y^2(\R^n)},
 \\
 \tfrac{\pa F_{n+1}}{\pa z_l}(u,k,z,a)
 &=
 (
 k^{\frac{n-2}{2}}u_{x_l}(x),
 \mathcal{Y}(y)
 )_{L_y^2(\R^n)},
 \\
 \tfrac{\pa F_j}{\pa a}(u,k,z,a)
 &=
 (
 -\mathcal{Y}(y),
 \Psi_j(y)
 )_{L_y^2(\R^n)},
 \\
 \tfrac{\pa F_{n+1}}{\pa a}(u,k,z,a)
 &=
 (
 -\mathcal{Y}(y),
 \mathcal{Y}(y)
 )_{L_y^2(\R^n)}.
 \end{align*}
 Therefore
 it holds that
 \begin{align*}
 \begin{pmatrix}
 \frac{\pa F_0}{\pa k} & \frac{\pa F_0}{\pa z_1} & \cdots & \frac{\pa F_0}{\pa z_n} & \frac{\pa F_0}{\pa a}
 \\
 \frac{\pa F_1}{\pa k} & \frac{\pa F_1}{\pa z_1} & \cdots & \frac{\pa F_1}{\pa z_n} & \frac{\pa F_1}{\pa a}
 \\
 \vdots & \vdots & \cdots & \vdots & \vdots
 \\
 \frac{\pa F_n}{\pa k} & \frac{\pa F_n}{\pa z_1} & \cdots & \frac{\pa F_n}{\pa z_n} & \frac{\pa F_n}{\pa a}
 \\
 \frac{\pa F_{n+1}}{\pa k} & \frac{\pa F_{n+1}}{\pa z_1} & \cdots & \frac{\pa F_{n+1}}{\pa z_n} &
 \frac{\pa F_{n+1}}{\pa a}
 \end{pmatrix}
 &=
 \begin{pmatrix}
 J_0 & 0 & \cdots & \cdots & 0
 \\
 0 & J_1 & 0 & \cdots & 0
 \\
 \vdots & & \ddots & & \vdots
 \\
 \vdots & \vdots & 0 & J_n & 0
 \\
 0 & \cdots & \cdots & 0 & J_{n+1}
 \end{pmatrix}
 \\
 &
 \text{for } (u,k,z,a)=({\sf Q},1,0,0),
 \end{align*}
 where
 \begin{align*}
 J_0
 &=
 (\Lambda_y{\sf Q}(y),\Psi_0(y) )_{L_y^2(\R^n)},
 \\
 J_j
 &=
 (\pa_{y_j}{\sf Q}(y),\Psi_j(y) )_{L_y^2(\R^n)} 
 \quad
 (j=1,\cdots,n),
 \\
 J_{n+1}
 &=
 -\|\mathcal{Y}\|_{L^2(\R^n)}^2.
 \end{align*}
 Since ${\bf F}({\sf Q},1,0,0)=0$,
 by the implicit function theorem,
 there exist two constants $\nu_1>0$ and $\nu_2>0$,
 and
 three $C^1$ functions
 \[
 (\mathcal{K}(u),\mathcal{Z}(u),\mathcal{A}(u)):
 \mathcal{B}_{{\sf Q}}(\nu_2)\to
 \R_+\times\R^n\times\R
 \]
 such that
 for $(u,k,z,a)\in\mathcal{B}_{{\sf Q}}(\nu_2)\times\Omega_{(1,0,0)}(\nu_1)$,
 the following conditions are equivalent.
 \begin{itemize}
 \item
 ${\bf F}(u,k,z,a)=0$,
 \item
 $(k,z,a)=(\mathcal{K}(u),\mathcal{Z}(u),\mathcal{A}(u))$.
 \end{itemize}
 Finally,
 we check that the mapping $(\mathcal{K}(u),\mathcal{Z}(u),\mathcal{A}(u)):
 \mathcal{B}_{{\sf Q}}(\nu_2)\to
 \R_+\times\R^n\times\R$ are $C^\infty$.
 To show this fact,
 it is enough to confirm that
 the mapping $F_j(u,k,z,a):\mathcal{B}_{{\sf Q}}(\nu_2)\to\R_+\times\R^n\times\R$
 is $C^\infty$ for any $j\in\{0,\cdots,n+1\}$.
 We observe that
 \begin{align*}
 F_j(u,k,z,a)
 &=
 (
 k^\frac{n-2}{2}u(k y+z)
 -
 {\sf Q}(y)
 -
 a\mathcal{Y}(y),
 \Psi_j(y)
 )_{L_y^2(\R^n)}
 \\
 &=
 (
 u(x),
 k^{-\frac{n+2}{2}}\Psi_j(\tfrac{x-z}{k})
 )_{L_x^2(\R^n)}
 -
 (
 {\sf Q}+a\mathcal{Y},
 \Psi_j
 )_{L_y^2(\R^n)}
 \\
 &
 \text{for }
 j\in\{0,\cdots,n\}.
 \end{align*}
 From this relation,
 it is clear that $F_j(u,k,z,a)$ is $C^\infty$ for $j\in\{0,\cdots,n\}$.
 In the same reason,
 $F_{n+1}(u,k,z,a)$ is $C^\infty$.
 The proof is completed.
 \end{proof}
%%%%%%%%%%%%%%%%%%%%%%%%%%%%%%%%%%%%%%%%%%%%%%%%%%%%%%%%%%%%%%%%
 \begin{lem}[Unique decomposition in a neighborhood of $\mathcal{M}$]
 \label{LEMMA_2.7}
 Let $n\geq3$ and define
 \[
 \mathcal{U}_{\eta}
 =
 \{u(x)\in\dot H^1(\R^n);\text{\rm dist}_{\dot H^1(\R^n)}(u(x),\mathcal{M})<\eta\}.
 \]
 There exist two constants $\eta_1^*>0$ and $\eta_2^*>0$,
 and three $C^\infty$ functions
 \begin{align*}
 (\mathcal{K}_{\mathcal M}(u),\mathcal{Z}_{\mathcal M}(u),\mathcal{A}_{\mathcal M}(u))
 :
 \mathcal{U}_{\eta_2^*}
 \to
 \R_+\times\R^n\times\R
 \end{align*}
 satisfying the following {\rm(i)} - {\rm(iii)}.
 \begin{enumerate}[\rm(i)]
 \item
 For each $u(x)\in\mathcal{U}_{\eta_2^*}$,
 define
 \begin{align}
 \label{equation_e2.17}
 v(y)
 =
 \mathcal{K}_{\mathcal M}(u)^\frac{n-2}{2}u(\mathcal{K}_{\mathcal M}(u) y+\mathcal{Z}_{\mathcal M}(u))
 -
 {\sf Q}(y)
 -
 \mathcal{A}_{\mathcal M}(u)\mathcal{Y}(y).
 \end{align}
 Then
 it holds that
 \[
 \|v(y)\|_{\dot H_y^1(\R^n)}^2+|\mathcal{A}_{\mathcal M}(u)|^2\|\mathcal{Y}\|_{\dot H^1(\R^n)}^2
 <
 (\eta_1^*)^2
 \quad
 \text{\rm for all } u\in \mathcal{U}_{\eta_2^*}.
 \]
 \item
 $\begin{cases}
 (
 v(y),\mathcal{Y}(y)
 )_{L_y^2(\R^n)}
 =0,
 \\
 (
 v(y),\Psi_j(y)
 )_{L_y^2(\R^n)}
 =0
 & \text{\rm for } j\in\{0,\cdots,n\}.
 \end{cases}$
 \item
 Furthermore
 for any $\eta_1\in(0,\eta_1^*)$,
 there exists $\eta_2\in(0,\eta_2^*)$ such that
 \[
 \|v(y)\|_{\dot H_y^1(\R^n)}^2
 +
 |\mathcal{A}_{\mathcal M}(u)|^2\|\mathcal{Y}\|_{\dot H^1(\R^n)}^2
 <
 \eta_1^2
 \quad\
 \text{\rm for all } u\in \mathcal{U}_{\eta_2}.
 \]
 \end{enumerate}
 \end{lem}
%%%%%%%%%%%%%%%%%%%%%%%%%%%%%%%%%%%%%%%%%%%%%%%%%%%%%%%%%%%%%%%%
 \begin{proof}
 We first prove the existence of a decomposition:
 \[
 u(x)\in\mathcal{U}_{\eta_2}\mapsto(k,z,a)
 \]
 satisfying the following (li) - (lii).
 \begin{enumerate}[(li)]
 \item
 For each $u\in\mathcal U_{\eta_2}$,
 let $w(y)=k^\frac{n-2}{2}u(ky+z)-{\sf Q}(y)-a\mathcal{Y}(y)$.
 Then
 it holds that
 \[
 \|w(y)\|_{\dot H_y^1(\R^n)}^2+|a|^2\|\mathcal{Y}\|_{\dot H^1(\R^n)}^2
 <
 \eta_1^2
 \quad\
 \text{for all }
 u\in \mathcal{U}_{\eta_2} \text{ and }
 \]
 \item
 $\begin{cases}
 (
 w(y),\mathcal{Y}(y)
 )_{L_y^2(\R^n)}
 =0,
 \\
 (
 w(y),\Psi_j(y)
 )_{L_y^2(\R^n)}
 =0
 & \text{\rm for } j\in\{0,\cdots,n\}.
 \end{cases}$
 \end{enumerate}
 Let $\nu_1$ and $\nu_2$ be positive constants given in Lemma \ref{LEMMA_2.6}.
 We introduce another small positive constant $\nu_1'\in(0,\nu_1)$.
 By the continuity of $(\mathcal{K}(u),\mathcal{Z}(u),\mathcal{A}(u))$ (see Lemma \ref{LEMMA_2.6}),
 there exists $\nu_2'\in(0,\nu_2)$ such that
 \begin{align}
 \label{equation_e2.18}
 |\mathcal{K}(u)-1|^2
 &+
 |\mathcal{Z}(u)|^2
 +
 |\mathcal{A}(u)|^2\|\mathcal{Y}\|_{\dot H^1(\R^n)}^2
 <
 (\nu_1')^2
 \\
 \nonumber
 & 
 \text{for all } u\in \mathcal{B}_{\sf Q}(\nu_2').
 \end{align}
 Let $\eta_2^*\in(0,\frac{\nu_2'}{2})$ and choose $\eta_2\in(0,\eta_2^*)$.
 For every $u(x)\in\mathcal{U}_{\eta_2}$,
 there exists $(\lambda,\sigma)\in\R_+\times\R^n$ such that
 \[
 \|u(x)-\lambda^{-\frac{n-2}{2}}{\sf Q}(\tfrac{x-\sigma}{\lambda})\|_{\dot H_x^1(\R^n)}
 <
 2\eta_2.
 \]
 By changing variables,
 we have
 \[
 \|\lambda^\frac{n-2}{2}u(\lambda x+\sigma)-{\sf Q}(x)\|_{\dot H_x^1(\R^n)}
 <
 2\eta_2.
 \]
 We put $U(x)=\lambda^\frac{n-2}{2}u(\lambda x+\sigma)$,
 then
 \begin{align}
 \label{equation_e2.19}
 \|U(x)-{\sf Q}(x)\|_{\dot H_x^1(\R^n)}
 <
 2\eta_2.
 \end{align}
 Since $2\eta_2<2\eta_2^*<\nu_2'<\nu_2$,
 from Lemma \ref{LEMMA_2.6},
 the function
 $w(y)=k^\frac{n-2}{2}U(ky+z)-{\sf Q}(y)-a\mathcal{Y}(y)$
 with $(k,z,a)=(\mathcal{K}(U),\mathcal{Z}(U),\mathcal{A}(U))$
 satisfies
 \[
 \begin{cases}
 (
 w(y),\mathcal{Y}(y)
 )_{L_y^2(\R^n)}
 =0,
 \\
 (
 w(y),\Psi_j(y)
 )_{L_y^2(\R^n)}
 =0
 & \text{\rm for } j\in\{0,\cdots,n\}.
 \end{cases}
 \]
 From the choice of $\nu_2'$,
 it follows from \eqref{equation_e2.18} that
 \begin{align}
 \label{equation_e2.20}
 |k-1|^2+|z|^2+|a|^2\|\mathcal{Y}\|_{\dot H^1(\R^n)}^2<(\nu_1')^2.
 \end{align}
%%%%%%%%%%%%%%%%%%%%%%%%%%%%%%%%%%%%%%%%%%%%%%%%%%%%%%%%%%%%%%%%
 Furthermore
 from the definition of $U(x)$ and $w(y)$,
 we immediately see that
 \begin{align}
 \label{equation_e2.21}
 w(y)
 &=
 k^\frac{n-2}{2}U(ky+z)-{\sf Q}(y)-a\mathcal{Y}(y)
 \\
 \nonumber
 &=
 k^\frac{n-2}{2}
 \lambda^\frac{n-2}{2}
 u(\lambda(ky+z)+\sigma)-{\sf Q}(y)-a\mathcal{Y}(y)
 \\
 \nonumber
 &=
 (k\lambda)^\frac{n-2}{2}
 u((k\lambda)y+\lambda z+\sigma)-{\sf Q}(y)-a\mathcal{Y}(y).
 \end{align}
 Let $K=k\lambda$ and $Z=\lambda z+\sigma$,
 then
 \begin{align}
 \label{equation_e2.22}
 w(y)
 =
 K^\frac{n-2}{2}
 u(Ky+Z)-{\sf Q}(y)-a\mathcal{Y}(y).
 \end{align}
%%%%%%%%%%%%%%%%%%%%%%%%%%%%%%%%%%%%%%%%%%%%%%%%%%%%%%%%%%%%%%%%
 From \eqref{equation_e2.19} - \eqref{equation_e2.21},
 $w(y)$ can be estimated as
 \begin{align*}
 \|w(y)&\|_{\dot H_y^1(\R^n)}
 =
 \|
 k^\frac{n-2}{2}U(ky+z)-{\sf Q}(y)-a\mathcal{Y}(y)
 \|_{\dot H_y^1(\R^n)}
 \\
 &<
 \|
 k^\frac{n-2}{2}U(ky+z)-k^\frac{n-2}{2}{\sf Q}(ky+z)
 \|_{\dot H_y^1(\R^n)}
 \\
 &\quad+
 \|
 k^\frac{n-2}{2}{\sf Q}(ky+z)-{\sf Q}(y)
 \|_{\dot H_y^1(\R^n)}
 +
 \|
 a\mathcal{Y}(y)
 \|_{\dot H_y^1(\R^n)}
 \\
 &<
 \|
 U(x)-{\sf Q}(x)
 \|_{\dot H_x^1(\R^n)}
 +
 \|
 k^\frac{n-2}{2}{\sf Q}(ky+z)-{\sf Q}(y)
 \|_{\dot H_y^1(\R^n)}
 +
 \nu_1'
 \\
 &<
 2\eta_2
 +
 \|
 k^\frac{n-2}{2}{\sf Q}(ky+z)-{\sf Q}(y)
 \|_{\dot H_y^1(\R^n)}
 +
 \nu_1'.
 \end{align*}
 Since $|k-1|^2+|z|^2<(\nu_1')^2$  (see \eqref{equation_e2.20}),
 we note from Lemma \ref{LEMMA_2.8} that
 $\|
 k^\frac{n-2}{2}{\sf Q}(ky+z)-{\sf Q}(y)
 \|_{\dot H_y^1(\R^n)}<C\nu_1'$.
 Therefore
 there exists a constant $C'>0$ depending only on $n$ such that
 \begin{align}
 \label{equation_e2.23}
 \|w(y)\|_{\dot H_y^1(\R^n)}
 <
 2\eta_2
 +
 C'\nu_1'.
 \end{align}
 We now choose $\nu_1'$, $\nu_2'$ and $\eta_2$ as follows.
 For any given $\eta_1>0$,
 we first fix $\nu_1'\in(0,\nu_1)$ such that $C'\nu_1'<\frac{\eta_1}{8}$ in \eqref{equation_e2.23}
 and $\nu_1'<\frac{\eta_1}{8}$.
 We next choose $\nu_2'\in(0,\nu_2)$ such that \eqref{equation_e2.18} holds for all
 $u\in\mathcal{B}_{{\sf Q}}(\nu_2')$.
 We finally take $\eta_2>0$ such that $\eta_2<\min\{\frac{\nu_2'}{2},\frac{\eta_1}{16}\}$.
 From the choice of $\nu_1'$, $\nu_2'$, $\eta_2$, and \eqref{equation_e2.23},
 we have
 \begin{align*}
 &
 \|w(y)\|_{\dot H_y^1(\R^n)}^2
 +
 |a|^2\|\mathcal{Y}\|_{\dot H^1(\R^n)}^2
 <
 (2\eta_2+C'\nu_1')^2+(\nu_1')^2
 \\
 &<
 (\tfrac{\eta_1}{8}+\tfrac{\eta_1}{8})^2+(\tfrac{\eta_1}{8})^2
 <
 \tfrac{5}{64}\eta_1^2
 \quad\
 \text{for all }
 u\in\mathcal{B}_{\sf Q}(\eta_2).
 \end{align*}
 Therefore
 we have proved that \eqref{equation_e2.22} gives the desired decomposition
 satisfying (li) - (lii).
%%%%%%%%%%%%%%%%%%%%%%%%%%%%%%%%%%%%%%%%%%%%%%%%%%%%%%%%%%%%%%%%
 To define functions
 $(\mathcal{K}_{\mathcal M}(u),\mathcal{Z}_{\mathcal M}(u),\mathcal{A}_{\mathcal M}(u)):
 \mathcal{U}_{\eta_2^*}$ $\to\R_+\times\R^n\times\R$ as described in the statement of this lemma,
 it is enough to prove the uniqueness of this decomposition.
 Assume that there exist a function
 $u(x)\in \dot H^1(\R^n)$ and $(K_i,Z_i,A_i)\in\R_+\times\R^n\times\R$ for $i=1,2$ such that
 \begin{align}
 \label{equation_e2.24}
 w_i(y)
 &=
 K_i^\frac{n-2}{2}
 u(K_iy+Z_i)
 -
 {\sf Q}(y)
 -
 A_i
 \mathcal{Y}(y)
 \qquad
 (i=1,2)
 \end{align}
 with
 \begin{align}
 \label{equation_e2.25}
 \|w_i(y)\|_{\dot H_y^1(\R^n)}^2
 +
 |A_i|^2\|\mathcal{Y}\|_{\dot H^1(\R^n)}^2
 <
 (\eta_1^*)^2
 \qquad
 (i=1,2) 
 \end{align}
 and
 \begin{align}
 \label{equation_e2.26}
 \begin{cases}
 (
 w_i(y),\mathcal{Y}(y)
 )_{L_y^2(\R^n)}
 =0,
 \\
 (
 w_i(y),\Psi_j(y)
 )_{L_y^2(\R^n)}
 =0
 &
 \text{for } j\in\{0,\cdots,n\}
 \end{cases}
 \quad\
 (i=1,2).
 \end{align}
 We now claim that
 if $\eta_1^*$ given in \eqref{equation_e2.25} is sufficiently small,
 then
 \begin{align}
 \label{equation_e2.27}
 K_1=K_2,
 \quad
 Z_1=Z_2
 \quad
 \text{and }
 \
 A_1=A_2.
 \end{align}
 From \eqref{equation_e2.24},
 we note that
 \[
 u(x)
 =
 K_i^{-\frac{n-2}{2}}
 {\sf Q}(\tfrac{x-Z_i}{K_i})
 +
 K_i^{-\frac{n-2}{2}}
 A_i
 \mathcal{Y}(\tfrac{x-Z_i}{K_i})
 +
 K_i^{-\frac{n-2}{2}}
 w_i(\tfrac{x-Z_i}{K_i})
 \]
 for each $i=1,2$.
 From this relation,
 we observe that
 \begin{align*}
 \|
 &
 K_2^{-\frac{n-2}{2}}
 {\sf Q}(\tfrac{x-Z_2}{K_2})
 -
 K_1^{-\frac{n-2}{2}}
 {\sf Q}(\tfrac{x-Z_1}{K_1})
 \|_{\dot H_x^1(\R^n)}
 \\
 &<
 \sum_{i=1}^2
 |A_i|
 \cdot
 \|K_i^{-\frac{n-2}{2}}\mathcal{Y}(\tfrac{x-Z_i}{K_i})\|_{\dot H_x^1(\R^n)}
 +
 \sum_{i=1}^2
 \|K_i^{-\frac{n-2}{2}}w_i(\tfrac{x-Z_i}{K_i})\|_{\dot H_x^1(\R^n)}
 \\
 &<
 \sum_{i=1}^2
 (
 |A_i|
 \cdot
 \|\mathcal{Y}(y)\|_{\dot H_y^1(\R^n)}
 +
 \|w_i(y)\|_{\dot H_y^1(\R^n)}
 ).
 \end{align*}
 Combining \eqref{equation_e2.25},
 we get
 \begin{align*}
 \|
 K_2^{-\frac{n-2}{2}}
 {\sf Q}(\tfrac{x-Z_2}{K_2})
 -
 K_1^{-\frac{n-2}{2}}
 {\sf Q}(\tfrac{x-Z_1}{K_1})
 \|_{\dot H_x^1(\R^n)}
 <
 2\sqrt{2}\eta_1^*.
 \end{align*}
%%%%%%%%%%%%%%%%%%%%%%%%%%%%%%%%%%%%%%%%%%%%%%%%%%%%%%%%%%%%%%%%
 By change of variables,
 we have
 \begin{align}
 \label{equation_e2.28}
 \left\|
 {\sf Q}(y)
 -
 (\tfrac{K_2}{K_1})^{\frac{n-2}{2}}
 {\sf Q}\left(
 (\tfrac{K_2}{K_1})y+\tfrac{Z_2-Z_1}{K_1}
 \right)
 \right\|_{\dot H_y^1(\R^n)}
 <
 2\sqrt{2}\eta_1^*.
 \end{align}
 Hence from Lemma \ref{LEMMA_2.8},
 there exists $\bar\eta_1^*>0$ depending only on $n$ such that
 if $\eta_1^*\in(0,\bar\eta_1^*)$,
 then we have that
 $|\frac{K_2}{K_1}-1|<\frac{1}{2}$, $|\frac{Z_2-Z_1}{K_1}|<1$
 and \eqref{equation_e2.28} can be written as
 \begin{align}
 \label{equation_e2.29}
 (\tfrac{K_2}{K_1}-1)^2
 \|\Lambda_y{\sf Q}\|_{\dot H^1(\R^n)}^2
 +
 |\tfrac{Z_2-Z_1}{K_1}|^2
 \|\nabla_y{\sf Q}\|_{\dot H^1(\R^n)}^2
 <
 16(\eta_1^*)^2.
 \end{align}
 Using \eqref{equation_e2.29},
 we now fix $\eta_1^*\in(0,\bar\eta_1^*)$ such that
 \begin{itemize}
 \item
 $\eta_1^*<\frac{\nu_1}{64}\min\{\|\Lambda_y{\sf Q}\|_{\dot H^1(\R^n)},\|\nabla_y{\sf Q}\|_{\dot H^1(\R^n)}\}$,
 \item
 $\eta_1^*<\tfrac{1}{4}\min\{\nu_1,\nu_2\}$
 \quad
 ($\nu_1$ and $\nu_2$ are constants given in Lemma \ref{LEMMA_2.6}).
 \end{itemize}
 From the choice of $\eta_1^*$,
 \eqref{equation_e2.29} implies
 \begin{align}
 \label{equation_e2.30}
 |\tfrac{K_2}{K_1}-1|^2
 +
 |\tfrac{Z_2-Z_1}{K_1}|^2
 <
 \tfrac{\nu_1^2}{128}.
 \end{align}
 Put
 $U_1(x)
 =
 K_1^\frac{n-2}{2}u(K_1x+Z_1)$.
 From the definition of $w_i(y)$ (see \eqref{equation_e2.24}),
 they are expressed as
 \begin{align}
 \nonumber
 w_1(y)
 &=
 K_1^\frac{n-2}{2}
 u(K_1y+Z_1)
 -
 {\sf Q}(y)
 -
 A_1
 \mathcal{Y}(y)
 \\
 \label{equation_e2.31}
 &=
 U_1(y)
 -
 {\sf Q}(y)
 -
 A_1
 \mathcal{Y}(y),
 \\
 \nonumber
 w_2(y)
 &=
 K_2^\frac{n-2}{2}
 u(K_2y+Z_2)
 -
 {\sf Q}(y)
 -
 A_2
 \mathcal{Y}(y)
 \\
 \label{equation_e2.32}
 &=
 (\tfrac{K_2}{K_1})^\frac{n-2}{2}
 U_1((\tfrac{K_2}{K_1})y+\tfrac{Z_2-Z_1}{K_1})
 -
 {\sf Q}(y)
 -
 A_2
 \mathcal{Y}(y).
 \end{align}
 From the assumption \eqref{equation_e2.25},
 we deduce from \eqref{equation_e2.31} that
 \begin{align*}
 \|U_1-{\sf Q}\|_{\dot H^1(\R^n)}
 =
 \|w_1+A_1\mathcal{Y}\|_{\dot H^1(\R^n)}
 <
 \sqrt{2}\eta_1^*
 <
 \tfrac{\sqrt{2}}{4}\nu_2.
 \end{align*}
 To summarize the above argument, we confirm the following four facts.
 \begin{itemize}
 \item
 $\|U_1-{\sf Q}\|_{\dot H^1(\R^n)}<\frac{\sqrt{2}}{4}\nu_2$,
 \item
 $|A_i|^2\|\mathcal{Y}\|_{\dot H^1(\R^n)}^2<(\eta_1^*)^2<\frac{\nu_1^2}{16}$ \quad ($i=1,2$),
 \item
 $|\tfrac{K_2}{K_1}-1|^2+|\tfrac{Z_2-Z_1}{K_1}|^2<\tfrac{\nu_1^2}{128}$
 \quad
 (see \eqref{equation_e2.30})
 and
 \item
 $w_1(y)$ and $w_2(y)$ defined in \eqref{equation_e2.31} - \eqref{equation_e2.32}
 satisfy the orthogonal conditions \eqref{equation_e2.26}.
 \end{itemize}
 Therefore
 Lemma \ref{LEMMA_2.6} implies
 \begin{align*}
 (1,0,A_1)
 &=
 (\mathcal{K}(U_1),\mathcal{Z}(U_1),\mathcal{A}(U_1)),
 \\
 (\tfrac{K_2}{K_1},\tfrac{Z_2-Z_1}{K_1},A_2)
 &=
 (\mathcal{K}(U_1),\mathcal{Z}(U_1),\mathcal{A}(U_1)).
 \end{align*}
 Thus the claim \eqref{equation_e2.27} is proved.
 Once the mapping
 \begin{align*}
 (\mathcal{K}_{\mathcal{M}}(u),\mathcal{Z}_{\mathcal{M}}(u),\mathcal{A}_{\mathcal{M}}(u))
 :
 \mathcal{U}_{\eta_2^*}
 \to
 \R_+\times\R^n\times\R
 \end{align*}
 is defined,
 the $C^\infty$ differentiability of this mapping follows from Lemma \ref{LEMMA_2.6}.
 The proof is completed.
 \end{proof}
%%%%%%%%%%%%%%%%%%%%%%%%%%%%%%%%%%%%%%%%%%%%%%%%%%%%%%%%%%%%%%%%

%%%%%%%%%%%%%%%%%%%%%%%%%%%%%%%%%%%%%%%%%%%%%%%%%%%%%%%%%%%%%%%%
 \begin{lem}
 \label{LEMMA_2.8}
 Let $n\geq3$.
 There exist two constants $C_1>0$ and $C_2>0$ depending only on $n$ such that
 \begin{align*}
 \|{\sf Q}(y)&-k^\frac{n-2}{2}{\sf Q}(ky+z)\|_{\dot H_y^1(\R^n)}
 >
 C_1
 \quad
 \text{ \rm if } |k-1|<\tfrac{1}{2} \text{ \rm or } |z|<1,
 \\
 \|{\sf Q}(y)&-k^\frac{n-2}{2}{\sf Q}(ky+z)\|_{\dot H_y^1(\R^n)}
 \\
 &=
 \sqrt{
 (k-1)^2\|\Lambda_y{\sf Q}\|_{\dot H^1(\R^n)}^2
 +
 |z|^2\|\nabla_y{\sf Q}\|_{\dot H^1(\R^n)}^2
 }
 +
 {\sf R}(k,z)
 \\
 &
 \text{\rm if } |k-1|<\tfrac{1}{2} \text{ \rm and } |z|<1,
 \end{align*}
 where
 ${\sf R}(k,z)$ is a certain function satisfying
 $|{\sf R}(k,z)|<C_2(|k-1|^\frac{3}{2}+|z|^\frac{3}{2})$.
 \end{lem}
%%%%%%%%%%%%%%%%%%%%%%%%%%%%%%%%%%%%%%%%%%%%%%%%%%%%%%%%%%%%%%%%
 \begin{proof}
 We first consider the case where $|k-1|<\frac{1}{2}$ and $|z|<1$.
 Put $G(k,z,y)=k^\frac{n-2}{2}{\sf Q}(ky+z)$.
 We apply a Taylor expansion.
 \begin{align*}
 G(k,z,y)
 &=
 G(1,0,y)
 +
 \tfrac{\pa G}{\pa k}(1,0,y)
 (k-1)
 +
 \sum_{j=1}^n
 \tfrac{\pa G}{\pa z_j}(1,0,y)
 z_j
 +
 R_2
 \\
 &=
 {\sf Q}(y)
 +
 (k-1)
 \Lambda_y{\sf Q}(y)
 +
 z\cdot
 \nabla_y{\sf Q}(y)
 +
 R_2(k,z,y).
 \end{align*}
 The remainder term $R_2$ can be expressed as
 \begin{align*}
 R_2(k,z,y)
 &=
 \tfrac{1}{2}
 \tfrac{\pa^2G}{\pa k^2}
 (k_\theta,z_\theta,y)
 \cdot 
 (k-1)^2
 +
 \tfrac{1}{2}
 \sum_{j=1}^n
 \tfrac{\pa^2G}{\pa z_j^2}
 (k_\theta,z_\theta,y)
 \cdot
 z_j^2
 \\
 &\quad
 +
 \sum_{j=1}^n
 \tfrac{\pa^2G}{\pa k\pa z_j}
 (k_\theta,z_\theta,y)
 \cdot
 (k-1)z_j
 +
 \sum_{i<j}
 \tfrac{\pa^2G}{\pa z_i\pa z_j}
 (k_\theta,z_\theta,y)
 \cdot
 z_iz_j,
 \end{align*}
 where $k_\theta=1+\theta(k-1)$, $z_\theta=\theta z$ for some $\theta\in(0,1)$.
 Let $x=ky+z$.
 A straightforward computation yields
 \begin{align*}
 \tfrac{\pa^2G}{\pa k^2}
 (k,z,y)
 &=
 \tfrac{(n-2)(n-4)}{4}
 k^\frac{n-6}{2}
 {\sf Q}(x)
 +
 (n-2)
 k^\frac{n-4}{2}
 \sum_{j=1}^n
 y_j
 \tfrac{\pa {\sf Q}}{\pa x_j}
 (x)
 \\
 &\quad
 +
 k^\frac{n-2}{2}
 \sum_{i<j}
 y_iy_j
 \tfrac{\pa^2{\sf Q}}{\pa x_i\pa x_j}(x),
 \\
 \tfrac{\pa^2G}{\pa z_i\pa z_j}
 (k,z,y)
 &=
 k^\frac{n-2}{2}
 \tfrac{\pa^2{\sf Q}}{\pa x_i\pa x_j}
 (x),
 \\
 \tfrac{\pa^2G}{\pa k\pa z_j}
 (k,z,y)
 &=
 \tfrac{n-2}{2}
 k^\frac{n-4}{2}
 \tfrac{\pa {\sf Q}}{\pa x_j}(x)
 +
 k^\frac{n-2}{2}
 \sum_{k=1}^n
 y_k
 \tfrac{\pa^2{\sf Q}}{\pa x_k\pa x_j}(x).
 \end{align*}
%%%%%%%%%%%%%%%%%%%%%%%%%%%%%%%%%%%%%%%%%%%%%%%%%%%%%%%%%%%%%%%%
 Since $\frac{1}{2}<|k_\theta|<2$ and $|z_\theta|<1$,
 we see that
 \begin{align*}
 &
 \|
 \tfrac{\pa^2G}{\pa k^2}
 (k_\theta,z_\theta,y)
 \|_{\dot H_y^1(\R^n)}
 \\
 &<
 c
 \|
 {\sf Q}(x)
 \|_{\dot H_y^1(\R^n)}
 +
 c
 \sum_{j=1}^n
 \|
 y_j\tfrac{\pa {\sf Q}}{\pa x_j}(x)
 \|_{\dot H_y^1(\R^n)}
 \\
 &\quad
 +
 c
 \sum_{i<j}
 \|
 y_iy_j
 \tfrac{\pa^2{\sf Q}}{\pa x_i\pa x_j}(x)
 \|_{\dot H_y^1(\R^n)}
 \\
 &<
 c
 \|
 {\sf Q}(x)
 \|_{\dot H_y^1(\R^n)}
 +
 c
 \|
 \tfrac{x_j-z_j}{k}
 \tfrac{\pa {\sf Q}}{\pa x_j}(x)
 \|_{\dot H_y^1(\R^n)}
 \\
 &\quad
 +
 c
 \sum_{i<j}
 \|
 \tfrac{(x_j-z_j)(x_i-z_i)}{k^2}
 \tfrac{\pa^2{\sf Q}}{\pa x_i\pa x_j}(x)
 \|_{\dot H_y^1(\R^n)}
 \\
 &<
 c
 \|
 {\sf Q}(x)
 \|_{\dot H_y^1(\R^n)}
 +
 c
 \|
 x_j
 \tfrac{\pa {\sf Q}}{\pa x_j}(x)
 \|_{\dot H_y^1(\R^n)}
 +
 c
 |z_j|
 \cdot
 \|
 \tfrac{\pa {\sf Q}}{\pa x_j}(x)
 \|_{\dot H_y^1(\R^n)}
 \\
 &\quad
 +
 c
 \sum_{i<j}
 \|
 x_jx_i
 \tfrac{\pa^2{\sf Q}}{\pa x_i\pa x_j}(x)
 \|_{\dot H_y^1(\R^n)}
 +
 c
 \sum_{i\not=j}
 |z_i|\cdot
 \|
 x_j
 \tfrac{\pa^2{\sf Q}}{\pa x_i\pa x_j}(x)
 \|_{\dot H_y^1(\R^n)}
 \\
 &\quad
 +
 c
 \sum_{i<j}
 |z_i||z_j|\cdot
 \|
 \tfrac{\pa^2{\sf Q}}{\pa x_i\pa x_j}(x)
 \|_{\dot H_y^1(\R^n)}.
 \end{align*}
 By changing variables,
 we get
 \begin{align*}
 &
 \|
 \tfrac{\pa^2G}{\pa k^2}
 (k_\theta,z_\theta,y)
 \|_{\dot H_y^1(\R^n)}
 \\
 &<
 c
 k_\theta^{-\frac{n-2}{2}}
 \|
 {\sf Q}(x)
 \|_{\dot H_x^1(\R^n)}
 +
 c
 k_\theta^{-\frac{n-2}{2}}
 \|
 x_j
 \tfrac{\pa {\sf Q}}{\pa x_j}(x)
 \|_{\dot H_x^1(\R^n)}
 +
 c
 k_\theta^{-\frac{n-2}{2}}
 \|
 \tfrac{\pa {\sf Q}}{\pa x_j}(x)
 \|_{\dot H_x^1(\R^n)}
 \\
 &\quad
 +
 c
 k^{-\frac{n-2}{2}}
 \sum_{i<j}
 \|
 x_jx_i
 \tfrac{\pa^2{\sf Q}}{\pa x_i\pa x_j}(x)
 \|_{\dot H_x^1(\R^n)}
 +
 c
 k_\theta^{-\frac{n-2}{2}}
 \sum_{i\not=j}
 \|
 x_j
 \tfrac{\pa^2{\sf Q}}{\pa x_i\pa x_j}(x)
 \|_{\dot H_x^1(\R^n)}
 \\
 &\quad
 +
 c
 k_\theta^{-\frac{n-2}{2}}
 \sum_{i\not=j}
 \|
 \tfrac{\pa^2{\sf Q}}{\pa x_i\pa x_j}(x)
 \|_{\dot H_x^1(\R^n)}
 <
 C.
 \end{align*}
 Here the constant $C$ depends only on $n$.
%%%%%%%%%%%%%%%%%%%%%%%%%%%%%%%%%%%%%%%%%%%%%%%%%%%%%%%%%%%%%%%%
 In the same manner,
 we can verify that there exists a constant $C>0$ depending only on $n$ such that
 \begin{align*}
 \sum_{i,j=1}^n
 \|
 \tfrac{\pa^2G}{\pa z_iz_j}
 (k_\theta,z_\theta,y)
 \|_{\dot H_y^1(\R^n)}
 +
 \sum_{j=1}^n
 \|
 \tfrac{\pa^2G}{\pa k\pa z_j}
 (k_\theta,z_\theta,y)
 \|_{\dot H_y^1(\R^n)}
 < 
 C.
 \end{align*}
 Thus
 we conclude that
 there exists a constant $C_1>0$ depending only on $n$ such that
 if
 $|k-1|<\tfrac{1}{2} \text{ and } |z|<1$,
 then
 \begin{align}
 \label{equation_e2.33}
 \|R_2(k,z,y)\|_{\dot H_y^1(\R^n)}
 <
 C_1((k-1)^2+|z|^2).
 \end{align}
 We now derive estimates for
 $\|{\sf Q}(y)-k^\frac{n-2}{2}{\sf Q}(ky+z)\|_{\dot H_y^1(\R^n)}$.
 Noting that
 $(\Lambda_y{\sf Q},z\cdot\nabla_y{\sf Q})_{\dot H_y^1(\R^n)}$ $=0$,
 $(\tfrac{\pa{\sf Q}}{\pa y_i},\tfrac{\pa{\sf Q}}{\pa y_j})_{\dot H_y^1(\R^n)}=0$ if $i\not=j$,
 and using \eqref{equation_e2.33},
 we have
 \begin{align*}
 \|&{\sf Q}(y)-k^\frac{n-2}{2}{\sf Q}(ky+z)\|_{\dot H_y^1(\R^n)}^2
 \\
 &=
 \|(k-1)\Lambda_y{\sf Q}(y)+z\cdot\nabla_y{\sf Q}(y)+R_2(k,z,y)\|_{\dot H_y^1(\R^n)}^2
 \\
 &=
 (k-1)^2
 \|\Lambda_y{\sf Q}(y)\|_{\dot H_y^1(\R^n)}^2
 +
 |z|^2
 \|\nabla_y{\sf Q}(y)\|_{\dot H_y^1(\R^n)}^2
 +
 {\sf R}(k,z)
 \\
 &
 \text{if }
 |k-1|<\tfrac{1}{2} \text{ and } |z|<1,
 \end{align*}
 where ${\sf R}(k,z)$ is a certain function satisfying
 $|{\sf R}(k,z)|<C(|k-1|^3+|z|^3)$.
 This proves the case $|k-1|<\frac{1}{2}$ and $|z|<1$.
 We next consider the case $|k-1|>\frac{1}{2}$ or $|z|>1$.
 It is enough to show that
 \begin{align*}
 \|{\sf Q}(y)&-k^\frac{n-2}{2}{\sf Q}(ky+z)\|_{L_y^\frac{2n}{n-2}(\R^n)}
 >
 C_1
 \end{align*}
 for any $(k,z)\in\R_+\times\R^n$ such that $|k-1|>\frac{1}{2}$ or $|z|>1$.
 Firs we assume that $0<k<\frac{1}{2}$.
 Let $\Omega_1=\{y\in\R^n;{\sf Q}(y)>\frac{3}{4}\}$.
 Then
 \begin{align*}
 \int_{\R^n}
 &|{\sf Q}(y)-k^\frac{n-2}{2}{\sf Q}(ky+z)|^\frac{2n}{n-2}
 dy
 \\
 &>
 \int_{\Omega_1}
 |{\sf Q}(y)-k^\frac{n-2}{2}{\sf Q}(ky+z)|^\frac{2n}{n-2}
 dy
 \\
 &>
 \int_{\Omega_1}
 (\tfrac{3}{4}-(\tfrac{1}{2})^\frac{n-2}{2})^\frac{2n}{n-2}
 dy
 >
 \int_{\Omega_1}
 (\tfrac{1}{4})^\frac{2n}{n-2}
 dy
 \\
 &=
 (\tfrac{1}{4})^\frac{2n}{n-2}
 |\Omega_1|.
 \end{align*}
 For the case $k>\frac{3}{2}$,
 we change variables.
 \begin{align*}
 \int_{\R^n}
 &|{\sf Q}(y)-k^\frac{n-2}{2}{\sf Q}(ky+z)|^\frac{2n}{n-2}
 dy
 \\
 &=
 \int_{\R^n}
 |k^{-\frac{n-2}{2}}{\sf Q}(\tfrac{x-z}{k})-{\sf Q}(x)|^\frac{2n}{n-2}
 dx.
 \end{align*}
 Put
 $d_1=1-(\frac{3}{2})^{-\frac{n-2}{2}}>0$
 and
 $\Omega_2=\{x\in\R^n;{\sf Q}(x)>1-\frac{d_1}{2}\}$.
 Similarly,
 it follows that
 \begin{align*}
 \int_{\R^n}
 &
 |k^{-\frac{n-2}{2}}{\sf Q}(\tfrac{x-z}{k})-{\sf Q}(x)|^\frac{2n}{n-2}
 dx
 \\
 &>
 \int_{\Omega_2}
 |{\sf Q}(x)-k^{-\frac{n-2}{2}}{\sf Q}(\tfrac{x-z}{k})|^\frac{2n}{n-2}
 dy
 \\
 &>
 \int_{\Omega_2}
 (1-\tfrac{d_1}{2}-k^{-\frac{n-2}{2}})
 dy
 \\
 &>
 \int_{\Omega_2}
 (1-\tfrac{d_1}{2}-(\tfrac{3}{2})^{-\frac{n-2}{2}})
 dy
 =
 \tfrac{d_1}{2}
 |\Omega_2|.
 \end{align*}
%%%%%%%%%%%%%%%%%%%%%%%%%%%%%%%%%%%%%%%%%%%%%%%%%%%%%%%%%%%%%%%%
 Finally,
 we address the case $|z|>1$.
 Suppose that $|z|>1$ and $k\in(0,1)$.
 From the explicit form of ${\sf Q}(x)$,
 it holds that
 \begin{align*}
 &
 \sup_{|y|<\frac{1}{2}}
 k^\frac{n-2}{2}{\sf Q}(ky+z)
 =
 k^\frac{n-2}{2}
 {\sf Q}((|z|-\tfrac{k}{2}){\bf e})
 \qquad
 (|{\bf e}|=1)
 \\
 &=
 k^\frac{n-2}{2}
 \left(
 1+\tfrac{(|z|-\tfrac{k}{2})^2}{n(n-2)}
 \right)^{-\frac{n-2}{2}}
 <
 \left(
 1+\tfrac{1}{4n(n-2)}
 \right)^{-\frac{n-2}{2}}.
 \end{align*}
 Hence
 we have
 \begin{align*}
 \int_{\R^n}
 &|{\sf Q}(y)-k^\frac{n-2}{2}{\sf Q}(ky+z)|^\frac{2n}{n-2}
 dy
 \\
 &>
 \int_{|y|<\frac{1}{2}}
 |{\sf Q}(y)-k^\frac{n-2}{2}{\sf Q}(ky+z)|^\frac{2n}{n-2}
 dy
 \\
 &>
 \left(
 1-
 (
 1+\tfrac{1}{4n(n-2)}
 )^{-\frac{n-2}{2}}
 \right)^\frac{2n}{n-2}
 \int_{|y|<\frac{1}{2}}
 dy
 \\
 &
 \text{if }
 |z|>1 \text{ and } k\in(0,1).
 \end{align*}
 Next consider the case $|z|>1$ and $k>1$.
 We easily see that
 \begin{align*}
 \int_{\R^n}
 &|{\sf Q}(y)-k^\frac{n-2}{2}{\sf Q}(ky+z)|^\frac{2n}{n-2}
 dy
 \\
 &=
 \int_{\R^n}
 |k^{-\frac{n-2}{2}}{\sf Q}(\tfrac{x-z}{k})-{\sf Q}(x)|^\frac{2n}{n-2}
 \\
 &>
 \int_{|x|<\frac{1}{2}}
 |k^{-\frac{n-2}{2}}{\sf Q}(\tfrac{x-z}{k})-{\sf Q}(x)|^\frac{2n}{n-2}
 dx
 \\
 &>
 \left\{
 1-k^{-\frac{n-2}{2}}(1+\tfrac{1}{n(n-2)}\tfrac{1}{4k^2})^{-\frac{n-2}{2}}
 \right\}^\frac{2n}{n-2}
 \int_{|x|<\frac{1}{2}}
 dx.
 \end{align*}
 Let $f(k)=k^{-\frac{n-2}{2}}(1+\tfrac{1}{n(n-2)}\tfrac{1}{4k^2})^{-\frac{n-2}{2}}$.
 Since $\sup_{k>1}f(k)=f(1)$ when $n\geq3$,
 it follows that
 \begin{align*}
 \int_{\R^n}
 |{\sf Q}(y)-k^\frac{n-2}{2}{\sf Q}(ky+z)|^\frac{2n}{n-2}
 dy
 >
 \left\{
 1-(1+\tfrac{1}{n(n-2)}\tfrac{1}{4})^{-\frac{n-2}{2}}
 \right\}^\frac{2n}{n-2}
 \int_{|x|<\frac{1}{2}}
 dx.
 \end{align*}
 This completes the proof.
 \end{proof}
%%%%%%%%%%%%%%%%%%%%%%%%%%%%%%%%%%%%%%%%%%%%%%%%%%%%%%%%%%%%%%%%

 \subsection{Ancient solutions}
%%%%%%%%%%%%%%%%%%%%%%%%%%%%%%%%%%%%%%%%%%%%%%%%%%%%%%%%%%%%%%%%
 As in \cite{Collot-Merle-Raphael},
 we will use the following lemma to show that
 a solution which escapes from the neighborhood of $\mathcal M$ exhibits type I blowup.
 \begin{lem}[Proposition 3.1 in \cite{Collot-Merle-Raphael} p.\,239]
 \label{LEMMA_2.9}
 Let $n\geq3$ and $p=\frac{n+2}{n-2}$.
 There exists a smooth, radially symmetric solution
 ${\sf Q}^+(x,t)$ of
 \[
 \pa_t u=\Delta_x u+|u|^{p-1}u
 \]
 defined on $(x,t)\in\R^n\times(-\infty,T)$ for some $T>0$
 satisfying {\rm (i1) - (i7)}.
 \begin{enumerate}[\rm ({i}1)]
 \item
 $\dis\lim_{t\to-\infty}\|\nabla_x{\sf Q}^+(x,t)-\nabla_x{\sf Q}(x)\|_2=0$,
 \item
 there exist $\epsilon_1>0$ and $C>0$ such that
 ${\sf Q}^+(x,t)$ can be expressed as
 \begin{align*}
 {\sf Q}^+(x,t)
 =
 {\sf Q}(x)
 +
 \epsilon_1
 e^{e_0t}
 \mathcal{Y}(x)
 +
 v(x,t),
 \end{align*}
 where $v(x,t )$ is a remainder satisfying
 \[
 \|v(t)\|_2
 +
 \|v(t)\|_\infty
 +
 \|\nabla_xv(t)\|_2
 <
 C\epsilon_1^2
 e^{2e_0t}
 \quad\
 \text{\rm for }
 t\in(-\infty,0),
 \]
 \item
 ${\sf Q}^+(x,t)>0$ \ {\rm for} $(x,t)\in\R^n\times(-\infty,T)$,
 \item
 ${\sf Q}^+(x,t)$ {\rm is non increasing in} $|x|$ {\rm for any fixed} $t\in(-\infty,T)$, 
 \item
 $\pa_t{\sf Q}^+(x,t)>0$ \ {\rm for} $(x,t)\in\R^n\times(-\infty,T)$,
 \item
 ${\sf Q}^+(x,t)$ blows up at $t=T$ and the blow up is of type I,
 \item
 ${\sf Q}^+(x,t)$ satisfies {\rm (a1) - (a6)} in {\rm Theorem 2} in {\rm\cite{Harada_ODE}} {\rm p.\,5 - p.\,6},
 which implies that this blowup is a nondegenerate ODE type blowup.
 \end{enumerate}
 \end{lem}
%%%%%%%%%%%%%%%%%%%%%%%%%%%%%%%%%%%%%%%%%%%%%%%%%%%%%%%%%%%%%%%%
 \begin{proof}
 We provide a brief explanation of the proof for (i4) - (i7).
 From their construction (see (3.2) in \cite{Collot-Merle-Raphael} p.\,239),
 we easily see that (i4) - (i5) hold.
 It is well known that (i5) implies (i6) when $p=\frac{n+2}{n-2}$.
 Furthermore
 (i7) is a direct consequence of (i4).
 \end{proof}
%%%%%%%%%%%%%%%%%%%%%%%%%%%%%%%%%%%%%%%%%%%%%%%%%%%%%%%%%%%%%%%%
% \begin{rem}
% In Section {\rm\ref{section_4.4}},
% we use the function ${\sf Q}^+(x,t)$
% to prove that a solution $u(x,t)$ of \eqref{equation_1.1} close to ${\sf Q}^+(x,t)$ exhibits a finite time blowup of type I
% by applying Theorem {\rm2} in {\rm\cite{Harada_ODE}} with $\varphi(x,t)={\sf Q}^+(x,t)$.
% By a standard argument,
% we can verify that
% ${\sf Q}^+(x,t)$ satisfies
% all the assumptions {\rm(a1) - (a6)} in Theorem {\rm2} in {\rm\cite{Harada_ODE}}.
% In particular,
% {\rm(a1)} and {\rm(a6)} in Theorem {\rm2} in {\rm\cite{Harada_ODE}} follow directly from {\rm(c4)}.
% \end{rem}
%%%%%%%%%%%%%%%%%%%%%%%%%%%%%%%%%%%%%%%%%%%%%%%%%%%%%%%%%%%%%%%%
 \begin{rem}
 \label{REMARK_2.10}
 The solution ${\sf Q}^+(x,t)$ in Lemma {\rm\ref{LEMMA_2.9}}
 coincides with the ancient solution obtained in Theorem {\rm2.3} of {\rm\cite{Fila-Yanagida}}.
 \end{rem}
%%%%%%%%%%%%%%%%%%%%%%%%%%%%%%%%%%%%%%%%%%%%%%%%%%%%%%%%%%%%%%%%

\section{Dynamics in a neighbourhood of $\mathcal{M}$}
 \label{section_3}
%%%%%%%%%%%%%%%%%%%%%%%%%%%%%%%%%%%%%%%%%%%%%%%%%%%%%%%%%%%%%%%%
 We recall the definition of $\mathcal{U}_\eta$.
 \[
 \mathcal{U}_{\eta}
 =
 \{u(x)\in\dot H^1(\R^n);\text{\rm dist}_{\dot H^1(\R^n)}(u(x),\mathcal{M})<\eta\},
 \]
 where
 $\mathcal{M}
 =
 \{\frac{1}{\lambda^\frac{n-2}{2}}{\sf Q}(\tfrac{x-\sigma}{\lambda})\in \dot H^1(\R^n);\lambda>0,\sigma\in\R^n\}$.
 A goal of this section is to study the dynamics in $\mathcal{U}_\eta$ for sufficiently small $\eta>0$.
 Let $\eta_1^*$ and $\eta_2^*$ be positive constants defined in Lemma \ref{LEMMA_2.7},
 and fix $\eta_1\in(0,\eta_1^*)$ which will be chosen in Section \ref{section_4}.
 From (iii) in Lemma \ref{LEMMA_2.7} ,
 there exists $\eta_2\in(0,\eta_2^*)$ such that
 \[
 \|v(y)\|_{\dot H_y^1(\R^n)}^2
 +
 |\mathcal{A}_{\mathcal M}(u)|^2\|\mathcal{Y}\|_{\dot H^1(\R^n)}^2
 <
 \eta_1^2
 \quad\
 \text{for all }
 u\in \mathcal{U}_{\eta_2},
 \]
 where $v(y)$ is defined by
 \begin{align*}
 v(y)
 =
 \mathcal{K}_{\mathcal M}(u)^\frac{n-2}{2}u(\mathcal{K}_{\mathcal M}(u) y+\mathcal{Z}_{\mathcal M}(u))
 -
 {\sf Q}(y)
 -
 \mathcal{A}_{\mathcal M}(u)\mathcal{Y}(y).
 \end{align*}
 In this section,
 we consider a solution $u(x,t)$ of \eqref{equation_1.1} for $t\in(t_1,t_2)$
 satisfying the following conditions (a1) - (a3).
 \vspace{2mm}
 \begin{enumerate}[({a}1)]
 \label{a_property}
 \item $u(t)\in C([t_1,t_2];H^1(\R^n))$,
 \item $u(x,t)\in L^\infty(\R^n\times(t_1',t_2))$ \ {\rm for any} $t_1'\in(t_1,t_2)$,
 \item $\text{\rm dist}_{\dot H^1(\R^n)}(u(t),\mathcal{M})<\eta_2$ \ {\rm for} $t\in (t_1,t_2)$,
 \end{enumerate}
 \vspace{2mm}
 where $(t_1,t_2)$ denotes a subinterval of $(0,T)$.
 From Lemma \ref{LEMMA_2.7},
 if a solution $u(x,t)$ of \eqref{equation_1.1} satisfies (a1) - (a3),
 then it can be written as
 \begin{align}
 \label{EQUATION_e3.1}
 u(x,t)
 &=
 \lambda(t)^{-\frac{n-2}{2}}
 \{
 {\sf Q}(y)
 +
 a(t)\mathcal{Y}(y)
 +
 \epsilon(y,t)
 \}
 \quad\
 \text{wiht }
 y=\tfrac{x-z(t)}{\lambda(t)}
 \\
 \nonumber
 &\
 \text{in }
 t\in(t_1,t_2),
 \end{align}
 where $\lambda(t),z(t),a(t)$ and $\epsilon(y,s)$ are defined by
%%%%%%%%%%%%%%%%%%%%%%%%%%%%%%%%%%%%%%%%%%%%%%%%%%%%%%%%%%%%%%%%
 \begin{align*}
 \begin{cases}
 (\lambda(t),z(t),a(t))
 =
 (\mathcal{K}_{\mathcal M}(u(t)),\mathcal{Z}_{\mathcal M}(u(t)),\mathcal{A}_{\mathcal M}(u(t))),
 \\
 \epsilon(y,t)
 =
 \lambda(t)^\frac{n-2}{2}u(\lambda(t)y+z(t),t)
 -
 {\sf Q}(y)
 -
 a(t)\mathcal{Y}(y).
 \end{cases}
 \end{align*}
 Here
 $\epsilon(y,t)$ satisfies the following orthogonal conditions.
 \begin{align}
 \label{EQUATION_e3.2}
 \begin{cases}
 (\epsilon(t),{\mathcal Y})_{L_y^2(\R^n)}=0 & \text{for } t\in(t_1,t_2), \\
 (\epsilon(t),\Psi_j)_{L_y^2(\R^n)}=0 & \text{for } t\in(t_1,t_2) \quad (j=0,1,\cdots,n).
 \end{cases}
 \end{align}
 From Lemma \ref{LEMMA_2.5},
 we have
 \begin{align}
 \label{EQUATION_e3.3}
 \begin{cases}
 \bar C_1
 \|\nabla_y\epsilon(t)\|_2^2
 <
 (-H_y\epsilon(t),\epsilon(t))_{L_y^2(\R^n)}
 <
 \bar C_2
 \|\nabla_y\epsilon(t)\|_2^2,
 \\
 \bar C_1
 \|\Delta_y\epsilon(t)\|_2^2
 <
 \|H_y\epsilon(t)\|_{L_y^2(\R^n)}^2
 <
 \bar C_2
 \|\Delta_y\epsilon(t)\|_2^2,
 \end{cases}
 \end{align}
 where both constants $\bar C_1$ and $\bar C_2$ are independent of $\epsilon(y,t)$.
 Furthermore
 by the choice of $\eta_2$,
 it follows that
 \begin{align}
 \label{EQUATION_e3.4}
 \|\nabla_y\epsilon(t)\|_2^2+|a(t)|^2\|\nabla_y\mathcal{Y}\|_2^2<\eta_1^2
 \quad\
 \text{for } t\in(t_1,t_2).
 \end{align}
%%%%%%%%%%%%%%%%%%%%%%%%%%%%%%%%%%%%%%%%%%%%%%%%%%%%%%%%%%%%%%%%
 Since
 $(\mathcal{K}_{\mathcal M}(u),\mathcal{Z}_{\mathcal M}(u),\mathcal{A}_{\mathcal M}(u)):
 \mathcal{U}_{\eta_2}\to\R_+\times\R^n\times\R$ are smooth,
 and
 by the regularity of the solution $u(x,t)$ (see (e1) - (e7) in Proposition \ref{PROPOSITION_2.2}),
 we can verify that
%%%%%%%%%%%%%%%%%%%%%%%%%%%%%%%%%%%%%%%%%%%%%%%%%%%%%%%%%%%%%%%%
 \begin{enumerate}[(b1)]
 \item $(\lambda(t),z(t),a(t))\in C([t_1,t_2];\R_+\times\R^n\times\R)\cap C^1((t_1,t_2);\R_+\times\R^n\times\R)$,
 \item $\epsilon(t)\in C([t_1,t_2];H_y^1(\R^n))\cap C^1((t_1,t_2);H_y^1(\R^n))\cap C((t_1,t_2);H_y^3(\R^n))$,
 \item $D_y^\alpha\epsilon(y,t),D_y^\beta\pa_t\epsilon(y,t)\in C(\R^n\times(t_1,t_2))$ \quad ($0\leq|\alpha|\leq3$, $0\leq|\beta|\leq2$),
 \item $D_y^\alpha\epsilon(y,t),D_y^\beta\pa_t\epsilon(y,t)\in L^\infty(\R^n\times(t_1+\Delta t,t_2))$ \ for any $\Delta t\in(0,t_2-t_1)$
 \\ ($0\leq|\alpha|\leq3$, $0\leq|\beta|\leq2$).
 \end{enumerate}
%%%%%%%%%%%%%%%%%%%%%%%%%%%%%%%%%%%%%%%%%%%%%%%%%%%%%%%%%%%%%%%%
 Substituting \eqref{EQUATION_e3.1} into $u_t=\Delta u+f(u)$,
 we get
 \begin{align}
 \nonumber
 -
 \lambda
 &
 \lambda_t
 (\Lambda_y{\sf Q})
 +
 \lambda^2
 a_t
 {\mathcal Y}
 -
 \lambda
 \lambda_t
 a
 (\Lambda_y{\mathcal Y})
 +
 \lambda^2
 \epsilon_t
 -
 \lambda
 \lambda_t
 (\Lambda_y\epsilon)
 \\
 \nonumber
 &\quad
 -
 \lambda
 \tfrac{dz}{dt}
 \cdot\nabla_y{\sf Q}
 -
 \lambda a
 \tfrac{dz}{dt}
 \cdot\nabla_y\mathcal{Y}
 -
 \lambda
 \tfrac{dz}{dt}
 \cdot\nabla_y\epsilon
 \\
 \nonumber
 &=
 \Delta_y{\sf Q}
 +
 a
 \Delta_y{\mathcal Y}
 +
 \Delta_y\epsilon
 +
 f({\sf Q}+a{\mathcal Y}+\epsilon)
 \\
 \label{EQUATION_e3.5}
 &=
 a
 H_y{\mathcal Y}
 +
 H_y\epsilon
 +
 N
 \qquad
 \text{for }
 t\in (t_1,t_2).
 \end{align}
 The nonlinear term denoted by $N$ is given by
 \begin{align}
 \label{EQUATION_e3.6}
 N=
 f({\sf Q}+a{\mathcal Y}+\epsilon)
 -
 f({\sf Q})
 -
 f'({\sf Q})
 (a{\mathcal Y}+\epsilon).
 \end{align}
 Since $p=\frac{n+2}{n-2}=2$ when $n=6$,
 from \eqref{EQUATION_e3.6},
 there exists a constant $C>0$ depending only on $n$ such that
 \begin{align}
 \label{EQUATION_e3.7}
 |N|
 <
 C(
 |a|^p\mathcal{Y}^p
 +
 |\epsilon|^p
 )
 \quad\
 (n=6).
 \end{align}
%%%%%%%%%%%%%%%%%%%%%%%%%%%%%%%%%%%%%%%%%%%%%%%%%%%%%%%%%%%%%%%%
 Equation \eqref{EQUATION_e3.5} can be rearranged as follows.
 \begin{align}
 \label{EQUATION_e3.8}
 \lambda^2
 \epsilon_t
 &=
 H_y\epsilon
 +
 \lambda
 \lambda_t
 (\Lambda_y\epsilon)
 +
 \lambda
 \tfrac{dz}{dt}
 \cdot
 \nabla_y\epsilon
 +
 \lambda
 \lambda_t
 (\Lambda_y{\sf Q})
 \\
 \nonumber
 &\quad
 -
 (
 \lambda^2
 a_t
 -
 e_0a
 )
 {\mathcal Y}
 +
 \lambda
 \lambda_t
 a
 (\Lambda_y{\mathcal Y})
 \\
 \nonumber
 &\quad
 +
 \lambda
 \tfrac{dz}{dt}
 \cdot
 \nabla_y{\sf Q}
 +
 \lambda
 a
 \tfrac{dz}{dt}
 \cdot\nabla_y
 \mathcal{Y}
 +
 N
 \quad\
 \text{for }
 t\in (t_1,t_2).
 \end{align}
 We introduce a new time variable.
 \begin{align*}
 s=s(t)=\int_0^t\tfrac{dt'}{\lambda(t')^2}.
 \end{align*}
 Since $\tfrac{\pa }{\pa t}=\tfrac{1}{\lambda^2}\tfrac{\pa}{\pa s}$,
 we rewrite \eqref{EQUATION_e3.8} in the new time variable $s$.
 \begin{align}
 \label{EQUATION_e3.9}
 \epsilon_s
 &=
 H_y\epsilon
 +
 \tfrac{\lambda_s}{\lambda}
 (\Lambda_y\epsilon)
 +
 \tfrac{1}{\lambda}
 \tfrac{dz}{ds}
 \cdot
 \nabla_y\epsilon
 +
 \tfrac{\lambda_s}{\lambda}
 (\Lambda_y{\sf Q})
 -
 (
 a_s
 -
 e_0a
 )
 {\mathcal Y}
 \\
 \nonumber
 &\quad
 +
 \tfrac{\lambda_s}{\lambda}
 a
 (\Lambda_y{\mathcal Y})
 +
 \tfrac{1}{\lambda}
 \tfrac{dz}{ds}\cdot\nabla_y
 {\sf Q}
 +
 \tfrac{a}{\lambda}
 \tfrac{dz}{ds}\cdot\nabla_y
 \mathcal{Y}
 +
 N.
 \end{align}
%%%%%%%%%%%%%%%%%%%%%%%%%%%%%%%%%%%%%%%%%%%%%%%%%%%%%%%%%%%%%%%%
 In this section,
 we study the behavior of the following quantities along the argument in \cite{Collot-Merle-Raphael}.
 \begin{itemize}
 \label{Scaling}
 \item Scaling parameter $\lambda(t)$ \quad (Lemma \ref{LEMMA_3.2} and Lemma \ref{LEMMA_3.4} - Lemma \ref{LEMMA_3.8}),
 \item Shift parameter $z(t)$ \quad (Lemma \ref{LEMMA_3.2}, Lemma \ref{LEMMA_3.4} and Lemma \ref{LEMMA_3.8}),
 \item Amplitude of unstable mode $a(t)$ \quad (Lemma \ref{LEMMA_3.1} and Lemma \ref{LEMMA_3.2}),
 \item Error norm $\|\nabla_y\epsilon(y,t)\|_2$ \quad (Lemma \ref{LEMMA_3.10}),
 \item Error norm $\|\Delta_y\epsilon(y,t)\|_2$ \quad (Lemma \ref{LEMMA_3.1}),
 \item Error norm $\|\epsilon(y,t)\|_2$ \quad
 (Lemma \ref{LEMMA_3.5} - Lemma \ref{LEMMA_3.8}),
 \item Error norm $\|\epsilon(y,t)\|_\infty$ \quad
 (Lemma \ref{LEMMA_3.9}).
 \end{itemize}
%%%%%%%%%%%%%%%%%%%%%%%%%%%%%%%%%%%%%%%%%%%%%%%%%%%%%%%%%%%%%%%%
 In the work of Collot-Merle-Rapha\"el \cite{Collot-Merle-Raphael} for $n\geq7$,
 they derived precise bounds for
 \begin{align*}
 \lambda(t),\ z(t),\ a(t),\ \|\nabla_y\epsilon(t)\|_2,\ \|\Delta_y\epsilon(t)\|_2
  \text{ and } \|\epsilon(t)\|_\infty.
 \end{align*}
 These quantities are typically central in the analysis of blowup dynamics of solutions in nonlinear PDEs.
 The key to obtaining these bounds is their detailed analysis of a coupled system of differential equations for
 \begin{center}
 $\lambda(t)$, $z(t)$, $a(t)$, $\|\nabla_y\epsilon(t)\|_2$ and $\|\Delta_y\epsilon(t)\|_2$.
 \end{center}
%%%%%%%%%%%%%%%%%%%%%%%%%%%%%%%%%%%%%%%%%%%%%%%%%%%%%%%%%%%%%%%%
 Our new approach for the case $n=6$ is to investigate a coupled system of differential equations for
 \begin{center}
 $\lambda(t)$, $z(t)$, $a(t)$, $\|\nabla_y\epsilon(t)\|_2$ and $\|\epsilon(t)\|_2$.
 \end{center}
 Therefore
 our argument requires a stronger integrability condition on the initial data,
 $u_0(x)\in L^2(\R^n)$.
% to obtain bounds for the quantities given in \eqref{B}.
%%%%%%%%%%%%%%%%%%%%%%%%%%%%%%%%%%%%%%%%%%%%%%%%%%%%%%%%%%%%%%%%
 We here remark that although Lemma \ref{LEMMA_3.3} is elementary,
 it plays a crucial role in our argument.
%%%%%%%%%%%%%%%%%%%%%%%%%%%%%%%%%%%%%%%%%%%%%%%%%%%%%%%%%%%%%%%%

 \noindent
 {\bf Notation.}
 Define
 \begin{align*}
 s_i=\int_0^{t_i}\tfrac{dt'}{\lambda(t')^2}
 \quad\
 \text{for }
 i=1,2.
 \end{align*}
 The constant $C$ represents a generic positive constant,
 which may vary from line to line, but is independent of all the parameters $\eta_1$, $\eta_2$, $t_1$, $t_2$, $s_1$ and $s_2$
 in this section.

 Since the proofs of (2.14) - (2.16) in Lemma 2.6, Lemma 2.7 and Lemma 2.9
 in \cite{Collot-Merle-Raphael} remain valid for $n=6$,
 we can establish
 Lemma \ref{LEMMA_3.1}, Lemma \ref{LEMMA_3.2} and
 Lemma \ref{LEMMA_3.10} in this section
 by using the same argument.
 However,
 for the sake of completeness,
 we provide their full proofs below.
%%%%%%%%%%%%%%%%%%%%%%%%%%%%%%%%%%%%%%%%%%%%%%%%%%%%%%%%%%%%%%%%
 \begin{lem}[Lemma 2.7 in \cite{Collot-Merle-Raphael} p.\,232 ]
 \label{LEMMA_3.1}
 Let $n=6$,
 and
 $u(x,t)$ be a solution of \eqref{equation_1.1} for $t\in(t_1,t_2)$ satisfying {\rm(a1) - (a3)} stated on {\rm p.\,\pageref{a_property}}.
 There exist constants $h_1>0$ and $C>0$,
 depending only on $n$
 such that if $\eta_1\in(0,h_1)$,
 then
 \begin{align}
 \label{EQUATION_3.10}
 \int_{s_1}^{s_2}
 \{
 \|H_y\epsilon(y,s)\|_{L_y^2(\R^n)}^2+a(s)^2
 \}
 ds
 <
 C
 \eta_1^2.
 \end{align}
 \end{lem}
%%%%%%%%%%%%%%%%%%%%%%%%%%%%%%%%%%%%%%%%%%%%%%%%%%%%%%%%%%%%%%%%
 \begin{proof}
 From the energy identity,
 we see that
 \begin{align*}
 E[u(t_1)]
 &-
 E[u(t_2)]
 =
 \int_{t_1}^{t_2}
 dt
 \int_{\R_x^n}
 u_t(x,t)^2
 dx
 \\
 &=
 \int_{t_1}^{t_2}
 dt
 \int_{\R_x^n}
 (\Delta_xu+f(u))^2
 dx
 \\
 &=
 \int_{t_1}^{t_2}
 \tfrac{dt}{\lambda^{n+2}}
 \int_{\R_x^n}
 \{
 \Delta_y({\sf Q}+a\mathcal{Y}+\epsilon)
 +f({\sf Q}+a\mathcal{Y}+\epsilon)
 \}^2
 dx
 \\
 &=
 \int_{t_1}^{t_2}
 \tfrac{dt}{\lambda^2}
 \int_{\R_y^n}
 (aH_y\mathcal{Y}+H_y\epsilon+N)^2
 dy.
 \end{align*}
 Since $H_y\mathcal{Y}=e_0\mathcal{Y}$,
 we obtain
 \begin{align}
 \label{EQUATION_3.11}
 E[u(t_1)]
 &-
 E[u(t_2)]
 =
 \int_{s_1}^{s_2}
 (
 a^2e_0^2\|\mathcal{Y}\|_{L_y^2(\R^n)}^2
 +
 \|H_y\epsilon\|_{L_y^2(\R^n)}^2
 )
 ds
 \\
 \nonumber
 & \quad
 +
 \int_{s_1}^{s_2}
 \{
 2(ae_0\mathcal{Y}+H_y\epsilon,N)_{L_y^2(\R^n)}
 +
 \|N\|_{L_y^2(\R^n)}^2
 \}
 ds.
 \end{align}
 From \eqref{EQUATION_e3.7},
 we observe that
 \begin{align*}
 \|N\|_{L_y^2(\R^n)}^2
 &<
 C|a|^{2p}
 +
 C\int_{\R_y^n}\epsilon^2\cdot|\epsilon|^{2(p-1)}dy
 \\
 &<
 C|a|^{2p}
 +
 C
 \|\epsilon^2\|_\frac{n}{n-4}
 \|\epsilon^{2(p-1)}\|_\frac{n}{4}
 \\
 &<
 C|a|^{2p}
 +
 C
 \|\epsilon\|_\frac{2n}{n-4}^2
 \|\epsilon\|_\frac{2n}{n-2}^{2(p-1)}
 \\
 &<
 C|a|^{2p}
 +
 C
 \|\Delta_y\epsilon\|_2^2\|\nabla_y\epsilon\|_2^{2(p-1)}.
 \end{align*}
 From \eqref{EQUATION_e3.3} - \eqref{EQUATION_e3.4},
 we deduce that
 \begin{align}
 \label{EQUATION_3.12}
 \|N\|_{L_y^2(\R^n)}^2
 &<
 C|a|^{2p}
 +
 C
 \eta_1^{2(p-1)}
 \|H_y\epsilon\|_2^2.
 \end{align}
 Therefore
 from \eqref{EQUATION_3.11} and \eqref{EQUATION_3.12},
 there exists $h_1>0$ depending only on $n$ such that
 if $\eta_1\in(0,h_1)$,
 then
 \begin{align}
 \label{EQUATION_3.13}
 E[u(t_1)]
 &-
 E[u(t_2)]
 >
 \tfrac{1}{2}
 \int_{s_1}^{s_2}
 (
 a^2e_0^2\|\mathcal{Y}\|_{L_y^2(\R^n)}^2
 +
 \|H_y\epsilon\|_{L_y^2(\R^n)}^2
 )
 ds.
 \end{align}
%%%%%%%%%%%%%%%%%%%%%%%%%%%%%%%%%%%%%%%%%%%%%%%%%%%%%%%%%%%%%%%%
 By the change of variable,
 $E[u(t_i)]$ ($i=1,2$) can be expressed as
 \begin{align}
 \nonumber
 E&[u(t_i)]
 =
 \int_{\R_x^n}
 (\tfrac{1}{2}|\nabla_xu(x,t_i)|^2-\tfrac{1}{p+1}|u(x,t_i)|^{p+1})
 dx
 \\
 \nonumber
 &=
 \tfrac{1}{2}
 \int_{\R_y^n}
 |\nabla_y({\sf Q}(y)+a(t_i)\mathcal{Y}(y)+\epsilon(y,t_i))|^2
 dy
 \\
 \nonumber
 &\quad
 -
 \tfrac{1}{p+1}
 \int_{\R_y^n}
 |{\sf Q}(y)+a(t_i)\mathcal{Y}(y)+\epsilon(y,t_i)|^{p+1}
 dy
 \\
 \label{EQUATION_3.14}
 &=
 E[{\sf Q}]
 -
 \tfrac{e_0}{2}
 a(t_i)^2
 \|\mathcal{Y}\|_2^2
 +
 \tfrac{1}{2}
 (-H_y\epsilon(t_i),\epsilon(t_i))_2
 -
 R_i
 \end{align}
 for $i=1,2$,
 where
 \begin{align*}
 R_i
 &=
 \tfrac{1}{p+1}
 \int_{\R_y^n}
 |{\sf Q}(y)+a(t_i)\mathcal{Y}(y)+\epsilon(y,t_i)|^{p+1}
 dy
 \\
 &
 -
 \tfrac{1}{p+1}
 \int_{\R_y^n}
 {\sf Q}(y)|^{p+1}
 dy
 -
 \int_{\R_y^n}
 {\sf Q}(y)^p\{a(t_i)\mathcal{Y}(y)+\epsilon(y,t_i)\}
 dy
 \\
 &
 -
 \int_{\R_y^n}
 p{\sf Q}(y)^{p-1}\{a(t_i)\mathcal{Y}(y)+\epsilon(y,t_i)\}^2
 dy.
 \end{align*}
 Since $p=\frac{n+2}{n-2}=2$ when $n=6$,
 it holds that
 \begin{align}
 \nonumber
 |R_i|
 &<
 C
 \int_{\R_y^n}
 |a(t_i)\mathcal{Y}(y)+\epsilon(y,t_i)|^{p+1}
 dy
 \\
 \label{EQUATION_3.15}
 &<
 C|a(t_i)|^{p+1}+C\|\nabla_y\epsilon(t_i)\|_2^{p+1}.
 \end{align}
 Combining \eqref{EQUATION_3.13} - \eqref{EQUATION_3.15} together with \eqref{EQUATION_e3.4},
 we obtain the desired estimate.
 \end{proof}
%%%%%%%%%%%%%%%%%%%%%%%%%%%%%%%%%%%%%%%%%%%%%%%%%%%%%%%%%%%%%%%%

\subsection{Coupled system of differential equations for
$\lambda(t)$, $z(t)$, $a(t)$ and $\|\epsilon(t)\|_2$}
%%%%%%%%%%%%%%%%%%%%%%%%%%%%%%%%%%%%%%%%%%%%%%%%%%%%%%%%%%%%%%%%
 \begin{lem}[(2.14) - (2.16) in Lemma 2.6 in \cite{Collot-Merle-Raphael} p.\,225 - p.\,226]
 \label{LEMMA_3.2}
 Let $n=6$,
 and
 $u(x,t)$ be as in Lemma {\rm\ref{LEMMA_3.1}}.
 There exist constants $h_1>0$ and $C>0$,
 depending only on $n$
 such that if $\eta_1\in(0,h_1)$,
 then
 \begin{align}
 \label{EQ_3.16}
 |a_s-e_0a|
 &<
 Ca^2
 +
 C
 \min\{
 \|\nabla_y\epsilon\|_2^2,
 \|\Delta_y\epsilon\|_2^2
 \}
 \quad\
 \text{\rm for }
 s\in(s_1,s_2),
 \\
 \label{EQ_3.17}
 \left|
 \frac{\lambda_s}{\lambda}
 \right|
 &<
 Ca^2
 +
 C
 \min\{
 \|\nabla_y\epsilon\|_2,
 \|\Delta_y\epsilon\|_2
 \}
 \quad\
 \text{\rm for }
 s\in(s_1,s_2),
 \\
 \label{EQ_3.18}
 \left|
 \frac{1}{\lambda}
 \frac{dz}{ds}
 \right|
 &<
 Ca^2
 +
 C
 \min\{
 \|\nabla_y\epsilon\|_2,
 \|\Delta_y\epsilon\|_2
 \}
 \quad\
 \text{\rm for }
 s\in(s_1,s_2).
 \end{align}
 \end{lem}
%%%%%%%%%%%%%%%%%%%%%%%%%%%%%%%%%%%%%%%%%%%%%%%%%%%%%%%%%%%%%%%%
 \begin{proof}
 By multiplying equation \eqref{EQUATION_e3.9} by $\mathcal{Y}$ and $\Psi_j$ ($j=0,\cdots,n$),
 and using \eqref{EQUATION_e3.2},
 we obtain
 \begin{align}
 \label{EQ_3.19}
 0
 &=
 0
 +
 \tfrac{\lambda_s}{\lambda}
 (\Lambda_y\epsilon,\mathcal{Y})_2
 +
 (\tfrac{1}{\lambda}
 \tfrac{dz}{ds}
 \cdot\nabla_y\epsilon,\mathcal{Y})_2
 +
 0
 -
 (
 a_s
 -
 e_0a
 )
 \|{\mathcal Y}\|_2^2
 \\
 \nonumber
 &\quad
 +
 \tfrac{\lambda_s}{\lambda}
 a
 (\Lambda_y{\mathcal Y},{\mathcal Y})_2
 +
 0
 +
 0
 +
 (N,{\mathcal Y})_2,
 \\
 \label{EQ_3.20}
 0
 &=
 (H_y\epsilon,\Psi_0)
 +
 \tfrac{\lambda_s}{\lambda}
 (\Lambda_y\epsilon,\Psi_0)_2
 +
 ( \tfrac{1}{\lambda}
 \tfrac{dz}{ds}
 \cdot
 \nabla_y\epsilon,\Psi_0)_2
 +
 \tfrac{\lambda_s}{\lambda}
 (\Lambda_y{\sf Q},\Psi_0)_2
 \\
 \nonumber
 &\quad
 -
 0
 +
 \tfrac{\lambda_s}{\lambda}
 a
 (\Lambda_y{\mathcal Y},\Psi_0)_2
 +
 0
 +
 0
 +
 (N,\Psi_0)_2,
 \\
 \label{EQ_3.21}
 0
 &=
 (H_y\epsilon,\Psi_j)
 +
 \tfrac{\lambda_s}{\lambda}
 (\Lambda_y\epsilon,\Psi_j)_2
 +
 (\tfrac{1}{\lambda}
 \tfrac{dz}{ds}
 \cdot
 \nabla_y\epsilon,\Psi_j)_2
 +
 0
 -
 0
 \\
 \nonumber
 &\quad
 +
 0
 +
 \underbrace{
 (
 \tfrac{1}{\lambda}
 \tfrac{dz}{ds}\cdot
 \nabla_y
 {\sf Q},
 \Psi_j)_2
 }_{=\tfrac{1}{\lambda}
 \tfrac{dz_j}{ds}(
 \pa_{y_j}
 {\sf Q},
 \Psi_j)_2}
 +
 \underbrace{
 (
 \tfrac{a}{\lambda}
 \tfrac{dz}{ds}\cdot
 \nabla_y
 \mathcal{Y},
 \Psi_j
 )_2
 }_{=\tfrac{a}{\lambda}
 \tfrac{dz_j}{ds}(
 \pa_{y_j}
 \mathcal{Y},
 \Psi_j)_2}
 +
 (N,\Psi_j)_2
 \\
 \nonumber
 &
 \text{for } j=1,\cdots,n.
 \end{align}
%%%%%%%%%%%%%%%%%%%%%%%%%%%%%%%%%%%%%%%%%%%%%%%%%%%%%%%%%%%%%%%%
 Since $p=\frac{n+2}{n-2}$ when $n=6$,
 from \eqref{EQUATION_e3.7} and the Hardy inequalities (see Lemma \ref{LEMMA_2.1}),
 we easily see that
 \begin{align*}
 |(N,{\mathcal Y})_2|
 &<
 C|a|^2
 +
 C(\epsilon^2,{\mathcal Y})_2
 \\
 &<
 C|a|^2
 +
 C\min\{\|\nabla_y\epsilon\|_2^2,\|\nabla_y\epsilon\|_2\|\Delta_y\epsilon\|_2,\|\Delta_y\epsilon\|_2^2\},
 \\
 |(N,\Psi_j)_2|
 &<
 C|a|^2
 +
 C\min\{\|\nabla_y\epsilon\|_2^2,\|\nabla_y\epsilon\|_2\|\Delta_y\epsilon\|_2,\|\Delta_y\epsilon\|_2^2\}
 \\
 &
 \text{for } j=0,\cdots,n.
 \end{align*}
 For simplicity,
 we put
 \begin{align*}
 X_1
 &=
 \min\{\|\nabla_y\epsilon\|_2,\|\Delta_y\epsilon\|_2\},
 \\
 X_2
 &=
 \min\{\|\nabla_y\epsilon\|_2^2,\|\nabla_y\epsilon\|_2\|\Delta_y\epsilon\|_2,\|\Delta_y\epsilon\|_2^2\}.
 \end{align*}
 From \eqref{EQ_3.19},
 we see that
 \begin{align}
 \label{EQ_3.22}
 |a_s-e_0a|
 &<
 C(
 |\tfrac{\lambda_s}{\lambda}|
 X_1
 +
 |\tfrac{1}{\lambda}\tfrac{dz}{ds}|
 X_1
 +
 |\tfrac{\lambda_s}{\lambda}||a|
 +
 a^2
 +
 X_2
 ).
 \end{align}
 Furthermore
 \eqref{EQ_3.20} immediately implies
 \begin{align*}
 |
 \tfrac{\lambda_s}{\lambda}
 |
 \cdot
 |
 &(\Lambda_y\epsilon,\Psi_0)_2
 +
 (\Lambda_y{\sf Q},\Psi_0)_2
 +
 a(\Lambda_y\mathcal{Y},\Psi_0)_2
 |
 \\
 \nonumber
 &<
 C(
 X_1
 +
 |\tfrac{1}{\lambda}\tfrac{dz}{ds}|
 X_1
 +
 a^2
 +
 X_2
 ).
 \end{align*}
 We note that
 there exists $h_1>0$ depends only on $n$ such that
 if $\eta_1\in(0,h_1)$,
 then we have
 \[
 |
 (\Lambda_y\epsilon,\Psi_0)_2
 +
 (\Lambda_y{\sf Q},\Psi_0)_2
 +
 a(\Lambda_y\mathcal{Y},\Psi_0)_2
 |
 >
 \tfrac{1}{2}(\Lambda_y{\sf Q},\Psi_0)_2.
 \]
 Therefore
 it follows that
 \begin{align}
 \label{EQ_3.23}
 |
 \tfrac{\lambda_s}{\lambda}
 |
 <
 C(
 X_1
 +
 |\tfrac{1}{\lambda}\tfrac{dz}{ds}|
 X_1
 +
 a^2
 +
 X_2
 )
 \quad\
 \text{if }
 \eta_1\in(0,h_1).
 \end{align}
 To rewrite \eqref{EQ_3.21} in a matrix form,
 we define a $n\times n$ matrix $D$ by
 \begin{align*}
 D_{ij}
 =
 \begin{cases}
 (\pa_{y_j}\epsilon,\Psi_j)_2
 +
 \tfrac{1}{n}\int_{\R_y^n}\chi_M
 (|\nabla_y{\sf Q}|^2+a\nabla_y\mathcal{Y}\cdot\nabla_y{\sf Q})dy
 & \text{if } i=j,
 \\
 (\pa_{y_j}\epsilon,\Psi_i)_2
 & \text{if } i\not=j,
 \end{cases}
 \end{align*}
 here we used the fact that
 $(\pa_{y_j}{\sf Q},\Psi_j)_2=\frac{1}{n}\int_{\R_y^n}\chi_M|\nabla_y{\sf Q}|^2dy$ and
 $(\pa_{y_j}\mathcal{Y},\Psi_j)_2=\frac{1}{n}$ $\int_{\R_y^n}\chi_M\nabla_y\mathcal{Y}\cdot\nabla_y{\sf Q}dy$.
 From the definition of $D$,
 there exists $h_1'\in(0,h_1)$ depending only on $n$ such that
 if $\eta_1\in(0,h_1')$,
 the inverse matrix $D^{-1}$ satisfies
 \[
 \sup_{1\leq i,j\leq n}
 |[D^{-1}]_{ij}|
 <
 \tfrac{2n}{\int_{\R_y^n}\chi_M|\nabla_y{\sf Q}|^2dy}.
 \]
 Since \eqref{EQ_3.21} is written as
 \begin{align*}
 D
 \begin{pmatrix}
 \tfrac{1}{\lambda}
 \tfrac{dz_1}{ds}
 \\
 \tfrac{1}{\lambda}
 \tfrac{dz_2}{ds}
 \\
 \vdots
 \\
 \tfrac{1}{\lambda}
 \tfrac{dz_n}{ds}
 \end{pmatrix}
 =
 -
 \begin{pmatrix}
 (H_y\epsilon,\Psi_1)_2+\frac{\lambda_s}{\lambda}(\Lambda_y\epsilon,\Psi_1)_2+(N,\Psi_1)_2
 \\
 (H_y\epsilon,\Psi_2)_2+\frac{\lambda_s}{\lambda}(\Lambda_y\epsilon,\Psi_2)_2+(N,\Psi_2)_2
 \\
 \vdots
 \\
 (H_y\epsilon,\Psi_n)_2+\frac{\lambda_s}{\lambda}(\Lambda_y\epsilon,\Psi_n)_2+(N,\Psi_n)_2
 \end{pmatrix},
 \end{align*}
 we deduce that
 \begin{align}
 \label{EQ_3.24}
 |\tfrac{1}{\lambda}\tfrac{dz}{ds}|
 <
 C(
 X_1+|\tfrac{\lambda_s}{\lambda}|X_1+a^2+X_2
 ).
 \end{align}
 Combining \eqref{EQ_3.22} - \eqref{EQ_3.24},
 we obtain the conclusion.
 \end{proof}
%%%%%%%%%%%%%%%%%%%%%%%%%%%%%%%%%%%%%%%%%%%%%%%%%%%%%%%%%%%%%%%%

%%%%%%%%%%%%%%%%%%%%%%%%%%%%%%%%%%%%%%%%%%%%%%%%%%%%%%%%%%%%%%%%
 \begin{lem}
 \label{LEMMA_3.3}
 Let $n\geq3$.
 Then
 \begin{align}
 \label{EQ_3.25}
 (\Lambda_yv,\Lambda_y{\sf Q})_2
 =
 \tfrac{n-6}{2}
 (v,\Lambda_y{\sf Q})_2
 +
 (v,R)_2
 \quad\
 \text{\rm for }
 v(y)\in H_y^1(\R^n).
 \end{align}
 The function $R(y)\in C^\infty(\R^n)$ is independent of $v(y)$,
 and satisfies
 \[
 |R(y)|+(1+|y|)|\nabla_yR(y)|
 <
 C(1+|y|)^{-n}
 \quad\
 \text{\rm for }
 y\in\R^n.
 \]
 \end{lem}
%%%%%%%%%%%%%%%%%%%%%%%%%%%%%%%%%%%%%%%%%%%%%%%%%%%%%%%%%%%%%%%%
 \begin{proof}
 For simplicity,
 we set
 \begin{align*}
 r=|y|.
 \end{align*}
 Let $v\in H^1(\R^n)$.
 It is straightforward to see that
 \begin{align}
 \nonumber
 (\Lambda_yv,\Lambda_y{\sf Q})_2
 &=
 \tfrac{n-2}{2}(v,\Lambda_y{\sf Q})_2
 +
 (y\cdot\nabla_yv,\Lambda_y{\sf Q})_2
 \\
 \nonumber
 &=
 \tfrac{n-2}{2}(v,\Lambda_y{\sf Q})_2
 -
 (v,\nabla_y\cdot(y\Lambda_y{\sf Q}))_2
 \\
 \nonumber
 &=
 (v,-(\tfrac{n+2}{2}+y\cdot\nabla_y)\Lambda_y{\sf Q})_2
 \\
 \label{EQ_3.26}
 &=
 (v,-(\tfrac{n+2}{2}+r\pa_r)\Lambda_y{\sf Q})_2.
 \end{align}
 Note that ${\sf Q}(y)$ is expressed as
 \begin{align}
 \nonumber
 {\sf Q}(y)
 &=
 (
 1+\tfrac{r^2}{n(n-2)}
 )^{-\frac{n-2}{2}}
 =
 r^{-(n-2)}
 (
 \tfrac{1}{r^2}+\tfrac{1}{n(n-2)}
 )^{-\frac{n-2}{2}}
 \\
 \nonumber
 &=
 (
 \tfrac{1}{\omega_n}
 )^{-\frac{n-2}{2}}
 r^{-(n-2)}
 +
 r^{-(n-2)}
 \underbrace{
 \{
 (
 \tfrac{1}{r^2}+\tfrac{1}{\omega_n}
 )^{-\frac{n-2}{2}}
 -
 (
 \tfrac{1}{\omega_n}
 )^{-\frac{n-2}{2}}
 \}
 }_{=h(r)}
 \\
 \label{EQ_3.27}
 &=
 \omega_n^\frac{n-2}{2}
 r^{-(n-2)}
 +
 h(r)r^{-(n-2)}
 \quad\
 \text{with }
 \omega_n=n(n-2).
 \end{align}
 From \eqref{EQ_3.27},
 we have
 \begin{align}
 \nonumber
 \Lambda_y{\sf Q}
 &=
 (
 \tfrac{n-2}{2}+r\pa_r
 )
 \{
 \omega_n^\frac{n-2}{2}
 r^{-(n-2)}
 +
 hr^{-(n-2)}
 \}
 \\
 \label{EQ_3.28}
 &=
 -
 \tfrac{n-2}{2}
 \omega_n^\frac{n-2}{2}
 r^{-(n-2)}
 +
 \underbrace{
 (
 -\tfrac{n-2}{2}h+r\pa_rh
 )
 }_{=h_1(r)}
 r^{-(n-2)},
 \\
 \nonumber
 r\pa_r
 (\Lambda_y{\sf Q})
 &=
 r\pa_r
 \{
 -
 \tfrac{n-2}{2}
 \omega_n^\frac{n-2}{2}
 r^{-(n-2)}
 +
 (
 -\tfrac{n-2}{2}h+r\pa_rh
 )
 r^{-(n-2)}
 \}
 \\
 \nonumber
 &=
 \tfrac{(n-2)^2}{2}
 \omega_n^\frac{n-2}{2}
 r^{-(n-2)}
 \\
 \nonumber
 &\quad
 +
 \underbrace{
 (
 \tfrac{(n-2)^2}{2}h
 -
 \tfrac{3(n-2)}{2}r\pa_rh
 +
 (r\pa_r)^2h
 )
 }_{=h_2(r)}
 r^{-(n-2)}
 \\
 \nonumber
 &=
 -(n-2)
 \cdot
 (-
 \tfrac{n-2}{2})
 \omega_n^\frac{n-2}{2}
 r^{-(n-2)}
 +
 h_2(r)
 r^{-(n-2)}
 \\
 \nonumber
 &=
 -(n-2)(\Lambda_y{\sf Q}
 -
 h_1(r)r^{-(n-2)})
 +
 h_2(r)
 r^{-(n-2)}
 \\
 \nonumber
 &=
 -(n-2)\Lambda_y{\sf Q}
 +
 \{
 (n-2)h_1(r)
 +
 h_2(r)
 \}
 r^{-(n-2)}.
 \end{align}
 Therefore
 we obtain
 \begin{align*}
 -(\tfrac{n+2}{2}+r\pa_r)
 \Lambda_y{\sf Q}
 &=
 -
 \tfrac{n+2}{2}
 \Lambda_y{\sf Q}
 +
 (n-2)
 \Lambda_y{\sf Q}
 \\
 &\quad
 -
 \{
 (n-2)h_1(r)
 +
 h_2(r)
 \}
 r^{-(n-2)}
 \\
 &=
 \tfrac{n-6}{2}
 \Lambda_y{\sf Q}
 -
 \{
 (n-2)h_1(r)
 +
 h_2(r)
 \}
 r^{-(n-2)}.
 \end{align*}
 Since $|h_1(r)|+|h_2(r)|+|r\pa_rh_1(r)|+|r\pa_rh_2(r)|<Cr^{-2}$ for $r>1$,
 the conclusion follows from \eqref{EQ_3.26}.
 \end{proof}
%%%%%%%%%%%%%%%%%%%%%%%%%%%%%%%%%%%%%%%%%%%%%%%%%%%%%%%%%%%%%%%%

%%%%%%%%%%%%%%%%%%%%%%%%%%%%%%%%%%%%%%%%%%%%%%%%%%%%%%%%%%%%%%%%
 \begin{lem}[(2.17) - (2.18) in Lemma 2.6 in \cite{Collot-Merle-Raphael} p.\,226]
 \label{LEMMA_3.4}
 Let $n=6$,
 and
 $u(x,t)$ be as in Lemma {\rm\ref{LEMMA_3.1}}.
 There exist constants $h_1>0$ and $C>0$,
 depending only on $n$
 such that if $\eta_1\in(0,h_1)$,
 then
 \begin{align}
 \label{EQ_3.29}
 &
 \left|
 \frac{d}{ds}
 \left(
 \log\lambda
 +
 \frac{n(n-2)}{4}
 \sum_{j=1}^n
 \frac{(\epsilon,\pa_{y_j}{\sf Q})_2^2}
 {\|\nabla_y{\sf Q}\|_2^2\|\Lambda_y{\sf Q}\|_2^2}
 -
 \frac{(\epsilon,\Lambda_y{\sf Q})_2}{\|\Lambda_y{\sf Q}\|_2^2}
 \right)
 \right|
 \\
 \nonumber
 &<
 C
 (a^2+\|\Delta_y\epsilon\|_2^2)
 \quad\
 \text{\rm for }
 s\in(s_1,s_2)
 \end{align}
 and
 \begin{align}
 \label{EQ_3.30}
 &
 \left|
 \frac{1}{\lambda}
 \frac{dz_j}{ds}
 -
 \frac{n(\epsilon_s,\pa_{y_j}{\sf Q})_2}{\|\nabla_y{\sf Q}\|_2^2}
 +
 \frac{\lambda_s}{\lambda}
 \frac{
 n(\Lambda_y\epsilon,\pa_{y_j}{\sf Q})_2
 }{\|\nabla_y{\sf Q}\|_2^2}
 \right|
 \\
 \nonumber
 &<
 C(a^2+\|\Delta_y\epsilon\|_2^2)
 \quad\
 \text{\rm for }
 s\in(s_1,s_2),
 \quad
 j\in\{1,\cdots,n\}.
 \end{align}
 \end{lem}
%%%%%%%%%%%%%%%%%%%%%%%%%%%%%%%%%%%%%%%%%%%%%%%%%%%%%%%%%%%%%%%%
 \begin{proof}
 We multiply equation \eqref{EQUATION_e3.9} by $\Lambda_y{\sf Q}$ to get
 \begin{align}
 \nonumber
 (\epsilon_s,\Lambda_y{\sf Q})_2
 &=
 0
 +
 \tfrac{\lambda_s}{\lambda}
 (\Lambda_y\epsilon,\Lambda_y{\sf Q})_2
 +
 (\tfrac{1}{\lambda}
 \tfrac{dz}{ds}
 \cdot\nabla_y\epsilon,\Lambda_y{\sf Q})_2
 +
 \tfrac{\lambda_s}{\lambda}
 \|\Lambda_y{\sf Q}\|_2^2
 \\
 \nonumber
 &\quad
 -
 0
 +
 \tfrac{\lambda_s}{\lambda}
 a
 (\Lambda_y{\mathcal Y},\Lambda_y{\sf Q})_2
 +
 0
 +
 0
 +
 (N,\Lambda_y{\sf Q})_2
 \\
 \label{EQ_3.31}
 &=
 \tfrac{\lambda_s}{\lambda}
 (\Lambda_y\epsilon,\Lambda_y{\sf Q})_2
 +
 \sum_{j=1}^n
 \tfrac{1}{\lambda}
 \tfrac{dz_j}{ds}
 (\pa_{y_j}\epsilon,\Lambda_y{\sf Q})_2
 +
 \tfrac{\lambda_s}{\lambda}
 \|\Lambda_y{\sf Q}\|_2^2
 \\
 \nonumber
 &\quad
 +
 \tfrac{\lambda_s}{\lambda}
 a
 (\Lambda_y{\mathcal Y},\Lambda_y{\sf Q})_2
 +
 (N,\Lambda_y{\sf Q})_2.
 \end{align}
 Note that
 due to the condition (a1),
 we have
 \[
 \epsilon(y,t)\in C([s_1,s_2];H^1(\R^n)).
 \]
 From Lemma \ref{LEMMA_3.3},
 we observe that $(\Lambda_y\epsilon,\Lambda_y{\sf Q})_2=(v,R)_2$ when $n=6$.
 Hence \eqref{EQ_3.31} can be simplified as follows.
 \begin{align}
 \label{EQ_3.32}
 &(\epsilon_s,\Lambda_y{\sf Q})_2
 =
 \underbrace{
 \tfrac{\lambda_s}{\lambda}
 (\Lambda_y\epsilon,R)_2
 }_{=J_1}
 +
 \sum_{j=1}^n
 \tfrac{1}{\lambda}
 \tfrac{dz_j}{ds}
 (\pa_{y_j}\epsilon,\Lambda_y{\sf Q})_2
 \\
 \nonumber
 &\quad
 +
 \tfrac{\lambda_s}{\lambda}
 \|\Lambda_y{\sf Q}\|_2^2
 +
 \underbrace{
 \tfrac{\lambda_s}{\lambda}
 a
 (\Lambda_y{\mathcal Y},\Lambda_y{\sf Q})_2
 }_{=J_2}
 +
 \underbrace{
 (N,\Lambda_y{\sf Q})_2
 }_{=J_3}.
 \end{align}
 In the same way,
 multiplying \eqref{EQUATION_e3.9} by $\pa_{y_j}{\sf Q}$ ($j=1,\cdots,n$),
 we get
 \begin{align}
 \nonumber
 (\epsilon_s&,\pa_{y_j}{\sf Q})_2
 =
 0
 +
 \tfrac{\lambda_s}{\lambda}
 (\Lambda_y\epsilon,\pa_{y_j}{\sf Q})_2
 +
 \sum_{i=1}^n
 \tfrac{1}{\lambda}
 \tfrac{dz_i}{ds}
 (\pa_{y_i}\epsilon,\pa_{y_j}{\sf Q})_2
 +
 0
 -
 0
 \\
 \nonumber
 &\quad
 +
 0
 +
 \tfrac{1}{\lambda}
 \tfrac{dz_j}{ds}
 \underbrace{
 (\pa_{y_j}{\sf Q},\pa_{y_j}{\sf Q})_2
 }_{=\frac{1}{n}\|\nabla_y{\sf Q}\|_2^2}
 +
 \tfrac{a}{\lambda}
 \tfrac{dz_j}{ds}
 \underbrace{
 (\pa_{y_j}\mathcal{Y},\pa_{y_j}{\sf Q})_2
 }_{=\frac{1}{n}(\nabla_y\mathcal{Y},\nabla_y{\sf Q})_2}
 \\
 \nonumber
 &\quad
 +
 (N,\pa_{y_j}{\sf Q})_2
 \\
 \label{EQ_3.33}
 &=
 \tfrac{\lambda_s}{\lambda}
 (\Lambda_y\epsilon,\pa_{y_j}{\sf Q})_2
 +
 \sum_{i=1}^n
 \tfrac{1}{\lambda}
 \tfrac{dz_i}{ds}
 (\pa_{y_i}\epsilon,\pa_{y_j}{\sf Q})_2
 +
 \tfrac{1}{\lambda}
 \tfrac{dz_j}{ds}
 \tfrac{\|\nabla_y{\sf Q}\|_2^2}{n}
 \\
 \nonumber
 &\quad
 +
 \tfrac{a}{\lambda}
 \tfrac{dz_j}{ds}
 \tfrac{(\nabla_y\mathcal{Y},\nabla_y{\sf Q})_2}{n}
 +
 (N,\pa_{y_j}{\sf Q})_2.
 \end{align}
 This implies
 \begin{align}
 \label{EQ_3.34}
 \tfrac{1}{\lambda}
 \tfrac{dz_j}{ds}
 &=
 \tfrac{n(\epsilon_s,\pa_{y_j}{\sf Q})_2}{\|\nabla_y{\sf Q}\|_2^2}
 -
 \tfrac{\lambda_s}{\lambda}
 \tfrac{
 n(\Lambda_y\epsilon,\pa_{y_j}{\sf Q})_2
 }{\|\nabla_y{\sf Q}\|_2^2}
 -
 \sum_{i=1}^n
 \tfrac{1}{\lambda}
 \tfrac{dz_i}{ds}
 \tfrac{n(\pa_{y_i}\epsilon,\pa_{y_j}{\sf Q})_2}
 {\|\nabla_y{\sf Q}\|_2^2}
 \\
 \nonumber
 &\quad
 -
 \tfrac{a}{\lambda}
 \tfrac{dz_j}{ds}
 \tfrac{(\nabla_y\mathcal{Y},\nabla_y{\sf Q})_2}{\|\nabla_y{\sf Q}\|_2^2}
 -
 \tfrac{n(N,\pa_{y_j}{\sf Q})_2}
 {\|\nabla_y{\sf Q}\|_2^2}
 \qquad
 (j=1,\cdots,n).
 \end{align}
%%%%%%%%%%%%%%%%%%%%%%%%%%%%%%%%%%%%%%%%%%%%%%%%%%%%%%%%%%%%%%%%
 By using \eqref{EQ_3.34},
 we compute the second term on the right hand side of \eqref{EQ_3.32}.
 \begin{align}
 \label{EQ_3.35}
 \tfrac{1}{\lambda}
 \tfrac{dz_j}{ds}
 &(\pa_{y_j}\epsilon,\Lambda_y{\sf Q})_2
 =
 \tfrac{n(\epsilon_s,\pa_{y_j}{\sf Q})_2\cdot(\pa_{y_j}\epsilon,\Lambda_y{\sf Q})_2}{\|\nabla_y{\sf Q}\|_2^2}
 +
 \underbrace{
 (-
 \tfrac{\lambda_s}{\lambda}
 )
 \tfrac{
 n(\Lambda_y\epsilon,\pa_{y_j}{\sf Q})_2\cdot(\pa_{y_j}\epsilon,\Lambda_y{\sf Q})_2}{\|\nabla_y{\sf Q}\|_2^2}
 }_{=I_1}
 \\
 \nonumber
 &
 +
 \underbrace{
 \sum_{i=1}^n
 (-\tfrac{1}{\lambda})
 \tfrac{dz_i}{ds}
 \tfrac{n(\pa_{y_i}\epsilon,\pa_{y_j}{\sf Q})_2\cdot(\pa_{y_j}\epsilon,\Lambda_y{\sf Q})_2}
 {\|\nabla_y{\sf Q}\|_2^2}
 }_{=I_2}
 +
 \underbrace{
 (-\tfrac{a}{\lambda})
 \tfrac{dz_j}{ds}
 \tfrac{(\nabla_y\mathcal{Y},\nabla_y{\sf Q})_2\cdot(\pa_{y_j}\epsilon,\Lambda_y{\sf Q})_2}{\|\nabla_y{\sf Q}\|_2^2}
 }_{=I_3}
 \\
 \nonumber
 &
 +
 \underbrace{
 \tfrac{n(-N,\pa_{y_j}{\sf Q})_2\cdot(\pa_{y_j}\epsilon,\Lambda_y{\sf Q})_2}
 {\|\nabla_y{\sf Q}\|_2^2}
 }_{=I_4}
 \quad\
 (j=1,\cdots,n).
 \end{align}
%%%%%%%%%%%%%%%%%%%%%%%%%%%%%%%%%%%%%%%%%%%%%%%%%%%%%%%%%%%%%%%%
 From \eqref{EQ_3.27} - \eqref{EQ_3.28},
 it is straightforward to see that
 \begin{align}
 \nonumber
 \pa_{y_j}{\sf Q}
 &=
 \tfrac{y_j}{r}\pa_r{\sf Q}
 \\
 \nonumber
 &=
 \tfrac{y_j}{r}\pa_r
 (\omega_n^\frac{n-2}{2}r^{-(n-2)}+hr^{-(n-2)})
 \\
 \nonumber
 &=
 -(n-2)
 \omega_n^\frac{n-2}{2}
 y_j
 r^{-n}
 +
 \underbrace{
 (r\pa_rh-(n-2)h)
 }_{=h_3(r)}
 y_jr^{-n},
%%%%%%%%%%%%%%%%%%%%%%%%%%%%%%%%%%%%%%%%%%%%%%%%%%%%%%%%%%%%%%%%
 \\
 \nonumber
 \pa_{y_j}
 \Lambda_y{\sf Q}
 &=
 \pa_{y_j}
 (-\tfrac{n-2}{2}\omega_nr^{-(n-2)}+h_1r^{-(n-2)})
 \\
 \nonumber
 &=
 \tfrac{y_n}{r}\pa_r
 (-\tfrac{n-2}{2}\omega_nr^{-(n-2)}+h_1r^{-(n-2)})
 \\
 \nonumber
 &=
 \tfrac{(n-2)^2}{2}\omega_ny_nr^{-n}
 +
 \underbrace{(r\pa_rh_1-(n-2)h_1)}_{=h_4(r)}
 y_jr^{-n}
 \\
 \nonumber
 &=
 -
 \tfrac{n-2}{2}
 (\pa_{y_j}{\sf Q}
 -
 h_3(r)
 y_jr^{-n}
 +
 h_4(r)
 y_jr^{-n}
 \\
 \label{EQ_3.36}
 &=
 -
 \tfrac{n-2}{2}
 \pa_{y_j}{\sf Q}
 +
 \underbrace{
 \{
 \tfrac{n-2}{2}
 h_3(r)
 +
 h_4(r)
 \}
 y_jr^{-n}
 }_{=k(r)}.
 \end{align}
 Here we can easily check that
 \begin{align}
 \label{EQ_3.37}
 |k(r)|
 <
 Cr^{-(n+1)}.
 \end{align}
 Relation \eqref{EQ_3.36} implies
 \begin{align}
 \label{EQ_3.38}
 (\pa_{y_j}\epsilon,\Lambda_y{\sf Q})_2
 &=
 -
 (\epsilon,\pa_{y_j}\Lambda_y{\sf Q})_2
 =
 \tfrac{n-2}{2}
 (\epsilon,\pa_{y_j}{\sf Q})_2
 -
 (\epsilon,k)_2.
 \end{align}
 Using \eqref{EQ_3.38},
 we arrange the first term on the right hand side of \eqref{EQ_3.35}.
 \begin{align*}
 \tfrac{n(\epsilon_s,\pa_{y_j}{\sf Q})_2\cdot(\pa_{y_j}\epsilon,\Lambda_y{\sf Q})_2}
 {\|\nabla_y{\sf Q}\|_2^2}
 &=
 \tfrac{n(n-2)}{2}
 \tfrac{(\epsilon_s,\pa_{y_j}{\sf Q})_2\cdot(\epsilon,\pa_{y_j}{\sf Q})_2}
 {\|\nabla_y{\sf Q}\|_2^2}
 -
 \tfrac{n(\epsilon_s,\pa_{y_j}{\sf Q})_2(\epsilon,k)_2}
 {\|\nabla_y{\sf Q}\|_2^2}
 \\
 &=
 \tfrac{d}{ds}
 (
 \tfrac{n(n-2)}{4}
 \tfrac{(\epsilon,\pa_{y_j}{\sf Q})_2^2}
 {\|\nabla_y{\sf Q}\|_2^2}
 )
 +
 \underbrace{
 \tfrac{-n(\epsilon_s,\pa_{y_j}{\sf Q})_2\cdot(\epsilon,k)_2}
 {\|\nabla_y{\sf Q}\|_2^2}
 }_{=I_5}.
 \end{align*}
%%%%%%%%%%%%%%%%%%%%%%%%%%%%%%%%%%%%%%%%%%%%%%%%%%%%%%%%%%%%%%%%
 Hence
 we can rewrite \eqref{EQ_3.35} as
 \begin{align}
 \label{EQ_3.39}
 \tfrac{1}{\lambda}
 \tfrac{dz_j}{ds}
 (\pa_{y_j}\epsilon,\Lambda_y{\sf Q})_2
 =
 \tfrac{d}{ds}
 (
 \tfrac{n(n-2)}{4}
 \tfrac{(\epsilon,\pa_{y_j}{\sf Q})_2^2}
 {\|\nabla_y{\sf Q}\|_2^2}
 )
 +
 \sum_{i=1}^5
 I_i
 \qquad
 (j=1,\cdots,n).
 \end{align}
%%%%%%%%%%%%%%%%%%%%%%%%%%%%%%%%%%%%%%%%%%%%%%%%%%%%%%%%%%%%%%%%
 Substituting \eqref{EQ_3.39} into \eqref{EQ_3.32},
 we obtain
 \begin{align*}
 (\epsilon_s,\Lambda_y{\sf Q})_2
 &=
 \sum_{j=1}^n
 \tfrac{1}{\lambda}
 \tfrac{dz_j}{ds}
 (\pa_{y_j}\epsilon,\Lambda_y{\sf Q})_2
 +
 \tfrac{\lambda_s}{\lambda}
 \|\Lambda_y{\sf Q}\|_2^2
 +
 \sum_{i=1}^3
 J_i
 \\
 &=
 \tfrac{d}{ds}
 (
 \tfrac{n(n-2)}{4}
 \tfrac{(\epsilon,\pa_{y_j}{\sf Q})_2^2}
 {\|\nabla_y{\sf Q}\|_2^2}
 )
 +
 \tfrac{\lambda_s}{\lambda}
 \|\Lambda_y{\sf Q}\|_2^2
 +
 \sum_{i=1}^3
 J_i
 +
 \sum_{i=1}^5
 I_i.
 \end{align*}
%%%%%%%%%%%%%%%%%%%%%%%%%%%%%%%%%%%%%%%%%%%%%%%%%%%%%%%%%%%%%%%%
 This immediately implies
 \begin{align}
 \label{EQ_3.40}
 \frac{d}{ds}
 &
 \left\{
 (\epsilon,\Lambda_y{\sf Q})_2
 -
 \tfrac{n(n-2)}{4}
 \sum_{j=1}^n
 \tfrac{(\epsilon,\pa_{y_j}{\sf Q})_2^2}
 {\|\nabla_y{\sf Q}\|_2^2}
 -
 \log\lambda
 \|\Lambda_y{\sf Q}\|_2^2
 \right\}
 \\
 \nonumber
 &=
 \sum_{i=1}^3
 J_i
 +
 \sum_{i=1}^5
 I_i.
 \end{align}
%%%%%%%%%%%%%%%%%%%%%%%%%%%%%%%%%%%%%%%%%%%%%%%%%%%%%%%%%%%%%%%%
 We collect $J_i$ ($i=1,2,3$) and $I_i$ ($i=1,\cdots,5$)  below.
 \begin{itemize}
 \item
 $J_1
 =
 \tfrac{\lambda_s}{\lambda}
 (\Lambda_y\epsilon,R)_2$,
 \item
 $J_2
 =
 \tfrac{\lambda_s}{\lambda}
 a
 (\Lambda_y{\mathcal Y},\Lambda_y{\sf Q})_2$,
 \item
 $J_3
 =
 (N,\Lambda_y{\sf Q})_2$,
 \end{itemize}
%%%%%%%%%%%%%%%%%%%%%%%%%%%%%%%%%%%%%%%%%%%%%%%%%%%%%%%%%%%%%%%%
 \begin{itemize}
 \item
 $I_1
 =
 (-
 \tfrac{\lambda_s}{\lambda}
 )
 \tfrac{
 n(\Lambda_y\epsilon,\pa_{y_j}{\sf Q})_2\cdot(\pa_{y_j}\epsilon,\Lambda_y{\sf Q})_2}{\|\nabla_y{\sf Q}\|_2^2}$,
 \item
 $\dis I_2=
 \sum_{i=1}^n
 (-\tfrac{1}{\lambda})
 \tfrac{dz_i}{ds}
 \tfrac{n(\pa_{y_i}\epsilon,\pa_{y_j}{\sf Q})_2\cdot(\pa_{y_j}\epsilon,\Lambda_y{\sf Q})_2}
 {\|\pa_{y_j}{\sf Q}\|_2^2}$,
 \item
 $I_3=
 (-\tfrac{a}{\lambda})
 \tfrac{dz_j}{ds}
 \tfrac{(\nabla_y\mathcal{Y},\nabla_y{\sf Q})_2\cdot(\pa_{y_j}\epsilon,\Lambda_y{\sf Q})_2}{\|\nabla_y{\sf Q}\|_2^2}$,
 \item
 $I_4=
 \tfrac{n(-N,\pa_{y_j}{\sf Q})_2\cdot(\pa_{y_j}\epsilon,\Lambda_y{\sf Q})_2}
 {\|\nabla_y{\sf Q}\|_2^2}$,
 \item
 $I_5
 =
 \tfrac{-n(\epsilon_s,\pa_{y_j}{\sf Q})_2\cdot(\epsilon,k)_2}
 {\|\nabla_y{\sf Q}\|_2^2}$.
 \end{itemize}
%%%%%%%%%%%%%%%%%%%%%%%%%%%%%%%%%%%%%%%%%%%%%%%%%%%%%%%%%%%%%%%%
 We first give estimates for $J_i$ ($i=1,2,3$).
 We recall from Lemma \ref{LEMMA_3.3} that $|R(y)|<C(1+|y|)^{-n}$ for $y\in\R^n$.
 Hence
 we get from \eqref{EQ_3.17} and the Hardy inequalities (see Lemma \ref{LEMMA_2.1}) that
 \begin{align}
 \label{EQ_3.41}
 |J_1|
 <
 |\tfrac{\lambda_s}{\lambda}|
 \cdot
 |(\Lambda_y\epsilon,R)_2|
 <
 C(a^2+\|\Delta_y\epsilon\|_2)
 \|\Delta_y\epsilon\|_2.
 \end{align}
 Furthermore
 from \eqref{EQ_3.17}, \eqref{EQUATION_e3.7} and the Hardy inequalities,
 we get
 \begin{align*}
 \nonumber
 |J_2|
 &<
 |\tfrac{\lambda_s}{\lambda}
 a(\Lambda_y{\mathcal Y},\Lambda_y{\sf Q})_2|
 <
 C(a^2+\|\Delta_y\epsilon\|_2)|a|,
 \\
 |J_3|
 &<
 |(N,\Lambda_y{\sf Q})_2|
 <
 C|a|^p+C|(|\epsilon|^p,|\Lambda_y{\sf Q}|)_2.
 \end{align*}
 Since $p=\frac{n+2}{n-2}=2$ when $n=6$,
 from the Hardy inequality,
 we observe that
 \begin{align*}
 |(|\epsilon|^p,|\Lambda_y{\sf Q}|)_2
 &<
 \int_{\R^n}
 \tfrac{C\epsilon^2}{(1+|y|^2)^\frac{n-2}{2}}
 dy
 =
 \int_{\R^n}
 \tfrac{C\epsilon^2}{(1+|y|^2)^2}
 dy
 <
 C
 \|\Delta_y\epsilon\|_2^2.
 \end{align*}
 Note that
 $a(t)$ is bounded by
 $|a(t)|^2\|\nabla_y\mathcal{Y}\|_2^2<\eta_1^2$ from \eqref{EQUATION_e3.4}.
 Therefore
 it follows that
 \begin{align}
 \label{EQ_3.42}
 |J_2|+|J_3|
 <
 C(a^2+\|\Delta_y\epsilon\|_2^2).
 \end{align}
%%%%%%%%%%%%%%%%%%%%%%%%%%%%%%%%%%%%%%%%%%%%%%%%%%%%%%%%%%%%%%%%
 In the same way,
 we easily see that
 \begin{align}
 \nonumber
 |I_2|
 &<
 \sum_{i=1}^n
 |\tfrac{1}{\lambda}
 \tfrac{dz_i}{ds}|
 \tfrac{n|(\pa_{y_i}\epsilon,\pa_{y_j}{\sf Q})_2|\cdot|(\pa_{y_j}\epsilon,\Lambda_y{\sf Q})_2|}
 {\|\nabla_y{\sf Q}\|_2^2}
 \\
 \nonumber
 &<
 C
 \sum_{i=1}^n
 (a^2+\|\Delta_y\epsilon\|_2)
 \tfrac{|(\pa_{y_j}\pa_{y_i}\epsilon,{\sf Q})_2|\cdot\|\nabla_y\epsilon\|_2}
 {\|\nabla_y{\sf Q}\|_2^2}
 \\
 \label{EQ_3.43}
 &<
 C
 (a^2+\|\Delta_y\epsilon\|_2)
 \|\Delta_y\epsilon\|_2
 \|\nabla_y\epsilon\|_2,
 \\
 \nonumber
 |I_3|
 &<
 |\tfrac{a}{\lambda}
 \tfrac{dz_j}{ds}|
 \tfrac{|(\nabla_y\mathcal{Y},\nabla_y{\sf Q})_2|\cdot|(\pa_{y_j}\epsilon,\Lambda_y{\sf Q})_2|}{\|\nabla_y{\sf Q}\|_2^2}
 \\
 \label{EQ_3.44}
 &<
 C|a|
 (a^2+\|\Delta_y\epsilon\|_2)
 \|\nabla_y\epsilon\|_2,
 \\
 \nonumber
 |I_4|
 &<
 \tfrac{n|(N,\pa_{y_j}{\sf Q})_2|\cdot|(\pa_{y_j}\epsilon,\Lambda_y{\sf Q})_2|}
 {\|\nabla_y{\sf Q}\|_2^2}
 \\
 \nonumber
 &<
 C\{|a|^p+(|\epsilon|^p,|\nabla_y{\sf Q}|)_2\}
 \cdot\|\nabla_y\epsilon\|_2
 \\
 \label{EQ_3.45}
 &<
 C(a^2+\|\Delta_y\epsilon\|_2^2)
 \|\nabla_y\epsilon\|_2.
 \end{align}
 To derive estimates for $I_1$,
 we compute $(\Lambda_y\epsilon,\pa_{y_j}{\sf Q})_2$ and $(\pa_{y_j}\epsilon,\Lambda_y{\sf Q})_2$.
 Let $r=\frac{2n}{n-3}\in(\frac{2n}{n-2},\frac{2n}{n-4})$.
 Since $(\Lambda_y\epsilon,\pa_{y_j}{\sf Q})_2=(\epsilon,-(\frac{n+2}{2}+y\cdot\nabla_y)\pa_{y_j}{\sf Q})_2$ (see \eqref{EQ_3.26}),
 we have
 \begin{align*}
 |(\Lambda_y\epsilon,\pa_{y_j}{\sf Q})_2|
 &=
 |(\epsilon,-(\tfrac{n+2}{2}+y\cdot\nabla_y)\pa_{y_j}{\sf Q})_2|
 \\
 &<
 C(|\epsilon|,(1+|y|)^{-(n-1)})_2
 \\
 &<
 C
 \|\epsilon\|_r
 \|(1+|y|)^{-(n-1)}\|_{r'}
 \\
 &<
 C
 \|\epsilon\|_\frac{2n}{n-4}^\frac{1}{2}
 \|\epsilon\|_\frac{2n}{n-2}^\frac{1}{2}
 \|(1+|y|)^{-(n-1)}\|_{r'}.
 \end{align*}
 Here we used $\frac{1}{r}=\frac{(\frac{1}{2})}{(\frac{2n}{n-4})}+\frac{(\frac{1}{2})}{(\frac{2n}{n-2})}$.
 Since $r'=\frac{2n}{n+3}$,
 it holds that
 \[
 \|(1+|y|)^{-(n-1)}\|_{r'}
 <
 \infty
 \quad\
 \text{if }
 n\geq6.
 \]
 Hence it follows that
 \begin{align}
 \label{EQ_3.46}
 |(\Lambda_y\epsilon,\pa_{y_j}{\sf Q})_2|
 <
 C
 \|\epsilon\|_\frac{2n}{n-4}^\frac{1}{2}
 \|\epsilon\|_\frac{2n}{n-2}^\frac{1}{2}
 <
 C
 \|\Delta_y\epsilon\|_2^\frac{1}{2}
 \|\nabla_y\epsilon\|_2^\frac{1}{2}.
 \end{align}
 In quite the same way,
 we can check that
 \begin{align}
 \label{EQ_3.47}
 |(\pa_{y_j}\epsilon,\Lambda_y{\sf Q})_2|
 <
 C
 \|\Delta_y\epsilon\|_2^\frac{1}{2}
 \|\nabla_y\epsilon\|_2^\frac{1}{2}.
 \end{align}
%%%%%%%%%%%%%%%%%%%%%%%%%%%%%%%%%%%%%%%%%%%%%%%%%%%%%%%%%%%%%%%%
 Therefore
 from \eqref{EQ_3.46} - \eqref{EQ_3.47},
 we deduce that
 \begin{align}
 \nonumber
 |I_1|
 &<
 |
 \tfrac{\lambda_s}{\lambda}
 |
 \tfrac{
 n|(\Lambda_y\epsilon,\pa_{y_j}{\sf Q})_2|\cdot|(\pa_{y_j}\epsilon,\Lambda_y{\sf Q})_2|}{\|\nabla_y{\sf Q}\|_2^2}
 \\
 \label{EQ_3.48}
 &<
 C(a^2+\|\Delta_y\epsilon\|_2)
 \|\Delta_y\epsilon\|_2
 \|\nabla_y\epsilon\|_2.
 \end{align}
 We finally provide estimates for $I_5$.
 From \eqref{EQ_3.37},
 we observe that
 \begin{align}
 \label{EQ_3.49}
 |I_5|
 <
 \tfrac{n|(\epsilon_s,\pa_{y_j}{\sf Q})_2|\cdot|(\epsilon,k)_2|}
 {\|\nabla_y{\sf Q}\|_2^2}
 <
 C|(\epsilon_s,\pa_{y_j}{\sf Q})_2|
 \|\Delta_y\epsilon\|_2.
 \end{align}
 Due to \eqref{EQ_3.33},
 we have
 \begin{align*}
 |(\epsilon_s&,\pa_{y_j}{\sf Q})_2|
 <
 |\tfrac{\lambda_s}{\lambda}|
 \cdot
 |(\Lambda_y\epsilon,\pa_{y_j}{\sf Q})_2|
 +
 \sum_{i=1}^n
 |\tfrac{1}{\lambda}
 \tfrac{dz_i}{ds}|
 \cdot
 |(\pa_{y_i}\epsilon,\pa_{y_j}{\sf Q})_2|
 \\
 \nonumber
 &\quad
 +
 |\tfrac{1}{\lambda}
 \tfrac{dz_j}{ds}|
 \tfrac{\|\nabla_y{\sf Q}\|_2^2}{n}
 +
 |\tfrac{a}{\lambda}
 \tfrac{dz_j}{ds}|
 \tfrac{|(\nabla_y\mathcal{Y},\nabla_y{\sf Q})_2|}{n}
 +
 |(N,\pa_{y_j}{\sf Q})_2|
 \\
 &<
 C
 |\tfrac{\lambda_s}{\lambda}|
 \cdot
 \|\nabla_y\epsilon\|_2
 +
 C
 \sum_{i=1}^n
 |\tfrac{1}{\lambda}
 \tfrac{dz_i}{ds}|
 \cdot
 \|\nabla_y\epsilon\|_2
 \\
 \nonumber
 &\quad
 +
 C
 |\tfrac{1}{\lambda}
 \tfrac{dz_j}{ds}|
 +
 C
 |a|\cdot
 |\tfrac{1}{\lambda}
 \tfrac{dz_j}{ds}|
 +
 C|(N,\pa_{y_j}{\sf Q})_2|
 \\
 &<
 C
 (a^2+\|\Delta_y\epsilon\|_2)
 (\|\nabla_y\epsilon\|_2+1+|a|)
 +
 |(N,\pa_{y_j}{\sf Q})_2|.
 \end{align*}
 Since
 $|(N,\pa_{y_j}{\sf Q})_2|<Ca^2+C\|\Delta_y\epsilon\|_2\|\nabla_y\epsilon\|_2$ when $n=6$
 (see \eqref{EQUATION_e3.7}),
 we obtain
 \begin{align}
 \nonumber
 |(\epsilon_s,\pa_{y_j}{\sf Q})_2|
 &<
 C
 (a^2+\|\Delta_y\epsilon\|_2)
 (1+\|\nabla_y\epsilon\|_2+|a|)
 \\
 \label{EQ_3.50}
 &<
 C
 (a^2+\|\Delta_y\epsilon\|_2).
 \end{align}
%%%%%%%%%%%%%%%%%%%%%%%%%%%%%%%%%%%%%%%%%%%%%%%%%%%%%%%%%%%%%%%%
 In the last inequality,
 we used \eqref{EQUATION_e3.4}.
 Combining \eqref{EQ_3.49} - \eqref{EQ_3.50},
 we conclude
 \begin{align}
 \label{EQ_3.51}
 |I_5|
 <
 C(a^2+\|\Delta_y\epsilon\|_2)\|\Delta_y\epsilon\|_2.
 \end{align}
 Estimate \eqref{EQ_3.29} follows from
 \eqref{EQ_3.40} with \eqref{EQ_3.41} - \eqref{EQ_3.45}, \eqref{EQ_3.48} and \eqref{EQ_3.51}.
 In the same way,
 \eqref{EQ_3.30} can be derived from \eqref{EQ_3.34}.
 The proof is completed.
 \end{proof}
%%%%%%%%%%%%%%%%%%%%%%%%%%%%%%%%%%%%%%%%%%%%%%%%%%%%%%%%%%%%%%%%

%%%%%%%%%%%%%%%%%%%%%%%%%%%%%%%%%%%%%%%%%%%%%%%%%%%%%%%%%%%%%%%%
 \begin{lem}
 \label{LEMMA_3.5}
 Let $n=6$,
 and
 $u(x,t)$ be as in Lemma {\rm\ref{LEMMA_3.1}}.
 There exist constants $h_1>0$ and $C>0$,
 depending only on $n$
 such that if $\eta_1\in(0,h_1)$,
 then
 \begin{align}
 \label{EQ_3.52}
 &
 \frac{d}{ds}
 \left(
 \frac{\lambda^2}{2}
 \|\epsilon\|_2^2
 -
 \frac{\lambda^2}{2}
 (\Lambda_y{\sf Q},\epsilon)_2
 +
 \frac{\lambda^2}{4}
 \|\Lambda_y{\sf Q}\|_2^2
 \right)
 \\
 \nonumber
 &=
 \lambda^2
 (H_y\epsilon,\epsilon)_2
 +
 \lambda^2
 (N,\epsilon)_2 
 +
 \lambda^2K(s)
 \quad\
 \text{\rm for }
 s\in(s_1,s_2),
 \end{align}
 where $K(s)$ is a function satisfying $|K(s)|<C(a^2+\|\Delta_y\epsilon\|_2^2)$.
 \end{lem}
%%%%%%%%%%%%%%%%%%%%%%%%%%%%%%%%%%%%%%%%%%%%%%%%%%%%%%%%%%%%%%%%
 \begin{proof}
 Multiplying equation \eqref{EQUATION_e3.9} by $\epsilon$,
 we obtain
 \begin{align}
 \label{EQ_3.53}
 \tfrac{1}{2}
 \tfrac{d}{ds}
 \|\epsilon\|_2^2
 &=
 (H_y\epsilon,\epsilon)_2
 +
 \tfrac{\lambda_s}{\lambda}
 (\Lambda_y\epsilon,\epsilon)_2
 +
 0
 +
 \tfrac{\lambda_s}{\lambda}
 (\Lambda_y{\sf Q},\epsilon)_2
 -
 0
 \\
 \nonumber
 &\quad
 +
 \tfrac{\lambda_s}{\lambda}
 a
 (\Lambda_y{\mathcal Y},\epsilon)_2
 +
 \sum_{i=1}^n
 \tfrac{1}{\lambda}
 \tfrac{dz_i}{ds}
 (\pa_{y_i}{\sf Q},\epsilon)_2
 \\
 \nonumber
 &\quad
 +
 \sum_{i=1}^n
 \tfrac{a}{\lambda}
 \tfrac{dz_i}{ds}
 (\pa_{y_i}\mathcal{Y},\epsilon)_2
 +
 (N,\epsilon)_2.
 \end{align}
 Since $\Lambda_y=\frac{n-2}{2}+y\cdot\nabla_y$,
 we verify that
%%%%%%%%%%%%%%%%%%%%%%%%%%%%%%%%%%%%%%%%%%%%%%%%%%%%%%%%%%%%%%%%
 \begin{align*}
 (\Lambda_y\epsilon,\epsilon)_2
 &=
 \tfrac{n-2}{2}
 (\epsilon,\epsilon)_2
 +
 (y\cdot\nabla\epsilon,\epsilon)_2
 \\
 &=
 \tfrac{n-2}{2}
 (\epsilon,\epsilon)_2
 +
 \tfrac{1}{2}
 \int y\cdot\nabla_y\epsilon^2
 \\
 &=
 \tfrac{n-2}{2}
 (\epsilon,\epsilon)_2
 -
 \tfrac{n}{2}
 (\epsilon,\epsilon)_2
 =
 -
 \|\epsilon\|_2^2.
 \end{align*}
 Hence
 we can rewrite \eqref{EQ_3.53} as
 \begin{align}
 \label{EQ_3.54}
 &
 \tfrac{1}{2}
 \tfrac{d}{ds}
 \|\epsilon\|_2^2
 +
 \tfrac{\lambda_s}{\lambda}
 \|\epsilon\|_2^2
 =
 (H_y\epsilon,\epsilon)_2
 +
 \tfrac{\lambda_s}{\lambda}
 (\Lambda_y{\sf Q},\epsilon)_2
 +
 \tfrac{\lambda_s}{\lambda}
 a
 (\Lambda_y{\mathcal Y},\epsilon)_2
 \\
 \nonumber
 &
 +
 \sum_{i=1}^n
 \tfrac{1}{\lambda}
 \tfrac{dz_i}{ds}
 (\pa_{y_i}{\sf Q},\epsilon)_2
 +
 \sum_{i=1}^n
 \tfrac{a}{\lambda}
 \tfrac{dz_i}{ds}
 (\pa_{y_i}\mathcal{Y},\epsilon)_2
 +
 (N,\epsilon)_2.
 \end{align}
 Multiplying \eqref{EQ_3.54} by $\lambda^2$,
 we have
 \begin{align}
 \nonumber
 &
 \tfrac{d}{ds}
 (
 \tfrac{\lambda^2}{2}
 \|\epsilon\|_2^2
 )
 =
 \lambda^2
 (H_y\epsilon,\epsilon)_2
 +
 \lambda
 \lambda_s
 (\Lambda_y{\sf Q},\epsilon)_2
 +
 \lambda^2
 \underbrace{
 \tfrac{\lambda_s}{\lambda}
 a
 (\Lambda_y{\mathcal Y},\epsilon)_2
 }_{=K_1}
 \\
 \nonumber
 &\quad
 +
 \lambda^2
 \sum_{i=1}^n
 \tfrac{1}{\lambda}
 \tfrac{dz_i}{ds}
 (\pa_{y_i}{\sf Q},\epsilon)_2
 +
 \lambda^2
 \underbrace{
 \sum_{i=1}^n
 \tfrac{a}{\lambda}
 \tfrac{dz_i}{ds}
 (\pa_{y_i}\mathcal{Y},\epsilon)_2
 }_{=K_2}
 +
 \lambda^2
 (N,\epsilon)_2
 \\
 \label{EQ_3.55}
 &=
 \lambda^2
 (H_y\epsilon,\epsilon)_2
 +
 \lambda
 \lambda_s
 (\Lambda_y{\sf Q},\epsilon)_2
 +
 \lambda^2
 \sum_{i=1}^n
 \tfrac{1}{\lambda}
 \tfrac{dz_i}{ds}
 (\pa_{y_i}{\sf Q},\epsilon)_2
 \\
 \nonumber
 &\quad
 +
 (K_1+K_2)
 \lambda^2
 +
 \lambda^2
 (N,\epsilon)_2.
 \end{align}
%%%%%%%%%%%%%%%%%%%%%%%%%%%%%%%%%%%%%%%%%%%%%%%%%%%%%%%%%%%%%%%%
 We arrange the second term on the right hand side of \eqref{EQ_3.55}.
 \begin{align}
 \nonumber
 \lambda
 \lambda_s
 &
 (\Lambda_y{\sf Q},\epsilon)_2
 =
 \tfrac{d}{ds}
 \tfrac{\lambda^2}{2}
 \cdot
 (\epsilon,\Lambda_y{\sf Q})_2
 \\
 \nonumber
 &=
 \tfrac{d}{ds}
 \{
 \tfrac{\lambda^2}{2}
 (\epsilon,\Lambda_y{\sf Q})_2
 \}
 -
 \tfrac{\lambda^2}{2}
 (\epsilon_s,\Lambda_y{\sf Q})_2
 \\
 \label{EQ_3.56}
 &=
 \tfrac{d}{ds}
 \{
 \tfrac{\lambda^2}{2}
 (\epsilon,\Lambda_y{\sf Q})_2 \}
 -
 \tfrac{\|\Lambda_y{\sf Q}\|_2^2\lambda^2}{2}
 \tfrac{(\epsilon_s,\Lambda_y{\sf Q})_2}{\|\Lambda_y{\sf Q}\|_2^2}.
 \end{align}
%%%%%%%%%%%%%%%%%%%%%%%%%%%%%%%%%%%%%%%%%%%%%%%%%%%%%%%%%%%%%%%%
 Note that
 \eqref{EQ_3.29} implies
 \begin{align}
 \label{EQ_3.57}
 \tfrac{(\epsilon_s,\Lambda_y{\sf Q})_2}{\|\Lambda_y{\sf Q}\|_2^2}
 =
 \tfrac{\lambda_s}{\lambda}
 +
 \tfrac{d}{ds}
 (
 \tfrac{n(n-2)}{4}
 \sum_{j=1}^n
 \tfrac{(\epsilon,\pa_{y_j}{\sf Q})_2^2}
 {\|\nabla_y{\sf Q}\|_2^2\|\Lambda_y{\sf Q}\|_2^2}
 )
 +
 O_1,
 \end{align}
 where $O_1$ satisfies
 \begin{align}
 \label{EQ_3.58}
 |O_1|
 <
 C
 (a^2+\|\Delta_y\epsilon\|_2^2).
 \end{align}
 We substitute \eqref{EQ_3.57} into \eqref{EQ_3.56} to get
 \begin{align}
 \nonumber
 &
 \lambda
 \lambda_s
 (\Lambda_y{\sf Q},\epsilon)_2
 \\
 \nonumber
 &=
 \tfrac{d}{ds}
 \{
 \tfrac{\lambda^2}{2}
 (\Lambda_y{\sf Q},\epsilon)_2 
 \}
 -
 \tfrac{\|\Lambda_y{\sf Q}\|_2^2\lambda^2}{2}
 \{
 \tfrac{\lambda_s}{\lambda}
 +
 \tfrac{d}{ds}
 (
 \tfrac{n(n-2)}{4}
 \sum_{i=1}^n
 \tfrac{(\epsilon,\pa_{y_i}{\sf Q})_2^2}
 {\|\nabla_y{\sf Q}\|_2^2\|\Lambda_y{\sf Q}\|_2^2}
 )
 +
 O_1
 \}
 \\
 \nonumber
 &=
 \tfrac{d}{ds}
 \{
 \tfrac{\lambda^2}{2}
 (\Lambda_y{\sf Q},\epsilon)_2 
 \}
 -
 \tfrac{\|\Lambda_y{\sf Q}\|_2^2}{2}
 \lambda\lambda_s
 -
 \tfrac{n(n-2)\lambda^2}{8}
 \tfrac{d}{ds}
 \sum_{i=1}^n
 \tfrac{(\epsilon,\pa_{y_i}{\sf Q})_2^2}
 {\|\nabla_y{\sf Q}\|_2^2}
 \\
 \nonumber
 &\quad
 -
 \tfrac{\|\Lambda_y{\sf Q}\|_2^2}{2}
 \lambda^2
 O_1
 \\
 \nonumber
 &=
 \tfrac{d}{ds}
 \{
 \tfrac{\lambda^2}{2}
 (\Lambda_y{\sf Q},\epsilon)_2 
 -
 \tfrac{\lambda^2}{4}
 \|\Lambda_y{\sf Q}\|_2^2
 -
 \tfrac{n(n-2)\lambda^2}{8}
 \sum_{j=1}^n
 \tfrac{(\epsilon,\pa_{y_j}{\sf Q})_2^2}
 {\|\nabla_y{\sf Q}\|_2^2}
 \}
 \\
 \nonumber
 &\quad
 +
 \lambda^2
 \underbrace{
 \tfrac{n(n-2)}{4}
 \tfrac{\lambda_s}{\lambda}
 \sum_{j=1}^n
 \tfrac{(\epsilon,\pa_{y_j}{\sf Q})_2^2}
 {\|\nabla_y{\sf Q}\|_2^2}
 }_{=K_3}
 +
 \underbrace{
 (-\tfrac{\|\Lambda_y{\sf Q}\|_2^2}{2})
 O_1
 }_{=K_4}
 \lambda^2
 \\
 \label{EQ_3.59}
 &=
 \tfrac{d}{ds}
 \{
 \tfrac{\lambda^2}{2}
 (\Lambda_y{\sf Q},\epsilon)_2 
 -
 \tfrac{\lambda^2}{4}
 \|\Lambda_y{\sf Q}\|_2^2
 -
 \tfrac{n(n-2)\lambda^2}{8}
 \sum_{j=1}^n
 \tfrac{(\epsilon,\pa_{y_j}{\sf Q})_2^2}
 {\|\nabla_y{\sf Q}\|_2^2}
 \}
 \\
 \nonumber
 &\quad
 +
 (K_3+K_4)
 \lambda^2.
 \end{align}
 To arrange the third term
 $\lambda^2
 \sum_{i=1}^n
 \tfrac{1}{\lambda}
 \tfrac{dz_i}{ds}
 (\pa_{y_i}{\sf Q},\epsilon)_2$
 on the right hand side of \eqref{EQ_3.55},
 we first recall \eqref{EQ_3.30}.
%%%%%%%%%%%%%%%%%%%%%%%%%%%%%%%%%%%%%%%%%%%%%%%%%%%%%%%%%%%%%%%%
 \begin{align}
 \label{EQ_3.60}
 \tfrac{1}{\lambda}
 \tfrac{dz_j}{ds}
 =
 \tfrac{n(\epsilon_s,\pa_{y_j}{\sf Q})_2}{\|\nabla_y{\sf Q}\|_2^2}
 -
 \tfrac{\lambda_s}{\lambda}
 \tfrac{
 n(\Lambda_y\epsilon,\pa_{y_j}{\sf Q})_2
 }{\|\nabla_y{\sf Q}\|_2^2}
 +
 O_2
 \qquad
 (j=1,\cdots,n),
 \end{align}
 where
 \begin{align}
 \label{EQ_3.61}
 |O_2|
 <
 C(a^2+\|\Delta_y\epsilon\|_2^2).
 \end{align}
%%%%%%%%%%%%%%%%%%%%%%%%%%%%%%%%%%%%%%%%%%%%%%%%%%%%%%%%%%%%%%%%
 From \eqref{EQ_3.60},
 we see that
 \begin{align*}
 \tfrac{1}{\lambda} 
 \tfrac{dz_j}{ds}
 (\pa_{y_j}{\sf Q},\epsilon)_2
 &=
 \tfrac{
 n(\epsilon_s,\pa_{y_j}{\sf Q})_2\cdot(\epsilon,\pa_{y_j}{\sf Q})_2
 }{\|\nabla_y{\sf Q}\|_2^2}
 +
 \underbrace{
 (-\tfrac{\lambda_s}{\lambda})
 \tfrac{
 n(\Lambda_y\epsilon,\pa_{y_j}{\sf Q})_2
 \cdot
 (\epsilon,\pa_{y_j}{\sf Q})_2
 }{\|\nabla_y{\sf Q}\|_2^2}
 }_{=K_5}
 \\
 &\quad
 +
 \underbrace{
 (\epsilon,\pa_{y_j}{\sf Q})_2
 O_2
 }_{=K_6}
 \\
 &=
 \tfrac{d}{ds}
 (\tfrac{n(\epsilon,\pa_{y_j}{\sf Q})_2^2
 }{2\|\nabla_y{\sf Q}\|_2^2}
 )
 +
 K_5+K_6
 \quad\
 (j=1,\cdots,n).
 \end{align*}
%%%%%%%%%%%%%%%%%%%%%%%%%%%%%%%%%%%%%%%%%%%%%%%%%%%%%%%%%%%%%%%%
 Therefore
 the third term on the right hand side of \eqref{EQ_3.55} can be simplified as
 \begin{align}
 \nonumber
 \lambda^2
 &
 \sum_{j=1}^n
 \tfrac{1}{\lambda}
 \tfrac{dz_j}{ds}
 (\pa_{y_j}{\sf Q},\epsilon)_2
 =
 \lambda^2
 \tfrac{d}{ds}
 \sum_{j=1}^n
 \tfrac{n(\epsilon,\pa_{y_j}{\sf Q})_2^2
 }{2\|\nabla_y{\sf Q}\|_2^2}
 +
 (K_5+K_6)
 \lambda^2
 \\
 \nonumber
 &=
 \tfrac{d}{ds}
 \sum_{j=1}^n
 \tfrac{\lambda^2n(\epsilon,\pa_{y_j}{\sf Q})_2^2}{2\|\nabla_y{\sf Q}\|_2^2}
 +
 \lambda^2
 \underbrace{
 \sum_{j=1}^n
 (-\tfrac{\lambda_s}{\lambda})
 \tfrac{n(\epsilon,\pa_{y_j}{\sf Q})_2^2}{\|\nabla_y{\sf Q}\|_2^2}
 }_{=K_7}
 +
 (K_5+K_6)
 \lambda^2
 \\
 \label{EQ_3.62}
 &=
 \tfrac{d}{ds}
 \{
 \tfrac{n\lambda^2}{2}
 \sum_{j=1}^n
 \tfrac{(\epsilon,\pa_{y_j}{\sf Q})_2^2
 }{\|\nabla_y{\sf Q}\|_2^2}
 \}
 +
 \lambda^2
 \sum_{i=5}^7
 K_i.
 \end{align}
 Substituting \eqref{EQ_3.59} and \eqref{EQ_3.62} into \eqref{EQ_3.55},
 we finally obtain
%%%%%%%%%%%%%%%%%%%%%%%%%%%%%%%%%%%%%%%%%%%%%%%%%%%%%%%%%%%%%%%%
 \begin{align*}
 &\tfrac{d}{ds}
 (
 \tfrac{\lambda^2}{2}
 \|\epsilon\|_2^2
 )
 \\
 &=
 \lambda^2
 (H_y\epsilon,\epsilon)_2
 +
 \lambda
 \lambda_s
 (\Lambda_y{\sf Q},\epsilon)_2
 +
 \lambda^2
 \sum_{i=1}^n
 \tfrac{1}{\lambda}
 \tfrac{dz_i}{ds}
 (\pa_{y_i}{\sf Q},\epsilon)_2
 \\
 \nonumber
 &\quad
 +
 (K_1+K_2)
 \lambda^2
 +
 \lambda^2
 (N,\epsilon)_2
 \\
 \nonumber
 &=
 \lambda^2
 (H_y\epsilon,\epsilon)_2
 \\
 \nonumber
 &\quad
 +
 \tfrac{d}{ds}
 \{
 \tfrac{\lambda^2}{2}
 (\Lambda_y{\sf Q},\epsilon)_2 
 -
 \tfrac{\lambda^2}{4}
 \|\Lambda_y{\sf Q}\|_2^2
 -
 \tfrac{n(n-2)\lambda^2}{8}
 \sum_{j=1}^n
 \tfrac{(\epsilon,\pa_{y_j}{\sf Q})_2^2}
 {\|\nabla_y{\sf Q}\|_2^2}
 \}
 \\
 \nonumber
 &\quad
 +
 (K_3+K_4)
 \lambda^2
 +
 \tfrac{d}{ds}
 \{
 \tfrac{n\lambda^2}{2}
 \sum_{j=1}^n
 \tfrac{(\epsilon,\pa_{y_j}{\sf Q})_2^2
 }{\|\nabla_y{\sf Q}\|_2^2}
 \}
 +
 \lambda^2
 \sum_{i=5}^7
 K_i
 \\
 &\quad
 +
 (K_1+K_2)
 \lambda^2
 +
 \lambda^2
 (N,\epsilon)_2
 \\
 \nonumber
 &=
 \lambda^2
 (H_y\epsilon,\epsilon)_2
 \\
 \nonumber
 &\quad
 +
 \tfrac{d}{ds}
 \{
 \tfrac{\lambda^2}{2}
 (\Lambda_y{\sf Q},\epsilon)_2 
 -
 \tfrac{\lambda^2}{4}
 \|\Lambda_y{\sf Q}\|_2^2
 -
 \tfrac{n(n-6)\lambda^2}{8}
 \sum_{j=1}^n
 \tfrac{(\epsilon,\pa_{y_j}{\sf Q})_2^2}
 {\|\nabla_y{\sf Q}\|_2^2}
 \}
 \\
 \nonumber
 &\quad
 +
 \lambda^2
 \sum_{i=1}^7
 K_i
 +
 \lambda^2
 (N,\epsilon)_2.
 \end{align*}
%%%%%%%%%%%%%%%%%%%%%%%%%%%%%%%%%%%%%%%%%%%%%%%%%%%%%%%%%%%%%%%%
 Therefore
 we conclude
 \begin{align}
 \label{EQ_3.63}
 \tfrac{d}{ds}
 &
 \{
 \tfrac{\lambda^2}{2}
 \|\epsilon\|_2^2
 -
 \tfrac{\lambda^2}{2}
 (\Lambda_y{\sf Q},\epsilon)_2
 +
 \tfrac{\lambda^2}{4}
 \|\Lambda_y{\sf Q}\|_2^2
 +
 \tfrac{n(n-6)\lambda^2}{8}
 \sum_{j=1}^n
 \tfrac{(\epsilon,\pa_{y_j}{\sf Q})_2^2}
 {\|\nabla_y{\sf Q}\|_2^2}
 \}
 \\
 \nonumber
 &=
 \lambda^2
 (H_y\epsilon,\epsilon)_2
 +
 \lambda^2
 (N,\epsilon)_2
 +
 \lambda^2
 \sum_{i=1}^7
 K_i.
 \end{align}
 We collect $K_i$ ($i=1,\cdots,6$) bellow.
 \begin{itemize}
 \item
 $K_1=\tfrac{\lambda_s}{\lambda}
 a
 (\Lambda_y{\mathcal Y},\epsilon)_2$,
 \item
 $\dis K_2=
 \sum_{i=1}^n
 \tfrac{a}{\lambda}
 \tfrac{dz_i}{ds}
 (\pa_{y_i}\mathcal{Y},\epsilon)_2$,
 \item
 $\dis K_3=
 \tfrac{n(n-2)}{4}
 \tfrac{\lambda_s}{\lambda}
 \sum_{j=1}^n
 \tfrac{(\epsilon,\pa_{y_j}{\sf Q})_2^2}
 {\|\nabla_y{\sf Q}\|_2^2}$,
 \item
 $K_4=
 (-\tfrac{\|\Lambda_y{\sf Q}\|_2^2}{2})
 O_1$,
 \item
 $K_5=
 (-\tfrac{\lambda_s}{\lambda})
 \tfrac{
 n(\Lambda_y\epsilon,\pa_{y_j}{\sf Q})_2
 \cdot
 (\epsilon,\pa_{y_j}{\sf Q})_2
 }{\|\nabla_y{\sf Q}\|_2^2}$,
 \item
 $K_6=
 (\epsilon,\pa_{y_j}{\sf Q})_2
 O_2$,
 \item
 $\dis K_7=
 \sum_{j=1}^n
 (-\tfrac{\lambda_s}{\lambda})
 \tfrac{n(\epsilon,\pa_{y_j}{\sf Q})_2^2}{\|\nabla_y{\sf Q}\|_2^2}$.
 \end{itemize}
 From the Hardy inequalities,
 we can verify that
 \begin{align}
 \label{EQ_3.64}
 |K_1|+|K_2|
 &<
 C
 |\tfrac{\lambda_s}{\lambda}|
 \cdot
 |a|\cdot
 \|\nabla_y\epsilon\|_2.
 \end{align}
 Due to \eqref{EQ_3.46} - \eqref{EQ_3.47},
 it follows that
 \begin{align}
 \label{EQ_3.65}
 |K_3|+|K_5|+|K_7|
 &<
 C
 |\tfrac{\lambda_s}{\lambda}|
 \cdot
 \|\Delta_y\epsilon\|_2 
 \|\nabla_y\epsilon\|_2.
 \end{align}
 Since
 $|O_1|<C(a^2+\|\Delta_y\epsilon\|_2^2)$ (see \eqref{EQ_3.58}) and
 $|O_2|<C(a^2+\|\Delta_y\epsilon\|_2^2)$ (see \eqref{EQ_3.61}),
 we have
 \begin{align}
 \nonumber
 |K_4|+|K_6|
 &<
 C(O_1+\|\nabla_y\epsilon\|_2O_2)
 \\
 \label{EQ_3.66}
 &<
 C
 (1+\|\nabla_y\epsilon\|_2)
 (a^2+\|\Delta_y\epsilon\|_2^2).
 \end{align}
 Combining \eqref{EQ_3.63} and \eqref{EQ_3.64} - \eqref{EQ_3.66} with \eqref{EQ_3.17},
 we obtain the conclusion.
 \end{proof}
%%%%%%%%%%%%%%%%%%%%%%%%%%%%%%%%%%%%%%%%%%%%%%%%%%%%%%%%%%%%%%%%

 \subsection{Bounds for $\lambda(t)$, $z(t)$, $\|\epsilon(t)\|_2$
 and $\|\epsilon(t)\|_\infty$}
%%%%%%%%%%%%%%%%%%%%%%%%%%%%%%%%%%%%%%%%%%%%%%%%%%%%%%%%%%%%%%%%
 \begin{lem}
 \label{LEMMA_3.6}
 Let $n=6$,
 and
 $u(x,t)$ be as in Lemma {\rm\ref{LEMMA_3.1}}.
 There exist constants $h_1>0$ and $C>0$,
 depending only on $n$
 such that if $\eta_1\in(0,h_1)$,
 then
 \begin{align}
 \label{EQ_3.67}
 \lambda(s)^2
 &(
 \|\epsilon(s)\|_2^2
 +
 \|\Lambda_y{\sf Q}\|_2^2
 )
 +
 \int_{s_1}^{s_2}
 \lambda(s)^2
 (-H_y\epsilon(s),\epsilon(s))_2
 ds
 \\
 \nonumber
 &<
 4\lambda(s_1)^2
 (
 \|\epsilon(s_1)\|_2^2
 +
 \|\Lambda_y{\sf Q}\|_2^2
 )
 \quad\
 \text{\rm for }
 s\in(s_1,s_2).
 \end{align}
 \end{lem}
%%%%%%%%%%%%%%%%%%%%%%%%%%%%%%%%%%%%%%%%%%%%%%%%%%%%%%%%%%%%%%%%
 \begin{proof}
 From \eqref{EQ_3.52},
 we recall that
%%%%%%%%%%%%%%%%%%%%%%%%%%%%%%%%%%%%%%%%%%%%%%%%%%%%%%%%%%%%%%%%
 \begin{align}
 \nonumber
 \tfrac{d}{ds}
 \{
 \tfrac{\lambda^2}{2}
 (
 \|\epsilon\|_2^2
 &-
 (\Lambda_y{\sf Q},\epsilon)_2
 +
 \tfrac{\|\Lambda_y{\sf Q}\|_2^2}{2}
 )
 \}
 \\
 \label{EQ_3.68}
 &=
 \lambda^2
 \{
 (H_y\epsilon,\epsilon)_2
 +
 (N,\epsilon)_2 
 +
 K(s)
 \}.
 \end{align}
%%%%%%%%%%%%%%%%%%%%%%%%%%%%%%%%%%%%%%%%%%%%%%%%%%%%%%%%%%%%%%%%
 From \eqref{EQUATION_e3.7} and \eqref{EQUATION_e3.4},
 we observe that
 \begin{align}
 \nonumber
 |(N,\epsilon)_2|
 &<
 C(\epsilon^2+a^2\mathcal{Y}^2,\epsilon)_2
 <
 C
 \|\nabla_y\epsilon\|_2^3
 +
 C
 a^2
 \|\nabla_y\epsilon\|_2
 \\
 \label{EQ_3.69}
 &<
 C
 \eta_1
 \|\nabla_y\epsilon\|_2^2
 +
 C\eta_1a^2.
 \end{align}
 Combining \eqref{EQ_3.69} and \eqref{EQUATION_e3.3},
 we have
 \begin{align*}
 (-H_y\epsilon,\epsilon)_2
 -
 (N,\epsilon)_2
 &>
 (-H_y\epsilon,\epsilon)_2
 -
 C
 \eta_1
 \|\nabla_y\epsilon\|_2^2
 -
 C\eta_1
 a^2
 \\
 &>
 (1-\tfrac{C}{\bar C_1}\eta_1)
 (-H_y\epsilon,\epsilon)_2
 -
 C\eta_1
 a^2.
 \end{align*}
 This implies
 \begin{align*}
 (-H_y\epsilon,\epsilon)_2
 -
 (N,\epsilon)_2
 &>
 \tfrac{1}{2}
 (-H_y\epsilon,\epsilon)_2
 -
 C
 a^2
 \quad\
 \text{if }
 \eta_1<\tfrac{\bar C_1}{2C}.
 \end{align*}
%%%%%%%%%%%%%%%%%%%%%%%%%%%%%%%%%%%%%%%%%%%%%%%%%%%%%%%%%%%%%%%%
 Hence
 \eqref{EQ_3.68} can be simplified as
 \begin{align}
 \label{EQ_3.70}
 \tfrac{d}{ds}
 \{
 \tfrac{\lambda^2}{2}
 (
 \|\epsilon\|_2^2
 &-
 (\Lambda_y{\sf Q},\epsilon)_2
 +
 \tfrac{\|\Lambda_y{\sf Q}\|_2^2}{2}
 )
 \}
 \\
 \nonumber
 &<
 \lambda^2
 \{
 \tfrac{1}{2}(H_y\epsilon,\epsilon)_2
 +
 Ca^2
 +
 |K(s)|
 \}.
 \end{align}
%%%%%%%%%%%%%%%%%%%%%%%%%%%%%%%%%%%%%%%%%%%%%%%%%%%%%%%%%%%%%%%%
 Integrating \eqref{EQ_3.70} over $s\in(s_1,s)$,
 we obtain
 \begin{align}
 \label{EQ_3.71}
 \tfrac{\lambda(s)^2}{2}
 &(
 \|\epsilon(s)\|_2^2
 -
 (\Lambda_y{\sf Q},\epsilon(s))_2
 +
 \tfrac{1}{2}
 \|\Lambda_y{\sf Q}\|_2^2
 )
 \\
 \nonumber
 &<
 \tfrac{\lambda(s_1)^2}{2}
 (
 \|\epsilon(s_1)\|_2^2
 -
 (\Lambda_y{\sf Q},\epsilon(s_1))_2
 +
 \tfrac{1}{2}
 \|\Lambda_y{\sf Q}\|_2^2
 )
 \\
 \nonumber
 &\quad
 +
 \int_{s_1}^s
 \lambda^2
 \{
 \tfrac{1}{2}
 (H_y\epsilon,\epsilon)_2
 +
 Ca^2
 +
 |K(s')|
 \}
 ds'.
 \end{align}
 We next provide estimates for $(\Lambda_{y}{\sf Q},\epsilon)_2$.
 Note that $\|\Lambda_{y}{\sf Q}\|_\frac{2n}{n+1}$ is finite when $n\geq6$.
 Therefore
 from \eqref{EQUATION_e3.4},
 we have
 \begin{align}
 \nonumber
 |(\Lambda_{y}{\sf Q},\epsilon)_2|
 &<
 \|\Lambda_{y}{\sf Q}\|_\frac{2n}{n+1}
 \|\epsilon\|_\frac{2n}{n-1}
 <
 \underbrace{
 \|\Lambda_{y}{\sf Q}\|_\frac{2n}{n+1}
 \|\epsilon\|_\frac{2n}{n-2}^\frac{1}{2}
 }_{<C\eta_1^\frac{1}{2}}
 \|\epsilon\|_2^\frac{1}{2}
 \\
 \label{EQ_3.72}
 &<
 \tfrac{3}{4}
 (
 C
 \eta_1^\frac{1}{2}
 )^\frac{4}{3}
 +
 \tfrac{1}{4}
 \|\epsilon\|_2^2
 <
 \tfrac{\|\Lambda_{y}{\sf Q}\|_2^2}{8}
 +
 \tfrac{1}{4}
 \|\epsilon\|_2^2
 \end{align}
 if
 $\tfrac{3}{4}
 (
 C^2
 \eta_1
 )^\frac{2}{3}<\tfrac{\|\Lambda_{y}{\sf Q}\|_2^2}{8}$.
 Combining \eqref{EQ_3.71} - \eqref{EQ_3.72},
 we obtain
 \begin{align}
 \label{EQ_3.73}
 &
 \tfrac{3\lambda(s)^2}{16}
 (
 \|\epsilon(s)\|_2^2
 +
 \|\Lambda_y{\sf Q}\|_2^2
 )
 +
 \int_{s_1}^s
 \tfrac{\lambda^2}{2}
 (-H_y\epsilon,\epsilon)_2
 ds'
 \\
 \nonumber
 &<
 \tfrac{5\lambda(s_1)^2}{8}
 (
 \|\epsilon(s_1)\|_2^2
 +
 \|\Lambda_y{\sf Q}\|_2^2
 )
 +
 \int_{s_1}^s
 \lambda^2
 \{
 Ca^2
 +
 |K(s')|
 \}
 ds'.
 \end{align}
 From \eqref{EQ_3.73},
 it follows that
 \begin{align*}
 \lambda(s)^2
 <
 \tfrac{10\lambda(s_1)^2}{3}
 (
 \tfrac{\|\epsilon(s_1)\|_2^2}{\|\Lambda_y{\sf Q}\|_2^2}
 +
 1
 )
 +
 \tfrac{16}{3\|\Lambda_y{\sf Q}\|_2^2}
 \int_{s_1}^s
 \lambda^2
 \{
 Ca(s')^2
 +
 |K(s')|
 \}
 ds'
 \end{align*}
 for $s\in(s_1,s_2)$.
 Applying the Gronwall inequality,
 we obtain
 \begin{align*}
 \lambda(s)^2
 <
 \tfrac{10\lambda(s_1)^2}{3}
 (
 \tfrac{\|\epsilon(s_1)\|_2^2}{\|\Lambda_y{\sf Q}\|_2^2}
 +
 1
 )
 e^{
 \tfrac{16}{3\|\Lambda_y{\sf Q}\|_2^2}
 \int_{s_1}^{s_2}
 \{
 Ca(s)^2
 +
 |K(s)|
 \}
 ds
 }
 \end{align*}
 for $s\in(s_1,s_2)$.
 We recall that $|K(s)|<C(a^2+\|\Delta_y\epsilon\|_2^2)$ (see Lemma \ref{LEMMA_3.5})
 and $\int_{s_1}^{s_2}(a(s)^2+\|\Delta_y\epsilon\|_2^2)ds<C\eta_1^2$ (see Lemma \ref{LEMMA_3.1}).
 Therefore
 it follows that
 \begin{align}
 \label{EQ_3.74}
 \lambda(s)^2
 <
 \tfrac{10\lambda(s_1)^2}{3}
 (
 \tfrac{\|\epsilon(s_1)\|_2^2}{\|\Lambda_y{\sf Q}\|_2^2}
 +
 1
 )
 e^{C\eta_1^2}
 \quad\
 \text{for }
 s\in(s_1,s_2).
 \end{align}
 Substituting \eqref{EQ_3.74} into the last term of \eqref{EQ_3.73},
 we obtain
 \begin{align*}
 &
 \tfrac{3\lambda(s)^2}{16}
 (
 \|\epsilon(s)\|_2^2
 +
 \|\Lambda_y{\sf Q}\|_2^2
 )
 +
 \int_{s_1}^s
 \tfrac{\lambda^2}{2}
 (-H_y\epsilon,\epsilon)_2
 ds'
 \\
 \nonumber
 &<
 \tfrac{5\lambda(s_1)^2}{8}
 (
 \|\epsilon(s_1)\|_2^2
 +
 \|\Lambda_y{\sf Q}\|_2^2
 )
 +
 \tfrac{10\lambda(s_1)^2}{3}
 (
 \tfrac{\|\epsilon(s_1)\|_2^2}{\|\Lambda_y{\sf Q}\|_2^2}
 +
 1
 )
 e^{C\eta_1^2}
 (C\eta_1^2).
 \end{align*}
 Therefore
 if we choose $\eta_1$ sufficiently small,
 we arrive at the conclusion.
 \end{proof}
%%%%%%%%%%%%%%%%%%%%%%%%%%%%%%%%%%%%%%%%%%%%%%%%%%%%%%%%%%%%%%%%

%%%%%%%%%%%%%%%%%%%%%%%%%%%%%%%%%%%%%%%%%%%%%%%%%%%%%%%%%%%%%%%%
 \begin{lem}
 \label{LEMMA_3.7}
 Let $n=6$,
 and
 $u(x,t)$ be as in Lemma {\rm\ref{LEMMA_3.1}}.
 There exist constants $h_1>0$ and $C>0$,
 depending only on $n$
 such that if $\eta_1\in(0,h_1)$,
 then
 \begin{align}
 \nonumber
 &|(\epsilon(s),\Lambda_y{\sf Q})_2|
 <
 C
 \left(
 \|\nabla_y\epsilon(s)\|_2
 +
 \int_{|\xi|>1}
 \frac{|u(\xi+z(s),t)|}{|\xi|^{n-2}}
 d\xi
 +1
 \right)
 \\
 \label{EQ_3.75}
 &\hspace{30mm}
 \text{\rm for } s\in(s_1,s_2)
 \quad\ \text{\rm if } \lambda(s)>\tfrac{1}{e},
 \\
 \nonumber
 &|(\epsilon(s),\Lambda_y{\sf Q})_2|
 \\
 \nonumber
 &\quad
 <
 C
 \left(
 |\log\lambda(s)|^\frac{n+2}{2n}
 \|\nabla_y\epsilon(s)\|_2
 +
 \int_{|\xi|>1}
 \frac{|u(\xi+z(s),t)|}{|\xi|^{n-2}}
 d\xi
 +
 1
 \right)
 \\
 \label{EQ_3.76}
 &\quad\
 \text{\rm for } s\in(s_1,s_2)
 \quad\
 \text{\rm if } \lambda(s)<\tfrac{1}{e}.
 \end{align}
 \end{lem}
%%%%%%%%%%%%%%%%%%%%%%%%%%%%%%%%%%%%%%%%%%%%%%%%%%%%%%%%%%%%%%%%
 \begin{proof}
 It is straightforward to see that
 \begin{align}
 \nonumber
 (\epsilon,\Lambda_y{\sf Q})_2
 &=
 \int
 \epsilon\Lambda_y{\sf Q} 
 dy
 \\
 \nonumber
 &=
 \int_{|y|<\frac{1}{\lambda(s)}}
 \epsilon\Lambda_y{\sf Q} 
 dy
 +
 \int_{|y|>\frac{1}{\lambda(s)}}
 \epsilon\Lambda_y{\sf Q} 
 dy
 \\
 \label{EQ_3.77}
 &<
 \|\epsilon\|_\frac{2n}{n-2}
 \|\Lambda_y{\sf Q}{\bf 1}_{|y|<\frac{1}{\lambda(s)}}\|_\frac{2n}{n+2}
 +
 \int_{|y|>\frac{1}{\lambda(s)}}
 \epsilon\Lambda_y{\sf Q} 
 dy.
 \end{align}
 We first compute $\|\Lambda_y{\sf Q}{\bf 1}_{|y|<\frac{1}{\lambda(s)}}\|_\frac{2n}{n+2}$.
 Put
 $q_1
 =
 \int_{|y|<e}|\Lambda_y{\sf Q}|^\frac{2n}{n+2}dy$.
 It is clear that
 \begin{align*}
 \|\Lambda_y{\sf Q}{\bf 1}_{|y|<\frac{1}{\lambda(s)}}\|_\frac{2n}{n+2}
 <
 q_1^\frac{n+2}{2n}
 \quad\
 \text{when }
 \lambda(s)>\tfrac{1}{e}.
 \end{align*}
 Note that when $n=6$
 \[
 |\Lambda_y{\sf Q}(y)|^\frac{2n}{n+2}
 <
 C|y|^{-(n-2)\frac{2n}{n+2}}
 =
 C|y|^{-n}
 \quad\
 \text{for } |y|>e.
 \]
 Therefore
 when $\lambda(s)<\frac{1}{e}$,
 we have
 \begin{align*}
 \|\Lambda_y&{\sf Q}{\bf 1}_{|y|<\frac{1}{\lambda(s)}}\|_\frac{2n}{n+2}
 =
 \left(
 \int_{|y|<e}
 +
 \int_{e<|y|<\frac{1}{\lambda(s)}}
 |\Lambda_y{\sf Q}|^\frac{2n}{n+2}
 dy
 \right)^\frac{n+2}{2n}
 \\
 &<
 \left(
 q_1
 +
 C
 \int_{e<|y|<\frac{1}{\lambda(s)}}
 |y|^{-n}
 dy
 \right)^\frac{n+2}{2n}
 \\
 &=
 (
 q_1
 +
 C
 +
 C
 |\log\lambda(s)|
 )^\frac{n+2}{2n}
 \\
 &<
 C^\frac{n+2}{2n}
 |\log\lambda(s)|^\frac{n+2}{2n}
 (
 \tfrac{q_1+1}{|\log\lambda(s)|}
 +
 1
 )^\frac{n+2}{2n}
 \\
 &<
 C^\frac{n+2}{2n}
 |\log\lambda(s)|^\frac{n+2}{2n}
 (
 q_1+1
 )^\frac{n+2}{2n}.
 \end{align*}
 In the last line,
 we used $|\log\lambda(s)|>1$ when $\lambda(s)<\frac{1}{e}$.
 Therefore
 there exists a constant $C>0$ depending only on $n$ such that
 \begin{align}
 \label{EQ_3.78}
 \|\Lambda_y{\sf Q}
 {\bf 1}_{|y|<\frac{1}{\lambda(s)}}\|_\frac{2n}{n+2}
 &<
 \begin{cases}
 C & \text{if } \lambda(s)>\frac{1}{e},
 \\
 C|\log\lambda(s)|^\frac{n+2}{2n}
 & \text{if } \lambda(s)<\frac{1}{e}.
 \end{cases}
 \end{align}
 We next estimate the last term of \eqref{EQ_3.77}.
 We recall that
 $\epsilon(y,s)$ is defied by
 $\epsilon(y,s)=\lambda^\frac{n-2}{2}u(x,t)-{\sf Q}(y)-a\mathcal{Y}(y)$ with $x=\lambda y+z$,
 which implies
 \begin{align}
 \label{EQ_3.79}
 \int_{|y|>\frac{1}{\lambda(s)}}
 &|\epsilon
 \Lambda_y{\sf Q}|
 dy
 =
 \int_{|y|>\frac{1}{\lambda(s)}}
 |(\lambda^\frac{n-2}{2}u-{\sf Q}-a\mathcal{Y})\Lambda_y{\sf Q}|
 dy
 \\
 \nonumber
 &<
 \int_{|y|>\frac{1}{\lambda(s)}}
 |\lambda^\frac{n-2}{2}u\Lambda_y{\sf Q}|
 dy
 +
 (\|{\sf Q}\|_2+|a|\cdot\|\mathcal{Y}\|_2)
 \|\Lambda_y{\sf Q}\|_2.
 \end{align}
 By the change of variables,
 we see that
 \begin{align}
 \nonumber
 \int_{|y|>\frac{1}{\lambda(s)}}
 &|\lambda^\frac{n-2}{2}u\Lambda_y{\sf Q}|
 dy
 <
 C
 \lambda^\frac{n-2}{2}
 \int_{|y|>\frac{1}{\lambda(s)}}
 |u(x,t)|
 \cdot
 |y|^{-(n-2)}
 dy
 \\
 \nonumber
 &=
 C
 \lambda(s)^\frac{n-2}{2}
 \int_{|x-z(s)|>1}
 |u(x,t)|
 \cdot
 |\tfrac{x-z(s)}{\lambda(s)}|^{-(n-2)}
 \cdot
 \tfrac{dx}{\lambda(s)^n}
 \\
 \nonumber
 &=
 C
 \lambda(s)^\frac{n-6}{2}
 \int_{|x-z(s)|>1}
 \tfrac{|u(x,t)|}{|x-z(s)|^{n-2}}
 dx
 \\
 \label{EQ_3.80}
 &=
 C
 \lambda(s)^\frac{n-6}{2}
 \int_{|x\xi|>1}
 \tfrac{|u(\xi+z(s),t)|}{|\xi|^{n-2}}
 d\xi.
 \end{align}
 Therefore
 \eqref{EQ_3.75} - \eqref{EQ_3.76} follow from
 \eqref{EQ_3.77} - \eqref{EQ_3.80}.
 \end{proof}
%%%%%%%%%%%%%%%%%%%%%%%%%%%%%%%%%%%%%%%%%%%%%%%%%%%%%%%%%%%%%%%%

%%%%%%%%%%%%%%%%%%%%%%%%%%%%%%%%%%%%%%%%%%%%%%%%%%%%%%%%%%%%%%%%
 \begin{lem}
 \label{LEMMA_3.8}
 Let $n=6$,
 and
 $u(x,t)$ be as in Lemma {\rm\ref{LEMMA_3.1}}.
 There exist a constant $h_1>0$
 depending only on $n$
 such that if $\eta_1\in(0,h_1)$,
 then
 \begin{align}
 \label{EQ_3.81}
 &\quad
 \lambda_{\inf}
 =
 \inf_{s\in(s_1,s_2)}
 \lambda(s)
 >0,
 \\
 \label{EQ_3.82}
 &\quad
 \lambda_{\sup}
 =
 \sup_{s\in(s_1,s_2)}
 \lambda(s)
 <\infty,
 \\
 \label{EQ_3.83}
 &\sup_{s\in(s_1,s_2)}
 |z(s)|
 <
 |z(s_1)|
 +
 C_1.
 \end{align}
 All constants $\lambda_{\inf}$, $\lambda_{\sup}$ and $C_1$ depend on
 $n$, $\lambda(s_1)$, $\|\epsilon(s_1)\|_2$.
 \end{lem}
%%%%%%%%%%%%%%%%%%%%%%%%%%%%%%%%%%%%%%%%%%%%%%%%%%%%%%%%%%%%%%%%
 \begin{proof}
 We first derive the uniform bound for $\|u(x,t)\|_{L_x^2(\R^n)}$.
 Note that
 $u(x,t)$ is expressed as
 \[
 u(x,t)=\lambda^{-\frac{n-2}{2}}\{{\sf Q}(y)+a\mathcal{Y}(y)+\epsilon(y,s)\}
 \quad\
 \text{with }
 x=\lambda y+z.
 \]
 Therefore
 we obtain
 \begin{align*}
 \|u(x&,t)\|_{L_x^2(\R^n)}^2
 =
 \int_{\R_x^n}
 u(x,t)^2
 dx
 \\
 &=
 \int_{\R_x^n}
 \{
 \lambda^{-\frac{n-2}{2}}
 ({\sf Q}+a\mathcal{Y}+\epsilon)
 \}^2
 dx 
 \\
 &=
 \lambda^{-(n-2)}
 \int_{\R_y^n}
 ({\sf Q}(y)+a\mathcal{Y}(y)+\epsilon(y,s))^2
 \cdot
 \lambda^n
 dy
 \\
 &<
 2
 \lambda^2
 (\|{\sf Q}\|_2^2+a^2\|\mathcal{Y}\|_2^2+\|\epsilon(y,s)\|_{L_y^2(\R^n)}^2)
 \\
 &<
 2
 \lambda^2
 (\|{\sf Q}\|_2^2+\tfrac{\|\mathcal{Y}\|_2^2\eta_1^2}{\|\nabla_y\mathcal{Y}\|_2^2}+\|\epsilon(s)\|_2^2).
 \end{align*}
 From \eqref{EQ_3.67},
 it is straightforward to see that
 \begin{align*}
 \lambda(s)^2
 &<
 4\lambda(s_1)^2
 (
 \tfrac{\|\epsilon(s_1)\|_2^2}{\|\Lambda_y{\sf Q}\|_2^2}
 +1
 )
 \quad\
 \text{\rm for }
 s\in(s_1,s_2),
 \\
 \|\epsilon(s)\|_2^2
 &<
 4
 \tfrac{\lambda(s_1)^2}{\lambda(s)^2}
 (
 \|\epsilon(s_1)\|_2^2+\|\Lambda_y{\sf Q}\|_2^2
 )
 \quad\
 \text{\rm for }
 s\in(s_1,s_2).
 \end{align*}
%%%%%%%%%%%%%%%%%%%%%%%%%%%%%%%%%%%%%%%%%%%%%%%%%%%%%%%%%%%%%%%%
 Therefore we obtain
 \begin{align*}
 \nonumber
 &\|u(x,t)\|_{L_x^2(\R^n)}^2
 \\
 &<
 8\lambda(s_1)^2
 (
 \tfrac{\|\epsilon(s_1)\|_2^2}{\|\Lambda_y{\sf Q}\|_2^2}
 +1
 )
 (\|{\sf Q}\|_2^2+\tfrac{\|\mathcal{Y}\|_2^2\eta_1^2}{\|\nabla_y\mathcal{Y}\|_2^2})
 +
 2\lambda^2
 \|\epsilon(s)\|_2^2
 \\
 &<
 C
 \lambda(s_1)^2
 (
 \|\epsilon(s_1)\|_2^2
 +
 \|\Lambda_y{\sf Q}\|_2^2
 ),
 \end{align*}
 where the constant $C>0$ depends only on $n$.
 As a consequence,
 it follows that
 \begin{align}
 \nonumber
 \int_{|\xi|>1}
 \tfrac{|u(\xi+z(s),t)|}{|\xi|^{n-2}}
 d\xi
 &<
 \|u(x,t)\|_{L_x^2(\R^n)}
 \||\xi|^{-(n-2)}{\bf 1}_{|\xi|>1}\|_{L_\xi^2(\R^n)}
 \\
 \label{EQ_3.84}
 &<
 C
 \lambda(s_1)
 (
 \|\epsilon(s_1)\|_2
 +
 \|\Lambda_y{\sf Q}\|_2
 ).
 \end{align}
 Here we used the fact that
 $\||\xi|^{-(n-2)}{\bf 1}_{|\xi|>1}\|_{L_\xi^2(\R^n)}<\infty$ when $n=6$.
 To derive the bound for $|(\epsilon(s),\Lambda_y{\sf Q})_2|$,
 we here apply Lemma \ref{LEMMA_3.7}.
 Let $h_1$ be the positive constant given in Lemma \ref{LEMMA_3.7}.
 Substituting \eqref{EQ_3.84} into \eqref{EQ_3.75} - \eqref{EQ_3.76},
 and using $\|\nabla_y\epsilon\|_2<\eta_1$,
 we find that
 if $\eta_1\in(0,h_1)$,
 then
 \begin{align}
 \nonumber
 |(\epsilon(s),\Lambda_y{\sf Q})_2|
 &<
 C
 (
 \|\nabla_y\epsilon\|_2
 +
 \underbrace{
 \lambda(s_1)
 (
 \|\epsilon(s_1)\|_2
 +
 \|\Lambda_y{\sf Q}\|_2
 )
 }_{=X_1}
 +
 1
 )
 \\
 \nonumber
 &<
 C
 (
 \eta_1
 +
 X_1
 +
 1
 )
 \\
 \label{EQ_3.85}
 &=
 C
 (
 X_1
 +
 1
 )
 \quad\
 \text{\rm if } \lambda(s)>\tfrac{1}{e},
 \\
 \label{EQ_3.86}
 |(\epsilon(s),\Lambda_y{\sf Q})_2|
 &<
 C
 (
 |\log\lambda(s)|^\frac{n+2}{2n}
 +
 X_1
 +
 1
 )
 \quad\
 \text{\rm if } \lambda(s)<\tfrac{1}{e}.
 \end{align}
 Integrating both sides of \eqref{EQ_3.29},
 we obtain
 \begin{align}
 \label{EQ_3.87}
 &
 |
 \log\lambda(s)
 +
 \tfrac{n(n-2)}{4}
 \sum_{j=1}^n
 \tfrac{(\epsilon(s),\pa_{y_j}{\sf Q})_2^2}
 {\|\nabla_y{\sf Q}\|_2^2\|\Lambda_y{\sf Q}\|_2^2}
 -
 \tfrac{(\epsilon(s),\Lambda_y{\sf Q})_2}{\|\Lambda_y{\sf Q}\|_2^2}
 |
 \\
 \nonumber
 &<
 \underbrace{
 |
 \log\lambda(s_1)
 +
 \tfrac{n(n-2)}{4}
 \sum_{j=1}^n
 \tfrac{(\epsilon(s_1),\pa_{y_j}{\sf Q})_2^2}
 {\|\nabla_y{\sf Q}\|_2^2\|\Lambda_y{\sf Q}\|_2^2}
 -
 \tfrac{(\epsilon(s_1),\Lambda_y{\sf Q})_2}{\|\Lambda_y{\sf Q}\|_2^2}
 |
 }_{=Y_1}
 \\
 \nonumber
 &\quad
 +
 C
 \int_{s_1}^s
 (a^2+\|\Delta_y\epsilon\|_2^2)
 ds'
 \quad\
 \text{for }
 s\in(s_1,s_2).
 \end{align}
 We recall from Lemma \ref{LEMMA_3.1} that $\int_{s_1}^{s_2}(a^2+\|\Delta_y\epsilon\|_2^2)ds<C\eta_1^2$.
 Due to \eqref{EQ_3.85} - \eqref{EQ_3.86},
 we can rewrite \eqref{EQ_3.87} as
 \begin{align}
 \nonumber
 &|\log\lambda(s)|
 <
 |
 \tfrac{n(n-2)}{4}
 \sum_{j=1}^n
 \tfrac{(\epsilon(s),\pa_{y_j}{\sf Q})_2^2}
 {\|\nabla_y{\sf Q}\|_2^2\|\Lambda_y{\sf Q}\|_2^2}
 -
 \tfrac{(\epsilon(s),\Lambda_y{\sf Q})_2}{\|\Lambda_y{\sf Q}\|_2^2}
 |
 +
 Y_1
 +
 C
 \\
 \nonumber
 &<
 \tfrac{n(n-2)}{4}
 \sum_{j=1}^n
 \tfrac{\|\pa_{y_j}\epsilon(s)\|_2^2\|{\sf Q}\|_2^2}
 {\|\nabla_y{\sf Q}\|_2^2\|\Lambda_y{\sf Q}\|_2^2}
 +
 \tfrac{|(\epsilon(s),\Lambda_y{\sf Q})_2|}{\|\Lambda_y{\sf Q}\|_2^2}
 +
 Y_1
 +
 C
 \\
 \nonumber
 &<
 \underbrace{
 \tfrac{n(n-2)}{4}
 \tfrac{\|\nabla_y\epsilon(s)\|_2^2\|{\sf Q}\|_2^2}
 {\|\nabla_y{\sf Q}\|_2^2\|\Lambda_y{\sf Q}\|_2^2}
 }_{<C\eta_1^2}
 +
 \tfrac{|(\epsilon(s),\Lambda_y{\sf Q})_2|}{\|\Lambda_y{\sf Q}\|_2^2}
 +
 Y_1
 +
 C
 \\
 \label{EQ_3.88}
 &<
 \begin{cases}
 C
 (
 X_1+1
 )
 +
 Y_1
 & \text{if } \lambda(s)>\frac{1}{e},
 \\
 |\log\lambda(s)|^\frac{n+2}{2n}
 +
 C(
 X_1+1
 )
 +
 Y_1
 & \text{if } \lambda(s)<\frac{1}{e}
 \end{cases}
 \end{align}
 for $s\in(s_1,s_2)$.
 The constant $C>0$ in \eqref{EQ_3.88} depends only on $n$.
 Since $\frac{n+2}{2n}=\frac{8}{12}$ when $n=6$,
 there exists a constant $C_1>0$ depending only on $n$, $X_1$ and $Y_1$ such that 
 \[
 \sup_{s\in(s_1,s_2)}
 |\log\lambda(s)|
 <
 C_1.
 \]
 Therefore
 \eqref{EQ_3.81} and \eqref{EQ_3.82} are proved.
%%%%%%%%%%%%%%%%%%%%%%%%%%%%%%%%%%%%%%%%%%%%%%%%%%%%%%%%%%%%%%%%
 We next show \eqref{EQ_3.83}.
 We recall from \eqref{EQ_3.30} that
 \begin{align*}
% \label{EQUATION_3.83}
 &
 \tfrac{1}{\lambda}
 \tfrac{dz_j}{ds}
 -
 \tfrac{n(\epsilon_s,\pa_{y_j}{\sf Q})_2}{\|\nabla_y{\sf Q}\|_2^2}
 +
 \tfrac{\lambda_s}{\lambda}
 \tfrac{
 n(\Lambda_y\epsilon,\pa_{y_j}{\sf Q})_2
 }{\|\nabla_y{\sf Q}\|_2^2}
 =
 K_1(s)
 \\
 &
 \text{for }
 s\in(s_1,s_2),\
 j\in\{1,\cdots,n\},
 \end{align*}
 where $|K_1(s)|<C(a^2+\|\Delta_y\epsilon\|_2^2)$.
 We rewrite this equation as
 \begin{align*}
 &
 \tfrac{dz_j}{ds}
 -
 \tfrac{d}{ds}
 (
 \tfrac{n\lambda(\epsilon,\pa_{y_j}{\sf Q})_2}{\|\nabla_y{\sf Q}\|_2^2}
 )
 \\
 &
 =
 \lambda
 \{
 K_1(s)
 +
 \underbrace{
 \tfrac{\lambda_s}{\lambda}
 (
 \tfrac{-n(\epsilon,\pa_{y_j}{\sf Q})_2}{\|\nabla_y{\sf Q}\|_2^2}
 ) 
 }_{=K_2(s)}
 +
 \underbrace{
 \tfrac{\lambda_s}{\lambda}
 \tfrac{
 -n(\Lambda_y\epsilon,\pa_{y_j}{\sf Q})_2
 }{\|\nabla_y{\sf Q}\|_2^2}
 }_{=K_3(s)}
 \}.
 \end{align*}
 Integrating both sides over $s\in(s_1,s)$,
 we get
 \begin{align}
 \label{EQ_3.89}
 z_j(s)
 &-
 \tfrac{n\lambda(s)(\epsilon(s),\pa_{y_j}{\sf Q})_2}{\|\nabla_y{\sf Q}\|_2^2}
 =
 z_j(s_1)
 -
 \tfrac{n\lambda(s_1)(\epsilon(s_1),\pa_{y_j}{\sf Q})_2}{\|\nabla_y{\sf Q}\|_2^2}
 \\
 \nonumber
 &+
 \sum_{i=1}^3
 \int
 K_i(s')
 \lambda(s')
 ds'.
 \end{align}
 From \eqref{EQUATION_e3.3},
 there exists a constant $C>0$ depending only on $n$ such that
 \begin{align}
 \nonumber
 |K_2(s)|
 &<
 |
 \tfrac{\lambda_s}{\lambda}
 (
 \tfrac{n(\epsilon,\pa_{y_j}{\sf Q})_2}{\|\nabla_y{\sf Q}\|_2^2}
 )
 |
 <
 |
 \tfrac{\lambda_s}{\lambda}
 |
 \cdot
 \tfrac{n\|\nabla_y\epsilon\|_2\|{\sf Q}\|_2}{\|\nabla_y{\sf Q}\|_2^2}
 \\
 \label{EQ_3.90}
 &<
 C(
 |
 \tfrac{\lambda_s}{\lambda}
 |^2
 +
 (-H_y\epsilon,\epsilon)
 ),
 \\
 \label{EQ_3.91}
 |K_3(s)|
 &<
 C(
 |
 \tfrac{\lambda_s}{\lambda}
 |^2
 +
 (-H_y\epsilon,\epsilon)
 ).
 \end{align}
 Note from \eqref{EQ_3.17} and \eqref{EQUATION_e3.4}
 that $|\frac{\lambda_s}{\lambda}|^2<C(a^2+\|\Delta_y\epsilon\|_2^2)$.
 Therefore
 combining \eqref{EQ_3.89} - \eqref{EQ_3.91},
 we obtain
 \begin{align}
 \nonumber
 &
 |
 z_j(s)
 -
 z_j(s_1)
 |
 \\
 \nonumber
 &<
 \sup_{s\in(s_1,s_2)}
 \tfrac{2n\lambda(s)\|\nabla_y\epsilon(s)\|_2\|{\sf Q}\|_2}{\|\nabla_y{\sf Q}\|_2^2}
 +
 \sum_{i=1}^3
 \int_{s_1}^s
 |K_i(s')|
 \lambda(s')
 ds'
 \\
 \nonumber
 &<
 \tfrac{n\lambda_{\sup}\eta_1\|{\sf Q}\|_2}{\|\nabla_y{\sf Q}\|_2^2}
 +
 \lambda_{\sup}
 \int_{s_1}^{s_2}
 C
 \{
 a^2+\|\Delta_y\epsilon\|_2^2+(-H_y\epsilon,\epsilon)_2
 \}
 ds
 \\
 \label{EQ_3.92}
 &<
 \tfrac{n\lambda_{\sup}\eta_1\|{\sf Q}\|_2}{\|\nabla_y{\sf Q}\|_2^2}
 +
 C
 \lambda_{\sup}
 \eta_1^2
 +
 C
 \lambda_{\sup}
 \int_{s_1}^{s_2}
 (-H_y\epsilon,\epsilon)_2
 ds
 \\
 \nonumber
 &
 \text{for }
 s\in(s_1,s_2),\
 j\in\{1,\cdots,n\}.
 \end{align}
 In the last inequality,
 we used Lemma \ref{LEMMA_3.1}.
 Furthermore
 from Lemma \ref{LEMMA_3.6},
 we observe that
%%%%%%%%%%%%%%%%%%%%%%%%%%%%%%%%%%%%%%%%%%%%%%%%%%%%%%%%%%%%%%%%
 \begin{align}
 \label{EQ_3.93}
 \int_{s_1}^{s_2}
 (-H_y\epsilon(s),\epsilon(s))_2
 ds
 &<
 4
 (\tfrac{\lambda(s_1)}{\lambda_{\inf}})^2
 (
 \|\epsilon(s_1)\|_2^2
 +
 \|\Lambda_y{\sf Q}\|_2^2
 )
 \\
 \nonumber
 &
 \text{\rm for }
 s\in(s_1,s_2).
 \end{align}
%%%%%%%%%%%%%%%%%%%%%%%%%%%%%%%%%%%%%%%%%%%%%%%%%%%%%%%%%%%%%%%%
 Without loss of generality,
 we can assume  that $\eta_1<1$.
 Combining \eqref{EQ_3.92} - \eqref{EQ_3.93},
 we conclude
 \begin{align*}
 |
 z_j(s)
 -
 z_j(s_1)
 |
 &<
 C\lambda_{\sup}
 +
 C
 \lambda_{\sup}
 (\tfrac{\lambda(s_1)}{\lambda_{\inf}})^2
 (\|\epsilon(s_1)\|_2^2+\|\Lambda_y{\sf Q}\|_2^2)
 \\
 &
 \text{for }
 s\in(s_1,s_2),\
 j\in\{1,\cdots,n\}.
 \end{align*}
 This proves \eqref{EQ_3.83}.
 \end{proof}
%%%%%%%%%%%%%%%%%%%%%%%%%%%%%%%%%%%%%%%%%%%%%%%%%%%%%%%%%%%%%%%%

%%%%%%%%%%%%%%%%%%%%%%%%%%%%%%%%%%%%%%%%%%%%%%%%%%%%%%%%%%%%%%%%
 \begin{lem}
 \label{LEMMA_3.9}
 Let $n=6$,
 and
 $u(x,t)$ be as in Lemma {\rm\ref{LEMMA_3.1}}.
 Furthermore
 let $(\Delta{\sf t}_1)$ and ${\sf M}(t)$
 be defined in Proposition {\rm\ref{PROPOSITION_2.3}},
 and define
 \begin{align*}
 {\sf M}_1(t)
 =
 \begin{cases}
 {\sf M}(t) & \text{\rm if } t\in(0,(\Delta{\sf t})_1),\\
 {\sf M}((\Delta{\sf t})_1) & \text{\rm if } t>(\Delta{\sf t})_1.
 \end{cases}
 \end{align*}
 There exist a constant $h_1>0$ depending only on $n$
 such that
 if $\eta_1\in(0,h_1)$,
 then
 \begin{align}
 \label{EQ_3.94}
 \|\epsilon(y,t)\|_{L_y^\infty(\R^n)}
 &<
 \left(\frac{\lambda_{\sup}}{\lambda_{\inf}}\right)^\frac{n-2}{2}
 {\sf M}_1\left(\frac{t-t_1}{\lambda(t_1)^2}\right)
 +
 \|{\sf Q}\|_\infty
 +
 \|\mathcal Y\|_\infty
 \\
 \nonumber
 &\text{\rm for }
 t\in(t_1,t_2).
 \end{align}
 \end{lem}
%%%%%%%%%%%%%%%%%%%%%%%%%%%%%%%%%%%%%%%%%%%%%%%%%%%%%%%%%%%%%%%%
 \begin{proof}
 We recall that
 \begin{align}
 \label{EQ_3.95}
 u(x,t)
 =
 \lambda(t)^{-\frac{n-2}{2}}
 \{
 {\sf Q}(\tfrac{x-z(t)}{\lambda(t)})
 +
 a(t)\mathcal{Y}(\tfrac{x-z(t)}{\lambda(t)})
 +
 \epsilon(\tfrac{x-z(t)}{\lambda(t)},t)
 \}.
 \end{align}
 Fix $t_0\in(t_1,t_2)$.
 We now provide estimates for $\|u(t_0)\|_\infty$.
 We consider the two cases
 $0<\tfrac{t_0-t_1}{\lambda(t_1)^2}<(\Delta{\sf t})_1$ and 
 $\tfrac{t_0-t_1}{\lambda(t_1)^2}>(\Delta{\sf t})_1$ separately.
 First assume that
 $0<\tfrac{t_0-t_1}{\lambda(t_1)^2}<(\Delta{\sf t})_1$,
 and put
 \[
 \tau_0
 =
 \tfrac{t_0-t_1}{\lambda(t_1)^2}
 <
 (\Delta{\sf t}_1).
 \]
 We define
 \begin{align*}
 w(\xi,\tau)
 &=
 \lambda_1^{\frac{n-2}{2}}u(\lambda_1\xi+z_1,t_1+\lambda_1^2\tau)
 \\
 &
 \text{with }
 \lambda_1=\lambda(t_1),\
 z_1=z(t_1).
 \end{align*}
 We easily check that
 $w(\xi,\tau)$ satisfies
 \begin{align*}
 \begin{cases}
 w_\tau=\Delta_\xi w+|w|^{p-1}w 
 \quad\
 \text{for } (\xi,\tau)\in\R^n\times(0,\tau_0),
 \\
 w(\xi,\tau)|_{\tau=0}
 =
 {\sf Q}(\xi)
 +
 a(t_1)\mathcal{Y}(\xi)
 +
 \epsilon(\xi,t_1).
 \end{cases}
 \end{align*}
 From Proposition \ref{PROPOSITION_2.3},
 we have
 \begin{align*}
 \|w(\xi,\tau_0)\|_{L_\xi^\infty(\R^n)}
 <
 {\sf M}(\tau_0).
 \end{align*}
 Note
 that
 $w(\xi,\tau_0)=
 \lambda_1^\frac{n-2}{2}
 u(\lambda_1\xi+z_1,t_0)$.
 Therefore
 it follows that
 \begin{align*}
 \|u(x,t_0)\|_{L_x^\infty(\R^n)}
 &=
 \lambda_1^{-\frac{n-2}{2}}
 \|w(\xi,\tau_0)\|_{L_\xi^\infty(\R^n)}
 <
 \lambda_1^{-\frac{n-2}{2}}
 {\sf M}(\tau_0).
 \end{align*}
 Combining \eqref{EQ_3.95},
 we conclude that
 \begin{align}
 \|\epsilon(y,t_0)\|_{L_y^\infty(\R^n)}
 <
 (\tfrac{\lambda(t_0)}{\lambda(t_1)})^\frac{n-2}{2}
 {\sf M}(\tau_0)
 +
 \|{\sf Q}\|_\infty+a(t_3)\|\mathcal Y\|_\infty.
 \end{align}
 Therefore
 \eqref{EQ_3.94}
 is established
 for $t\in(t_1,t_2)$ such that
 $\frac{t-t_1}{\lambda(t_1)^2}<(\Delta{\sf t})_1$.
 In the same way as above,
 we can derive \eqref{EQ_3.94}
 for the case $\frac{t-t_1}{\lambda(t_1)^2}>(\Delta{\sf t})_1$,
 which completes the proof.
 \end{proof}
%%%%%%%%%%%%%%%%%%%%%%%%%%%%%%%%%%%%%%%%%%%%%%%%%%%%%%%%%%%%%%%%

 \subsection{Differential equation for $\|\nabla_y\epsilon(t)\|_2$}
%%%%%%%%%%%%%%%%%%%%%%%%%%%%%%%%%%%%%%%%%%%%%%%%%%%%%%%%%%%%%%%%
 We here provide a differential equation for $\|\nabla_y\epsilon(t)\|_2$,
 which will be often used in Section \ref{section_4}.
 \begin{lem}[(2.42) in Lemma 2.9 in \cite{Collot-Merle-Raphael} p.\,234]
 \label{LEMMA_3.10}
 Let $n=6$,
 and
 $u(x,t)$ be as in Lemma {\rm\ref{LEMMA_3.1}}.
 There exist constants $h_1>0$ and $C>0$,
 depending only on $n$
 such that if $\eta_1\in(0,h_1)$,
 then
 \begin{align}
 \label{EQ_3.97}
 \frac{d}{ds}
 [-(H_y\epsilon,\epsilon)_2]
 <
 -
 \frac{1}{2}
 \|H_y\epsilon\|_2^2
 +
 Ca^4
 \quad\
 \text{\rm for }
 s\in(s_1,s_2).
 \end{align}
 \end{lem}
%%%%%%%%%%%%%%%%%%%%%%%%%%%%%%%%%%%%%%%%%%%%%%%%%%%%%%%%%%%%%%%%
 \begin{proof}
%%%%%%%%%%%%%%%%%%%%%%%%%%%%%%%%%%%%%%%%%%%%%%%%%%%%%%%%%%%%%%%%
 Multiplying \eqref{EQUATION_e3.9} by $H_y\epsilon$,
 we get
 \begin{align*}
 &
 (\epsilon_s,H_y\epsilon)_2
 =
 \|H_y\epsilon\|_2^2
 +
 \tfrac{\lambda_s}{\lambda}
 (\Lambda_y\epsilon,H_y\epsilon)_2
 +
 \tfrac{1}{\lambda}
 (
 \tfrac{dz}{ds}
 \cdot
 \nabla_y\epsilon,H_y\epsilon)_2
 \\
 & \quad
 +
 \tfrac{\lambda_s}{\lambda}
 \underbrace{
 (\Lambda_y{\sf Q},H_y\epsilon)_2
 }_{=0}
 -
 (a_s-e_0a)
 \underbrace{
 ({\mathcal Y},H_y\epsilon)_2
 }_{=0}
 +
 \tfrac{\lambda_s}{\lambda}
 a
 (\Lambda_y{\mathcal Y},H_y\epsilon)_2
 \\
 &\quad
 +
 \tfrac{1}{\lambda}
 \underbrace{
 (
 \tfrac{dz}{ds}
 \cdot
 \nabla_y{\sf Q},H_y\epsilon)_2
 }_{=0}
 +
 \tfrac{a}{\lambda}
 (\tfrac{dz}{ds}\cdot\nabla_y
 \mathcal{Y},
 H_y\epsilon)_2
 +
 (N,H_y\epsilon)_2.
 \end{align*}
 We apply Lemma \ref{LEMMA_3.11} to get
 \begin{align*}
 &
 (\epsilon_s,H_y\epsilon)_2
 =
 \|H_y\epsilon\|_2^2
 -
 \tfrac{\lambda_s}{\lambda}
 (\epsilon^2,V+\tfrac{y\cdot\nabla_yV}{2})_2
 -
 \tfrac{1}{\lambda}
 (
 \epsilon^2,
 \tfrac{dz}{ds}
 \cdot
 \nabla_y(\tfrac{V}{2}))_2
 \\
 & \quad
 +
 \tfrac{\lambda_s}{\lambda}
 a
 (\Lambda_y{\mathcal Y},H_y\epsilon)_2
 +
 \tfrac{a}{\lambda}
 (\tfrac{dz}{ds}\cdot\nabla_y
 \mathcal{Y},
 H_y\epsilon)_2
 +
 (N,H_y\epsilon)_2.
 \end{align*}
 From the fact that $|V(y)|<(1+|y|^2)^{-2}$
 and the Sobolev inequality,
 we have
 \begin{align*}
 (\epsilon^2,|V|)_2+(\epsilon^2,(1+|y|)|\nabla_yV|)_2
 <
 C\|\Delta_y\epsilon\|_2^2.
 \end{align*}
 Therefore
 we obtain
 \begin{align}
 \nonumber
 -&(\epsilon_s,H_y\epsilon)_2
 <
 -\|H_y\epsilon\|_2^2
 +
 C
 (
 |\tfrac{\lambda_s}{\lambda}|
 \cdot
 \|\Delta_y\epsilon\|_2^2
 +
 \tfrac{1}{\lambda}
 |\tfrac{dz}{ds}|
 \cdot
 \|\Delta_y\epsilon\|_2^2
 \\
 \label{EQ_3.98}
 & \quad
 +
 |\tfrac{\lambda_s}{\lambda}|
 |a|
 \|H_y\epsilon\|_2
 +
 |\tfrac{a}{\lambda}\tfrac{dz}{ds}|
 \cdot
 \|H_y\epsilon\|_2
 +
 \|N\|_2\|H_y\epsilon\|_2
 ).
 \end{align}
 Combining \eqref{EQ_3.98}, \eqref{EQ_3.17}, \eqref{EQ_3.18} and
 \eqref{EQUATION_3.12},
 we conclude \eqref{EQ_3.97}.
 \end{proof}
%%%%%%%%%%%%%%%%%%%%%%%%%%%%%%%%%%%%%%%%%%%%%%%%%%%%%%%%%%%%%%%%
 \begin{lem}[(1.16)-(1.18) in \cite{Collot-Merle-Raphael} p.\,220]
 \label{LEMMA_3.11}
 Let $n\geq3$.
 Then it holds that
 \begin{align}
 \label{EQ_3.99}
 (\Lambda_yv,H_yv)_2
 &=
 -
 (v^2,V+\tfrac{y}{2}\cdot\nabla_yV)_2
 \quad
 \text{\rm for }
 v\in C_\text{\rm c}^\infty(\R^n),
 \\
 \label{EQ_3.100}
 (\nabla_yv,H_yv)_2
 &=
 -
 (v^2,\tfrac{1}{2}\nabla_yV)_2
 \quad
 \text{\rm for }
 v\in C_\text{\rm c}^\infty(\R^n).
 \end{align}
 \end{lem}
%%%%%%%%%%%%%%%%%%%%%%%%%%%%%%%%%%%%%%%%%%%%%%%%%%%%%%%%%%%%%%%%
 \begin{proof}
 From the definition of $\Lambda_y$,
 we have
 \begin{align}
 \label{EQ_3.101}
 (\Lambda_yv,H_yv)_2
 =
 \tfrac{n-2}{2}
 (v,H_yv)_2
 +
 (y\cdot\nabla_yv,H_yv)_2.
 \end{align}
 To simplify the last term of \eqref{EQ_3.101},
 we first compute $(y\cdot\nabla_yv,\Delta_yv)_2$.
 \begin{align*}
 &
 (y\cdot\nabla_yv,\Delta_yv)_2
 =
 -
 n(v,\Delta_yv)_2
 - 
 (v,y\cdot\nabla_y\Delta_yv)_2
 \\
 &=
 -
 n(v,\Delta_yv)_2
 -
 (v,y\cdot\Delta_y(\nabla_yv))_2
 \\
 &=
 -
 n(v,\Delta_yv)_2
 -
 (v,\Delta_y(y\cdot\nabla_yv))_2
 +
 2(v,\Delta_yv)_2
 \\
 &=
 -
 (n-2)
 (v,\Delta_yv)_2
 -
 (\Delta_yv,y\cdot\nabla_yv)_2.
 \end{align*}
 This implies
 \begin{align}
 \label{EQ_3.102}
 (y\cdot\nabla_yv,\Delta_yv)_2
 =
 -
 \tfrac{n-2}{2}
 (v,\Delta_yv)_2.
 \end{align}
 Combining \eqref{EQ_3.101} - \eqref{EQ_3.102},
 we obtain
 \begin{align*}
 (\Lambda_yv,H_yv)_2
 &=
 \tfrac{n-2}{2}
 (v,H_yv)_2
 -
 \tfrac{n-2}{2}
 (v,\Delta_yv)_2
 +
 (y\cdot\nabla_yv,Vv)
 \\
 &=
 \tfrac{n-2}{2}
 (v,\Delta_yv+Vv)_2
 -
 \tfrac{n-2}{2}
 (v,\Delta_yv)_2
 +
 (y\cdot\nabla_y(\tfrac{v^2}{2}),V)_2
 \\
 &=
 \tfrac{n-2}{2}
 (v,Vv)_2
 -
 \tfrac{n}{2}
 (v^2,V)_2
 -
 (v^2,\tfrac{y\cdot\nabla_yV}{2})_2
 \\
 &=
 -
 (v^2,V)_2
 -
 (v^2,\tfrac{y\cdot\nabla_yV}{2})_2.
 \end{align*}
 This shows \eqref{EQ_3.99}.
 Since the proof of \eqref{EQ_3.100} is similar,
 we omit it.
 \end{proof}
%%%%%%%%%%%%%%%%%%%%%%%%%%%%%%%%%%%%%%%%%%%%%%%%%%%%%%%%%%%%%%%%

 \section{Section}
 \label{section_4}
%%%%%%%%%%%%%%%%%%%%%%%%%%%%%%%%%%%%%%%%%%%%%%%%%%%%%%%%%%%%%%%%
 Let $h_1^*$ be the smallest constant among the values of $h_1$
 given in Lemma \ref{LEMMA_3.1} - Lemma \ref{LEMMA_3.11}.
 We fix $\bar\alpha\in(0,h_1^*)$.
 From (iii) of Lemma \ref{LEMMA_2.7},
 there exist $\bar \eta>0$
 such that,
 if we define $v(y)$ by
 \begin{align}
 \label{equation_Q4.1}
 v(y)
 =
 \mathcal{K}_{\mathcal M}(u)^\frac{n-2}{2}u(\mathcal{K}_{\mathcal M}(u) y+\mathcal{Z}_{\mathcal M}(u))
 -
 {\sf Q}(y)
 -
 \mathcal{A}_{\mathcal M}(u)\mathcal{Y}(y),
 \end{align}
 then the following inequality holds.
 \begin{align*}
 \|v(y)\|_{\dot H_y^1(\R^n)}^2
 +
 |\mathcal{A}_{\mathcal M}(u)|^2\|\mathcal{Y}\|_{\dot H^1(\R^n)}^2
 &<
 (\bar\alpha)^2
 \quad\
 \text{for all }
 u\in \mathcal{U}_{\bar\eta}.
 \end{align*}
%%%%%%%%%%%%%%%%%%%%%%%%%%%%%%%%%%%%%%%%%%%%%%%%%%%%%%%%%%%%%%%%
 Fix $\alpha$ such that
 \begin{align*}
 0
 <
 \alpha
 <
 \min\{\tfrac{\bar\eta}{2},\bar\alpha\}.
 \end{align*}
 Furthermore
 let $\delta_1\in(0,\alpha)$ be a sufficiently small constant,
 which will be chosen later.
 We again apply (iii) of Lemma \ref{LEMMA_2.7}.
 There exists $\eta_1\in(0,\bar\eta)$ such that
 for all $u\in \mathcal{U}_{\eta_1}$,
 it holds that
 \begin{align*}
 \|v(y)\|_{\dot H_y^1(\R^n)}^2
 +
 |\mathcal{A}_{\mathcal M}(u)|^2\|\mathcal{Y}\|_{\dot H^1(\R^n)}^2
 &<
 \delta_1^2,
 \end{align*}
 where $v(y)$ is a function as defined in \eqref{equation_Q4.1}.
 Throughout this section,
 we assume
 \begin{align*}
 \|\nabla_xu_0(x)-\nabla_x{\sf Q}(x)\|_{L_x^2(\R^n)}
 <
 \eta_1.
 \end{align*}
 This implies $u_0\in\mathcal{U}_{\eta_1}\subset\mathcal{U}_{\bar\eta}$.
 We denote by $T$ the maximal existence time of $u(x,t)$.
 Regarding the behavior of $u(x,t)$,
 there are four possibilities.
 \vspace{2mm}
 \begin{enumerate}[(c1)]
 \item
 \label{(c1)}
 $T<\infty$ and $\text{dist}_{\dot H^1(\R^n)}(u(t),\mathcal{M})<\bar\eta$ for $t\in(0,T)$,

 \item
 \label{(c2)}
 $T<\infty$ and there exists $T_1\in(0,T)$ such that $\text{dist}_{\dot H^1(\R^n)}(u(T_1),\mathcal{M})=\bar\eta$,
 
 \item
 \label{(c3)}
 $T=\infty$
 and $\text{dist}_{\dot H^1(\R^n)}(u(t),\mathcal{M})<\bar\eta$ for $t\in(0,\infty)$,

 \item
 \label{(c4)}
 $T=\infty$
 and there exists $T_1\in(0,T)$ such that $\text{dist}_{\dot H^1(\R^n)}(u(T_1),\mathcal{M})=\bar\eta$.
 \end{enumerate}
 \vspace{2mm}
 Throughout this section,
 we normalize $\mathcal Y(y)>0$ such that
 \[
 \|\mathcal Y\|_2=1
 \label{mathcal_Y}.
 \]
 While the argument in this section is essentially the same as that in \cite{Collot-Merle-Raphael} p.\,244 - p.\,267,
 we provide full proofs for the sake of completeness.
 The difference from \cite{Collot-Merle-Raphael} is that
 our discussion does not require the control of $\dot H^2(\R^n)$ norm of the solution.
 To derive the $L^\infty$ bound of the solution,
 we apply Proposition \ref{PROPOSITION_2.3} here,
 instead of using the boundedness of the $\dot H^2$ norm of the solution.

 \subsection{The proof for Theorem \ref{THEOREM_2}}
%%%%%%%%%%%%%%%%%%%%%%%%%%%%%%%%%%%%%%%%%%%%%%%%%%%%%%%%%%%%%%%%
 From Lemma \ref{LEMMA_3.8} - Lemma \ref{LEMMA_3.9},
 we conclude that (c1) cannot occur.
 Consequently
 Theorem \ref{THEOREM_2} follows from Proposition \ref{PROPOSITION_4.1},
 Proposition \ref{PROPOSITION_4.6} and Proposition \ref{PROPOSITION_4.7}.
 In the rest of this section,
 we discuss the cases (c2) - (c4) separately.
%%%%%%%%%%%%%%%%%%%%%%%%%%%%%%%%%%%%%%%%%%%%%%%%%%%%%%%%%%%%%%%%

 \subsection{Behavior of a solution for case (c3)}
 \label{section_4.2}
%%%%%%%%%%%%%%%%%%%%%%%%%%%%%%%%%%%%%%%%%%%%%%%%%%%%%%%%%%%%%%%%
 We first investigate the behavior of the solution for case (c3).
 We again use the time variable $s$ defined by
 \begin{align}
 \label{EQ_4.2}
 s
 =
 s(t)
 =
 \int_0^{t}
 \tfrac{dt'}{\lambda(t')^2}.
 \end{align}
 From Lemma \ref{LEMMA_3.8},
 we recall that $\inf_{t\in(0,\infty)}\lambda(t)=\lambda_{\inf}>0$ .
 Hence
 under the transformation \eqref{EQ_4.2}, the time interval $0<t<\infty$ is mapped to $0<s<\infty$.
%%%%%%%%%%%%%%%%%%%%%%%%%%%%%%%%%%%%%%%%%%%%%%%%%%%%%%%%%%%%%%%%
 \begin{pro}
 \label{PROPOSITION_4.1}
 Let $n=6$ and assume {\rm (c3)} on {\rm p.\,\pageref{(c3)}}.
 There exist $\lambda_\infty\in(0,\infty)$ and $z_\infty\in\R^n$
 such that
 \begin{align*}
 \lim_{t\to\infty}
 u(x,t)
 =
 \tfrac{1}{\lambda_\infty^\frac{n-2}{2}}{\sf Q}(\tfrac{x-z_\infty}{\lambda_\infty})
 \quad\
 \text{\rm in }
 \dot H^1(\R^n).
 \end{align*}
 \end{pro}
%%%%%%%%%%%%%%%%%%%%%%%%%%%%%%%%%%%%%%%%%%%%%%%%%%%%%%%%%%%%%%%%

%%%%%%%%%%%%%%%%%%%%%%%%%%%%%%%%%%%%%%%%%%%%%%%%%%%%%%%%%%%%%%%%
 \begin{lem}
 \label{LEMMA_4.2}
 Let $n=6$ and assume {\rm (c3)}.
 For any fixed $R>0$,
 we have
 \[
 \lim_{s\to\infty}
 \sup_{|y|<R}|\epsilon(y,s)|
 =0.
 \]
 \end{lem}
%%%%%%%%%%%%%%%%%%%%%%%%%%%%%%%%%%%%%%%%%%%%%%%%%%%%%%%%%%%%%%%%
 \begin{proof}
 We first rewrite the nonlinear term $N$ defined in \eqref{EQUATION_e3.6} on p.\,\pageref{EQUATION_e3.6} as
 \begin{align*}
 N
 &=
 f({\sf Q}+a{\mathcal Y}+\epsilon)
 -
 f({\sf Q})
 -
 f'({\sf Q})
 (a{\mathcal Y}+\epsilon)
 \\
 &=
 \underbrace{
 f({\sf Q}+a{\mathcal Y}+\epsilon)
 -
 f({\sf Q}+a{\mathcal Y})
 }_{=N_1}
 -
 f'({\sf Q})
 \epsilon
 \\
 &\quad
 +
 \underbrace{
 f({\sf Q}+a{\mathcal Y})
 -
 f({\sf Q})
 -
 f'({\sf Q})
 (a{\mathcal Y})
 }_{N_2}.
 \end{align*}
 We express \eqref{EQUATION_e3.9} as
 \begin{align}
 \nonumber
 \epsilon_s
 &=
 \Delta_y\epsilon
 +
 \{
 \underbrace{
 V+\tfrac{N_1}{\epsilon}-f'({\sf Q})+\tfrac{\lambda_s}{\lambda}\tfrac{n-2}{2}
 }_{=V_1}
 \}
 \epsilon
 +
 \{
 \underbrace{
 \tfrac{\lambda_s}{\lambda}
 y
 +
 \tfrac{1}{\lambda}
 \tfrac{dz}{ds}
 }_{={\bf b}_1}
 \}
 \cdot
 \nabla_y\epsilon
 \\
 \nonumber
 &\quad
 +
 \underbrace{
 \tfrac{\lambda_s}{\lambda}
 (\Lambda_y{\sf Q})
 -
 (
 a_s
 -
 e_0a
 )
 {\mathcal Y}
 +
 \tfrac{\lambda_s}{\lambda}
 a
 (\Lambda_y{\mathcal Y})
 +
 \tfrac{1}{\lambda}
 \tfrac{dz}{ds}\cdot\nabla_y
 {\sf Q}
 }_{=g}
 \\
 \nonumber
 &\quad
 \underbrace{
 +
 \tfrac{a}{\lambda}
 \tfrac{dz}{ds}\cdot\nabla_y
 \mathcal{Y}
 +
 N_2
 }_{=g}
 \\
 \label{EQ_4.3}
 &=
 \Delta_y\epsilon
 +
 V_1\epsilon
 +
 {\bf b}_1
 \cdot
 \nabla_y\epsilon
 +
 g.
 \end{align}
 From Lemma \ref{LEMMA_3.8},
 there exist positive constants $\lambda_{\inf}$ and $\lambda_{\sup}$
 depending on $n$, $\lambda(s)|_{s=0}$ and $\|\epsilon(s)|_{s=0}\|_2$
 such that
 \begin{align}
 \label{EQ_4.4}
 \lambda_{\inf}
 <
 \lambda(s)
 <
 \lambda_{\sup}
 \quad\
 \text{for }
 s\in(0,\infty).
 \end{align}
%%%%%%%%%%%%%%%%%%%%%%%%%%%%%%%%%%%%%%%%%%%%%%%%%%%%%%%%%%%%%%%%
 Let $(\Delta{\sf t})_1$ be a positive constant given in Lemma \ref{LEMMA_3.9},
 and define
 \begin{align*}
 (\Delta{\sf s})_1
 =
 \int_0^{(\Delta{\sf t})_1}
 \tfrac{dt}{\lambda(t)^2}.
 \end{align*}
 From Lemma \ref{LEMMA_3.9},
 there exists a positive constant $m_{\sup}$
 depending on $n$, $\lambda_{\inf}$ and $\lambda_{\sup}$
 such that
 \begin{align}
 \label{EQ_4.5}
 \|\epsilon(y,s)\|_{L_y^\infty(\R^n)}<m_{\sup}
 \quad\
 \text{for }
 s\in((\Delta{\sf s})_1,\infty).
 \end{align}
%%%%%%%%%%%%%%%%%%%%%%%%%%%%%%%%%%%%%%%%%%%%%%%%%%%%%%%%%%%%%%%%
 Due to \eqref{EQ_4.4} - \eqref{EQ_4.5} and \eqref{EQ_3.17} - \eqref{EQ_3.18},
 there exists a constant $\bar C>0$ depending on $\lambda_{\inf}$, $\lambda_{\sup}$ and $m_{\sup}$
 such that 
 \begin{align*}
 |V_1(y,s)|
 <
 \bar C
 \quad\
 \text{for }
 (y,s)\in\R^n\times((\Delta{\sf s})_1,\infty),
 \\
 |{\bf b}_1(y,s)|
 <
 \bar CR
 \quad\
 \text{for }
 |y|<R,\ s\in((\Delta{\sf s})_1,\infty).
 \end{align*}
 Hence
 applying a standard local parabolic estimate to \eqref{EQ_4.3},
 we find that
 there exists a constant $\bar D_R>0$ such that
 \begin{align*}
 &
 \sup_{\sigma\in(s,s-\frac{1}{2})}
 \sup_{\xi\in B(y;1)}
 |\epsilon(\xi,\sigma)|
 <
 \bar D_R
 \left(
 \int_{s-1}^s
 \|\epsilon(\xi,\sigma)\|_{L_\xi^2(B(y;2))}^2
 d\sigma
 \right)^\frac{1}{2}
 \\
 &
 +
 \bar D_R
 \left(
 \int_{s-1}^s
 \|g(\xi,\sigma)\|_{L_\xi^\infty(B(y;2))}^2
 d\sigma
 \right)^\frac{1}{2}
 \quad\
 \text{for }
 |y|<R,\ s>(\Delta{\sf s})_1+1.
 \end{align*}
 Since $|N_2|<Ca^2\mathcal{Y}^2$ when $n=6$,
 combining Lemma \ref{LEMMA_3.2},
 we get
 \begin{align*}
 \|g(\xi,\sigma)\|_{L_\xi^\infty(\R^n)}
 &<
 C(
 |\tfrac{\lambda_s}{\lambda}|+|a_s-e_0a|+|\tfrac{1}{\lambda}\tfrac{d\lambda}{ds}|+\tfrac{a}{\lambda}|\tfrac{dz}{ds}|+a^2
 )
 \\
 &<
 C(|a|+\|\Delta_y\epsilon\|_2).
 \end{align*}
 Furthermore
 the Sobolev inequality implies
 \begin{align*}
 \|\epsilon(\xi,\sigma)\|_{L_\xi^2(B(y;2))}^2
 &<
 C
 \|\epsilon(\xi,\sigma)\|_{L_\xi^\frac{2n}{n-4}(B(y;2))}^2
 <
 C
 \|\epsilon(\xi,\sigma)\|_{L_\xi^\frac{2n}{n-4}(\R^n)}^2
 \\
 &<
 C
 \|\Delta_y\epsilon(y,\sigma)\|_{L_y^2(\R^n)}^2.
 \end{align*}
 Hence
 we conclude that
 \begin{align}
 \label{EQ_4.6}
 \sup_{\sigma\in(s,s-\frac{1}{2})}
 \sup_{\xi\in B(y;1)}
 |\epsilon(\xi,\sigma)|
 &<
 \bar D_R
 \sup_{|z|<2R}
 \left(
 \int_{s-1}^s
 (a^2+\|\Delta_y\epsilon\|_2^2)
 d\sigma
 \right)^\frac{1}{2}
 \\
 \nonumber
 &
 \text{for }
 |y|<R,\ s>(\Delta{\sf s})_1+1.
 \end{align}
 From Lemma \ref{LEMMA_3.1},
 we recall that $\int_0^\infty(a^2+\|\Delta_y\epsilon\|_2^2)ds<\infty$.
 Therefore
 the integral on the right hand side of \eqref{EQ_4.6} converges to zero as $s\to\infty$.
 The proof is completed.
 \end{proof}
%%%%%%%%%%%%%%%%%%%%%%%%%%%%%%%%%%%%%%%%%%%%%%%%%%%%%%%%%%%%%%%%

%%%%%%%%%%%%%%%%%%%%%%%%%%%%%%%%%%%%%%%%%%%%%%%%%%%%%%%%%%%%%%%%
 \begin{lem}
 \label{LEMMA_4.3}
 Let $n=6$ and assume {\rm (c3)}.
 There exist $\lambda_\infty\in(0,\infty)$ and $z_\infty\in\R^n$
 such that
 \begin{align}
 \label{EQ_4.7}
 \lim_{s\to\infty}
 \lambda(s)=\lambda_\infty,
 \\
 \label{EQ_4.8}
 \lim_{s\to\infty}
 z(s)=z_\infty.
 \end{align}
 \end{lem}
%%%%%%%%%%%%%%%%%%%%%%%%%%%%%%%%%%%%%%%%%%%%%%%%%%%%%%%%%%%%%%%%
 \begin{proof}
 From \eqref{EQ_3.81} - \eqref{EQ_3.83} in Lemma \ref{LEMMA_3.8},
 there exists a sequence $\{s_k\}_{k=1}^\infty$ with $s_k\to\infty$ such that
 \begin{align*}
 \lim_{k\to\infty}
 \lambda(s_k)
 &=
 \lambda_\infty
 \quad\
 \text{for some } \lambda_\infty>0,
 \\
 \lim_{k\to\infty}
 z(s_k)
 &=
 z_\infty
 \quad\
 \text{for some } z_\infty\in\R^n.
 \end{align*}
 Integrating both sides of \eqref{EQ_3.29} from $s_k$ to $s$ ($s_k<s$),
 we get
 \begin{align}
 \nonumber
 &|\log\tfrac{\lambda(s)}{\lambda(s_k)}|
 <
 \tfrac{n(n-2)}{4}
 \sum_{j=1}^n
 \tfrac{(\epsilon(s),\pa_{y_j}{\sf Q})_2^2+(\epsilon(s_k),\pa_{y_j}{\sf Q})_2^2}
 {\|\nabla_y{\sf Q}\|_2^2\|\Lambda_y{\sf Q}\|_2^2}
 \\
 \nonumber
 &\quad
 +
 \tfrac{|(\epsilon(s),\Lambda_y{\sf Q})_2|+|(\epsilon(s_k),\Lambda_y{\sf Q})_2|}{\|\Lambda_y{\sf Q}\|_2^2}
 +
 C
 \int_{s_k}^s
 (a^2+\|\Delta_y\epsilon\|_2^2)
 ds'
 \\
 \label{EQ_4.9}
 &<
 C
 \sup_{s>s_k}
 \left\{
 \sum_{j=1}^n
 (\epsilon(s),\pa_{y_j}{\sf Q})_2^2
 +
 |(\epsilon(s),\Lambda_y{\sf Q})_2|
 \right\}
 \\
 \nonumber
 &\quad
 +
 C
 \int_{s_k}^s
 (a^2+\|\Delta_y\epsilon\|_2^2)
 ds'.
 \end{align}
 Since $|\Lambda_y{\sf Q}(y)|<(1+|y|^2)^{-\frac{n-2}{2}}$,
 we easily see that
 \begin{align}
 \nonumber
 &
 |(\epsilon(s),\Lambda_y{\sf Q})_2|
 <
 C
 \int_{|y|<R}
 |\epsilon(y,s)|
 |\Lambda_y{\sf Q}(y)|
 dy
 \\
 \nonumber
 &\quad
 +
 C
 \int_{|y|>R}
 |\epsilon(y,s)|
 |\Lambda_y{\sf Q}(y)|
 dy
 \\
 \label{EQ_4.10}
 &<
 CR^2
 \sup_{|y|<R}|\epsilon(y,s)|
 +
 C
 R^{-\frac{n-4}{2}}
 \|\epsilon(s)\|_2.
 \end{align}
 From Lemma \ref{LEMMA_3.6} and Lemma \ref{LEMMA_3.8},
 we have
 \begin{align}
 \label{EQ_4.11}
 \sup_{s\in(0,\infty)}\|\epsilon(s)\|_2
 <
 \infty.
 \end{align}
 Therefore
 from \eqref{EQ_4.10} - \eqref{EQ_4.11} and Lemma \ref{LEMMA_4.2},
 we deduce that
 \begin{align}
 \label{EQ_4.12}
 \lim_{s\to\infty}
 |(\epsilon(s),\Lambda_y{\sf Q})_2|
 =0.
 \end{align}
 In the same manner,
 we can verify that
 \begin{align}
 \label{EQ_4.13}
 \lim_{s\to\infty}
 \sum_{j=1}^n
 (\epsilon(s),\pa_{y_j}{\sf Q})_2^2
 =0.
 \end{align}
 Since $\int_0^\infty(a^2+\|\Delta_y\epsilon\|_2^2)ds<\infty$ (see Lemma \ref{LEMMA_3.1}),
 \eqref{EQ_4.12} - \eqref{EQ_4.13} imply that the right hand side of \eqref{EQ_4.9} converges to $0$
 as $k\to\infty$,
 which shows \eqref{EQ_4.7}. 
 We next prove \eqref{EQ_4.8}.
 We rewrite \eqref{EQ_3.30} as
%%%%%%%%%%%%%%%%%%%%%%%%%%%%%%%%%%%%%%%%%%%%%%%%%%%%%%%%%%%%%%%%
 \begin{align*}
 &
 \left|
 \tfrac{dz_j}{ds}
 -
 \tfrac{n\lambda(\epsilon_s,\pa_{y_j}{\sf Q})_2}{\|\nabla_y{\sf Q}\|_2^2}
 +
 \tfrac{\lambda_s}{\lambda}
 \tfrac{
 n\lambda(\Lambda_y\epsilon,\pa_{y_j}{\sf Q})_2
 }{\|\nabla_y{\sf Q}\|_2^2}
 \right|
 \\
 \nonumber
 &=
 \left|
 \tfrac{dz_j}{ds}
 -
 \tfrac{d}{ds}
 \tfrac{n\lambda(\epsilon,\pa_{y_j}{\sf Q})_2}{\|\nabla_y{\sf Q}\|_2^2}
 +
 \tfrac{\lambda_s}{\lambda}
 \tfrac{n\lambda(\epsilon,\pa_{y_j}{\sf Q})_2}{\|\nabla_y{\sf Q}\|_2^2} 
 +
 \tfrac{\lambda_s}{\lambda}
 \tfrac{
 n\lambda(\Lambda_y\epsilon,\pa_{y_j}{\sf Q})_2
 }{\|\nabla_y{\sf Q}\|_2^2}
 \right|
 \\
 \nonumber
 &<
 C\lambda(a^2+\|\Delta_y\epsilon\|_2^2)
 \quad\
 \text{for }
 j\in\{1,\cdots,n\}.
 \end{align*}
 Integrating both sides from $s_k$ to $s$ ($s_k<s$),
 we get
 \begin{align*}
 &|
 z_j(s)-z_j(s_k)
 |
 <
 \tfrac{n\{\lambda(s)|(\epsilon(s),\pa_{y_j}{\sf Q})_2|+\lambda(s_k)|(\epsilon(s_k),\pa_{y_j}{\sf Q})_2|\}}
 {\|\nabla_y{\sf Q}\|_2^2}
 \\
 &\quad
 +
 \int_{s_k}^s
 |\tfrac{\lambda_s}{\lambda}|
 \{
 \tfrac{n\lambda|(\epsilon,\pa_{y_j}{\sf Q})_2|}{\|\nabla_y{\sf Q}\|_2^2}
 +
 \tfrac{
 n\lambda|(\Lambda_y\epsilon,\pa_{y_j}{\sf Q})_2|
 }{\|\nabla_y{\sf Q}\|_2^2}
 \}
 ds'
 \\
 &\quad
 +
 C
 \int_{s_k}^s
 \lambda(a^2+\|\Delta_y\epsilon\|_2^2)
 ds'
 \\
 &<
 C
 \left(
 \sup_{s\in(0,\infty)}
 \lambda(s)
 \right)
 \left\{
 \sup_{s>s_k}
 (|\epsilon(s)|,|\nabla_y{\sf Q}|)_2
 \right.
 \\
 &\quad
 \left.
 +
 \int_{s_k}^s
 (
 |\tfrac{\lambda_s}{\lambda}|^2
 +
 \|\nabla_y\epsilon\|_2^2
 +
 a^2+\|\Delta_y\epsilon\|_2^2
 )
 \right\}
 ds'.
 \end{align*}
 From Lemma \ref{LEMMA_3.6} and Lemma \ref{LEMMA_3.8},
 we find that
 \begin{align}
 \label{EQ_4.14}
 \int_0^\infty\|\nabla_y\epsilon\|_2^2ds
 <
 \infty.
 \end{align}
 Therefore
 in the same reason as above,
 we conclude \eqref{EQ_4.8}.
 \end{proof}
%%%%%%%%%%%%%%%%%%%%%%%%%%%%%%%%%%%%%%%%%%%%%%%%%%%%%%%%%%%%%%%%

%%%%%%%%%%%%%%%%%%%%%%%%%%%%%%%%%%%%%%%%%%%%%%%%%%%%%%%%%%%%%%%%
 \begin{proof}[Proof of Proposition {\rm\ref{PROPOSITION_4.1}}]
 From \eqref{EQ_4.14} and Lemma \ref{LEMMA_3.1},
 we have
 \begin{align}
 \label{equation_4.15}
 \int_0^\infty
 (\|\nabla_y\epsilon(s)\|_2^2+a(s)^2)ds
 <
 \infty.
 \end{align}
 Note from \eqref{EQUATION_e3.3} that
 $(-H_y\epsilon(t),\epsilon(t))_{L_y^2(\R^n)}<\bar C_2\|\nabla_y\epsilon(t)\|_2^2$.
 Therefore
 there exists a sequence $\{s_k\}_{k\in\N}$ satisfing $s_k\to\infty$
 such that
 \begin{align}
 \label{equation_4.16}
 \lim_{k\to\infty}(-H_y\epsilon(s_k),\epsilon(s_k))_2=0
 \quad
 \text{and}
 \quad
 \lim_{k\to\infty}a(s_k)=0.
 \end{align}
 Integrating both sides of \eqref{EQ_3.97} from $s_k$ to $s$ ($s_k<s$),
 we obtain
 \begin{align}
 \label{equation_4.17}
 (-H_y\epsilon(s),\epsilon(s))_2
 <
 (-H_y\epsilon(s_k),\epsilon(s_k))_2
 +
 C
 \int_{s_k}^s
 a^4
 ds'.
 \end{align}
 Therefore
 from \eqref{equation_4.15} - \eqref{equation_4.17},
 we deduce that
 $\lim_{s\to\infty}(-H_y\epsilon(s),\epsilon(s))_2=0$.
 As a consequence,
 it follows from \eqref{EQUATION_e3.3} that
 \begin{align}
 \label{equation_4.18}
 \lim_{s\to\infty}
 \|\nabla_y\epsilon(s)\|_2=0.
 \end{align}
 We recall from \eqref{EQ_3.16} that $|a_s-e_0a|<Ca^2+\|\Delta_y\epsilon\|_2^2$.
 In the same manner as above,
 we can shows that
 \begin{align}
 \label{equation_4.19}
 \lim_{s\to\infty}a(s)=0.
 \end{align}
 Since $u(x,t)$ is expressed as
 $u(x,t)
 =
 \lambda(t)^{-\frac{n-2}{2}}
 \{
 {\sf Q}(y)
 +
 a(t)\mathcal{Y}(y)
 +
 \epsilon(y,t)
 \}$
 with $y=\tfrac{x-z(t)}{\lambda(t)}$,
 from \eqref{EQ_4.7} - \eqref{EQ_4.8} and \eqref{equation_4.18} - \eqref{equation_4.19},
 we conclude
 \begin{align*}
 \lim_{t\to\infty}
 u(x,t)
 =
 \tfrac{1}{\lambda_\infty^\frac{n-2}{2}}{\sf Q}(\tfrac{x-z_\infty}{\lambda_\infty})
 \quad\
 \text{\rm in }
 \dot H^1(\R^n).
 \end{align*}
 This completes the proof of Proposition \ref{PROPOSITION_4.1}.
 \end{proof}
%%%%%%%%%%%%%%%%%%%%%%%%%%%%%%%%%%%%%%%%%%%%%%%%%%%%%%%%%%%%%%%%

 \subsection{Behavior of a solution for case (c4)}
 \label{section_4.3}
%%%%%%%%%%%%%%%%%%%%%%%%%%%%%%%%%%%%%%%%%%%%%%%%%%%%%%%%%%%%%%%%
 Let $\alpha$, $\bar\alpha$, $\delta_1$, $\bar\eta$ and $\eta_1$
 be small positive constants
 introduced at the beginning of Section \ref{section_4} (see p.\,\pageref{section_4}).
 We introduce another small positive constant $\delta$ defined by
 \begin{align*}
 \delta
 =
 \delta_1^2.
 \end{align*}
 From (iii) of Lemma \ref{LEMMA_2.7},
 there exists $\eta\in(0,\eta_1)$
 such that
 \begin{align*}
 \|v(y)\|_{\dot H_y^1(\R^n)}^2
 +
 |\mathcal{A}_{\mathcal M}(u)|^2\|\mathcal{Y}\|_{\dot H^1(\R^n)}^2
 &<
 \delta^2
 \quad\
 \text{for all }
 u\in \mathcal{U}_{\eta},
 \end{align*}
 where $v(y)$ is a function as defined in \eqref{equation_Q4.1}.
 Here we have
 \begin{align}
 \label{equation_4.20}
 \begin{cases}
 0<\delta<\delta_1<\alpha<\{\tfrac{\bar\eta}{2},\bar\alpha\}\leq\bar\alpha<h_1^* & \text{and}
 \\
 0<\eta<\eta_1<\bar\eta.
 \end{cases}
 \end{align}
 We now assume
 \begin{align}
 \label{equation_4.21}
 \|\nabla_xu_0(x)-\nabla_x{\sf Q}(x)\|_{L_x^2(\R^n)}
 <
 \eta.
 \end{align}
 This implies $u_0\in\mathcal{U}_{\eta}\subset\mathcal{U}_{\bar\eta}$.
%%%%%%%%%%%%%%%%%%%%%%%%%%%%%%%%%%%%%%%%%%%%%%%%%%%%%%%%%%%%%%%%
 By the continuity of the solution,
 there exists $t_0>0$ such that
 \[
 u(t)\in\mathcal{U}_{\bar\eta}
 \quad
 \text{for }
 t\in[0,t_0].
 \]
 From the definition of $\bar\eta$ (see the beginning of Section \ref{section_4} p.\,\pageref{section_4}),
 $u(x,t)$ can be decomposed as
 \begin{align}
 \label{equation_4.22}
 u(x,t)
 =
 \lambda(t)^{-\frac{n-2}{2}}
 \{
 {\sf Q}(y)
 +
 a(t)\mathcal{Y}(y)
 +
 \epsilon(y,t)
 \}
 \end{align}
 with
 $y=\tfrac{x-z(t)}{\lambda(t)}$ for $t\in[0,t_1]$.
 Furthermore
 by the choice of $\eta$,
 we have
 \begin{align}
 \label{equation_4.23}
 a(t)^2\|\nabla_y\mathcal{Y}\|_2^2
 +
 \|\nabla_y\epsilon(t)\|_2^2
 <
 \delta^2
 =
 \delta_1^4
 \quad\
 \text{for }
 t=0.
 \end{align}
 For case (c4),
 there exists $T_1\in(0,\infty)$ such that
 \begin{align}
 \label{equation_4.24}
 \begin{cases}
 \text{dist}_{\dot H^1(\R^n)}(u(t),\mathcal{M})<\bar\eta$ \quad for $t\in[0,T_1),
 \\
 \text{dist}_{\dot H^1(\R^n)}(u(t),\mathcal{M})=\bar\eta$ \quad for $t=T_1.
 \end{cases}
 \end{align}
 Suppose that
 \begin{align}
 \label{equation_4.25}
 a(t)^2\|\nabla_y\mathcal{Y}\|_2^2
 +
 \|\nabla_y\epsilon(t)\|_2^2
 <
 \alpha^2
 \quad\
 \text{for }
 t\in[0,T_1].
 \end{align}
 From \eqref{equation_4.22} and \eqref{equation_4.25},
 it holds that
 \begin{align*}
 &\|\nabla_xu(x,t)-\nabla_x\lambda(t)^{-\frac{n-2}{2}}{\sf Q}(\tfrac{x-z(t)}{\lambda(t)})\|_{L_x^2(\R^n)}
 \\
 &\quad
 =
 \|
 \nabla_x\lambda(t)^{-\frac{n-2}{2}}
 \{
 a(t)\mathcal{Y}(\tfrac{x-z(t)}{\lambda(t)})
 +
 \epsilon(\tfrac{x-z(t)}{\lambda(t)},t)
 \}
 \|_{L_x^2(\R^n)}
 \\
 &\quad
 =
 \|
 a(t)\nabla_y\mathcal{Y}(y)
 +
 \nabla_y\epsilon(y,t)
 \|_{L_y^2(\R^n)}
 \\
 &\quad
 <
 |a(t)|
 \cdot
 \|
 \nabla_y\mathcal{Y}
 \|_2
 +
 \|
 \nabla_y\epsilon(t)
 \|_2
 \\
 &\quad
 <
 \sqrt{2}
 \sqrt{
 |a(t)|^2
 \|
 \nabla_y\mathcal{Y}
 \|_2^2
 +
 \|
 \nabla_y\epsilon(t)
 \|_2^2
 }
 \\
 &\quad
 <
 \sqrt{2}\alpha
 \quad\
 \text{for }
 t\in[0,T_1].
 \end{align*}
 Since $\alpha<\frac{\bar \eta}{2}$ (see \eqref{equation_4.20}),
 we obtain
 \begin{align*}
 \|\nabla_xu(x,t)-\nabla_x\lambda(t)^{-\frac{n-2}{2}}{\sf Q}(\tfrac{x-z(t)}{\lambda(t)})\|_{L_x^2(\R^n)}
 <
 \tfrac{\sqrt{2}}{2}
 \bar\eta
 \quad\
 \text{for }
 t\in[0,T_1].
 \end{align*}
 This contradicts \eqref{equation_4.24}.
 Hence
 \eqref{equation_4.25} does not hold.
 Therefore
 under the conditions \eqref{equation_4.21} and \eqref{equation_4.24},
 there exist $T_2$ and $T_3$ satisfying $0<T_2<T_3<T_1$ such that
%%%%%%%%%%%%%%%%%%%%%%%%%%%%%%%%%%%%%%%%%%%%%%%%%%%%%%%%%%%%%%%%
 \begin{align}
 \label{equation_4.26}
 &
 \begin{cases}
 a(t)^2\|\nabla_y\mathcal{Y}\|_2^2
 +
 \|\nabla_y\epsilon(t)\|_2^2
 <
 \delta_1^2
 \quad\
 \text{for }
 t\in[0,T_2),
 \\
 a(t)^2\|\nabla_y\mathcal{Y}\|_2^2
 +
 \|\nabla_y\epsilon(t)\|_2^2
 =
 \delta_1^2
 \quad\
 \text{for }
 t=T_2,
 \end{cases}
 \\[2mm]
 \label{equation_4.27}
 &
 \begin{cases}
 a(t)^2\|\nabla_y\mathcal{Y}\|_2^2
 +
 \|\nabla_y\epsilon(t)\|_2^2
 <
 \alpha^2
 \quad\
 \text{for }
 t\in[0,T_3),
 \\
 a(t)^2\|\nabla_y\mathcal{Y}\|_2^2
 +
 \|\nabla_y\epsilon(t)\|_2^2
 =
 \alpha^2
 \quad\
 \text{for }
 t=T_3.
 \end{cases}
 \end{align}
 For convenience,
 we write
 \begin{align}
 \label{equation_4.28}
 S_1=s(T_1),
 \quad\
 S_2=s(T_1),
 \quad\
 S_3=s(T_3).
 \end{align}
 From Lemma \ref{LEMMA_3.1},
 we note that
 $\int_0^{S_2}(a^2+\|\Delta_y\epsilon\|_2^2)ds<C\delta_1^2$.
 Integrating both sides of \eqref{EQ_3.97} over $s\in(0,S_2)$,
 and applying \eqref{EQUATION_e3.3}, \eqref{equation_4.23} and \eqref{equation_4.26},
 we get
 \begin{align*}
 \nonumber
 (-H_y\epsilon(S_2),\epsilon(S_2))_2
 &<
 (-H_y\epsilon(s),\epsilon(s))_2|_{s=0}
 +
 C
 \int_0^{S_2}
 a^4ds'
 \\
 \nonumber
 &<
 \bar C_2
 \|\nabla_y\epsilon(s) |_{s=0}\|_2^2
 +
 \tfrac{C\delta_1^2}{\|\nabla_y\mathcal{Y}\|_2^2}
 \int_0^{S_2}
 a^2ds'
 \\
 &<
 \bar C_2
 \delta_1^4
 +
 C
 \delta_1^4.
 \end{align*}
 Hence
 from the choice of $T_2$ (see \eqref{equation_4.26}),
 there exists a constant $C>0$ depending only on $n$ such that
 \begin{align}
 \label{equation_4.29}
 \begin{cases}
 \|\nabla_y\epsilon(s)\|_2^2
 <
 C\delta_1^4
 &
 \text{for }
 s=S_2,
 \\
 \delta_1^2
 -
 C\delta_1^4
 <
 a(s)^2\|\nabla_y\mathcal{Y}\|_2^2
 <
 \delta_1^2
 &
 \text{for }
 s=S_2.
 \end{cases}
 \end{align}
 By using
 \eqref{EQUATION_3.14} - \eqref{EQUATION_3.15} and \eqref{equation_4.29},
 we can estimate the energy value $E[u(T_2)]$ as
 \begin{align*}
 &
 E[u(T_2)]
 \\
 &=
 E[{\sf Q}]
 -
 \tfrac{e_0}{2}
 a(T_2)^2
 \|\mathcal{Y}\|_2^2
 +
 \tfrac{1}{2}
 \underbrace{
 (-H_y\epsilon(T_2),\epsilon(T_2))_2
 }_{<\bar C_2\|\epsilon(T_2)\|_2^2 \text{ (see \eqref{EQUATION_e3.3})}}
 -
 R(T_2)
 \\
 &<
 E[{\sf Q}]
 -
 \tfrac{e_0}{2}
 a(T_2)^2
 \|\mathcal{Y}\|_2^2
 +
 \tfrac{\bar C_2}{2}
 \|\nabla_y\epsilon(T_2)\|_2^2
% \\
% &\quad
 +
 C(|a(T_2)|^3+\|\nabla_y\epsilon(T_2)\|_2^3)
 \\
 &<
 E[{\sf Q}]
 -
 \tfrac{e_0}{2}
 (\delta_1^2-C\delta_1^4)
 +
 C
 \delta_1^4
 +
 C\delta_1^3
 +
 C\delta_1^6
 )
 \quad\
 \text{when }
 n=6.
 \end{align*}
 Therefore
 there exists $\delta_1^*>0$ depending only on $n$ such that
 if $\delta_1\in(0,\delta_1^*)$,
 then
 \begin{align}
 \label{equation_4.30}
 E[u(T_2)]<E[{\sf Q}].
 \end{align}
 Furthermore
 we see from \eqref{equation_4.22} that
 \begin{align}
 \nonumber
 &
 I[u(T_2)]
 =
 -
 \int_{\R^n}
 |\nabla_xu(x,T_2)|^2
 dx
 +
 \int_{\R^n}
 |u(x,T_2)|^{p+1}
 dx
 \\
 \nonumber
 &=
 -
 \int_{\R^n}
 |\nabla_y\{{\sf Q}(y)+a(T_2)\mathcal{Y}(y)+\epsilon(y,T_2)\}|^2
 dy
 \\
 \nonumber
 &\quad
 +
 \int_{\R^n}
 |{\sf Q}(y)+a(T_2)\mathcal{Y}(y)+\epsilon(y,T_2)|^{p+1}
 dy
 \\
 \label{equation_4.31}
 &=
 (p-1)
 \left(
 a(T_2)
 \int_{\R^n}
 {\sf Q}^p
 \mathcal{Y}
 dy
 +
 \int_{\R^n}
 {\sf Q}^p
 \epsilon(T_2)
 dy
 \right)
 \\
 \nonumber
 &\quad
 -
 a(T_2)^2\|\nabla_y\mathcal{Y}\|_2^2
 -
 \|\nabla_y\epsilon(T_2)\|_2^2
 +
 N_I.
 \end{align}
 The nonlinear term $N_I$ is given by
 \begin{align}
 \label{equation_4.32}
 N_I
 &=
 \int_{\R^n}
 |{\sf Q}+a(T_2)\mathcal{Y}+\epsilon(T_2)|^{p+1}
 dy
 -
 \int_{\R^n}
 {\sf Q}^{p+1}
 dy
 \\
 \nonumber
 &\quad
 -
 (p+1)
 \int_{\R^n}
 {\sf Q}^p
 \{a(T_2)\mathcal{Y}+\epsilon(T_2)\}
 dy.
 \end{align}
 The second term on the right hand side of \eqref{equation_4.31}
 is bounded by
 \begin{align}
 \label{equation_4.33}
 \int_{\R^n}
 {\sf Q}^p
 \epsilon(T_2)
 dy
 &<
 \|{\sf Q}\|_{p+1}^p
 \|\epsilon(T_2)\|_{p+1}
 <
 C
 \|\nabla_y\epsilon(T_2)\|_2.
 \end{align}
 Furthermore
 $N_I$ given by \eqref{equation_4.32} can be estimated as
 \begin{align}
 \nonumber
 &
 |N_I|
 <
 C
 \int_{\R^n}
 \{
 {\sf Q}^{p-1}
 (a(T_2)\mathcal{Y}+\epsilon)^2
 +
 C|a(T_2)\mathcal{Y}+\epsilon|^{p+1}
 \}
 dy
 \\
 \nonumber
 &<
 C
 (
 a(T_2)^2
 +
 \|{\sf Q}\|_{p+1}^{p-1}
 \|\epsilon(T_2)\|_{p+1}^2
 +
 \underbrace{
 |a(T_2)|^{p+1}
 }_{<Ca(T_2)^2}
 +
 \underbrace{
 \|\epsilon(T_2)\|_{p+1}^{p+1}
 }_{<C\|\epsilon(T_2)\|_{p+1}^2}
 )
 \\
 \label{equation_4.34}
 &<
 C
 (
 a(T_2)^2
 +
 \|\nabla_y\epsilon(T_2)\|_2^2
 ).
 \end{align}
 Therefore
 substituting \eqref{equation_4.33} - \eqref{equation_4.34} into \eqref{equation_4.31},
 and combining \eqref{equation_4.29},
 we obtain
 \begin{align}
 I[u(T_2)]
 \label{equation_4.35}
 &=
 (p-1)
 \left(
 a(T_2)
 \int_{\R^n}
 {\sf Q}^p
 \mathcal{Y}
 dy
 +
 h_1(T_2)
 \right)
 +
 h(T_2)
 \end{align}
 with
 \begin{align*}
 |h_1(T_2)|
 &<
 C\|\nabla_y\epsilon(T_2)\|_2
 <
 C\delta_1^2,
 \\
 |h_2(T_2)|
 &<
 Ca(T_2)^2
 +
 \|\nabla_y\epsilon(T_2)\|_2^2
 +
 |N_I|
 \\
 &<
 Ca(T_2)^2
 +
 C\delta_1^4.
 \end{align*}
 From \eqref{equation_4.29},
 we note that $\frac{\delta_1}{2}<a(T_2)\|\nabla_y\mathcal Y\|_2<\delta_1$
 if $\delta_1$ is sufficiently small.
 Therefore
 it follows from \eqref{equation_4.35} that
 \begin{align}
 \label{equation_4.36}
 I[u(T_2)]<0
 \quad\
 \text{if } a(T_2)<0.
 \end{align}
 Since we have \eqref{equation_4.30} and \eqref{equation_4.36},
 Theorem 1.1 in \cite{Ishiwata} (p.\,352) implies that
 if $a(T_2)<0$,
 then $u(x,t)$ is globally defined and satisfies
 \begin{align*}
 \lim_{t\to\infty}
 \|\nabla_xu(x,t)\|_2=0.
 \end{align*}
 On the other hand,
 according to Theorem 1.2 in \cite{Ishiwata} (p.\,352),
 if $a(T_2)>0$,
 then $u(x,t)$ blows up in finite time.
 Consequently
 since we are now assuming (c4),
 $a(T_2)$ must be negative.
 To summarize this subsection,
 we obtain the following.
%%%%%%%%%%%%%%%%%%%%%%%%%%%%%%%%%%%%%%%%%%%%%%%%%%%%%%%%%%%%%%%%
 \begin{pro}
 \label{PROPOSITION_4.6}
 Let $n=6$ and assume {\rm (c4)} on {\rm p.\,\pageref{(c4)}}.
 There exists a constant $\delta_1^*>0$ depending only on $n$ such that
 if $\delta_1\in(0,\delta_1^*)$,
 then $u(x,t)$ is globally defined and satisfies
 \begin{align*}
 \lim_{t\to\infty}
 u(x,t)
 =
 0
 \quad\
 \text{\rm in }
 \dot H^1(\R^n).
 \end{align*}
 \end{pro}
%%%%%%%%%%%%%%%%%%%%%%%%%%%%%%%%%%%%%%%%%%%%%%%%%%%%%%%%%%%%%%%%

 \subsection{Behavior of a solution for case (c2)}
 \label{section_4.4}
%%%%%%%%%%%%%%%%%%%%%%%%%%%%%%%%%%%%%%%%%%%%%%%%%%%%%%%%%%%%%%%%
 This subsection is a continuation of Section \ref{section_4.3}.
 Here
 we address the case
 \[
 a(T_2)>0
 \]
 in \eqref{equation_4.35}.
 In this case,
 as mentioned above,
 $u(x,t)$ blows up in finite time.
 A goal of this subsection is to provide a proof of Proposition \ref{PROPOSITION_4.7}.
%%%%%%%%%%%%%%%%%%%%%%%%%%%%%%%%%%%%%%%%%%%%%%%%%%%%%%%%%%%%%%%%
 \begin{pro}
 \label{PROPOSITION_4.7}
 Let $n=6$ and assume {\rm (c2)} on {\rm p.\,\pageref{(c4)}}.
 There exists $\alpha^*>0$ depending only on $n$,
 such that
 for any $\alpha\in(0,\alpha^*)$,
 there exists $\delta_1^*=\delta_1^*(\alpha)>0$ such that
 if $\delta_1\in(0,\delta_1^*)$,
 then $u(x,t)$ blows up in finite time $T(u)\in(0,\infty)$,
 and satisfies
 \begin{align*}
 \sup_{t\in(0,T(u))}
 (T(u)-t)^\frac{1}{p-1}
 \|u(x,t)\|_{L_x^\infty(\R^n)}
 <
 \infty.
 \end{align*}
 \end{pro}
%%%%%%%%%%%%%%%%%%%%%%%%%%%%%%%%%%%%%%%%%%%%%%%%%%%%%%%%%%%%%%%%
 \noindent
 {\bf Notation.}
 Note that the constant $C$ represents a generic positive constant,
 which may vary from line to line, but is independent of both $\alpha$ and $\delta_1$.

 First we claim that (see \eqref{equation_4.26} - \eqref{equation_4.28} for the definition of $S_2,S_3$)
 \begin{align}
 \label{equation_4.37}
 (-H_y\epsilon(s),\epsilon(s))_2<\alpha a(s)^2
 \quad\
 \text{for } s\in(S_2,S_3).
 \end{align}
 From \eqref{equation_4.29},
 we observe that
 \[
 (-H_y\epsilon(s),\epsilon(s))_2
 <
 C\delta_1^4
 <
 C\delta_1^2a(s)^2
 \ll
 \alpha a(s)^2
 \quad\
 \text{at }
 s=S_2
 \]
 if $\delta_1$ is sufficiently small.
 By the continuity of the solution,
 there exists $\Delta S>0$ such that
 \begin{align*}
 (-H_y\epsilon(s),\epsilon(s))_2<\alpha a(s)^2
 \quad\
 \text{for } s\in(S_2,S_2+\Delta S).
 \end{align*}
 Suppose that there exists $S_4\in(S_2,S_3)$ such that
 \begin{align}
 \label{equation_4.38}
 \begin{cases}
 (-H_y\epsilon(s),\epsilon(s))_2<\alpha a(s)^2
 &
 \text{for } s\in(S_2,S_4),
 \\
 (-H_y\epsilon(s),\epsilon(s))_2=\alpha a(s)^2
 &
 \text{for } s=S_4.
 \end{cases}
 \end{align}
 %%%%%%%%%%%%%%%%%%%%%%%%%%%%%%%%%%%%%%%%%%%%%%%%%%%%%%%%%%%%%%%%
 From \eqref{EQ_3.16},
 it follows that
 \begin{align*}
 |a_s-e_0a|
 <
 Ca^2+C\|\nabla_y\epsilon\|_2^2
 <
 Ca^2
 \quad\
 \text{for }
 s\in(S_2,S_4).
 \end{align*}
 Hence
 from Lemma \ref{LEMMA_3.10},
 we have
 \begin{align*}
 \tfrac{d}{ds}
 &
 \{
 (-H_y\epsilon,\epsilon)_2
 -
 \alpha
 a^2
 \}
 =
 \tfrac{d}{ds}
 (-H_y\epsilon,\epsilon)_2
 -
 2\alpha
 a
 a_s
 \\
 &<
 -
 \tfrac{1}{2}
 \|H_y\epsilon\|_2^2
 +
 Ca^4
 -
 2e_0
 \alpha
 a^2
 +
 2
 \alpha
 |a|\cdot|a_s-e_0a|
 \\
 &<
 -
 \tfrac{1}{2}
 \|H_y\epsilon\|_2^2
 +
 \alpha
 a^2
 (-2e_0+\tfrac{Ca^2}{\alpha}+C|a|)
 \quad\
 \text{for }
 s\in(S_2,S_4).
 \end{align*}
 Since
 $|a(s)|\cdot\|\nabla_y\mathcal{Y}\|_2<\alpha$ for $s\in(S_2,S_3)$,
 we deduce that
 \begin{align*}
 \tfrac{d}{ds}
 \{
 (-H_y\epsilon,\epsilon)_2
 -
 \alpha
 a^2
 \}
 <
 \alpha
 a^2
 (
 -2e_0+C\alpha
 )
 <0
 \quad\
 \text{for }
 s\in(S_2,S_4)
 \end{align*}
 if $\alpha<\tfrac{e_0}{C}$.
 Integrating both sides,
 we obtain
 \begin{align*}
 (-H_y\epsilon(S_4)&,\epsilon(S_4))_2
 -
 \alpha
 a(S_4)^2
 \\
 &<
 (-H_y\epsilon(S_2),\epsilon(S_2))_2
 -
 \alpha
 a(S_2)^2
 <0.
 \end{align*}
 This contradicts \eqref{equation_4.38},
 which proves \eqref{equation_4.37}.
%%%%%%%%%%%%%%%%%%%%%%%%%%%%%%%%%%%%%%%%%%%%%%%%%%%%%%%%%%%%%%%%
 From \eqref{equation_4.37}, \eqref{EQ_3.16} and \eqref{equation_4.29},
 there exist constants $C_1>0$ and $C_2>0$ depending only on $n$ such that
 \begin{align}
 \label{equation_4.39}
 \begin{cases}
 |a_s-e_0a|<C_1a^2
 \quad\
 \text{for } s\in(S_2,S_3),
 \\
 |a(S_2)-\frac{\delta_1}{\|\nabla_y\mathcal{Y}\|_2}|<\frac{C_2\delta_1^2}{\|\nabla_y\mathcal Y\|_2^2}.
 \end{cases}
 \end{align}
 Here we used the fact that $a(S_2)>0$
 (see the beginning of this subsection p.\,\pageref{section_4.4}).
 Using the constants $C_1$ and $C_2$ in \eqref{equation_4.39}, we define
 \[
 K_1=2(\tfrac{5}{4}\tfrac{C_1}{e_0}+C_2).
 \]
%%%%%%%%%%%%%%%%%%%%%%%%%%%%%%%%%%%%%%%%%%%%%%%%%%%%%%%%%%%%%%%%
 We now prove the following inequalities.
 \begin{align}
 \label{equation_4.40}
 \begin{cases}
 |a(s)-\tfrac{\delta_1}{\|\nabla_y\mathcal Y\|_2}e^{e_0(s-S_2)}|
 <
 \tfrac{K_1\delta_1^2}{\|\nabla_y\mathcal Y\|_2^2}
 e^{2e_0(s-S_2)}
 \quad\
 \text{for }
 s\in(S_2,S_3),
 \\
 \tfrac{\delta_1}{2\|\nabla_y\mathcal Y\|_2}e^{e_0(s-S_2)}
 <
 a(s)
 <
 \tfrac{2\delta_1}{\|\nabla_y\mathcal Y\|_2}e^{e_0(s-S_2)}
 \quad\
 \text{for }
 s\in(S_2,S_3),
 \\
 \delta_1
 e^{e_0(s-S_2)}
 <
 2\alpha
 \quad\
 \text{for }
 s\in(S_2,S_3).
 \end{cases}
 \end{align}
%%%%%%%%%%%%%%%%%%%%%%%%%%%%%%%%%%%%%%%%%%%%%%%%%%%%%%%%%%%%%%%%
 Assume that
 there exists $S_5\in(S_2,S_3)$ such that
 \begin{align}
 \label{equation_4.41}
 \begin{cases}
 |a(s)-\tfrac{\delta_1}{\|\nabla_y\mathcal Y\|_2}e^{e_0(s-S_2)}|
 <
 \tfrac{K_1\delta_1^2}{\|\nabla_y\mathcal Y\|_2^2}
 e^{2e_0(s-S_2)}
 \quad
 \text{for }
 s\in(S_2,S_5),
 \\
 \tfrac{\delta_1}{2\|\nabla_y\mathcal Y\|_2}e^{e_0(s-S_2)}
 <
 a(s)
 <
 \tfrac{2\delta_1}{\|\nabla_y\mathcal Y\|_2}
 e^{e_0(s-S_2)}
 \quad\
 \text{for }
 s\in(S_2,S_5),
 \\
 \delta_1
 e^{e_0(s-S_2)}
 <
 2\alpha
 \quad\
 \text{for }
 s\in(S_2,S_5).
 \end{cases}
 \end{align}
 Integrating both sides of the first inequality of \eqref{equation_4.39} over $s\in(S_2,s)$,
 we get
 \begin{align}
 \label{equation_4.42}
 |a(s)-a(S_2)e^{e_0(s-S_2)}|
 &<
 C_1
 e^{e_0s}
 \int_{S_2}^{s}
 e^{-e_0s'}
 a(s')^2
 ds'.
 \end{align}
 From the first and the third inequalities of \eqref{equation_4.41},
 it holds that for $s'\in(S_2,S_5)$
 \begin{align*}
 a(s')^2
 &<
 (
 \tfrac{\delta_1}{\|\nabla_y\mathcal Y\|_2}e^{e_0(s'-S_2)}+\tfrac{K_1\delta_1^2}{\|\nabla_y\mathcal Y\|_2^2}e^{2e_0(s'-S_2)}
 )^2
 \\
 &<
 \tfrac{\delta_1^2}{\|\nabla_y\mathcal Y\|_2^2}
 e^{2e_0(s'-S_2)}
 +
 \tfrac{K_1^2\delta_1^4}{\|\nabla_y\mathcal Y\|_2^4}
 e^{4e_0(s'-S_2)}
 +
 \tfrac{2K_1\delta_1^3}{\|\nabla_y\mathcal Y\|_2^3}
 e^{3e_0(s'-S_2)}
 \\
 &<
 \tfrac{\delta_1^2}{\|\nabla_y\mathcal Y\|_2^2}
 e^{2e_0(s'-S_2)}
 (1+\tfrac{4\alpha^2K_1^2}{\|\nabla_y\mathcal Y\|_2^2}+\tfrac{4\alpha K_1}{\|\nabla_y\mathcal Y\|_2}).
 \end{align*}
 If $\alpha$ is sufficiently small such that
 $\frac{4\alpha^2K_1^2}{\|\nabla_y\mathcal Y\|_2^2}+\tfrac{4\alpha K_1}{\|\nabla_y\mathcal Y\|_2}<\frac{1}{4}$,
 then we have
 \begin{align}
 \label{equation_4.43}
 a(s')^2
 <
 \tfrac{5}{4}
 \tfrac{\delta_1^2}{\|\nabla_y\mathcal Y\|_2^2}
 e^{2e_0(s'-S_2)}
 \quad\
 \text{for }
 s'\in(S_2,S_5).
 \end{align}
 Substituting \eqref{equation_4.43} into \eqref{equation_4.42},
 we get
 \begin{align}
 \label{equation_4.44}
 |a(s)-a(S_2)&e^{e_0(s-S_2)}|
 <
 \tfrac{5}{4}
 \tfrac{\delta_1^2}{\|\nabla_y\mathcal Y\|_2^2}
 \tfrac{C_1}{e_0}
 e^{2e_0(s-S_2)}
 \quad\
 \text{for }
 s\in(S_2,S_5).
 \end{align}
 Combining \eqref{equation_4.44} and the second line of \eqref{equation_4.39},
 we obtain
 \begin{align}
 \nonumber
 &
 |a(s)-\tfrac{\delta_1}{\|\nabla_y\mathcal Y\|_2}e^{e_0(s-S_2)}|
 \\
 \nonumber
 &<
 |a(s)-a(S_2)e^{e_0(s-S_2)}|
 +
 |a(S_2)e^{e_0(s-S_2)}-\tfrac{\delta_1}{\|\nabla_y\mathcal Y\|_2}e^{e_0(s-S_2)}|
 \\
 \nonumber
 &<
 \tfrac{5}{4}
 \tfrac{\delta_1^2}{\|\nabla_y\mathcal Y\|_2^2}
 \tfrac{C_1}{e_0}
 e^{2e_0(s-S_2)}
 +
 \tfrac{C_2\delta_1^2}{\|\nabla_y\mathcal Y\|_2^2}
 e^{e_0(s-S_2)}
 \\
 \nonumber
 &=
 (
 \tfrac{5}{4}
 \tfrac{C_1}{e_0}
 +
 C_2
 )
 \tfrac{\delta_1^2}{\|\nabla_y\mathcal Y\|_2^2}
 e^{2e_0(s-S_2)}
 \\
 \label{equation_4.45}
 &=
 \tfrac{1}{2}
 \tfrac{K_1\delta_1^2}{\|\nabla_y\mathcal Y\|_2^2}
 e^{2e_0(s-S_2)}
 \quad\
 \text{for }
 s\in(S_2,S_5).
 \end{align}
%%%%%%%%%%%%%%%%%%%%%%%%%%%%%%%%%%%%%%%%%%%%%%%%%%%%%%%%%%%%%%%%
 From the third line of \eqref{equation_4.41},
 it can be seen that the right hand side of \eqref{equation_4.45} is bounded by
 \begin{align*}
 \tfrac{1}{2}
 \tfrac{K_1\delta_1^2}{\|\nabla_y\mathcal Y\|_2^2}
 e^{2e_0(s-S_2)}
 <
 \tfrac{\alpha K_1\delta_1}{\|\nabla_y\mathcal Y\|_2}
 e^{e_0(s-S_2)}
 \quad\
 \text{for }
 s\in(S_2,S_5).
 \end{align*}
 Hence
 if $\alpha<\frac{1}{16K_1}$,
 then
 it follows from \eqref{equation_4.45} that
 for $s\in(S_2,S_5)$
 \begin{align}
 \label{equation_4.46}
 \tfrac{15}{16}
 \tfrac{\delta_1}{\|\nabla_y\mathcal Y\|_2}
 e^{e_0(s-S_2)}
 <
 a(s)
 <
 \tfrac{17}{16}
 \tfrac{\delta_1}{\|\nabla_y\mathcal Y\|_2}
 e^{e_0(s-S_2)}.
 \end{align}
%%%%%%%%%%%%%%%%%%%%%%%%%%%%%%%%%%%%%%%%%%%%%%%%%%%%%%%%%%%%%%%%
 From the definition of $T_3$ (see \eqref{equation_4.27}),
 we recall that $a(s)^2\|\nabla_y\mathcal{Y}\|_2^2<\alpha^2$ for $s\in(S_2,S_3)$.
 Therefore
 from \eqref{equation_4.46},
 we have
 \begin{align}
 \label{equation_4.47}
 \tfrac{15}{16}
 \tfrac{\delta_1}{\|\nabla_y\mathcal Y\|_2}
 e^{e_0(s-S_2)}
 <
 a(s)
 <
 \tfrac{\alpha}{\|\nabla_y\mathcal{Y}\|_2}
 \quad\
 \text{for }
 s\in(S_2,S_5).
 \end{align}
 Since the estimates \eqref{equation_4.45} - \eqref{equation_4.47} are independent of $S_5$,
 we conclude that all the inequalities in \eqref{equation_4.40} hold.
 Integrating both sides of \eqref{EQ_3.97},
 and applying the first estimate of \eqref{equation_4.29} and
 the second line of \eqref{equation_4.40},
 we obtain
 \begin{align}
 \nonumber
 (-H_y\epsilon(s)&,\epsilon(s))_2
 <
 (-H_y\epsilon(S_2),\epsilon(S_2))_2
 +
 C
 \int_{S_2}^s
 a^4ds'
 \\
 \nonumber
 &<
 C
 \delta_1^4
 +
 C
 \int_{S_2}^s
 \delta_1^4
 e^{4e_0(s'-S_2)}
 ds'
 \\
 \label{equation_4.48}
 &<
 C
 \delta_1^4
 +
 C
 \delta_1^4
 e^{4e_0(s-S_2)}
 \quad\
 \text{for }
 s\in(S_2,S_3).
 \end{align}
 From the second line of \eqref{equation_4.40},
 we note that $\frac{\delta_1}{2\|\nabla_y\mathcal Y\|_2}e^{e_0(s-S_2)}<a(s)$ for $s\in(S_2,S_3)$.
 Hence
 \eqref{equation_4.48} can be rewritten as
 \begin{align}
 \label{equation_4.49}
 (-H_y\epsilon(s),\epsilon(s))_2
 <
 Ca(s)^4
 \quad\
 \text{for }
 s\in(S_2,S_3).
 \end{align}
%%%%%%%%%%%%%%%%%%%%%%%%%%%%%%%%%%%%%%%%%%%%%%%%%%%%%%%%%%%%%%%%
 Using \eqref{equation_4.49},
 we can rewrite \eqref{EQ_3.17} as
 \begin{align*}
 |
 \tfrac{d}{ds}\log\lambda
 |
 <
 Ca^2
 +
 C\|\nabla_y\epsilon\|_2
 <
 Ca^2
 \quad\
 \text{for }
 s\in(S_2,S_3).
 \end{align*}
 We again use the second inequality of \eqref{equation_4.40} to get
 \begin{align}
 \label{equation_4.50}
 |
 \tfrac{d}{ds}\log\lambda
 |
 <
 C\delta_1^2
 e^{2e_0(s-S_2)}
 \quad\
 \text{for }
 s\in(S_2,S_3).
 \end{align}
 Integrating both sides of \eqref{equation_4.50},
 we see that
 \begin{align*}
 |\log\tfrac{\lambda(s)}{\lambda(S_2)}|
 <
 C\delta_1^2
 e^{2e_0(s-S_2)}
 \quad\
 \text{for }
 s\in(S_2,S_3).
 \end{align*}
 This is equivalent to
 \begin{align}
 \label{equation_4.51}
 e^{-C\delta_1^2e^{2e_0(s-S_2)}}
 &<
 \tfrac{\lambda(s)}{\lambda(S_2)}
 <
 e^{C\delta_1^2e^{2e_0(s-S_2)}}
 \quad\
 \text{for }
 s\in(S_2,S_3).
 \end{align}
 Note that
 the following elementary relation holds.
 \begin{align}
 \label{equation_4.52}
 1-2A<e^{-A}<e^{A}<1+2A
 \quad\
 \text{if }
 A<\tfrac{1}{e}.
 \end{align}
 Since $\delta_1e^{e_0(s-S_2)}<2\alpha\ll 1$ for $s\in(S_2,S_3)$ (see the third line of \eqref{equation_4.40}),
 combining \eqref{equation_4.51} - \eqref{equation_4.52},
 we obtain
 \begin{align}
 \label{equation_4.53}
 |
 \tfrac{\lambda(S_2)}{\lambda(s)}-1
 |
 +
 |
 \tfrac{\lambda(s)}{\lambda(S_2)}-1
 |
 <
 C\delta_1^2
 e^{2e_0(s-S_2)}
 \quad\
 \text{for }
 s\in(S_2,S_3).
 \end{align}
 We here go back to the original time variable $t$.
 \begin{align*}
 t-T_2
 &=
 \int_{T_2}^t
 dt
 =
 \int_{S_2}^s
 \tfrac{dt}{ds}
 ds'
 =
 \int_{S_2}^s
 \lambda(s')^2
 ds'
 \\
 &=
 \lambda(S_2)^2
 \int_{S_2}^s
 (\tfrac{\lambda(s')}{\lambda(S_2)})^2
 ds'
 \\
 &=
 \lambda(S_2)^2
 (s-S_2)
 +
 \lambda(S_2)^2
 \int_{S_2}^s
 \{
 (\tfrac{\lambda(s')}{\lambda(S_2)})^2-1
 \}
 ds'
 \\
 &
 \text{for }
 t\in(T_2,T_3).
 \end{align*}
 From \eqref{equation_4.53},
 we observe that
 there exists $C>0$ depending only on $n$ such that for $s'\in(S_2,S_3)$
 \begin{align*}
 |
 (\tfrac{\lambda(s')}{\lambda(S_2)})^2-1
 |
 =
 |\tfrac{\lambda(s')}{\lambda(S_2)}+1|
 |\tfrac{\lambda(s')}{\lambda(S_2)}-1|
 <
 C\delta_1^2e^{2e_0(s'-S_2)}.
 \end{align*}
 Hence
 we obtain
 for $t\in(T_2,T_3)$
 \begin{align*}
 |
 (t-T_2)
 -
 \lambda(S_2)^2
 (s-S_2)
 |
 <
 C
 \lambda(S_2)^2
 \delta_1^2
 e^{2e_0(s-S_2)}.
 \end{align*}
 This implies
 for $t\in(T_2,T_3)$
 \begin{align*}
 \tfrac{t-T_2}{\lambda(S_2)^2}
 -
 C\delta_1^2e^{2e_0(s-S_2)}
 <
 s-S_2
 <
 \tfrac{t-T_2}{\lambda(S_2)^2}
 +
 C\delta_1^2e^{2e_0(s-S_2)}
 \end{align*}
 and
 for $t\in(T_2,T_3)$
 \begin{align*}
 e^{\frac{e_0(t-T_2)}{\lambda(S_2)^2}}
 e^{-e_0C\delta_1^2e^{2e_0(s-S_2)}}
 <
 e^{e_0(s-S_2)}
 <
 e^{\frac{e_0(t-T_2)}{\lambda(S_2)^2}}
 e^{e_0C\delta_1^2e^{2e_0(s-S_2)}}.
 \end{align*}
 Therefore
 due to \eqref{equation_4.52},
 we obtain
 \begin{align}
 \label{equation_4.54}
 &
 e^{\frac{e_0(t-T_2)}{\lambda(S_2)^2}}
 (1-C\delta_1^2e^{2e_0(s-S_2)})
 <
 e^{e_0(s-S_2)}
 \\
 \nonumber
 &<
 e^{\frac{e_0(t-T_2)}{\lambda(S_2)^2}}
 (1+C\delta_1^2e^{2e_0(s-S_2)})
 \quad\
 \text{for }
 t\in(T_2,T_3).
 \end{align}
 From \eqref{equation_4.54} and the third line of \eqref{equation_4.40},
 it holds that
 \begin{align}
 \label{equation_4.55}
 \tfrac{15}{16}
 e^{\frac{e_0(t-T_2)}{\lambda(S_2)^2}}
 <
 e^{e_0(s-S_2)}
 <
 \tfrac{17}{16}
 e^{\frac{e_0(t-T_2)}{\lambda(S_2)^2}}
 \quad\
 \text{for } t\in(T_2,T_3).
 \end{align}
 Combining \eqref{equation_4.54} - \eqref{equation_4.55} and
 the third line of \eqref{equation_4.40},
 we obtain
 \begin{align}
 \label{equation_4.56}
 \begin{cases}
 |e^{e_0(s-S_2)}-e^{\frac{e_0(t-T_2)}{\lambda(S_2)^2}}|
 <
 C
 \delta_1^2
 e^{\frac{3e_0(t-T_2)}{\lambda(S_2)^2}}
 \quad\
 \text{for }
 t\in(T_2,T_3),
 \\
 e^{e_0(s-S_2)}
 <
 \frac{17}{16}
 e^{\frac{e_0(t-T_2)}{\lambda(S_2)^2}}
 \quad\
 \text{for }
 t\in(T_2,T_3),
 \\[1mm]
 \delta_1
 e^{\frac{e_0(t-T_2)}{\lambda(S_2)^2}}
 <
 \frac{32}{15}
 \alpha
 \quad\
 \text{for }
 t\in(T_2,T_3).
 \end{cases}
 \end{align}
 We rewrite the first relation of \eqref{equation_4.40} in the original variable $t$
 by using \eqref{equation_4.56}.
%%%%%%%%%%%%%%%%%%%%%%%%%%%%%%%%%%%%%%%%%%%%%%%%%%%%%%%%%%%%%%%%
 \begin{align}
 \nonumber
 &
 |a(t)-\tfrac{\delta_1}{\|\nabla_y\mathcal Y\|_2}e^{\frac{e_0(t-T_2)}{\lambda(S_2)^2}}|
 \\
 \nonumber
 &<
 |a(s)-\tfrac{\delta_1}{\|\nabla_y\mathcal Y\|_2}e^{e_0(s-S_2)}|
 +
 \tfrac{\delta_1}{\|\nabla_y\mathcal Y\|_2}
 |e^{e_0(s-S_2)}-e^{\frac{e_0(t-T_2)}{\lambda(S_2)^2}}|
 \\
 \nonumber
 &<
 \tfrac{K_1\delta_1^2}{\|\nabla_y\mathcal Y\|_2^2}
 e^{2e_0(s-S_2)}
 +
 \tfrac{C\delta_1^3}{\|\nabla_y\mathcal Y\|_2}
 e^{\frac{3e_0(t-T_2)}{\lambda(S_2)^2}}
 \\
 \nonumber
 &<
 (\tfrac{17}{16})^2
 \tfrac{K_1\delta_1^2}{\|\nabla_y\mathcal Y\|_2^2}
 e^{\frac{2e_0(t-T_2)}{\lambda(S_2)^2}}
 +
 \tfrac{C\delta_1^3}{\|\nabla_y\mathcal Y\|_2}
 e^{\frac{3e_0(t-T_2)}{\lambda(S_2)^2}}
 \\
 \label{equation_4.57}
 &<
 C
 \delta_1^2
 e^{\frac{2e_0(t-T_2)}{\lambda(S_2)^2}}
 \quad\
 \text{for }
 t\in(T_2,T_3).
 \end{align}
 Furthermore
 from \eqref{equation_4.57} and the third line of \eqref{equation_4.56},
 we deduce that
 \begin{align}
 \label{equation_4.58}
 \tfrac{15}{16}
 \tfrac{\delta_1}{\|\nabla_y\mathcal Y\|_2}e^{\frac{e_0(t-T_2)}{\lambda(S_2)^2}}
 <
 a(t)
 <
 \tfrac{17}{16}
 \tfrac{\delta_1}{\|\nabla_y\mathcal Y\|_2}e^{\frac{e_0(t-T_2)}{\lambda(S_2)^2}}
 \quad\
 \text{for }
 t\in(T_2,T_3).
 \end{align}
%%%%%%%%%%%%%%%%%%%%%%%%%%%%%%%%%%%%%%%%%%%%%%%%%%%%%%%%%%%%%%%%
 Due to \eqref{equation_4.49} and \eqref{equation_4.58},
 we immediately see that
 \begin{align}
 \label{equation_4.59}
 (-H_y\epsilon(t),\epsilon(t))_2
 <
 Ca(t)^4
 <
 C\delta_1^4
 e^{\frac{4e_0(t-T_2)}{\lambda(S_2)^2}}
 \quad\
 \text{for }
 t\in(T_2,T_3).
 \end{align}
 We now estimate the value of $T_3-T_2$.
 From the fact that $\|\nabla_y\epsilon(t)\|_2^2<C\alpha^4$ for $t\in(T_2,T_3)$
 (see \eqref{equation_4.59} and the third line of \eqref{equation_4.56})
 and
 the definition of $T_3$ (see \eqref{equation_4.27}),
 we have
 \begin{align}
 \label{equation_4.60}
 |a(T_3)-\tfrac{\alpha}{\|\nabla_y\mathcal{Y}\|_2}|<C\alpha^3.
 \end{align}
 Due to \eqref{equation_4.57} and \eqref{equation_4.60},
 it follows that
%%%%%%%%%%%%%%%%%%%%%%%%%%%%%%%%%%%%%%%%%%%%%%%%%%%%%%%%%%%%%%%%
 \begin{align*}
 &
 |\tfrac{\alpha}{\|\nabla_y\mathcal{Y}\|_2}
 -
 \tfrac{\delta_1}{\|\nabla_y\mathcal{Y}\|_2}
 e^{\frac{e_0(T_3-T_2)}{\lambda(S_2)^2}}|
 \\
 &<
 |\tfrac{\alpha}{\|\nabla_y\mathcal{Y}\|_2}-a(T_3)|
 +
 |a(T_3)-\tfrac{\delta_1}{\|\nabla_y\mathcal{Y}\|_2}e^{\frac{e_0(T_3-T_2)}{\lambda(S_2)^2}}|
 \\
 &<
 C
 \alpha^3
 +
 C\delta_1^2
 e^{\frac{2e_0(T_3-T_2)}{\lambda(S_2)^2}}.
 \end{align*}
 Therefore
 applying the third line of \eqref{equation_4.56},
 we obtain
 \begin{align*}
 |
 \tfrac{\alpha}{\|\nabla_y\mathcal{Y}\|_2}
 -
 \tfrac{\delta_1}{\|\nabla_y\mathcal{Y}\|_2}
 e^{\frac{e_0(T_3-T_2)}{\lambda(S_2)^2}}
 |
 <
 C
 \alpha^2.
 \end{align*}
 This implies that
 $\log\frac{\alpha-C\alpha^2}{\delta_1}
 <
 \frac{e_0(T_3-T_2)}{\lambda(T_2)^2}
 <
 \log\frac{\alpha+C\alpha^2}{\delta_1}$.
 Therefore
 we conclude that
 \begin{align}
 \label{equation_4.61}
 \log\tfrac{\alpha}{\delta_1}
 -
 \log2
 <
 \tfrac{e_0(T_3-T_2)}{\lambda(T_2)^2}
 <
 \log\tfrac{\alpha}{\delta_1}
 +
 \log2.
 \end{align}
 We next derive the bound of $z(t)$ for $t\in(T_2,T_3)$.
 From \eqref{EQ_3.18},
 we observe that
 $|\frac{dz}{ds}|<C\lambda(a^2+\|\nabla_y\epsilon\|_2)$ for $s\in(S_2,S_3)$.
 Therefore
 combining \eqref{equation_4.49}, \eqref{equation_4.53} and the second line of \eqref{equation_4.40},
 we get
 \begin{align}
 \nonumber
 &
 |z(s)-z(S_2)|
 <
 C
 \int_{S_2}^s
 \lambda(s')
 (a(s')^2+\underbrace{\|\nabla_y\epsilon(s')\|_2}_{<Ca(s')^2})
 ds'
 \\
 \nonumber
 &<
 C
 \int_{S_2}^s
 \lambda(s')
 a(s')^2
 ds'
 \\
 \nonumber
 &<
 C
 \lambda(S_2)
 \int_{S_2}^s
 \tfrac{\lambda(s')}{\lambda(S_2)}
 \delta_1^2
 e^{2e_0(s'-S_2)}
 ds'
 \\
 \label{equation_4.62}
 &<
 C
 \lambda(S_2)
 \delta_1^2
 e^{2e_0(s-S_2)}
 \quad\
 \text{for }
 s\in(S_2,S_3).
 \end{align}
%%%%%%%%%%%%%%%%%%%%%%%%%%%%%%%%%%%%%%%%%%%%%%%%%%%%%%%%%%%%%%%%
 We now introduce a new rescaling function.
 \begin{align}
 \label{equation_4.63}
 \hat
 u(\xi,\sigma)
 =
 \lambda(T_2)^\frac{2}{p-1}
 u(z(T_2)+\lambda(T_2)\xi,T_2+\lambda(T_2)^2\sigma).
 \end{align}
 We define $\sigma_3>0$ by $T_2+\lambda(T_2)^2\sigma_3=T_3$.
 From \eqref{equation_4.61},
 $\sigma_3$ is bounded by
 \begin{align}
 \label{equation_4.64}
 \tfrac{\log\frac{\alpha}{\delta_1}-\log 2}{e_0}
 <
 \tfrac{T_3-T_2}{\lambda(T_2)^2}
 =
 \sigma_3
 <
 \tfrac{\log\frac{\alpha}{\delta_1}+\log 2}{e_0}.
 \end{align}
%%%%%%%%%%%%%%%%%%%%%%%%%%%%%%%%%%%%%%%%%%%%%%%%%%%%%%%%%%%%%%%%
 Furthermore
 we put
 \begin{align}
 \label{equation_4.65}
 \begin{cases}
 \hat a(\sigma)
 =
 (\hat u(\xi,\sigma)-{\sf Q}(\xi),\mathcal Y(\xi))_{L_\xi^2(\R^n)},
 \\
 \hat b_0(\sigma)
 =
 (\hat u(\xi,\sigma)-{\sf Q}(\xi),
 \tfrac{\Psi_0(\xi)}{\sqrt{(\Lambda_\xi{\sf Q},\Psi_0)_2}})_{L_\xi^2(\R^n)},
 \\
 \hat b_j(\sigma)
 =
 (\hat u(\xi,\sigma)-{\sf Q}(\xi),
 \tfrac{\Psi_j(\xi)}{\sqrt{(\pa_{\xi_j}{\sf Q},\Psi_j)_2}})_{L_\xi^2(\R^n)}
 \quad
 \text{for } j\in\{1,\cdots,n\}.
 \end{cases}
 \end{align}
 By using $\hat a(\sigma)$ and $\hat b_j(\sigma)$ in \eqref{equation_4.65},
 we define $\hat v(\xi,\sigma)$ as
 \begin{align}
 \label{equation_4.66}
 \hat u(\xi,\sigma)
 &=
 {\sf Q}(\xi)
 +
 \hat a(\sigma)
 \mathcal{Y}(\xi)
 +
 \tfrac{\hat b_0(\sigma)\Lambda_\xi{\sf Q}(\xi)}{\sqrt{(\Lambda_\xi{\sf Q},\Psi_0)_2}}
 +
 \sum_{j=1}^n
 \tfrac{\hat b_j(\sigma)\pa_{\xi_j}{\sf Q}(\xi)}{\sqrt{(\pa_{\xi_j}{\sf Q},\Psi_j)_2}}
 +
 \hat v(\xi,\sigma).
 \end{align}
 From the definition of $\hat a(\sigma)$ and $\hat b_j(\sigma)$,
 it holds that
 \begin{align}
 \label{equation_4.67}
 (\hat v(\sigma),\mathcal{Y})_{L_\xi^2(\R^n)}
 &=
 (\hat v(\sigma),\Psi_j)_{L_\xi^2(\R~n)}
 =
 0
 \\
 \nonumber
 &
 \text{for }
 \sigma\in(0,\sigma_3),\
 j\in\{0,\cdots,n\}.
 \end{align}
 Next,
 we establish the bounds of
 $\hat a(\sigma),\hat b_j(\sigma)$ and $\hat v(\xi,\sigma)$.
%%%%%%%%%%%%%%%%%%%%%%%%%%%%%%%%%%%%%%%%%%%%%%%%%%%%%%%%%%%%%%%%
 Since $u(x,t)$ is expressed as \eqref{EQUATION_e3.1},
 we can write $\hat u(\xi,\sigma)$ as
 (see \eqref{equation_4.63} for the definition of $\hat u(\xi,\sigma)$)
 \begin{align}
 \label{equation_4.68}
 \hat u(\xi,\sigma)
 &=
 (
 \tfrac{\lambda(T_2)}{\lambda(t)}
 )^{\frac{2}{p-1}}
 \{
 {\sf Q}(\tfrac{\lambda(T_2)\xi+z(T_2)-z(t)}{\lambda(t)})
 +
 a(t)\mathcal{Y}(\tfrac{\lambda(T_2)\xi+z(T_2)-z(t)}{\lambda(t)})
 \\
 \nonumber
 & \quad
 +\epsilon(\tfrac{\lambda(T_2)\xi+z(T_2)-z(t)}{\lambda(t)},t)
 \}
 \quad\
 \text{with }
 t=T_2+\lambda(T_2)^2\sigma.
 \end{align}
%%%%%%%%%%%%%%%%%%%%%%%%%%%%%%%%%%%%%%%%%%%%%%%%%%%%%%%%%%%%%%%%
 For simplicity,
 we put
 \begin{align}
 \label{equation_4.69}
 \hat\lambda(t)
 =
 \tfrac{\lambda(t)}{\lambda(T_2)},
 \quad
 \hat z(t)
 =
 \tfrac{z(t)-z(T_2)}{\lambda(T_2)}
 \quad
 \text{and } \
 \hat X
 =
 \tfrac{\xi-\hat z(t)}{\hat\lambda(t)}.
 \end{align}
 Then
 \eqref{equation_4.68} can be written as
 \begin{align}
 \nonumber
 \hat u(\xi,\sigma)
 &=
 \hat\lambda(t)^{-\frac{2}{p-1}}
 \{
 {\sf Q}(\tfrac{\xi-\hat z(t)}{\hat\lambda(t)})
 +
 a(t)\mathcal{Y}(\tfrac{\xi-\hat z(t)}{\hat\lambda(t)})
 +\epsilon(\tfrac{\xi-\hat z(t)}{\hat\lambda(t)},t)
 \}
 \\
 \label{equation_4.70}
 &=
 \hat\lambda(t)^{-\frac{2}{p-1}}
 \{
 {\sf Q}(\hat X)
 +
 a(t)\mathcal{Y}(\hat X)
 +\epsilon(\hat X,t)
 \}.
 \end{align}
%%%%%%%%%%%%%%%%%%%%%%%%%%%%%%%%%%%%%%%%%%%%%%%%%%%%%%%%%%%%%%%%
 We recall that $\|\mathcal Y\|_2=1$ (see p.\,\pageref{mathcal_Y}).
 From the definition of $\hat a(\sigma)$ (see \eqref{equation_4.65}),
 we see that
 \begin{align}
 \nonumber
 &\hat a(\sigma)
 =
 (\hat u(\xi,\sigma)-{\sf Q}(\xi),\mathcal{Y}(\xi))_{L_\xi^2(\R^n)}
 \\
 \nonumber
 &=
 (
 {\sf Q}(\hat X)
 +
 a(t)\mathcal{Y}(\hat X)
 +
 \epsilon(\hat X,t)
 -
 {\sf Q}(\xi),\mathcal{Y}(\xi))_{L_\xi^2(\R^n)}
 \\
 \nonumber
 &\quad
 +
 \underbrace{
 (
 \hat u(\xi,\sigma)
 -
 \{
 {\sf Q}(\hat X)
 +
 a(t)\mathcal{Y}(\hat X)
 +
 \epsilon(\hat X,t)
 \}
 ,\mathcal{Y}(\xi))_{L_\xi^2(\R^n)}
 }_{=D_0}
 \\
 \nonumber
 &=
 \underbrace{
 ({\sf Q}(\hat X)-{\sf Q}(\xi),\mathcal{Y}(\xi))_{L_\xi^2(\R^n)}
 }_{=D_1}
 +
 \underbrace{
 a(t)
 (\mathcal{Y}(\hat X)-\mathcal{Y}(\xi),\mathcal{Y}(\xi))_{L_\xi^2(\R^n)}
 }_{=D_2}
 \\
 \label{equation_4.71}
 &\quad
 +
 a(t)
 +
 \underbrace{
 (\epsilon(\hat X,t),\mathcal{Y}(\xi))_{L_\xi^2(\R^n)}
 }_{=D_3}
 +
 D_0.
 \end{align}
 From \eqref{equation_4.53} and \eqref{equation_4.62},
 we note that
 \begin{align*}
 |\hat\lambda(t)-1|+|\hat z(t)|
 &<
 C\delta_1^2
 e^{2e_0(s-S_2)}
 \quad\
 \text{for }
 t\in(T_2,T_3).
 \end{align*}
 Therefore
 from the second and the third line of \eqref{equation_4.56} and the definition of $\sigma$ (see \eqref{equation_4.63}),
 it follows that
 \begin{align}
 \label{equation_4.72}
 \begin{cases}
 |\hat\lambda(t)-1|+|\hat z(t)|
 <
 C\delta_1^2
 e^{\frac{2e_0(t-T_2)}{\lambda(S_2)^2}}
 =
 C\delta_1^2
 e^{2e_0\sigma}
 &
 \text{for }
 \sigma\in(0,\sigma_3),
 \\
 \delta_1
 e^{e_0\sigma}
 <
 \frac{32}{15}
 \alpha
 &
 \text{for }
 \sigma\in(0,\sigma_3).
 \end{cases}
 \end{align}
%%%%%%%%%%%%%%%%%%%%%%%%%%%%%%%%%%%%%%%%%%%%%%%%%%%%%%%%%%%%%%%%
 By using \eqref{equation_4.72},
 we can verify that
 \begin{align}
 \nonumber
 |D_1|
 &<
 |({\sf Q}(\tfrac{\xi-\hat z(t)}{\hat\lambda(t)})-{\sf Q}(\tfrac{\xi}{\hat\lambda(t)}),
 \mathcal{Y}(\xi))_2|
 +
 |({\sf Q}(\tfrac{\xi}{\hat\lambda(t)})-{\sf Q}(\xi),
 \mathcal{Y}(\xi))_2|
 \\
 \nonumber
 &<
 C(
 |\tfrac{\hat z(t)}{\hat\lambda(t)}|
 +
 |\tfrac{1}{\hat\lambda(t)}-1|
 )
 \\
 \label{equation_4.73}
 &<
 C\delta_1^2
 e^{2e_0\sigma}
 \quad\
 \text{for }
 \sigma\in(0,\sigma_3).
 \end{align}
 Furthermore
 from \eqref{equation_4.72}, \eqref{equation_4.58} and \eqref{equation_4.49},
 we see that
 \begin{align}
 \nonumber
 |D_2|
 &<
 C|a(t)|
 (
 |\tfrac{\hat z(t)}{\hat\lambda(t)}|
 +
 |\tfrac{1}{\hat\lambda(t)}-1|
 )
 \\
 \label{equation_4.74}
 &<
 C
 \delta_1^3
 e^{3e_0\sigma}
 \quad\
 \text{for }
 \sigma\in(0,\sigma_3),
 \\
 \nonumber
 |D_3|
 &<
 C\|\nabla_y\epsilon(y,t)\|_{L_y^2(\R^n)}
 <
 Ca(t)^2
 \\
 \label{equation_4.75}
 &<
 C\delta_1^2
 e^{2e_0\sigma}
 \quad\
 \text{for }
 \sigma\in(0,\sigma_3).
 \end{align}
 To obtain the bound of $D_0$,
 we use the explicit form of $\hat u(\xi,\sigma)$ give in \eqref{equation_4.70}.
 \begin{align}
 \nonumber
 |D_0|
 &<
 |\hat\lambda(t)^{-\frac{2}{p-1}}-1|
 \cdot
 |(
 {\sf Q}(\hat X)
 +
 a(t)\mathcal{Y}(\hat X)
 +
 \epsilon(\hat X,t),\mathcal{Y}(\xi))_{L_\xi^2(\R^n)}
 |
 \\
 \nonumber
 &<
 C|\hat\lambda(t)^{-\frac{2}{p-1}}-1|
 <
 C|\hat\lambda(t)-1|
 \\
 \label{equation_4.76}
 &<
 C\delta_1^2
 e^{2e_0\sigma}
 \quad\
 \text{for }
 \sigma\in(0,\sigma_3).
 \end{align}
 Therefore
 it follows from \eqref{equation_4.71} and \eqref{equation_4.73} - \eqref{equation_4.76} that
 \begin{align}
 \label{equation_4.77}
 |\hat a(\sigma)-a(t)|
 <
 C\delta_1^2e^{2e_0\sigma}
 \quad\
 \text{for }
 \sigma\in(0,\sigma_3).
 \end{align}
%%%%%%%%%%%%%%%%%%%%%%%%%%%%%%%%%%%%%%%%%%%%%%%%%%%%%%%%%%%%%%%%
 We rewrite \eqref{equation_4.57} as
 \begin{align}
 \label{equation_4.78}
 |a(t)-\tfrac{\delta_1}{\|\nabla_y\mathcal{Y}\|_2}e^{e_0\sigma}|
 &<
 C\delta_1^2
 e^{2e_0\sigma}
 \quad\
 \text{for }
 \sigma\in(0,\sigma_3).
 \end{align}
 From \eqref{equation_4.77} - \eqref{equation_4.78},
 we observe that
 \begin{align}
 \label{equation_4.79}
 |\hat a(\sigma)-\tfrac{\delta_1}{\|\nabla_y\mathcal Y\|_2}e^{e_0\sigma}|
 <
 C\delta_1^2e^{2e_0\sigma}
 \quad\
 \text{for }
 \sigma\in(0,\sigma_3).
 \end{align}
 Furthermore
 \eqref{equation_4.79} together with the second line of \eqref{equation_4.72} implies
 \begin{align}
 \nonumber
 |\hat a(\sigma)|
 &<
 (1+C\alpha)
 \tfrac{\delta_1}{\|\nabla_y\mathcal{Y}\|_2}
 e^{e_0\sigma} 
 \\
 \label{equation_4.80}
 &<
 \tfrac{5}{4}
 \tfrac{\delta_1}{\|\nabla_y\mathcal{Y}\|_2}
 e^{e_0\sigma} 
 <
 \tfrac{8}{3}
 \tfrac{\alpha}{\|\nabla_y\mathcal{Y}\|_2} 
 \quad\
 \text{for }
 \sigma\in(0,\sigma_3).
 \end{align}
%%%%%%%%%%%%%%%%%%%%%%%%%%%%%%%%%%%%%%%%%%%%%%%%%%%%%%%%%%%%%%%%
 We repeat the same computation as above
 to derive the bound of $\hat b_0(\sigma)$ (see \eqref{equation_4.65}  for the definition of $\hat b_0(\sigma)$).
 \begin{align*}
 &
 \hat b_0(\sigma)
 =
 (\hat u(\xi,\sigma)-{\sf Q}(\xi),
 \tfrac{\Psi_0(\xi)}{\sqrt{(\Lambda_\xi{\sf Q},\Psi_0)_2}})_{L_\xi^2(\R^n)}
 \\
 &=
 (
 {\sf Q}(\hat X)
 +
 a(t)\mathcal{Y}(\hat X)
 +
 \epsilon(\hat X,t)-{\sf Q}(\xi),
 \tfrac{\Psi_0(\xi)}{\sqrt{(\Lambda_\xi{\sf Q},\Psi_0)_2}})_{L_\xi^2(\R^n)}
 \\
 &\quad
 +
 \underbrace{
 (
 \hat u(\xi,\sigma)
 -
 \{
 {\sf Q}(\hat X)
 +
 a(t)\mathcal{Y}(\hat X)
 +
 \epsilon(\hat X,t)
 \},
 \tfrac{\Psi_0(\xi)}{\sqrt{(\Lambda_\xi{\sf Q},\Psi_0)_2}})_{L_\xi^2(\R^n)}
 }_{=E_0}
 \\
 &=
 \underbrace{
 ({\sf Q}(\hat X)-{\sf Q}(\xi),
 \tfrac{\Psi_0(\xi)}{\sqrt{(\Lambda_\xi{\sf Q},\Psi_0)_2}})_{L_\xi^2(\R^n)}
 }_{=E_1}
 \\
 &\quad
 +
 a(t)
 \underbrace{
 (\mathcal{Y}(\hat X)-\mathcal{Y}(\xi),
 \tfrac{\Psi_0(\xi)}{\sqrt{(\Lambda_\xi{\sf Q},\Psi_0)_2}})_{L_\xi^2(\R^n)}
 }_{=E_2}
 \\
 &\quad
 +
 \underbrace{
 (\epsilon(\hat X,t),
 \tfrac{\Psi_0(\xi)}{\sqrt{(\Lambda_\xi{\sf Q},\Psi_0)_2}})_{L_\xi^2(\R^n)}
 }_{=E_3}
 +
 E_0.
 \end{align*}
 As in \eqref{equation_4.71} - \eqref{equation_4.80},
 we can verify that
 \begin{align}
 \label{equation_4.81}
 |\hat b_0(\sigma)|
 <
 C\delta_1^2
 e^{2e_0\sigma}
 \quad\
 \text{for }
 \sigma\in(0,\sigma_3).
 \end{align}
 In the same way,
 we obtain
 \begin{align}
 \label{equation_4.82}
 \sup_{j\in\{1\cdots,n\}}
 |\hat b_j(\sigma)|
 <
 C\delta_1^2
 e^{2e_0\sigma}
 \quad\
 \text{for }
 \sigma\in(0,\sigma_3).
 \end{align}
%%%%%%%%%%%%%%%%%%%%%%%%%%%%%%%%%%%%%%%%%%%%%%%%%%%%%%%%%%%%%%%%
 Since
 both \eqref{equation_4.66} and \eqref{equation_4.70}
 express the same function $\hat u(\xi,\sigma)$,
 they must coincide.
 \begin{align*}
 &
 {\sf Q}(\xi)
 +
 \hat a(\sigma)\mathcal{Y}(\xi)
 +
 \tfrac{\hat b_0(\sigma)\Lambda_\xi{\sf Q}(\xi)}{\sqrt{(\Lambda_\xi{\sf Q},\Psi_0)_2}}
 +
 \sum_{j=1}^n
 \tfrac{\hat b_j(\sigma)\pa_{\xi_j}{\sf Q}(\xi)}{\sqrt{(\pa_{\xi_j}{\sf Q},\Psi_j)_2}}
 +
 \hat v(\xi,\sigma)
 \\
 &=
 \hat\lambda(t)^{-\frac{2}{p-1}}
 \{
 {\sf Q}(\tfrac{\xi-\hat z(t)}{\hat\lambda(t)})
 +
 a(t)\mathcal{Y}(\tfrac{\xi-\hat z(t)}{\hat\lambda(t)})
 +\epsilon(\tfrac{\xi-\hat z(t)}{\hat\lambda(t)},t)
 \}.
 \end{align*}
 From this expression,
 we can derive the bound of $\|\nabla_\xi\hat v(\xi,\sigma)\|_{L_\xi^2(\R^n)}$.
%%%%%%%%%%%%%%%%%%%%%%%%%%%%%%%%%%%%%%%%%%%%%%%%%%%%%%%%%%%%%%%%
 \begin{align*}
 &
 \|\nabla_\xi\hat v(\xi,\sigma)\|_{L_\xi^2(\R^n)}
 \\
 &<
 \|
 \nabla_\xi
 \{
 \hat\lambda(t)^{-\frac{2}{p-1}}
 {\sf Q}
 (\hat X)
 -
 {\sf Q}(\xi)
 \}
 \|_{L_\xi^2(\R^n)}
 \\
 &\quad
 +
 \|
 \nabla_\xi
 \{
 \hat\lambda(t)^{-\frac{2}{p-1}}
 a(t)\mathcal{Y}(\hat X)
 - 
 \hat a(\sigma)
 \mathcal{Y}(\xi)
 \}
 \|_{L_\xi^2(\R^n)}
 \\
 &\quad
 +
 \|
 \nabla_\xi\{\lambda(t)^{-\frac{2}{p-1}}\epsilon(\hat X,t)\}
 \|_{L_\xi^2(\R^n)}
 \\
 &\quad
 +
 \tfrac{|\hat b_0(\sigma)|}{\sqrt{(\Lambda_\xi{\sf Q},\Psi_0)_2}}
 \|
 \nabla_\xi\Lambda_\xi{\sf Q}(\xi)
 \|_{L_\xi^2(\R^n)}
 \\
 &\quad
 +
 \sum_{j=1}^n
 \tfrac{|\hat b_j(\sigma)|}{\sqrt{(\pa_{\xi_j}{\sf Q},\Psi_j)_2}}
 \|
 \nabla_\xi\pa_{\xi_j}{\sf Q}(\xi)
 \|_{L_\xi^2(\R^n)}
 \\
 &<
 |\hat\lambda(t)^{-\frac{2}{p-1}}-1|
 \cdot
 \|
 \nabla_\xi
 {\sf Q}
 (\tfrac{\xi-\hat z(t)}{\hat\lambda(t)})
 \|_{L_\xi^2(\R^n)}
 \\
 &\quad
 +
 \|
 \nabla_\xi
 \{
 {\sf Q}
 (\hat X)
 -
 {\sf Q}(\xi)
 \}
 \|_{L_\xi^2(\R^n)}
 \\
 &\quad
 +
 |\hat\lambda(t)^{-\frac{2}{p-1}}-1|
 |a(t)|
 \cdot
 \|
 \nabla_\xi
 \mathcal{Y}(\hat X)
 \|_{L_\xi^2(\R^n)}
 \\
 &\quad
 +
 |a(t)-\hat a(\sigma)|
 \cdot
 \|
 \nabla_\xi
 \mathcal{Y}(\hat X)
 \|_{L_\xi^2(\R^n)}
 \\
 &\quad
 +
 |\hat a(\sigma)|
 \cdot
 \|
 \nabla_\xi
 \{
 \mathcal{Y}(\hat X)
 - 
 \mathcal{Y}(\xi)
 \}
 \|_{L_\xi^2(\R^n)}
 \\
 &\quad
 +
 \lambda(t)^{-\frac{2}{p-1}}
 \|
 \nabla_\xi\epsilon(\hat X,t)
 \|_{L_\xi^2(\R^n)}
 +
 \tfrac{|\hat b_0(\sigma)|\cdot\|\nabla_\xi\Lambda_\xi{\sf Q}\|_2}
 {\sqrt{(\Lambda_\xi{\sf Q},\Psi_0)_2}}
 \\
 &\quad
 +
 \sum_{j=1}^n
 \tfrac{|\hat b_j(\sigma)\cdot|\|\nabla_\xi\pa_{\xi_j}{\sf Q}\|_2}
 {\sqrt{(\pa_{\xi_j}{\sf Q},\Psi_j)_2}}.
 \end{align*}
 Therefore
 from \eqref{equation_4.72}, \eqref{equation_4.77} and \eqref{equation_4.80} - \eqref{equation_4.82},
 we deduce that
 \begin{align}
 \label{equation_4.83}
 \|\nabla_\xi\hat v(\xi,\sigma)\|_{L_\xi^2(\R^n)}
 <
 C\delta_1^2
 e^{2e_0\sigma}
 \quad\
 \text{for }
 \sigma\in(0,\sigma_3).
 \end{align}
 For a later argument,
 we collect
 \eqref{equation_4.79},
 \eqref{equation_4.81} - \eqref{equation_4.83}
 and the second line of \eqref{equation_4.72} here.
 \begin{align}
 \label{equation_4.84}
 \begin{cases}
 |\hat a(\sigma)-\tfrac{\delta_1}{\|\nabla_y\mathcal Y\|_2}e^{e_0\sigma}|
 <
 C\delta_1^2e^{2e_0\sigma}
 &
 \text{for }
 \sigma\in(0,\sigma_3),
 \\[2mm]
 \dis
 \sup_{j\in\{0,\cdots,n\}}
 |\hat b_j(\sigma)|
  <
 C\delta_1^2e^{2e_0\sigma}
 &
 \text{for }
 \sigma\in(0,\sigma_3),
 \\
 \|\nabla_\xi\hat v(\xi,\sigma)\|_{L_\xi^2(\R^n)}
 <
 C\delta_1^2
 e^{2e_0\sigma}
 &
 \text{for }
 \sigma\in(0,\sigma_3),
 \\
 \delta_1e^{e_0\sigma}
 <
 \frac{32}{15}
 \alpha
 &
 \text{for }
 \sigma\in(0,\sigma_3).
 \end{cases}
 \end{align}
 As a consequence of \eqref{equation_4.84},
 we observe from \eqref{equation_4.66} that
 \begin{align}
 \nonumber
 \|\nabla_\xi\hat u(\xi,\sigma)-\nabla_\xi{\sf Q}(\xi)\|_2
 &<
 C(|\hat a|+\sum_{j=0}^n|\hat b_j|+\|\nabla_\xi\hat v(\xi,\sigma)\|_{L_\xi^2(\R^n)})
 \\
 \label{equation_4.85}
 &<
 C\alpha
 \quad\
 \text{for }
 \sigma\in(0,\sigma_3).
 \end{align}
%%%%%%%%%%%%%%%%%%%%%%%%%%%%%%%%%%%%%%%%%%%%%%%%%%%%%%%%%%%%%%%%
 Let $(\Delta{\sf t}_1)$, ${\sf h}_1$ and ${\sf M}(t)$
 be given in Proposition {\rm\ref{PROPOSITION_2.3}},
 and define
 \begin{align}
 \label{equation_4.86}
 {\sf M}_1(t)
 =
 \begin{cases}
 {\sf M}(t) & \text{\rm if } t\in(0,(\Delta{\sf t})_1),\\
 {\sf M}((\Delta{\sf t})_1) & \text{\rm if } t>(\Delta{\sf t})_1.
 \end{cases}
 \end{align}
 Proposition {\rm\ref{PROPOSITION_2.3}} implies that
 if $\alpha$ is sufficiently small such that $C\alpha<\frac{{\sf h}_1}{2}$ in \eqref{equation_4.85},
 then we have
 \begin{align*}
 \|\hat u(\xi,\sigma)\|_{L_\xi^\infty(\R^n)}
 <
 {\sf M}_1(\sigma)
 \quad\
 \text{for }
 \sigma\in(0,\sigma_3).
 \end{align*}
 From this relation, \eqref{equation_4.66} and
 the bounds for $\hat a(\sigma)$, $\hat b_j(\sigma)$ ($j=0,\cdots,n$) (see \eqref{equation_4.80} - \eqref{equation_4.82}),
 there exists $C_1>0$ depending only on $n$ such that
 \begin{align}
 \label{equation_4.87}
 \|\hat v(\xi,\sigma)\|_{L_\xi^\infty(\R^n)}
 <
 {\sf M}_1(\sigma)
 +
 C_1
 \quad\
 \text{for }
 \sigma\in(0,\sigma_3).
 \end{align}
%%%%%%%%%%%%%%%%%%%%%%%%%%%%%%%%%%%%%%%%%%%%%%%%%%%%%%%%%%%%%%%%
 Let ${\sf Q}^+(\xi,\sigma)$ and $\epsilon_1>0$ be given in Lemma \ref{LEMMA_2.9}.
 Without loss of generality,
 we can assume that
 \begin{align}
 \label{equation_4.88}
 \alpha<\tfrac{\|\nabla_y\mathcal Y\|_2}{4}\epsilon_1.
 \end{align}
 For any $\delta_1$ such that $0<\delta_1<\alpha<\frac{\|\nabla_y\mathcal Y\|_2}{4}\epsilon_1$,
 we can choose $t_1<0$ such that
 \begin{align}
 \label{equation_4.89}
 \epsilon_1e^{e_0t_1}
 =
 \tfrac{\delta_1}{\|\nabla_y\mathcal{Y}\|_2},
 \end{align}
 and define
 \begin{align}
 \label{equation_4.90}
 U(\xi,\sigma;\delta_1)
 =
 {\sf Q}^+(\xi,t_1+\sigma).
 \end{align}
 From Lemma \ref{LEMMA_2.9},
 ${\sf Q}^+(x,t_1+\sigma)$ is expressed as
 \begin{align}
 \label{equation_4.91}
 U(\xi,\sigma;\delta_1)
 =
 {\sf Q}
 (x)
 +
 \tfrac{\delta_1e^{e_0\sigma}}{\|\nabla_y\mathcal Y\|_2}
 \mathcal{Y}(\xi)
 +
 \underbrace{
 v(\xi,t_1+\sigma)
 }_{=h(\xi,\sigma;\delta_1)},
 \end{align}
 where $h(\xi,\sigma,\delta_1)$ is a remainder term satisfying
 \begin{align}
 \nonumber
 &
 \|h(\sigma;\delta_1)\|_2
 +
 \|h(\sigma;\delta_1)\|_\infty
 +
 \|\nabla_\xi h(\sigma;\delta_1)\|_2
 <
 C\epsilon_1^2
 e^{2e_0(t_1+\sigma)}
 \\
 \label{equation_4.92}
 &<
 \tfrac{C\delta_1^2}{\|\nabla_y\mathcal Y\|_2^2}
 e^{2e_0\sigma}
 \quad\
 \text{for }
 \sigma\in(0,-t_1).
 \end{align}
 The constant $C>0$ in \eqref{equation_4.92} depends only on $n$.
 From the definition of $\sigma_3$ (see \eqref{equation_4.64}) and \eqref{equation_4.88} - \eqref{equation_4.89},
 we have
 \begin{align}
 \nonumber
 \sigma_3
 &<
 \tfrac{1}{e_0}
 (\log\tfrac{\alpha}{\delta_1}+\log 2)
 =
 \tfrac{1}{e_0}
 (\log\tfrac{\alpha}{\|\nabla_y\mathcal Y\|_2\epsilon_1e^{e_0t_1}}+\log 2)
 \\
 \nonumber
 &<
 \tfrac{1}{e_0}
 (\log\tfrac{\|\nabla_y\mathcal Y\|_2\epsilon_1}{4\|\nabla_y\mathcal Y\|_2\epsilon_1e^{e_0t_1}}+\log 2)
 \\
 \label{equation_4.93}
 &=
 \tfrac{1}{e_0}(-e_0t_1-\log 4+\log 2)
 <
 -t_1.
 \end{align}
%%%%%%%%%%%%%%%%%%%%%%%%%%%%%%%%%%%%%%%%%%%%%%%%%%%%%%%%%%%%%%%%
 We decompose $h(\xi,\sigma;\delta_1)$ as
 \begin{align}
 \label{equation_4.94}
 h(\xi,\sigma;\delta_1)
 &=
 A(\sigma;\delta_1)
 \mathcal{Y}(\xi)
 +
 \tfrac{B_0(\sigma;\delta_1)\Lambda_\xi{\sf Q}(\xi)}{\sqrt{(\Lambda_\xi{\sf Q},\Psi_0)_2}}
 +
 \sum_{j=1}^n
 \tfrac{B_j(\sigma;\delta_1)\pa_{\xi_j}{\sf Q}(\xi)}{\sqrt{(\pa_{\xi_j}{\sf Q},\Psi_j)_2}}
 \\
 \nonumber
 &\quad
 +
 h^\perp(\xi,\sigma;\delta_1)
 \end{align}
 with
 \begin{align}
 \label{equation_4.95}
 \begin{cases}
 A(\sigma;\delta_1)
 =
 (h(\xi,\sigma;\delta_1),\mathcal Y(\xi)
 )_{L_\xi^2(\R^n)},
 \\
 B_0(\sigma;\delta_1)
 =
 (h(\xi,\sigma;\delta_1),
 \tfrac{\Psi_0(\xi)}{\sqrt{(\Lambda_\xi{\sf Q},\Psi_0)_2}}
 )_{L_\xi^2(\R^n)},
 \\
 B_j(\sigma;\delta_1)
 =
 (h(\xi,\sigma;\delta_1),
 \tfrac{\Psi_j(\xi)}{\sqrt{(\pa_{\xi_j}{\sf Q},\Psi_j)_2}}
 )_{L_\xi^2(\R^n)}
 \quad
 \text{for }
 j\in\{1,\cdots,n\}.
 \end{cases}
 \end{align}
 From \eqref{equation_4.92} - \eqref{equation_4.93},
 there exists a constant $C>0$ depending only on $n$ such that
 \begin{align}
 \label{equation_4.96}
 \begin{cases}
 |A(\sigma;\delta_1)|
 <
 C
 \delta_1^2
 e^{2e_0\sigma}
 &
 \text{\rm for }
 \sigma\in(0,\sigma_3),
 \\
 \dis
 \sup_{j\in\{0,\cdots,n\}}
 |B_j(\sigma;\delta_1)|
 <
 C
 \delta_1^2
 e^{2e_0\sigma}
 &
 \text{\rm for }
 \sigma\in(0,\sigma_3),
 \\
 \|h^\perp(\xi,\sigma;\delta_1)\|_2
 <
 C
 \delta_1^2
 e^{2e_0\sigma}
 &
 \text{\rm for }
 \sigma\in(0,\sigma_3),
 \\
 \|h^\perp(\xi,\sigma;\delta_1)\|_\infty
 <
 C
 \delta_1^2
 e^{2e_0\sigma}
 &
 \text{\rm for }
 \sigma\in(0,\sigma_3),
 \\
 \|\nabla_\xi h^\perp(\xi,\sigma;\delta_1)\|_2
 <
 C
 \delta_1^2
 e^{2e_0\sigma}
 &
 \text{\rm for }
 \sigma\in(0,\sigma_3).
 \end{cases}
 \end{align}
%%%%%%%%%%%%%%%%%%%%%%%%%%%%%%%%%%%%%%%%%%%%%%%%%%%%%%%%%%%%%%%%
 We put
 \begin{align}
 \label{equation_4.97}
 J(\xi,\sigma;\delta_1)
 =
 \hat u(\xi,\sigma)-U(\xi,\sigma;\delta_1). 
 \end{align}
 It satisfies
 \begin{align}
 \label{equation_4.98}
 \pa_\sigma
 J
 &=
 H_\xi J
 +
 \mathcal{N}^{(J)},
 \end{align}
 where $\mathcal{N}^{(J)}$ is given by
 \begin{align}
 \label{equation_4.99}
 \mathcal{N}^{(J)}
 =
 f(\hat u)
 -
 f(U)
 -
 f'({\sf Q})
 (\hat u-U).
 \end{align}
 We decompose $J(\xi,\sigma;\delta_1)$ as
 \begin{align}
 \label{equation_4.100}
 J(\xi,\sigma;\delta_1)
 &=
 a^{(J)}(\sigma;\delta_1)
 \mathcal Y(\xi)
 +
 \tfrac{b_0^{(J)}(\sigma;\delta_1)\Lambda_\xi{\sf Q}(\xi)}{\sqrt{(\Lambda_\xi{\sf Q},\Psi_0)_2}}
 \\
 \nonumber
 &\quad
 +
 \sum_{j=1}^n
 \tfrac{b_j^{(J)}(\sigma;\delta_1)\pa_{\xi_j}{\sf Q}(\xi)}{\sqrt{(\pa_{\xi_j}{\sf Q},\Psi_j)_2}}
 +
 J^\perp(\xi,\sigma;\delta_1)
 \end{align}
 with
 \begin{align}
 \label{equation_4.101}
 \begin{cases}
 a^{(J)}(\sigma;\delta_1)
 =
 (
 J(\xi,\sigma;\delta_1),\mathcal Y(\xi)
 )_{L_\xi^2(\R^n)},
 \\
 b_0^{(J)}(\sigma;\delta_1)
 =
 (
 J(\xi,\sigma;\delta_1),\tfrac{\Psi_0(\xi)}{\sqrt{(\Lambda_\xi{\sf Q},\Psi_0)_2}}
 )_{L_\xi^2(\R^n)},
 \\
 b_j^{(J)}(\sigma;\delta_1)
 =
 (
 J(\xi,\sigma;\delta_1),\tfrac{\Psi_j(\xi)}{\sqrt{(\pa_{\xi_j}{\sf Q},\Psi_j)_2}}
 )_{L_\xi^2(\R^n)}
 \quad\
 (j=1,\cdots,n).
 \end{cases}
 \end{align}
 Substituting \eqref{equation_4.100} into \eqref{equation_4.98},
 we see that $J^\perp(\xi,\sigma)$ satisfies
 \begin{align}
 \label{equation_4.102}
 \pa_\sigma J^\perp
 &=
 H_\xi J^\perp
 -
 (\tfrac{da^{(J)}}{d\sigma}-e_0a^{(J)})
 \mathcal Y
 -
 \tfrac{db_0^{(J)}}{d\sigma}
  \tfrac{\Lambda_\xi{\sf Q}}{\sqrt{(\Lambda_\xi{\sf Q},\Psi_0)_2}}
 \\
 \nonumber
 &\quad
 -
 \sum_{j=1}^n
 \tfrac{db_j^{(J)}}{d\sigma}
 \tfrac{\pa_{\xi_j}{\sf Q}}
 {\sqrt{(\pa_{\xi_j}{\sf Q},\Psi_j)_2}}
 +
 \mathcal{N}^{(J)}.
 \end{align}
%%%%%%%%%%%%%%%%%%%%%%%%%%%%%%%%%%%%%%%%%%%%%%%%%%%%%%%%%%%%%%%%
 Since $J(\xi,\sigma;\delta_1)$ is given by \eqref{equation_4.97},
 combining \eqref{equation_4.65}, \eqref{equation_4.91}, \eqref{equation_4.94} - \eqref{equation_4.95},
 we easily see that
 \begin{align}
 \nonumber
 a^{(J)}(\sigma;\delta_1)
 &=
 (
 J(\sigma;\delta_1),\mathcal Y
 )_2
 \\
 \nonumber
 &=
 (
 \hat u(\sigma)-{\sf Q},\mathcal Y
 )_2
 -
 (
 U(\sigma;\delta_1)-{\sf Q},\mathcal Y
 )_2
 \\
 \label{equation_4.103}
 &=
 \hat a(\sigma)
 -
 \tfrac{\delta_1}{\|\nabla_y\mathcal Y\|_2}
 e^{e_0\sigma}
 -
 A(\sigma;\delta_1).
 \end{align}
 In the same reason as in \eqref{equation_4.103},
 from \eqref{equation_4.65} - \eqref{equation_4.66}, \eqref{equation_4.91},
 \eqref{equation_4.94} - \eqref{equation_4.95},
 and \eqref{equation_4.100} - \eqref{equation_4.101},
 we can verify that
%%%%%%%%%%%%%%%%%%%%%%%%%%%%%%%%%%%%%%%%%%%%%%%%%%%%%%%%%%%%%%%%
 \begin{align}
 \label{equation_4.104}
 \begin{cases}
 b_j^{(J)}(\sigma;\delta_1)
 =
 \hat b_j(\sigma)
 -
 B_j(\sigma;\delta_1)
 \quad\
 (j=0,\cdots,n),
 \\
 J^\perp(\xi,\sigma;\delta_1)
 =
 \hat v(\xi,\sigma)
 -
 h^\perp(\xi,\sigma;\delta_1).
 \end{cases}
 \end{align}
 From \eqref{equation_4.103} - \eqref{equation_4.104}, \eqref{equation_4.84} and \eqref{equation_4.96},
 there exists a constant $C>0$ depending only on $n$ such that
 \begin{align}
 \label{equation_4.105}
 |a^{(J)}(\sigma;\delta_1)|
 +
 \sum_{j=0}^n
 |b_j^{(J)}(\sigma;\delta_1)|
 +
 \|\nabla_\xi J^\perp(\sigma;\delta_1)\|_2
 <
 C\delta_1^2
 \quad\
 \text{for }
 \sigma=0.
 \end{align}
 For simplicity,
 we put
 \begin{align}
 \label{equation_4.106}
 \Theta^{(J)}(\xi,\sigma;\delta_1)
 =
 a^{(J)}(\sigma;\delta_1)
 \mathcal Y(\xi)
 +
 \tfrac{b_0^{(J)}(\sigma;\delta_1)\Lambda_\xi{\sf Q}(\xi)}{\sqrt{(\Lambda_\xi{\sf Q},\Psi_0)_2}}
 +
 \sum_{j=1}^n
 \tfrac{b_j^{(J)}(\sigma;\delta_1)\pa_{\xi_j}{\sf Q}(\xi)}{\sqrt{(\pa_{\xi_j}{\sf Q},\Psi_j)_2}}.
 \end{align}
%%%%%%%%%%%%%%%%%%%%%%%%%%%%%%%%%%%%%%%%%%%%%%%%%%%%%%%%%%%%%%%%
 To establish new bounds for $a^{(J)}(\sigma;\delta_1)$, $b_j^{(J)}(\sigma;\delta_1)$ ($j=0,\cdots,n$)
 and $J^\perp(\xi,\sigma;\delta_1)$,
 we first arrange $\mathcal N^{(J)}$ defined in \eqref{equation_4.99} as
 \begin{align}
 \nonumber
 \mathcal{N}^{(J)}
 &=
 f(U+J)
 -
 f(U)
 -
 f'({\sf Q})J
 \\
 \nonumber
 &=
 \underbrace{
 f(U+\Theta^{(J)}+J^\perp)
 -
 f(U+\Theta^{(J)})
 -
 f'(U+\Theta^{(J)})
 J^\perp
 }_{=V_1J^\perp}
 \\
 \nonumber
 &\quad
 +
 \underbrace{
 f(U+\Theta^{(J)})
 -
 f(U)
 -
 f'(U)\Theta^{(J)}
 }_{=G_1}
 \\
 \nonumber
 &\quad
 +
 f'(U+\Theta^{(J)})
 J^\perp
 +
 f'(U)\Theta^{(J)}
 -
 f'({\sf Q})
 J
 \\
 \nonumber
 &=V_1J^\perp+G_1
 +
 \underbrace{
 \{
 f'(U+\Theta^{(J)})
 -
 f'({\sf Q})
 \}
 J^\perp
 }_{=V_2J^\perp}
 \\
 \nonumber
 &\quad
 +
 \underbrace{
 \{f'(U)-f'({\sf Q})\}\Theta^{(J)}
 }_{=G_2}
 \\
 \label{equation_4.107}
 &=
 (V_1+V_2)J^\perp
 +
 G_1+G_2.
 \end{align}
%%%%%%%%%%%%%%%%%%%%%%%%%%%%%%%%%%%%%%%%%%%%%%%%%%%%%%%%%%%%%%%%
 We can verify that
 there exists a constant $C>0$ depending only on $n$ such that
 when $n=6$,
 we have
 \begin{align}
 \label{equation_4.108}
 \begin{cases}
 |V_1|
 <
 C|J^\perp|,
 \\
 |V_2|
 <
 C
 (|U-{\sf Q}|+|\Theta^{(J)}|),
 \\
 |G_1|
 <
 C(\Theta^{(J)})^2,
 \\
 |G_2|
 <
 C|U-{\sf Q}||\Theta^{(J)}|.
 \end{cases}
 \end{align}
%%%%%%%%%%%%%%%%%%%%%%%%%%%%%%%%%%%%%%%%%%%%%%%%%%%%%%%%%%%%%%%%
 We rewrite \eqref{equation_4.102} as
 \begin{align}
 \label{equation_4.109}
 \pa_\sigma J^\perp
 &=
 H_\xi J^\perp
 +
 (V_1+V_2)
 J^\perp
 +
 G_1
 +
 G_2
 -
 (\tfrac{da^{(J)}}{d\sigma}-e_0a^{(J)})
 \mathcal Y
 \\
 \nonumber
 &\quad
 -
 \tfrac{db_0^{(J)}}{d\sigma}
  \tfrac{\Lambda_\xi{\sf Q}}{\sqrt{(\Lambda_\xi{\sf Q},\Psi_0)_2}}
 -
 \sum_{j=1}^n
 \tfrac{db_j^{(J)}}{d\sigma}
 \tfrac{\pa_{\xi_j}{\sf Q}}
 {\sqrt{(\pa_{\xi_j}{\sf Q},\Psi_j)_2}}.
 \end{align}
%%%%%%%%%%%%%%%%%%%%%%%%%%%%%%%%%%%%%%%%%%%%%%%%%%%%%%%%%%%%%%%%
 We now claim that
 \begin{align}
 \nonumber
 |a^{(J)}(\sigma;\delta_1)|
 &+
 \sup_{j\in\{0,\cdots,n\}}|b_j^{(J)}(\sigma;\delta_1)|
 +
 \|\nabla_\xi J^\perp(\xi,\sigma;\delta_1)\|_{L_\xi^2(\R^n)}
 \\
 \label{equation_4.110}
 &<
 \delta_1^\frac{3}{2}
 e^{e_0\sigma}
 \quad\
 \text{for }
 \sigma\in(0,\sigma_3).
 \end{align}
%%%%%%%%%%%%%%%%%%%%%%%%%%%%%%%%%%%%%%%%%%%%%%%%%%%%%%%%%%%%%%%%
 To show \eqref{equation_4.110},
 we suppose that
 there exists $\sigma_4\in(0,\sigma_3)$ such that
 \begin{align}
 \nonumber
 |a^{(J)}(\sigma;\delta_1)|
 &+
 \sup_{j\in\{0,\cdots,n\}}|b_j^{(J)}(\sigma;\delta_1)|
 +
 \|\nabla_\xi J^\perp(\xi,\sigma;\delta_1)\|_{L_\xi^2(\R^n)}
 \\
 \label{equation_4.111}
 &<
 \delta_1^\frac{3}{2}
 e^{e_0\sigma}
 \quad\
 \text{for }
 \sigma\in(0,\sigma_4).
 \end{align}
%%%%%%%%%%%%%%%%%%%%%%%%%%%%%%%%%%%%%%%%%%%%%%%%%%%%%%%%%%%%%%%%
 From \eqref{equation_4.108} - \eqref{equation_4.109},
 we see that
 \begin{align}
 \nonumber
 &
 |\tfrac{da^{(J)}}{d\sigma}-e_0a^{(J)}|
 =
 |((V_1+V_2)J^\perp,\mathcal Y)_2|
 +
 |(G_1+G_2,\mathcal Y)_2|
 \\
 \nonumber
 &<
 C
 \{
 ((J^\perp)^2,\mathcal Y)_2
 +
 (|U-{\sf Q}||J^\perp|,\mathcal Y)_2
 +
 (|\Theta^{(J)}||J^\perp|,\mathcal Y)_2
 \}
 \\
 \nonumber
 &\quad
 +
 C
 \{
 \|\Theta^{(J)}\|_2^2
 +
 \|U-{\sf Q}\|_2
 \|\Theta^{(J)}\|_2
 \}
 \\
 \label{equation_4.112}
 &<
 C
 \{
 \|\nabla_\xi J^\perp\|_2^2
 +
 \|U-{\sf Q}\|_2
 \|\nabla_\xi J^\perp\|_2
 +
 \|\Theta^{(J)}\|_2
 \|\nabla_\xi J^\perp\|_2
 \}
 \\
 \nonumber
 &\quad
 +
 C
 \{
 \|\Theta^{(J)}\|_2^2
 +
 \|U-{\sf Q}\|_2
 \|\Theta^{(J)}\|_2
 \}.
 \end{align}
 Since $U(\xi,\sigma;\delta_1)$ is expressed as \eqref{equation_4.91},
 from \eqref{equation_4.92},
 there exists a constant $C>0$ depending only on $n$ such that
 \begin{align}
 \nonumber
 &
 \|U(\xi,\sigma;\delta_1)-{\sf Q}(\xi)\|_2
 +
 \|U(\xi,\sigma;\delta_1)-{\sf Q}(\xi)\|_\infty
 \\
 \nonumber
 &<
 \tfrac{(1+\|\mathcal Y\|_\infty)\delta_1}{\|\nabla_y\mathcal Y\|_2}
 e^{e_0\sigma}
 +
 \tfrac{C\delta_1^2}{\|\nabla_y\mathcal Y\|_2^2}
 e^{2e_0\sigma}
 \\
 \label{equation_4.113}
 &<
 \tfrac{1+\|\mathcal Y\|_\infty+C\alpha}{\|\nabla_y\mathcal Y\|_2}
 \delta_1
 e^{e_0\sigma}
 \quad\
 \text{for }
 \sigma\in(0,\sigma_3).
 \end{align}
 Furthermore
 from the definition of $\Theta^{(J)}(\xi,\sigma)$ (see \eqref{equation_4.106})
 and \eqref{equation_4.111},
 we observe that
 \begin{align}
 \label{equation_4.114}
 |\Theta^{(J)}(\xi,\sigma)|
 <
 C\delta_1^\frac{3}{2}
 e^{e_0\sigma}
 (1+|\xi|^2)^{-\frac{n-2}{2}}
 \quad\
 \text{for }
 \sigma\in(0,\sigma_4).
 \end{align}
 Therefore
 substituting \eqref{equation_4.111} and \eqref{equation_4.113} - \eqref{equation_4.114} into \eqref{equation_4.112},
 we get
 \begin{align}
 \nonumber
 |\tfrac{da^{(J)}}{d\sigma}-e_0a^{(J)}|
 &<
 C
 (
 \delta_1^3
 e^{2e_0\sigma}
 +
 \delta_1^\frac{5}{2}
 e^{2e_0\sigma}
 )
 \\
 \label{equation_4.115}
 &<
 C
 \delta_1^\frac{5}{2}
 e^{2e_0\sigma}
 \quad\
 \text{for }
 \sigma\in(0,\sigma_4).
 \end{align}
%%%%%%%%%%%%%%%%%%%%%%%%%%%%%%%%%%%%%%%%%%%%%%%%%%%%%%%%%%%%%%%%
 Integrating \eqref{equation_4.115},
 and applying \eqref{equation_4.105}
 together with the last line of \eqref{equation_4.84},
 we obtain
 \begin{align}
 \nonumber
 |a^{(J)}(\sigma)|
 &<
 e^{e_0\sigma}
 |a^{(J)}(\sigma)|_{\sigma=0}|
 +
 C
 \delta_1^\frac{5}{2}
 e^{2e_0\sigma}
 \\
 \nonumber
 &<
 C\delta_1^2
 e^{e_0\sigma}
 +
 C
 \delta_1^\frac{5}{2}
 e^{2e_0\sigma}
 \\
 \label{equation_4.116}
 &<
 C
 (
 \delta_1^\frac{1}{2}
 +
 \alpha
 )
 \delta_1^\frac{3}{2}
 e^{e_0\sigma}
 \quad\
 \text{for }
 \sigma\in(0,\sigma_4).
 \end{align}
 We next derive the bound of $(J^\perp,-H_\xi J^\perp)$
 by using \eqref{equation_4.108} - \eqref{equation_4.109}.
 \begin{align*}
 \nonumber
 &
 \tfrac{1}{2}
 \tfrac{d}{d\sigma}
 (J^\perp,-H_\xi J^\perp)_2
 =
 -
 \|H_\xi J^\perp\|_2^2
 +
 ((V_1+V_2)J^\perp,H_\xi J^\perp)_2
 \\
 \nonumber
 &\quad
 +
 (G_1+G_2,H_\xi J^\perp)_2
 \\
 \nonumber
 &<
 -
 \tfrac{1}{2}
 \|H_\xi J^\perp\|_2^2
 +
 2
 \|(V_1+V_2)J^\perp\|_2^2
 +
 2
 \|G_1+G_2\|_2^2
 \\
 \nonumber
 &<
 -
 \tfrac{1}{2}
 \|H_\xi J^\perp\|_2^2
 +
 C
 \|J^\perp\|_4^4
 +
 C
 \|(U-{\sf Q})J^\perp\|_2^2
 +
 C
 \|\Theta^{(J)} J^\perp\|_2^2
 \\
 \nonumber
 &\quad
 +
 C\|\Theta^{(J)}\|_4^4
 +
 C\|(U-{\sf Q})\Theta^{(J)}\|_2^2
 \\
 \nonumber
 &<
 -
 \tfrac{1}{2}
 \|H_\xi J^\perp\|_2^2
 +
 C
 \underbrace{
 \|J^\perp\|_3^2
 \|J^\perp\|_6^2
 }_{=\|J^\perp\|_\frac{2n}{n-2}^2\|J^\perp\|_\frac{2n}{n-4}^2}
 +
 C
 \|U-{\sf Q}\|_\frac{n}{2}^2
 \|J^\perp\|_\frac{2n}{n-4}^2
 \\
 &\quad
 +
 C
 \|\Theta^{(J)}\|_\frac{2n}{n-2}^2
 \|J^\perp\|_\frac{2n}{n-4}^2
 +
 C\|\Theta^{(J)}\|_4^4
 +
 C\|U-{\sf Q}\|_\frac{n}{2}^2
 \|\Theta^{(J)}\|_\frac{2n}{n-4}^2.
 \end{align*}
 From the definition of $J^\perp$ (see \eqref{equation_4.100} - \eqref{equation_4.101}),
 we note that
 $(J^\perp,\mathcal Y)_2=0$
 and
 $(J^\perp,\Psi_j)_2=0$ for $j\in\{0,\cdots,n\}$. 
 Hence
 from Lemma \ref{LEMMA_2.5},
 there exists a constant $C>0$ depending only on $n$ such that
 $\|J^\perp\|_\frac{2n}{n-4}<C\|H_\xi J^\perp\|_2$.
 Therefore
 from \eqref{equation_4.113} - \eqref{equation_4.114},
 there exists $\alpha^*>0$ depending only on $n$ such that
 if $\alpha\in(0,\alpha^*)$,
 then we have
 \begin{align}
 \nonumber
 &
 \tfrac{1}{2}
 \tfrac{d}{d\sigma}
 (J^\perp,-H_\xi J^\perp)_2
 \\
 \nonumber
 &<
 -\tfrac{1}{4}
 \|H_\xi J^\perp\|_2^2
 +
 C\|\Theta^{(J)}\|_4^4
 +
 C\|U-{\sf Q}\|_\frac{n}{2}^2
 \|\Theta^{(J)}\|_\frac{2n}{n-4}^2.
 \\
 \nonumber
 &<
 -\tfrac{1}{4}
 \|H_\xi J^\perp\|_2^2
 +
 C\delta_1^6
 e^{4e_0\sigma}
 +
 C
 \delta_1^5
 e^{4e_0\sigma}
 \\
 \label{equation_4.117}
 &<
 -\tfrac{1}{4}
 \|H_\xi J^\perp\|_2^2
 +
 C\delta_1^5
 e^{4e_0\sigma}
 \quad\
 \text{for }
 \sigma\in(0,\sigma_4).
 \end{align}
 Integrating this estimate with \eqref{equation_4.105},
 we obtain
 \begin{align}
 \nonumber
 &
 (J^\perp(\sigma),-H_\xi J^\perp(\sigma))_2
 \\
 \nonumber
 &<
 \underbrace{
 (J^\perp(\sigma),-H_\xi J^\perp(\sigma))_2|_{\sigma=0}
 }_{<C\delta_1^4}
 +
 C\delta_1^5
 e^{4e_0\sigma}
 \\
 \nonumber
 &=
 C(\delta_1e^{-2e_0\sigma}+\delta_1^2e^{2e_0\sigma})
 \delta_1^3
 e^{2e_0\sigma}
 \\
 \nonumber
 &<
 C(\delta_1+\alpha^2)
 \delta_1^3
 e^{2e_0\sigma}
 \\
 \label{equation_4.118}
 &<
 C\alpha
 \delta_1^3
 e^{2e_0\sigma}
 \quad\
 \text{for }
 \sigma\in(0,\sigma_4).
 \end{align}
 We finally provide estimates for $b_j^{(J)}(\sigma)$.
 Repeating the same argument as in \eqref{equation_4.112} - \eqref{equation_4.115},
 we see that
 \begin{align*}
 |\tfrac{db_0^{(J)}}{d\sigma}|
 &<
 |(H_\xi J^\perp,\tfrac{\Psi_0}{\sqrt{(\Lambda_\xi{\sf Q},\Psi_0)_2}})_2|
 +
 |((V_1+V_2)J^\perp,\tfrac{\Psi_0}{\sqrt{(\Lambda_\xi{\sf Q},\Psi_0)_2}})_2|
 \\
 &\quad
 +
 |(G_1+G_2,\tfrac{\Psi_0}{\sqrt{(\Lambda_\xi{\sf Q},\Psi_0)_2}})_2|
 \\
 &<
 |(H_\xi J^\perp,\tfrac{\Psi_0}{\sqrt{(\Lambda_\xi{\sf Q},\Psi_0)_2}})_2|
 +
 C
 \delta_1^\frac{5}{2}
 e^{2e_0\sigma}
 \quad\
 \text{for }
 \sigma\in(0,\sigma_4).
 \end{align*}
 Here
 we apply \eqref{equation_4.118} to obtain
 \begin{align}
 \nonumber
 |\tfrac{db_0^{(J)}}{d\sigma}|
 &<
 C
 \sqrt{\alpha}
 \delta_1^\frac{3}{2}
 e^{e_0\sigma}
 +
 C
 \delta_1^\frac{5}{2}
 e^{2e_0\sigma}
 \\
 \nonumber
 &<
 C
 (
 \sqrt{\alpha}
 +
 \alpha
 )
 \delta_1^\frac{3}{2}
 e^{e_0\sigma}
 \\
 \label{equation_4.119}
 &<
 C
 \sqrt{\alpha}
 \delta_1^\frac{3}{2}
 e^{e_0\sigma}
 \quad\
 \text{for }
 \sigma\in(0,\sigma_4).
 \end{align}
 Integrating this relation with \eqref{equation_4.105},
 we conclude
 \begin{align}
 \nonumber
 |b_0^{(J)}(\sigma)|
 &<
 \underbrace{
 |b_0^{(J)}(\sigma)|_{\sigma=0}|
 }_{C\delta_1^2}
 +
 C
 \sqrt{\alpha}
 \delta_1^\frac{3}{2}
 e^{e_0\sigma}
 \\
 \nonumber
 &<
 C
 (
 \sqrt{\delta_1}
 +
 \sqrt{\alpha}
 )
 \delta_1^\frac{3}{2}
 e^{e_0\sigma}
 \\
 \label{equation_4.120}
 &<
 C
 \sqrt{\alpha}
 \delta_1^\frac{3}{2}
 e^{e_0\sigma}
 \quad\
 \text{for }
 \sigma\in(0,\sigma_4).
 \end{align}
 In the same manner,
 we can derive
 \begin{align}
 \label{equation_4.121}
 \begin{cases}
 \dis
 \sup_{j\in\{1,\cdots,n\}}
 |\tfrac{db_j^{(J)}}{d\sigma}|
 <
 C
 \sqrt{\alpha}
 \delta_1^\frac{3}{2}
 e^{e_0\sigma}
 &
 \text{for }
 \sigma\in(0,\sigma_4),
 \\
 \dis
 \sup_{j\in\{1,\cdots,n\}}
 |b_j^{(J)}(\sigma)|
 <
 C
 \sqrt{\alpha}
 \delta_1^\frac{3}{2}
 e^{e_0\sigma}
 &
 \text{for }
 \sigma\in(0,\sigma_4).
 \end{cases}
 \end{align}
 Since
 all constants $C$ in \eqref{equation_4.116}, \eqref{equation_4.118} and \eqref{equation_4.120} - \eqref{equation_4.121}
 are independent of $\alpha$ and $\delta_1$,
 there exists $\alpha^*>0$ depending only on $n$ such that
 if $\alpha\in(0,\alpha^*)$,
 then we have
 \begin{align}
 \label{equation_4.122}
 \begin{cases}
 |a^{(J)}(\sigma;\delta_1)|
 <
 \frac{1}{4}\delta_1^\frac{3}{2}
 e^{e_0\sigma}
 & \text{for } \sigma\in(0,\sigma_4),
 \\
 \dis
 \sup_{j\in\{0,\cdots,n\}}|b_j^{(J)}(\sigma;\delta_1)|
 <
 \tfrac{1}{4}\delta_1^\frac{3}{2}
 e^{e_0\sigma}
 & \text{for } \sigma\in(0,\sigma_4),
 \\
 \|\nabla_\xi J^\perp(\sigma;\delta_1)\|_2
 <
 \tfrac{1}{4}\delta_1^\frac{3}{2}
 e^{e_0\sigma}
 & \text{for } \sigma\in(0,\sigma_4).
 \end{cases}
 \end{align}
 As a consequence of \eqref{equation_4.122},
 \eqref{equation_4.110} follows.
 We finally establish the $L^\infty$ bound of $J^\perp(\xi,\sigma)$.
 Let $(\Delta{\sf t})_1$ and ${\sf M}_1(t)$ be the same as in \eqref{equation_4.86},
 and assume that
 \begin{align}
 \label{equation_4.123}
 \sigma_3>2(\Delta{\sf t}_1).
 \end{align}
 Since $J^\perp(\xi,\sigma;\delta_1)$ is given by \eqref{equation_4.104},
 from the $L^\infty$ bound of $\hat v(\xi,\sigma)$ and $h(\xi,\sigma;\delta_1)$
 (see \eqref{equation_4.87} and \eqref{equation_4.96}),
 we have
 \begin{align*}
 &
 \|J^\perp(\xi,\sigma;\delta_1)\|_{L_\xi^\infty(\R^n)}
 <
 \|\hat v(\xi,\sigma)\|_\infty
 +
 \|h^\perp(\xi,\sigma;\delta_1)\|_\infty
 \\
 &<
 {\sf M}_1(\sigma)
 +
 C_1
 +
 C\delta_1^2
 e^{2e_0\sigma}
 \quad\
 \text{for }
 \sigma\in(0,\sigma_3).
 \end{align*}
 From the definition of ${\sf M}_1(\sigma)$ and \eqref{equation_4.123},
 there exists $M>0$ depending only on $n$ such that
 \begin{align*}
 \|J^\perp(\xi,\sigma;\delta_1)\|_{L_\xi^\infty(\R^n)}
 <
 M
 \quad\
 \text{for }
 \sigma\in(\tfrac{(\Delta{\sf t})_1}{2},\sigma_3).
 \end{align*}
 Therefore
 from the definition of $V_1$ and $V_2$ (see \eqref{equation_4.108}),
 there exists a constant $\bar M>0$ depending only on $n$ such that
 $|V_1(\xi,\sigma)|+|V_2(\xi,\sigma)|<\bar M$ for $\sigma\in(\frac{(\Delta{\sf t})_1}{2},\sigma_3)$.
 We now apply a standard local parabolic estimate to \eqref{equation_4.109}.
 There exists a constant $C>0$ depending only on $n$, $\bar M$ and $(\Delta{\sf t})_1$ such that
 for $\sigma\in((\Delta{\sf t})_1,\sigma_3)$
 \begin{align*}
 &
 \|J^\perp(\xi,\sigma;\delta_1)\|_{L_\xi^\infty(\R^n)}
 <
 C
 \sup_{\zeta\in\R^n}
 \left(
 \int_{\sigma-\frac{(\Delta{\sf t})_1}{2}}^\sigma
 \int_{|\xi-\zeta|<\sqrt{\frac{(\Delta{\sf t})_1}{2}}}
 |J^\perp(\xi,\sigma')|^2
 d\xi
 d\sigma'
 \right)^\frac{1}{2}
 \\
 &\quad
 +
 C
 \sup_{\sigma\in(\sigma-\frac{(\Delta{\sf t})_1}{2},\sigma)}
 (
 \|G_1(\sigma')\|_{L_\xi^\infty(\R^n)}
 +
 \|G_2(\sigma')\|_{L_\xi^\infty(\R^n)}
 )
 \\
 \nonumber
 &\quad
 +
 C
 \sup_{\sigma\in(\sigma-\frac{(\Delta{\sf t})_1}{2},\sigma)}
 \left(
 |\tfrac{da^{(J)}}{d\sigma}(\sigma')-e_0a^{(J)}(\sigma')|
 +
 \sum_{j=0}^n
 |\tfrac{db_j^{(J)}}{d\sigma}(\sigma')|
 \right).
 \end{align*}
 From \eqref{equation_4.108}, \eqref{equation_4.110} and \eqref{equation_4.113},
 we see that
 \begin{align*}
 \|G_1(\sigma)\|_{L_\xi^\infty(\R^n)}
 &<
 C\|\Theta^{(J)}(\sigma)\|_{L_\xi^\infty(\R^n)}^2
 \\
 &<
 C\delta_1^3
 e^{2e_0\sigma}
 \quad
 \text{for }
 \sigma\in(0,\sigma_3),
 \\
 \|G_2(\sigma)\|_{L_\xi^\infty(\R^n)}
 &<
 C
 \|U-{\sf Q}\|_\infty
 \|\Theta^{(J)}(\sigma)\|_{L_\xi^\infty(\R^n)}
 \\
 &<
 C
 \delta_1^\frac{5}{2}
 e^{2e_0\sigma}
 \quad
 \text{for }
 \sigma\in(0,\sigma_3).
 \end{align*}
 Note from the last line of \eqref{equation_4.84}
 that $\delta_1e^{e_0\sigma}<\frac{32}{15}\alpha$.
 Therefore
 combining \eqref{equation_4.110}, \eqref{equation_4.115}, \eqref{equation_4.119} and \eqref{equation_4.121},
 we conclude
 \begin{align*}
 \|J^\perp(\xi,\sigma;\delta_1)\|_{L_\xi^\infty(\R^n)}
 &<
 C
 \delta_1^\frac{3}{2}
 e^{e_0\sigma}
 \quad\
 \text{for }
 \sigma\in((\Delta{\sf t})_1,\sigma_3).
 \end{align*}
 Since $J(\xi,\sigma;\delta_1)$ is expressed as \eqref{equation_4.100},
 it follows that
 \begin{align}
 \nonumber
 &
 \|J(\xi,\sigma;\delta_1)\|_{L_\xi^\infty(\R^n)}
 \\
 \nonumber
 &<
 C\left(
 |a^{(J)}(\sigma)|+\sum_{j=0}^n|b_j^{(J)}(\sigma)|
 \right)
 +
 \|J^\perp(\xi,\sigma;\delta_1)\|_{L_\xi^\infty(\R^n)}
 \\
 \label{equation_4.124}
 &<
 C
 \delta_1^\frac{3}{2}
 e^{e_0\sigma}
 \quad\
 \text{for }
 \sigma\in((\Delta{\sf t})_1,\sigma_3).
 \end{align}
 From \eqref{equation_4.64},
 we recall that $\sigma_3$ is bounded by
 $\sigma_3>\frac{1}{e_0}(\log\frac{\alpha}{\delta_1}-\log 2)=\frac{1}{e_0}\log\frac{\alpha}{2\delta_1}$.
 Let
 \begin{align*}
 \hat\sigma_1(\alpha,\delta_1)
 =
 \tfrac{1}{e_0}\log\tfrac{\alpha}{2\delta_1}.
 \end{align*}
 Then we have
 \begin{align}
 \label{equation_4.125}
 \sigma_3
 >
 \hat\sigma_1
 =
 \tfrac{1}{e_0}\log\tfrac{\alpha}{2\delta_1}
 >
 2(\Delta{\sf t})_1
 \quad\
 \text{if }
 \delta_1<\tfrac{\alpha}{2}e^{-2e_0(\Delta{\sf t})_1}.
 \end{align}
 This implies that the assumption \eqref{equation_4.123} is satisfied if $\delta_1<\tfrac{\alpha}{2}e^{-2e_0(\Delta{\sf t})_1}$.
 As a consequence of \eqref{equation_4.124},
 we conclude 
 \begin{align}
 \label{equation_4.126}
 \|J(\xi,\hat\sigma_1;\delta_1)\|_{L_\xi^\infty(\R^n)}
 <
 C\delta_1^\frac{3}{2}
 e^{e_0\hat\sigma_1}
 <
 C\sqrt{\delta_1}
 \alpha.
 \end{align}
 From the definition of $U(\xi,\sigma;\delta_1)$ (see \eqref{equation_4.90}),
 we observe that
 \begin{align*}
 U(\xi,\hat\sigma_1;\delta_1)
 &=
 {\sf Q}^+(\xi,t_1+\hat\sigma_1)
 \quad\
 \text{with }
 \epsilon_1e^{e_0t_1}=\tfrac{\delta_1}{\|\nabla_y{\mathcal Y}\|_2}
 \\
 &=
 {\sf Q}^+(\xi,\hat\sigma_0)
 \quad\
 \text{with }
 \hat\sigma_0=t_1+\hat\sigma_1=\tfrac{1}{e_0}\log(\tfrac{\alpha}{2\epsilon_1\|\nabla_Y\mathcal Y\|_2}).
 \end{align*}
 Therefore from the definition of $J(\xi,\sigma;\delta_1)$ (see \eqref{equation_4.97}),
 we have
 \begin{align*}
 \hat u(\xi,\hat\sigma_1)
 =
 {\sf Q}^+(\xi,\hat\sigma_0)
 +
 J(\xi,\hat\sigma_1;\delta_1)
 \quad\
 \text{for } \xi\in\R^n.
 \end{align*}
 We here apply the stability result obtained in Theorem 2 of \cite{Harada_ODE}
 (see also Theorem 1.1 in \cite{Collot-Merle-Raphael_ODE} p.\,66).
 First  we fix $\alpha>0$ sufficiently small so that \eqref{equation_4.126} holds.
 We emphasize that $\hat\sigma_0$ depends on $\alpha$ but not on $\delta_1$.
 As stated in Remark \ref{REMARK_2.10},
 ${\sf Q}^+(\xi,\sigma)$ satisfies all the assumptions (a1) - (a6) in Theorem 2 of \cite{Harada_ODE}.
 Therefore
 there exists a constant $d_1=d_1(\alpha)>0$ independent of $\delta_1$ such that if $\|J(\xi,\hat\sigma_1;\delta_1)\|_\infty<d_1$,
 then
 $\hat u(\xi,\sigma)$ exhibits a finite time blow up of type I.
 From \eqref{equation_4.126},
 if $\delta_1$ satisfies $C\sqrt{\delta_1}\alpha<\frac{d_1}{2}$,
 then we have $\|J(\xi,\hat\sigma_1;\delta_1)\|_\infty<d_1$.
 We conclude that
 $\hat u(\xi,\sigma)$ exhibits finite time blowup of type I
 if $C\sqrt{\delta_1}\alpha<\frac{d_1}{2}$.
 In this case,
 $u(x,t)$ also exhibits finite time blowup of type I,
 which completes the proof of Proposition \ref{PROPOSITION_4.7}.
%%%%%%%%%%%%%%%%%%%%%%%%%%%%%%%%%%%%%%%%%%%%%%%%%%%%%%%%%%%%%%%%

%%%%%%%%%%%%%%%%%%%%%%%%%%%%%%%%%%%%%%%%%%%%%%%%%%%%%%%%%%%%%%%%
\section*{Acknowledgement}
The author is partly supported by JSPS KAKENHI Grant Number 23K03161.
%%%%%%%%%%%%%%%%%%%%%%%%%%%%%%%%%%%%%%%%%%%%%%%%%%%%%%%%%%%%%%%%

%%%%%%%%%%%%%%%%%%%%%%%%%%%%%%%%%%%%%%%%%%%%%%%%%%%%%%%%%%%%%%%%
\section*{Data availability}
Data sharing is not applicable to this article
as no datasets were generated or analyzed during the current study.
%%%%%%%%%%%%%%%%%%%%%%%%%%%%%%%%%%%%%%%%%%%%%%%%%%%%%%%%%%%%%%%%

%%%%%%%%%%%%%%%%%%%%%%%%%%%%%%%%%%%%%%%%%%%%%%%%%%%%%%%%%%%%%%%%
\section*{Conflict of interest}
The author declares that he has no conflict of interest.
%%%%%%%%%%%%%%%%%%%%%%%%%%%%%%%%%%%%%%%%%%%%%%%%%%%%%%%%%%%%%%%%

%%%%%%%%%%%%%%%%%%%%%%%%%%%%%%%%%%%%%%%%%%%%%%%%%%%%%%%%%%%%%%%%
 
%%%%%%%%%%%%%%%%%%%%%%%%%%%%%%%%%%%%%%%%%%%%%%%%%%%%%%%%%%%%%%%%
\end{document}